\documentclass[14pt]{siamart171218}

\usepackage{braket,amsfonts}
\usepackage{dutchcal}
\usepackage{array}
\usepackage{pgfplots}
\usepackage{overpic}
\pgfplotsset{compat=newest}
\usepackage[caption=false]{subfig}
\usepackage{mathtools}
\usepackage{eqparbox}
\usepackage{enumerate}
\usepackage{etoolbox}
\usepackage{xcolor}
\usepackage{multirow}
\usepackage{bbm}
\usepackage{hhline}
\usepackage{graphicx}
\usepackage{bbold}

\newsiamthm{claim}{Claim}
\newsiamthm{remark}{Remark}
\newsiamthm{example}{Example}
\newsiamremark{hypothesis}{Hypothesis}
\newsiamremark{problem}{Problem}
\crefname{hypothesis}{Hypothesis}{Hypotheses}
\AtEndEnvironment{problem}{\null\hfill$\Diamond$}
\AtEndEnvironment{example}{\null\hfill$\Diamond$}
\newsiamthm{assumption}{Assumption}

\newsiamthm{mainresult}{Main Result}
\newsiamthm{conjecture}{Conjecture}

\usepackage{algorithmic}

\usepackage{graphicx,epstopdf}
\usepackage{float}

\usepackage{amsopn}

\usepackage{verbatim}

\usepackage{MnSymbol}

\usepackage{dsfont}

\makeatletter
\def\contentsname{Contents}
\def\tableofcontents{%
    \section*{\MakeUppercase{\contentsname}}
    \@starttoc{toc}%
    }
\makeatother

\newcommand{\PDF}{probability density function\,}

\newcommand{\mL}{\mathcal L}
\newcommand{\mcl}{\mathcal}
\newcommand{\mbf}{\mathbf}
\newcommand{\mbb}{\mathbb}

\newcommand{\dist}{\text{dist}}

\newcommand{\eps}{\epsilon}
\newcommand{\bD}{\bar{\mcl Z}}
\newcommand{\R}{\mathbb{R}}

\newcommand{\bu}{\mbf u}
\newcommand{\bv}{\mbf v}

\newcommand{\vrhoo}{\varrho_0}

\newcommand{\Z}{\mcl Z}
\newcommand{\Dp}{\mcl{Z}'}
\renewcommand{\L}{{\mcl L}}

\newcommand{\Langle}{\left\langle}
  \newcommand{\Rangle}{\right\rangle}

\definecolor{darkred}{rgb}{.7,0,0}

\definecolor{darkgreen}{rgb}{0,0.5,0}

\definecolor{darkblue}{rgb}{0,0,0.7}

\definecolor{grey}{rgb}{0.5,0.5,0.9}

\newcommand{\fh}[1]{{#1}}

\newcommand{\as}[1]{{#1}}

\newcommand{\changed}[1]{{#1}}

\reversemarginpar
\usepackage[colorinlistoftodos,prependcaption,textsize=tiny]{todonotes}
 
\title{Spectral Analysis Of Weighted Laplacians Arising In Data Clustering
}

\author{Franca Hoffmann\and
Bamdad Hosseini\and
Assad A. Oberai
\thanks{Department of Aerospace and Mechanical Engineering, University of Southern California, Los Angeles, CA 90089, USA
(\email{aoberai@usc.edu})}
\and
Andrew M. Stuart
\thanks{Computing and Mathematical Sciences, Caltech, Pasadena, CA (\email{fkoh@caltech.edu },
\email{bamdadh@caltech.edu }, \email{astuart@caltech.edu }).}} 

\begin{document}
\date{today}
\maketitle

\begin{abstract}
Graph Laplacians computed from weighted adjacency matrices are 
widely used to identify geometric structure in data, and clusters 
in particular; their spectral properties play a central role in a 
number of unsupervised and semi-supervised learning algorithms. 
When suitably scaled, graph Laplacians approach limiting continuum
operators in the large data limit. Studying these limiting 
operators, therefore, sheds light on learning algorithms. This 
paper is devoted to the study of a parameterized family of 
divergence form elliptic operators that arise as the large 
data limit of graph Laplacians. The link between a three-parameter
family of graph Laplacians and a three-parameter family of
differential operators is explained. The spectral properties of these 
differential operators are analyzed in the situation where the data comprises
two nearly separated clusters, in a sense which is made precise. 
In particular, we investigate
how the spectral gap depends on the three parameters entering the
graph Laplacian, and on a parameter measuring the size of the perturbation 
from the perfectly clustered case.
Numerical results are presented which exemplify and extend the analysis: 
\changed{the computations study situations in which there are two
nearly separated clusters, but which violate the assumptions used in
our theory; situations in which more than two clusters are present,
also going beyond our theory; and situations which demonstrate the 
relevance of our studies of differential operators for the understanding 
of finite data problems via the graph Laplacian.}
The findings provide insight into parameter choices 
made in learning algorithms which are based on weighted adjacency matrices;
they also provide the basis for analysis of the consistency of
various unsupervised and semi-supervised learning algorithms,
in the large data limit.
\end{abstract}







\section{Introduction}
\label{sec:I}

\subsection{Overview}
\label{ssec:O}

This article presents a spectral analysis of 
differential operators of the form
\begin{equation}
  \label{general-weighted-Laplacian}
  \left\{
    \begin{aligned}
      &\mcl{L} u  := -\frac{1}{\varrho^p} {\rm div} \left( \varrho^q \nabla \left( \frac{u}{\varrho^r} \right) \right),  &&\text{ in } \Z,\\
      &\varrho^q \frac{\partial}{\partial n} \left( \frac{u}{\varrho^r} \right) = 0,
      &&\text{ on } \partial \Z,
  \end{aligned}
\right.
\end{equation}
for parameters $p,q,r\in\R$ fixed.
The analysis is focused on the situation where the density $\varrho$ 
concentrates on two disjoint connected sets (clusters), and
numerical results extend our conclusions to multiple clusters \changed{and
to more general two cluster data densities $\varrho$ 
not covered by our analysis.} 
Our motivation is to understand a range of algorithms which
learn about geometric information in data, and clusters in
particular, by means of graph Laplacians
constructed from adjacency matrices whose edge weights reflect affinities
between data points at each vertex. Operators of the
  form \eqref{general-weighted-Laplacian}
  arise as a large data limit of graph Laplacian operators of the form
\begin{equation}\label{defLN-intro}
L_N := 
\begin{cases}
D_N^{\frac{1-p}{q-1}}\left(D_N- W_N\right)D_N^{-\frac{r}{q-1}},
&\text{ if } q\neq 1\,,\\
D_N-W_N, &\text{ if } q= 1,
\end{cases}
\end{equation}
where the symmetric weighted adjacency \changed{matrix $W_N=W_N(q)$ is constructed 
via a suitably reweighted kernel capturing the similarities between 
discrete data points and $D_N=D_N(q)$ is} an associated weighted degree matrix \changed{(see Subsection~\ref{sec:discrete-setting}
 for precise definitions of these matrices)}.

\changed{The three primary contributions of this paper are as follows:

\begin{enumerate}

\item Under assumptions on $\varrho$ capturing
the notion of data approximately clustered into two sets, we study the low
lying spectrum of $\mcl{L}$,
the corresponding eigenfunctions and their dependence on $(p,q,r)$; 
these results reveal the special properties of the parametric family 
$q=p+r$ for clustering tasks, and we refer to $\mcl{L}$ and $L_N$
as  \emph{balanced} in this case.

\item We present numerical experiments which exemplify 
the analysis in both the continuum and discrete regimes,
leading to conjectures concerning aspects of our analysis
which are not sharp, and extending our understanding to 
mixture models and to multiple clusters, situations not 
covered by the analysis.

\item We explain how $\mcl L$ arises from $L_N$, and provide 
numerical simulations illustrating that the characteristic behavior 
identified for the limiting operators $\mcl L$ in point 1 also 
manifests in the finite data setting when using $L_N$. 

\end{enumerate}

These results may also be of independent interest  
in the spectral theory of elliptic differential operators.
Subsection \ref{ssec:BLR} is devoted to the background 
to our work, and a literature review. In Subsection \ref{ssec:CONT} 
we describe the three contributions above in detail; Subsection \ref{ssec:CNUM}
contains illustrative numerical experiments which demonstrate our
 contributions; and Subsection \ref{ssec:OUT} concludes the
introduction with an outline of the paper, by section.
}

\subsection{Literature Review}
\label{ssec:BLR}

Clustering is a fundamental task in data analysis and in unsupervised 
and semi-supervised learning in particular; algorithms in these areas
seek to detect clusters, and more generally coarse structures, geometry 
and patterns in data. Our focus is on Euclidean data. Our starting point 
is a dataset  $X = \{x_1,...,x_N\}$ comprising $N$ points  
$x_i \in \mbb R^d$, assumed to be drawn i.i.d. from a (typically unknown) 
probability distribution with (Lebesgue) density $\varrho$.
The goal of clustering algorithms is to split  $X$ into meaningful clusters. Many such algorithms proceed as follows:
The data points $x_i$ are associated with the vertices of a graph
and a weighted adjacency matrix $W_N$, measuring affinities between 
data points, is defined on the edges of the graph.
From this matrix, and from a weighted diagonal degree matrix $D_N$ found
from summing edge weights originating from a given node, various 
 graph Laplacian matrices $L_N$ can be defined. 
The success of clustering algorithms 
is closely tied to the spectrum of $L_N.$ 
At a high level, $k$ clusters will manifest in $k$ small eigenvalues
of $L_N$, and then a spectral gap; and the $k$ associated eigenvectors 
will have geometry which encodes the clusters. Unsupervised
learning leverages this structure to identify clusters 
\cite{belkin2003laplacian,Ng01onspectral,spielmat1996spectral, vLuxburg2007} and semi-supervised learning
uses this structure as prior information which is enhanced
by labeled data \cite{bertozzi2012diffuse,bertozzi2017uncertainty, zhu2003semi}. It is thus of considerable interest to study 
the spectral properties of $L_N$, and the dependence of the
spectral properties  on the data and on the design
parameters chosen in constructing $L_N.$

The operator $L_N$ in \eqref{defLN-intro} corresponds to different
normalizations of the graph Laplacian. A number of special cases within
this general class arise frequently in the implementation of unsupervised 
and supervised learning algorithms. 
The \emph{unnormalized graph Laplacian} 
refers to the choice $(p,q,r)=(1,2,0),$
giving the symmetric matrix $L_N=D_N-W_N$; 
another popular choice is the {\it normalized graph Laplacian}
where 
$(p,q,r)=(3/2,2,1/2);$ the choice  $(p,q,r)=(2,2,0)$ also gives a 
widely used normalized operator. The graph Laplacian 
for $(p, q,r)= (3/2,2,1/2)$ is symmetric and studied in
\cite{trillos2016variational, Ng01onspectral, schiebinger2015geometry, shimalik2000,vLuxburg2007,vLBB2008}, whereas the choice $(2,2,0)$ gives an operator that is not symmetric, but can be interpreted as a transition probability of a random walk on a graph
\cite{CoifmanLafon2006, shimalik2000}. A number of other choices for $(p,q,r)$ appear in the literature. For example, the spectrum of the graph Laplacian with $(1,2,0)$ is related to the ratio cut, whereas $(2,2,0)$ is connected to the Ncut problem. The success of the spectral clustering procedure for the graph Laplacian with parameters $(1,1,0)$ was investigated in \cite{NGTFHBH} in the setting of non-parametric mixture models;  in 
this case the Dirichlet energy with respect to the natural density 
weighted $L^2$ inner-product  is linear in $\varrho$.
In \cite{CoifmanLafon2006, wormell2020spectral}, general choices of $p=q\ge 0$ and $r=0$ are investigated in the context of diffusion maps with \cite{wormell2020spectral} presenting sharp pointwise error
  bounds on the spectrum as well as norm convergence of $L_N$  
 to  $\mcl L$. In this case, the limiting operator $\mathcal{L}$ 
is the generator of a reversible
diffusion process, a connection first
established in the celebrated paper \cite{CoifmanLafon2006}
by Coifman and Lafon.

\changed{Whilst many different normalizations of the graph Laplacian have been used for a variety of data analysis tasks, a thorough understanding of the advantages and disadvantages of different parameter choices is still lacking. 
The papers \cite{vLuxburg2007,vLBB2008} contain comparisons between the
normalized, unnormalized and  random walk Laplacians. But, 
to the best of our knowledge, there is a gap in the current literature 
concerning a systematic understanding of the effects of the
entire family of weighted graph Laplacian matrices $L_N$
  depending on the family of parameters $(p,q,r).$ 
Of particular interest  is the case where $N$ is large,
relevant in large data applications, and in \cite{vLBB2008} the 
authors showed that the normalized and random walk Laplacians
give consistent spectral clustering as opposed to the unnormalized Laplacian
operator in this large $N$ limit. This behavior is attributed to
different integral operators to which the normalized and unnormalized 
Laplacians converge. The normalized Laplacian converges to a 
compact perturbation of the identity with a discrete spectrum 
while it is demonstrated that
the unnormalized Laplacian may not possess a purely discrete spectrum.
}

  The large data limit convergence of graph Laplacians to integral or differential operators has been
  the subject of many recent studies including \cite{Belkin2006ConvergenceOL, 
    belkin2008towards, calder2019improved, trillos2016variational,trillos2018error,
    gine2006empirical,schiebinger2015geometry, shi2009data,
    slepvcev2017analysis, vLBB2008, wormell2020spectral}.
The point of departure in  these papers is a kernel $\eta$
defined on $\R^d \times \R^d$, from which the weighted adjacency 
matrix $W_N$ defined on the edges of a graph is constructed. In
\cite{Belkin2006ConvergenceOL, belkin2008towards, schiebinger2015geometry, shi2009data, vLBB2008}
the authors fix a kernel and let $N \to \infty$ obtaining an integral operator
as the limit of graph Laplacians. These limiting integral operators are dependent on
the kernel $\eta$ and subsequently  the results of these articles also depend on the choice of
the kernel. 
The more recent articles 
\cite{calder2019improved,trillos2016variational,trillos2018error,gine2006empirical,
  slepvcev2017analysis, wormell2020spectral} consider the joint limit
as $N \to \infty$ and the width of the kernel $\eta$ vanishes sufficiently slowly thereby controlling
the local connectivity of the graph. It then follows that in taking this joint limit graph Laplacian matrices
$L_N$ converge to differential operators of a similar form to our $\mcl L$ operator; \changed{under this type of limiting procedure the resulting differential
operator is independent of the weight kernel $\eta$, up to scaling.}

\changed{The aforementioned articles suggest the potential for further analysis 
of the continuum limits of graph Laplacians as a means to advance our 
understanding of  clustering algorithms on finite but large data sets. 
Such continuum approaches, often refereed to as {\it population level analyses},
proceed by studying graph Laplacian operators and 
subsequently spectral clustering algorithms in the continuum 
regime \cite{NGTFHBH,schiebinger2015geometry, shi2009data}.
The continuum analysis may then be extended to the finite data setting
using discrete-to-continuum approximation results such as those in
\cite{calder2019improved,trillos2018error,schiebinger2015geometry,wormell2020spectral}.
We employ the same perspective in this work, focusing primarily on the
analysis of the continuum operators and providing numerical experiments
and formal calculations demonstrating the relevance of the continuum analysis
to finite data settings. We note that the paper \cite{ng2002spectral} studies
consistency of spectral clustering for finite graph problems, and that
similar ideas from linear algebra are used to study large data limits 
in \cite{de2020consistency}, albeit with very restrictive assumptions on the
clusters; no limiting operator is employed, or identified, in this analysis.}

\changed{
We also note that mathematical studies which are conceptually 
similar to the spectral
analysis that we present here have been prevalent 
in the study of metastability in chemically reacting systems 
for some time; see \cite{deuflhard1999computation,deuflhard2000identification,
  huisinga2004phase, schutte2001transfer} and the references therein for applications.
This body of work has led to very subtle and deep analyses of the generators of
Markov processes \cite{bovier2004metastability,bovier2005metastability};
this analysis might, in principle, be used to extend some of the
work undertaken here to a wider range of sampling densities.
}

\changed{
Finally, 
the tools developed in this paper may be used to study 
consistency of semi-supervised learning algorithms in \cite{HHOS2}. 
In particular, we provide the  spectral perturbation results needed  to generalize the 
work in \cite{HHRS1}, which studies consistency of graph-based 
semi-supervised learning algorithms for finite $N$ and using
the graph Laplacian $L_N$, to the large data limit where $N \to \infty$ and  $L_N$
is replaced by $\mcl{L}$ \cite{HHOS2}.}

\subsection{Our Contributions}
\label{ssec:CONT}

We now detail the three contributions outlined in Subsection \ref{ssec:O}.
\changed{Contribution 1 is summarized in our main theoretical result characterizing the
  low-lying spectrum of $\L$ and the effect of the $(p,q,r)$ parameters;
  Contribution 2 extends our theoretical analyses by various
  numerical experiments (i) in the unbalanced
  regime where $q \neq p +r$, revealing that some of our bounds on
  the eigenvalues of $\L$ can be sharpened, and (ii) to the setting of multiple clusters and more general data densities $\varrho$, suggesting that the theory provided under Contribution 1 reveals fundamental concepts that hold in more generality than the specific setting considered in Contribution 1; Contribution 3 combines formal calculations and
  numerical experiments to
  reveal the relationship between the $(p,q,r)$ parameterized family of differential
  operators $\L$ and
  various weightings of discrete graph Laplacians $L_N$.}

\subsubsection{Contribution 1}\label{sec:contribution-1}

\changed{Let us define the notion of a  \emph{perfectly separated} density.  
Let $\Z \subset\R^d$ be bounded and $\varrho_0$
be a (Lebesgue) probability density with support $\Dp \subset \Z$ 
strictly contained in $\Z$ and concentrated on two 
disjoint subsets $\Z^+$ and $\Z^-$ of $\Z$; that is, 
$\Dp=\Z^+\cup \Z^-$ and $\Z^+\cap \Z^-=\emptyset$.
We refer to $\Z^{\pm}$ as clusters, and denote the 
operator of the form \eqref{general-weighted-Laplacian} 
based on $\varrho_0$ by $\mcl{L}_0$. Consequently, a \emph{nearly separated} 
density  comprises a class of smooth densities $\varrho_\eps$
that are $\mcl{O}(\eps)$ perturbations of the perfectly 
separated case $\varrho_0$, with density supported everywhere 
on $\Z$ and such that $\varrho_\eps = C \eps$ away from $\Z'$ with $C > 0$ a constant;
we define this concept precisely in Section \ref{sec:sa}.
We denote the operator of the form \eqref{general-weighted-Laplacian}
based on $\varrho_\eps$ by $\mcl{L}_\eps$. To this end, our main 
theoretical result characterizes the low-lying spectrum of $\L_\eps$ 
in the nearly separated regime.}

\changed{
\begin{mainresult} \label{main-result}
Assume $q>0$ and $p+r>0$.

\begin{enumerate}[(i)]
\item The first eigenpair of $\mcl{L}_\eps$ is given by
\begin{align*}
  \sigma_{1,\epsilon}=0\,,\qquad
   \varphi_{1,\epsilon} = \frac{1}{|\Z|_{\varrho_\eps^{p +r}}^{1/2}}  \varrho_\eps^r(x)\mbf{1}_{\Z}(x),
  \qquad \forall x\in\Z\,
\end{align*}
where $| \Z |_{\varrho_\eps^{p+r}} : = \int_\Z \varrho_\eps^{p+r}(x) dx$.
\item The second eigenvalue scales as $\sigma_{2,\eps} = \mcl{O}( \epsilon^{q})$
and the corresponding eigenvector is given, approximately in a density weighted
$L^2$ space, by the formula
\begin{equation}
  \label{eq:FV}
\varphi_{2,\epsilon} \approx \frac{1}{ | \Z'|_{\varrho_\eps^{p+r}}^{1/2}}\varrho_\eps^r(x)\bigl(\mbf{1}_{\Z^+}(x)-\mbf{1}_{\Z^-}(x)\bigr),\qquad \forall x\in\Z\,.
\end{equation} 

\item The behavior of the third eigenvalue $\sigma_{3,\eps}$ varies depending on the relationship between the parameters $q$ and $p+r$:
\begin{itemize}
\item if $p+r<q<2(p+r)$, then a spectral ratio gap
  manifests with $\sigma_{2,\eps}/\sigma_{3,\eps}=
{\mathcal O}(\eps^{2(p+r)-q})$  as $\eps \to 0$;
\item if $q=p+r$, then $\sigma_{3,\eps} \asymp 1$ and  a uniform spectral gap manifests, i.e., $\sigma_{3,\eps}-\sigma_{2,\eps} \asymp 1$ and
    $\sigma_{2,\eps}/\sigma_{3,\eps}= \mcl O( \eps^{q})$ as $\eps \to 0$;
\item if $q < p+r<2q$, then a spectral ratio gap
  manifests with $\sigma_{2,\eps}/\sigma_{3,\eps}=
{\mathcal O}(\eps^{2q-(p+r)})$  as $\eps \to 0$.
\end{itemize}

\end{enumerate}

\end{mainresult}

We precisely state this result, with fully detailed 
assumptions, in  Section~\ref{sec:sa}; the statement is comprised
of a combination of theorems and corollaries. Part (i) is contained in 
Theorem~\ref{thm:L-eps-low-eigenvals}(i) while part (ii) follows by  combining
Theorem~\ref{thm:L-eps-low-eigenvals}(ii) with Theorem~\ref{thm:geometry-of-fiedler-vector}.
Finally part (iii) is encompassed by Corollary~\ref{cor:ratiogap}. A roadmap of the
proofs of these results is explained in Section~\ref{sec:sa} with 
the detailed proofs  postponed to Section~\ref{sec:sap}.
}

\subsubsection{Contribution 2}\label{sec:contribution-2}

\changed{We present detailed numerical experiments 
in Section \ref{sec:num} that both support our Main Result~\ref{main-result}
and make two substantial extensions. These extensions sharpen
our results in the unbalanced cases and extend our results to $K>2$
clusters. In particular, our experiments in case $K=2$ demonstrate that 
the rates for $\sigma_{2,\eps}/\sigma_{3,\eps}$ in 
Main Result~\ref{main-result}(iii) are sharp in the balanced setting 
where $q = p +r$ but show clear evidence that
the theoretical rates obtained in the unbalanced settings 
where $q \neq p +r$ are slower than the observed rates. The results
obtained by combining Main Result~\ref{main-result} and this
empirical improvement in the unbalanced case are then shown numerically 
to extend naturally to $K>2$ clusters. For clarity we summarize 
these numerical results in the conjecture that follows.


\begin{conjecture}\label{conj:SG}
  Suppose that the conditions of Main Result \ref{main-result} are satisfied with the
  data density $\varrho_\eps$ concentrating on $K \ge 2$ clusters in the small $\eps$ limit.  Then  
\begin{equation*}
  \sigma_{K, \eps} \asymp \eps^q, \qquad \frac{\sigma_{K,\eps}}{\sigma_{K+1, \eps}}
  \asymp \eps^{\min\{q, p +r\}}.
\end{equation*}
\end{conjecture}

Our numerical simulations in Section~\ref{sec:num},
and in particular Tables~\ref{tab:balanced-r-eq-p} to
\ref{tab:unbalanced-qlepr}, suggest the above conjecture in the binary
  cluster setting 
that sharpens the decay rate of $\sigma_{2,\eps}/\sigma_{3,\eps}$
as a function of $\eps$, in the unbalanced settings when $q \neq p +r $.
Put simply, this  conjecture states that  when $K=2$ and   $q < p +r$ 
the third eigenvalue $\sigma_{3,\eps}$ exhibits similar behavior to 
the balanced setting where $q = p +r$ and hence
a uniform gap in the spectrum manifests as $\eps \to 0$. However, 
when  $q>p+r$ the third eigenvalue $\sigma_{3,\eps}$ vanishes 
like $\eps^{q-p -r}$ and a spectral ratio gap manifests. 
Moreover, if this conjecture holds then it
allows us to sharpen the approximation error of the
second eigenfunction $\varphi_{2,\eps}$ in Theorem~\ref{thm:geometry-of-fiedler-vector},
as this result heavily depends on a lower bound for $\sigma_{3,\eps}$.
We attribute this discrepancy to the lower bound on $\sigma_{3,\eps}$
obtained in Theorem~\ref{thm:L-eps-low-eigenvals}(iii) that in turn
relies on a generalization of Cheeger's inequality from
Appendix~\ref{sec:proof-isoperim-ineqa}.
}

\subsubsection{Contribution 3}\label{sec:contribution-3}

\changed{
We demonstrate the relationship between the $(p,q,r)$ dependent family of
operators $\L$ in \eqref{general-weighted-Laplacian} showing how they
arise as the limit of graph Laplacian matrices $L_N$ of the form 
\eqref{defLN-intro}.  Subsection~\ref{sec:cvEnergies} presents an informal 
limiting argument
to identify the operator $\L$ by considering the large data $N$ limit, followed
by small kernel bandwidth $\delta$ limit of $L_N=L_N(\delta)$. 
Our informal calculations in 
Subsection~\ref{sec:discrete-vs-continuum-eigenproblems} extend these
arguments from Dirichlet energies to eigenvalue problems,
and indicate that the spectrum of
the matrix $C \delta^{-2} N^{2r- q} L_N$ converges to that of $\L$,
for a suitable constant $C >0$,
as $(N,\delta^{-1}) \to \infty$. 
Our numerical experiments in Subsection~\ref{sec:discrete-numerics}
support these informal calculations, demonstrating the convergence of 
the eigenvalues of $L_N$
to numerically computed eigenvalues of $\L$ for different choices of $(p,q,r)$
and for two different types of mixture models.
The numerical experiments and informal arguments are developed
in the following setting: we assume that the data at the
$N$ vertices of the graph, $\{x_1,...,x_N\}$, 
are sampled i.i.d. from the probability density 
$\varrho$ and we suppose that the resulting weight matrix $W_N$ is constructed
using a kernel $\eta_\delta$ with the parameter $\delta >0$ controlling
the local connectivity of the vertices; see Subsection~\ref{sec:discrete-setting} for details.

To make a precise theory supporting these observations requires
specification of the relationship between $N$ and $\delta$ in
the limiting process $(N,\delta^{-1}) \to \infty$.
The convergence of $L_N$ to $\L$  for specific choices of
$(p,q,r)$ has been established in the literature, and this issue was
addressed in those papers. In particular, in
\cite{trillos2016variational} convergence of the spectrum of 
$L_N$ $\Gamma$-converges to that of $\mcl L$,
and that the eigenfunctions of $L_N$  converge to those of $\mcl L$ in 
the $TL^2$ topology.
More recently, the articles \cite{calder2019improved, trillos2018error, wormell2020spectral}
further extend these results
giving rates for the convergence of eigenvalues and eigenfunctions 
  for $(p,q,r) = (1, 2, 0)$ and also for the convergence of $L_N$ on $k$-nearest neighbor ($k$-NN) graphs
  to $\mcl L$ with $(p,q,r) = (1, 1 - 2/d, 0)$. We postulate that
the methods of proof
introduced in \cite{calder2019improved, trillos2018error},
and extensions to spectral convergence properties proved there, can be
generalized to the $(p,q,r)-$dependent family of graph Laplacian 
operators introduced here; with the analysis for $k$-NN graphs departing from the
proximity graphs considered here in particular in the construction of the
discrete operator $L_N$ and its normalization with different choices of $(p,q,r)$. However space considerations preclude a full analysis within the
confines of this paper.
}

\subsection{Illustrative Numerical Experiments}
\label{ssec:CNUM}

\changed{The contributions detailed in the preceding subsection
demonstrate that the manner in which clustering is manifest 
in the spectral properties of the graph Laplacian
depend subtly on the choice of the parameters $(p,q,r)$. 
Making the balanced choice $q= p+r$ one obtains a 
family of operators whose second eigenvalue decays rapidly, while 
the gap between the second and third eigenvalues remains of
order one as the parameter $\epsilon$, measuring closeness to
perfect clustering, decreases to zero; this uniform separation
of second and third eigenvalues does not happen when $q > p+r.$
Furthermore the form of the  \emph{Fiedler vector} (the second 
eigenfunction), whilst always exhibiting 
the clusters present in the data, can have different behavior away from the
clusters, depending on $(p,q,r)$. We demonstrate these facts
in Example \ref{ex:ex}, exemplifying Contributions 1 and 2. Additionally, Example~\ref{ex:mixture-model} shows that
our theory likely applies without the rather specific assumptions used to
define clustering as mentioned in Contribution 2; furthermore, Example~\ref{ex:mixture-model}
illustrates that the spectral properties of the
limiting operator $\L$  reflect the properties of the
discrete graph Laplacian arising when $N < + \infty$ as outlined in our Contribution 3. } 

\begin{example}[Comparison of unnormalized and normalized graph Laplacians]
\label{ex:ex}
\changed{We study the spectral properties of operator $\mcl{L}_{\eps}$
with parameter choices $(p,q,r)$ given by
 $(1, 2, 0)$ and $(3/2, 2, 1/2)$ respectively,
corresponding to the unnormalized and normalized graph Laplacians
respectively. All our numerical experiments are for
a data density $\varrho_\eps$ of the form \eqref{numerics-rho-def} with two distinct clusters; see Figure~\ref{fig:density-plots-two-clusters}(a) for a plot of $\varrho_\eps$ with $\eps= 0.0125$. 

In the unnormalized case $q > p +r$ it follows 
from our Main Result~\ref{main-result} that as $\eps \downarrow 0$ the
second eigenvalue of $\mcl L_\eps$ scales as $\eps^{2}$ and that 
a spectral gap is present only in ratio form. 
In Figure~\ref{fig:density-plots-two-clusters}(b)
we plot the second and third eigenvalues $\sigma_2$ and $\sigma_3$ against
$\eps$, on a log scale, and calculate best linear fits to the data;
this demonstrates that they converge to zero like $\eps^2$ and
$\eps$ respectively, in agreement with our Main Result~\ref{main-result} (second eigenvalue)
and the first component of Conjecture \ref{conj:SG} (third eigenvalue). 
We also compute the second eigenfunction (Fielder vector)
$\varphi_{2,\eps}$ shown in Figure~\ref{fig:density-plots-two-clusters}(d).
Note that in this case the pointwise distance between $\varphi_{2,\eps}$ and
the right hand side of \eqref{eq:FV} in Main Result~\ref{main-result}(ii) is only small  within the
clusters; this reflects the fact that the weighted $L^2(\Z,\varrho_\eps^{p-r})$-norm arising
in Theorem~\ref{thm:geometry-of-fiedler-vector} for this choice of $(p,q,r)$ is 
not sensitive to large pointwise values of functions in areas 
where $\varrho_\eps$ is small. 
}

 \begin{figure}[htp!]
   \centering
   \subfloat[]{%
   \includegraphics[height=.35\textwidth]{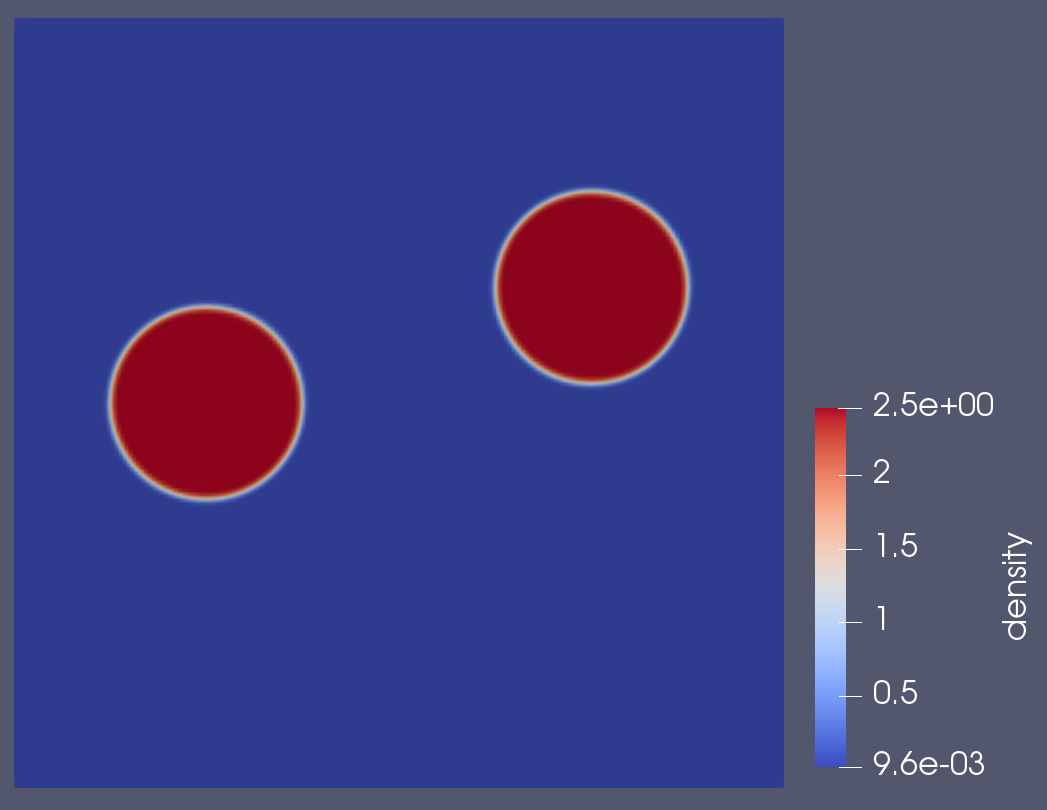}
 } \\
  \subfloat[]{%
    \includegraphics[height=.37\textwidth, clip=true, trim=1cm 6cm 2cm 7cm]{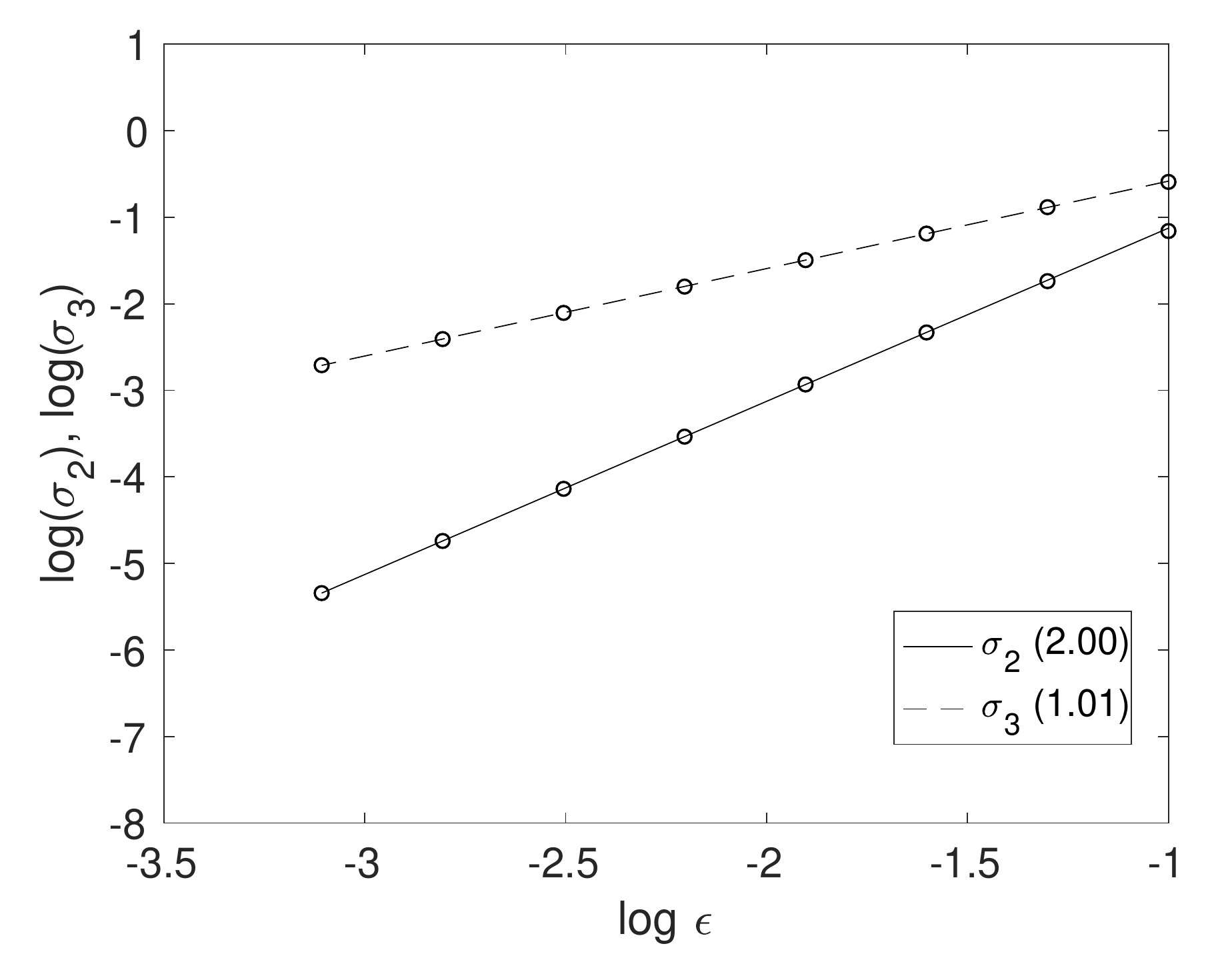}
    } \hspace{2ex}
  \subfloat[]{%
   \includegraphics[height=.37\textwidth, clip=true, trim=1cm 6cm 2cm 7cm]{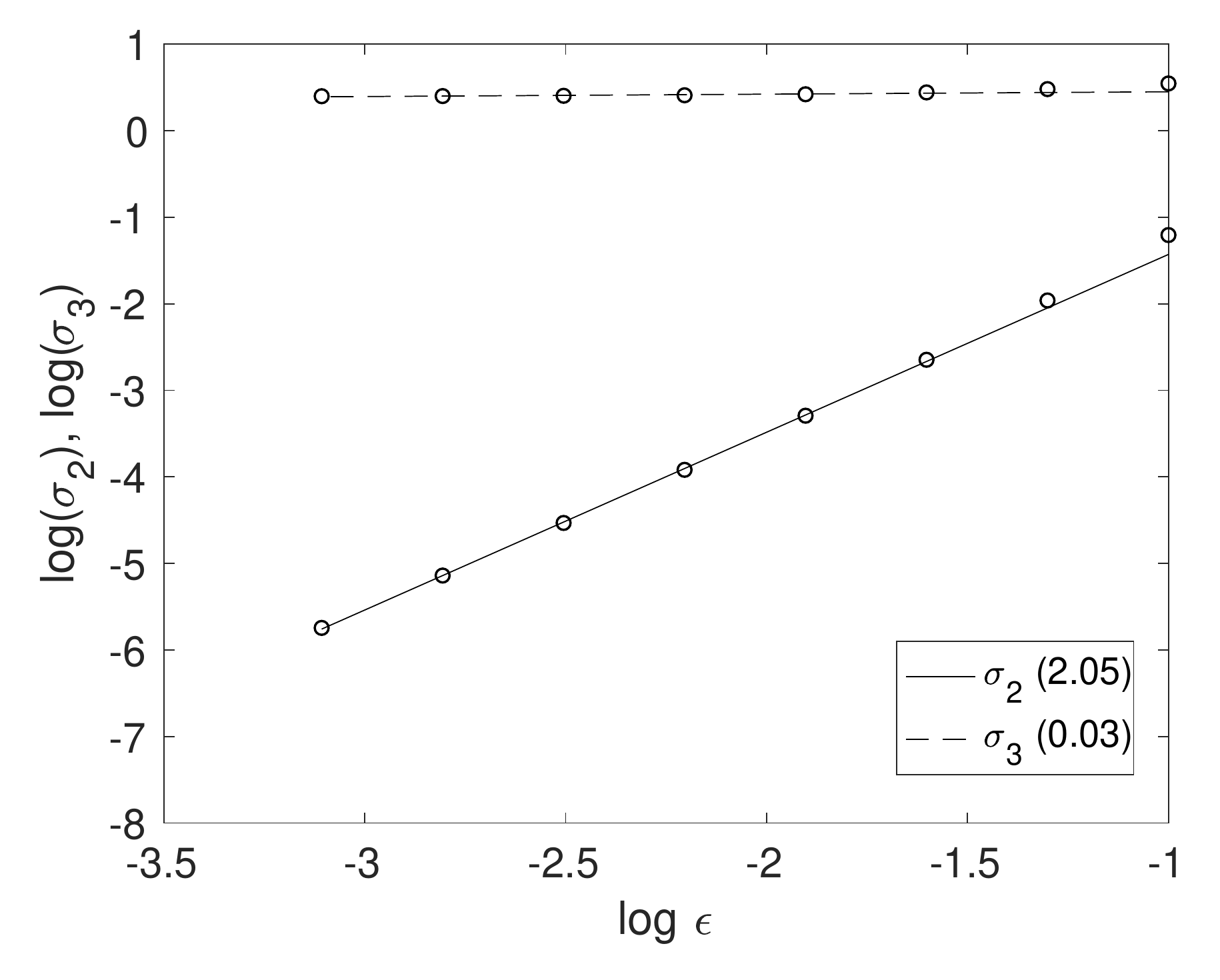}
 }\\
 \subfloat[]{%
    \includegraphics[height=.35\textwidth, clip=true, trim=18cm 5cm 8cm 5cm]{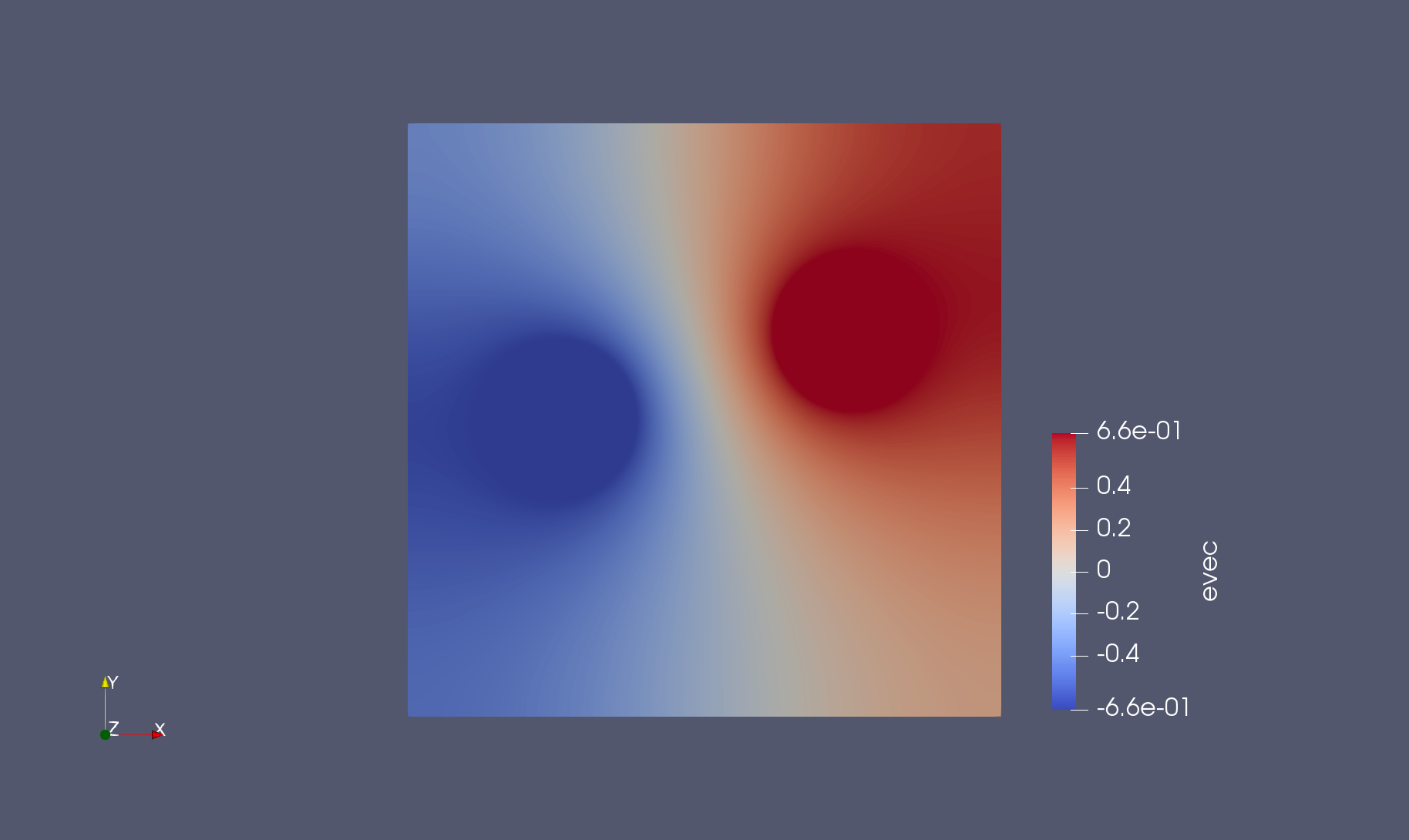}
 }
\subfloat[]{%
   \includegraphics[height=.35\textwidth, clip=true, trim=18cm 5cm 8cm 5cm]{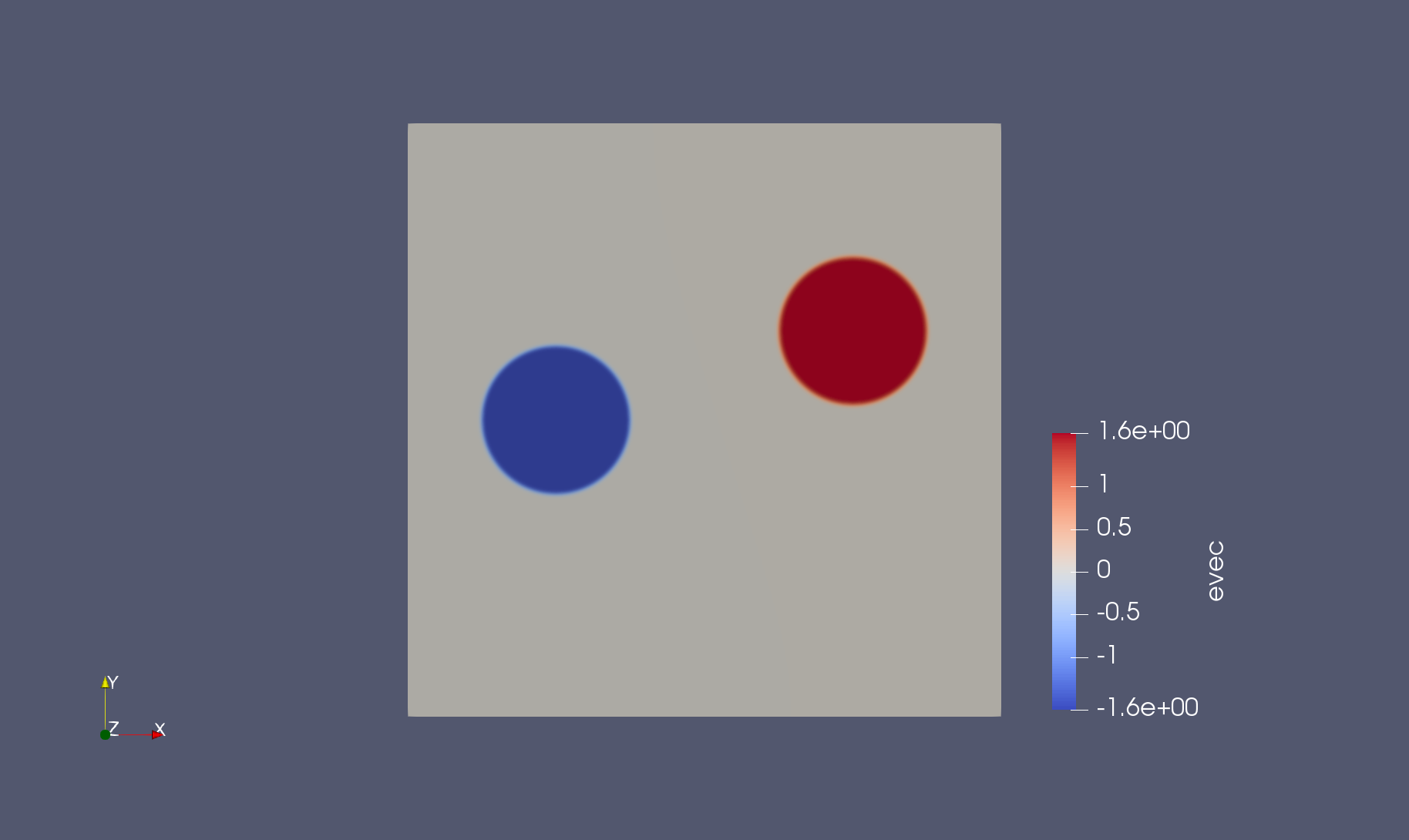}
 }
 \caption{
     (a) Plot of a density $\varrho_\eps$ of the form \eqref{numerics-rho-def} with two distinct clusters for $\eps= 0.0125$. (b) Showing $\log(\sigma_2)$ and $\log(\sigma_3)$,
     the second
     and third eigenvalues of the unnormalized operator
     $\mcl L_\eps$ with $(p,q,r) = (1, 2, 0)$ as functions of $\eps$.
     Values in brackets in the legends indicate numerical slope of the lines.
     (c) Showing $\log(\sigma_2)$ and $\log(\sigma_3)$
     for the normalized operator $\mcl L_\eps$ for $(p,q,r) = (3/2, 2, 1/2)$, as functions of $\eps$. (d) and (e) The Fiedler vector of 
   $\mcl L_\eps$ with $(p,q,r) = (1, 2, 0)$ and $(p,q,r) = (3/2, 2, 1/2)$ respectively for $\eps= 0.0125$.}
     \label{fig:density-plots-two-clusters}
 \end{figure}

\changed{For comparison we now consider the 
normalized setting. For $q = p+ r$ 
our Main Result~\ref{main-result} predicts that,
as $\eps \downarrow 0$, there exists a uniform spectral gap between the first 
two eigenvalues of $\mcl L_\eps$: for $(p,q,r)=(3/2, 2, 1/2)$, the second eigenvalue scales
as $\eps^2$ and the third is of order one with respect to $\eps$.
In Figure~\ref{fig:density-plots-two-clusters}(c)
we plot the second and third eigenvalues of $\L_\eps$ against $\eps$ in that case,
on a log-scale, and provide best fits to the data; the results support the theory. 
The corresponding Fiedler vector $\varphi_{2,\eps}$ is shown in Figure~\ref{fig:density-plots-two-clusters}(e).
In this case $\varphi_{2,\eps}$  appears to converge pointwise to the
right hand side of \eqref{eq:FV}, in contrast to the unnormalized case.

It is well-known that the Fiedler vectors
encode information on the clusters $\Z^\pm$ 
that we are trying to detect. They play
a significant role in the context of spectral clustering
and binary classification \cite{vLuxburg2007}. However, it
is noteworthy that the Fiedler vectors in the unnormalized and
normalized cases differ substantially
within $\Z\setminus\Z':$ in the unnormalized case a smooth transition is
made between $\Z^+$ and $\Z^-$, whereas in the normalized case abrupt transitions
are made to near zero on the boundaries of  $\Z^+$ and $\Z^-$. 
}

\end{example}

\changed{Since our primary motivation is data clustering, it is relevant
to interpret our contributions in that context.
In the following example we demonstrate that
although our theory is developed under rather strict assumptions
on the sampling density of the data and in the limit $N \to \infty$, 
our results concerning the dependence of spectral ratio gaps on 
the $(p,q,r)$ parameters appear to generalize to mixture models 
that violate some of our assumptions. The mixture model assumption
is a natural model for population level analysis of clustering 
algorithms and is considered in the articles 
\cite{NGTFHBH,schiebinger2015geometry}. It can be argued to be a 
more realistic data model for the density $\varrho$ than the one
for which our theory is developed and it is therefore of
interest to demonstrate that our theory is predictive in this setting.

\begin{example}[Clustering a mixture model]\label{ex:mixture-model}
  Consider the following mixture on the unit square
  \begin{equation}\label{exp-mixture-model}
    \varrho_\omega(t) := \frac{1}{2\omega} \left(1-\exp \left( -\frac{1}{\omega} \right)\right)^{-1}
    \left[ \exp \left(  -\frac{t_1}{\omega} \right)
        + \exp \left(\frac{t_1-1 }{\omega} \right) \right], \qquad t = (t_1, t_2)^T \in [0,1]^2. 
    \end{equation}
This density is simply the mixture of two exponential distributions
restricted to the unit interval $[0, 1]$ in the $t_1$ direction,  with a  
uniform distribution in the $t_2$ direction; see
Figure~\ref{fig:exp-mixture-example}(a). The parameter $\omega$ controls the
overlap of the mixture components. This model clearly violates our assumptions
on the density $\varrho$ outlined in Section~\ref{ssec:pod}, most notably, (i) letting
    $\omega \to 0$ the density $\varrho_{\omega}$  concentrates on sets of measure zero as opposed to
    clusters $\Z^\pm$ of positive measure, and (ii) we cannot ensure that $\varrho_\omega = C \omega$
    outside of clusters since the tails of the exponential  components decay exponentially
    as we let $\omega \to 0$.

    We generate $N$ samples from $\varrho_\omega$ and construct a weighted proximity
    graph on this dataset using a weight kernel of width $\delta >0$ as detailed 
    in Subsection~\ref{sec:discrete-numerics}. We then proceed to define a discrete
    graph Laplacian $L_N$ of the form \eqref{defLN-intro} and compute the first four
    non-trivial eigenvalues $\sigma_{N, \delta}$ of this discrete operator
    (this  notation for the eigenvalues is defined in Subsection~\ref{sec:cvEnergies}).
    Figure~\ref{fig:exp-mixture-example}(b,c,d) show the variation of the
    first few eigenvalues as a function of $\omega$ for $N= 2^{13}$ vertices.
    We consider three choices of the $(p,q,r)$ parameters, a balanced case with
    $(1, 2,1)$ and two unbalanced cases with $(1/2, 2, 1/2)$ and $(1, 3/2, 1)$.
    While our theory does not make a prediction regarding the rate at which the second eigenvalue
    vanishes with $\omega$,  we can still use our theoretical insights to postulate uniform or ratio gaps between the second and third eigenvalues.
    
    In the balanced case where $q = p +r$ we observe that the second eigenvalue
    vanishes with $\omega $ while the rest of the spectrum remains bounded away
    from zero; in contrast, in the unbalanced case $q > p +r$ the third eigenvalue also
    vanishes and only a spectral ratio gap manifests. The results in the unbalanced
    case $q < p +r$ are less clear since the higher eigenvalues still vanish, 
but they do so rather slowly; this may be attributed to numerical error. 
    The results are in agreement with our analysis and numerical results
    in the continuum limit and suggest that the characteristic behavior we 
prove for our specific construction of the sampling density $\varrho$ is 
in fact a more general phenomenon that applies for other type of clustered
data and on finite data sets. Further details regarding
this experiment are summarized in Subsection~\ref{sec:discrete-numerics}.

    \begin{figure}[htp]
      \centering
         \subfloat[]{%
           \includegraphics[width=.45\textwidth]{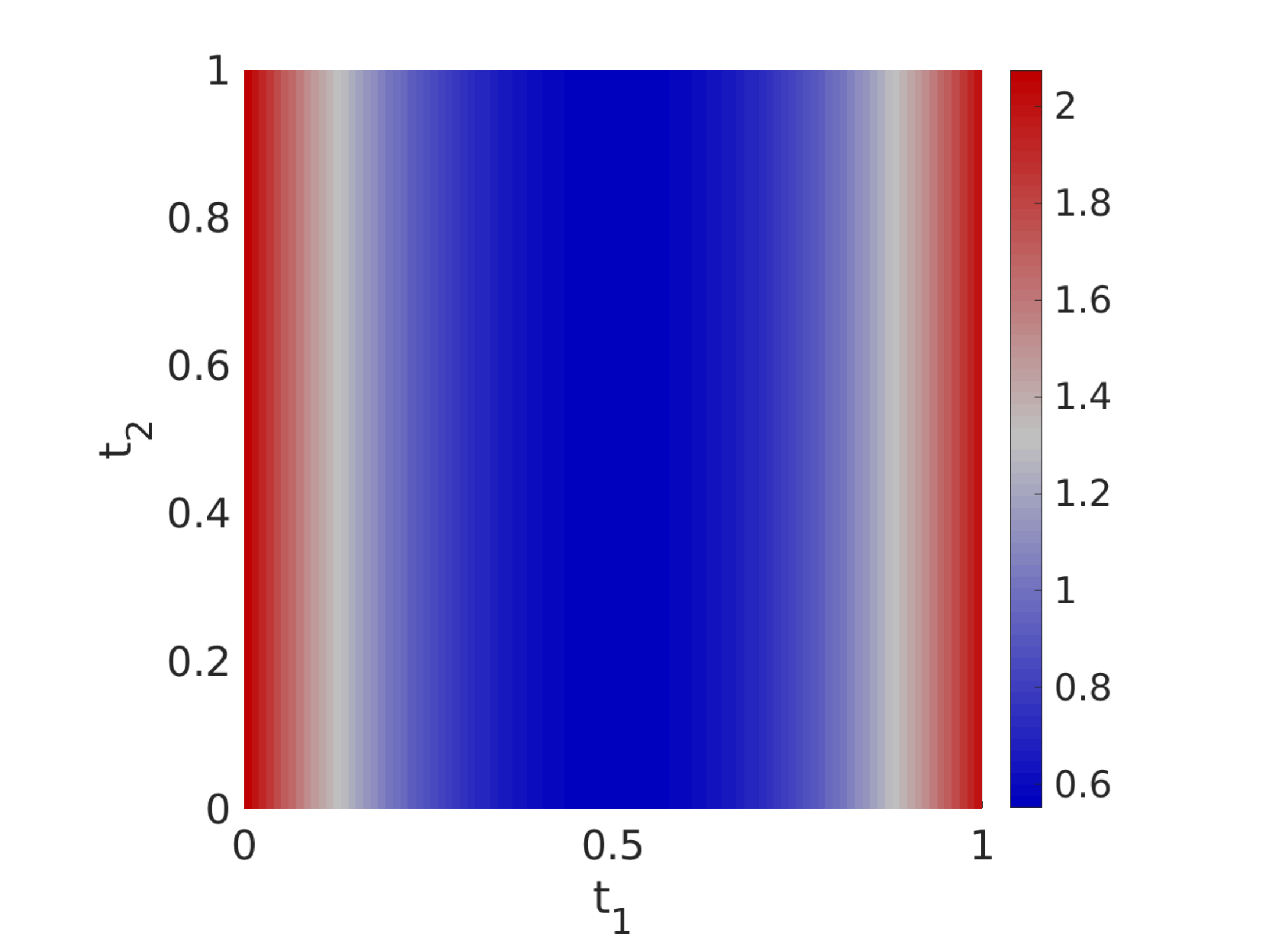}
         } 
         \subfloat[]{%
           \includegraphics[width=.45\textwidth]{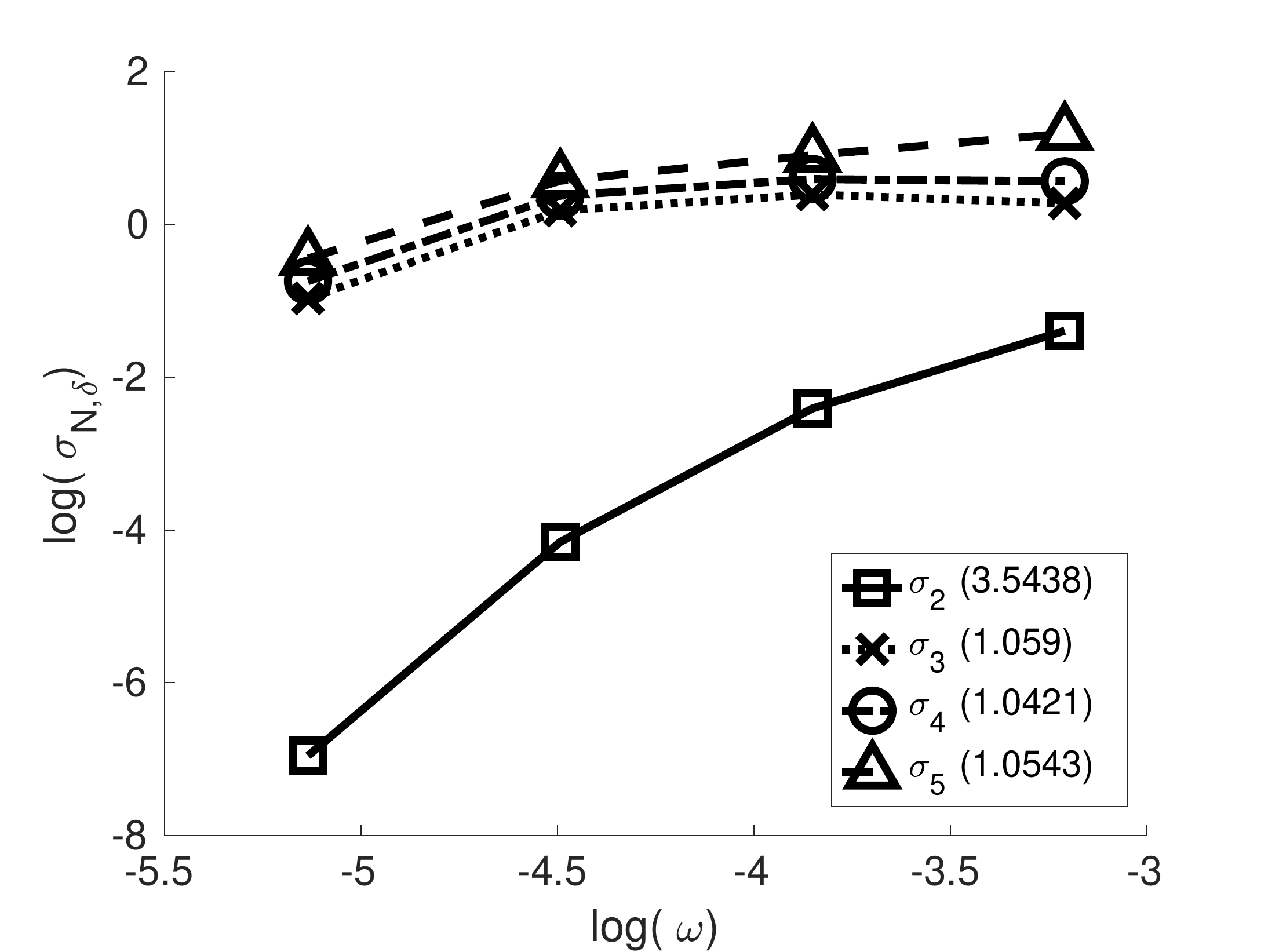}
         } \\
         
         \subfloat[]{%
           \includegraphics[width=.45\textwidth]{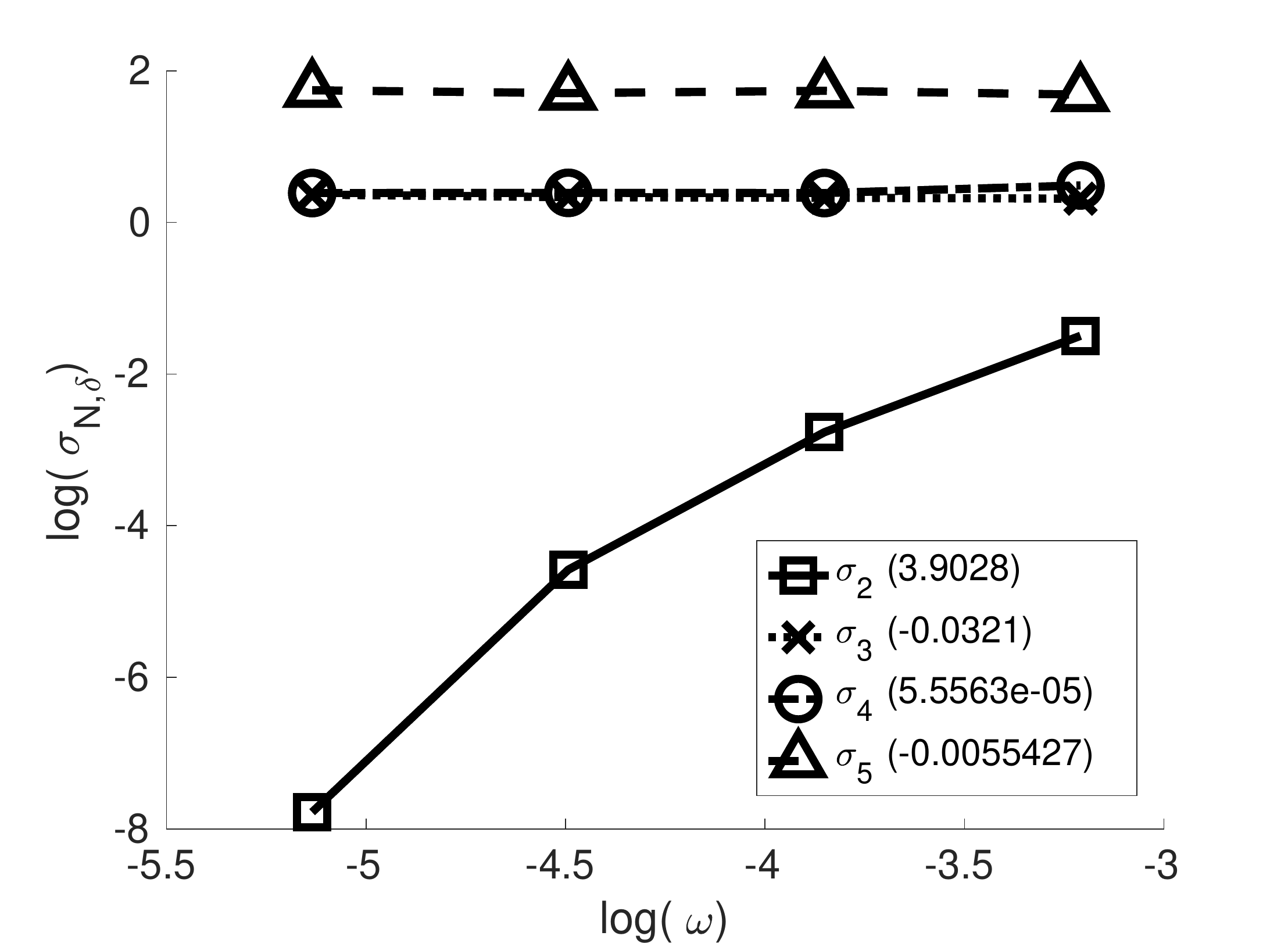}
         }
         \subfloat[]{%
           \includegraphics[width=.45\textwidth]{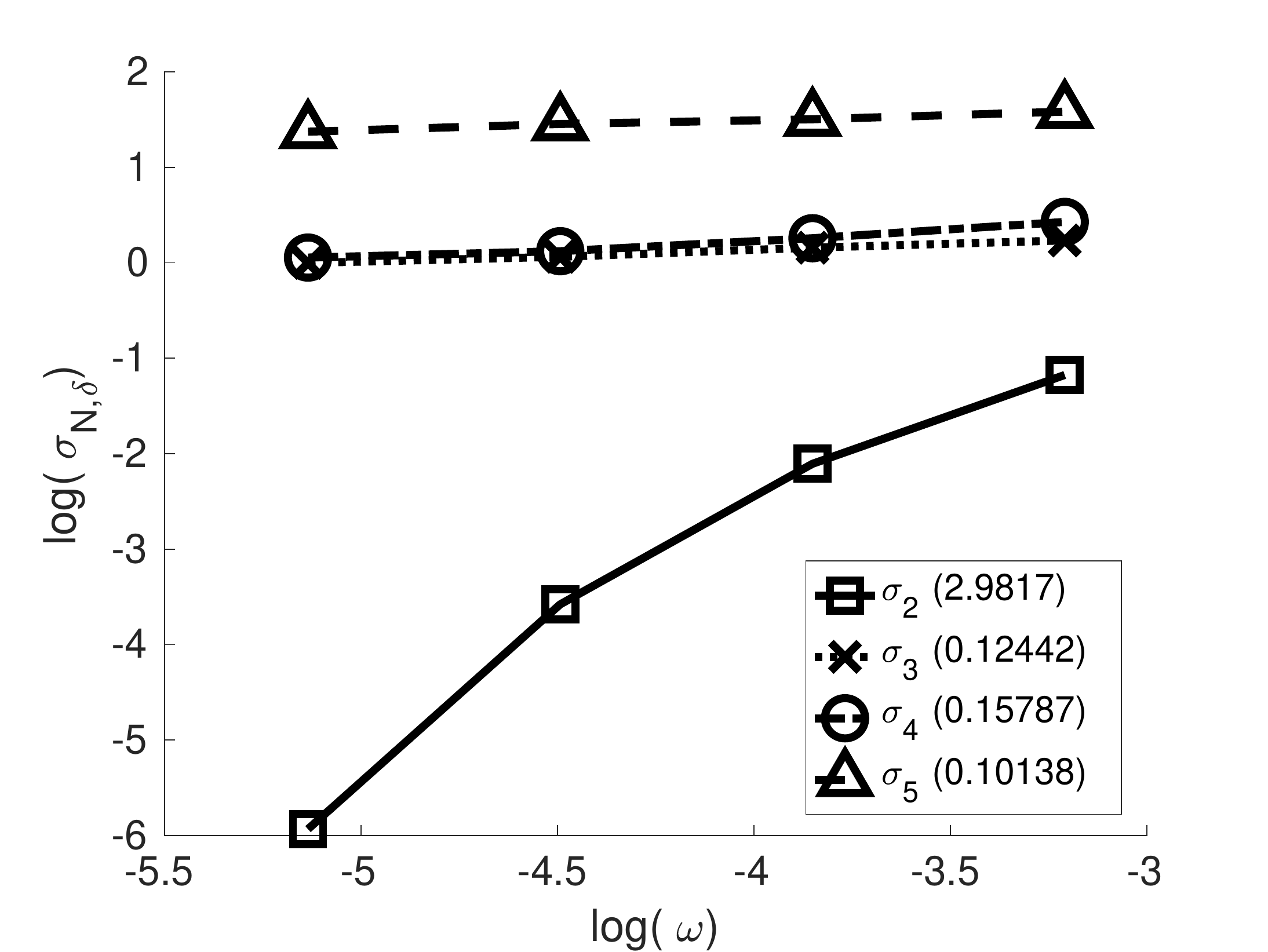}
         }

         \caption{(a) A plot of the mixture density \eqref{exp-mixture-model} for $\omega = 0.25$.
           (b) The first four non-trivial eigenvalues of the discrete graph Laplacian $L_N$
           with parameters $(p,q,r) = (1/2, 2, 1/2)$ as a function of
           the mean parameter $\omega$. Values reported in brackets in the legends
           indicate numerical slope of the lines fitted to the data. (c)  Showing the
           first four non-trivial eigenvalues of $L_N$ with $(p,q,r) = (1, 2, 1)$. (d)
           Showing the same results for parameters $(p,q,r) = (1, 3/2, 1)$. 
         }
         \label{fig:exp-mixture-example}
    \end{figure}

\end{example}
}

\subsection{Outline}
\label{ssec:OUT}

\changed{
The remainder of the paper is organized as follows.
Section~\ref{sec:set} sets up the necessary framework and notation. Section~\ref{sec:sa} contains the precise statements 
of the key results Theorems \ref{thm:L-eps-low-eigenvals} and \ref{thm:geometry-of-fiedler-vector},
relating to  Main Result~\ref{main-result}; 
proofs of these results are postponed to Section~\ref{sec:sap}. 
Numerical results illustrating, and extending  the Main Result~\ref{main-result} and leading to the 
  Conjecture~\ref{conj:SG}
  are presented in Section \ref{sec:num}. Section~\ref{app:AA} 
contains the informal derivation 
of \eqref{general-weighted-Laplacian} 
from the  parameterized family of graph Laplacians \eqref{defLN-intro},
and presents the formal calculations and numerical experiments that were summarized under
Contribution 3. 
Our conclusions are given in Section \ref{sec:con}.
Appendices \ref{sec:diffmaps}, \ref{app:AC}, \ref{app:AB}, and
\ref{sec:proof-isoperim-ineqa} contain, respectively:
connections between the diffusion maps and $\mcl L$;
discussion of function spaces;
the min-max principle; and a weighted
Cheeger inequality.
}




\section{The Set-Up}
\label{sec:set}

In this section we set-up the functional analytic
framework for our theory and numerics. Subsection
\ref{ssec:prelim} describes the notation  and introduces weighted Laplacian
operators in this framework, and Subsection \ref{ssec:pod}
is devoted to our precise formulation of binary clustered
data in the perfect or nearly separated clustered data setting.

\subsection{Preliminaries}
\label{ssec:prelim}

For an open subset $\Omega \subseteq \Z \subset \mbb{R}^d$ with $C^{1,1}$ boundary, consider a  probability density function $\varrho$ satisfying
\begin{equation}
  \label{conditions-on-rho}
\varrho \in C^\infty(\bar{\Omega}), \qquad
  \int_\Omega \varrho(x) dx = 1, \qquad 
  \varrho^-< \varrho(x) < \varrho^+, \qquad \forall x \in \bar{\Omega}, 
\end{equation}
with constants $\varrho^-, \varrho^+ >0$. We also denote  the measure of subsets  $\Omega'$
of $\Omega$
with respect to $\varrho$ with the following notation
\begin{equation}\label{measure-of-sets}
|\Omega'|_\varrho:=\int_{\Omega'} \varrho(x)dx.
\end{equation}

Given a continuous probability density function $\varrho$ as above
with full support on $\Omega \subseteq \Z$
we define the weighted space
 \begin{equation}
  \label{L2-rho-definition}
  L^2(\Omega, \varrho^s) := \left\{ u :  \int_\Omega |u(x)|^2 \varrho(x)^s dx< +\infty  \right\},  
\end{equation}
with inner product 
\begin{equation}
 \label{L2-rho-inner-product}
  \langle u,v \rangle_{\varrho^s} := \int_\Omega u(x) v(x) \varrho^s(x) dx,
\end{equation}
for any $s \in \R.$
This reduces to the standard $L^2(\Omega)$ space 
with norm $\| \cdot \|_{L^2(\Omega)}$ and inner product
$\langle \cdot, \cdot \rangle$ if
  $\varrho = 1$ on $\Omega$.
Furthermore, for $\varrho>0$ a.e. on $\Omega$ and parameters $(p, q, r)\in \mbb R^3$, we define the weighted Sobolev spaces
\begin{equation*}
    H^1(\Omega, \varrho):=\left\{ \frac{u}{\varrho^r} \in L^2(\Omega, \varrho^{p+r}) :
      \| u\|_{H^1(\Omega, \varrho)}:= \Langle u, u \Rangle_V < +\infty
    \right\},
  \end{equation*}
  where the $\langle \cdot, \cdot \rangle_V$ inner product is defined as
\begin{equation}\label{V1-inner-prod}
  \Langle u ,v \Rangle_{V}  := \Langle \nabla \left( \frac{ u}{\varrho^r} \right)
  ,  \nabla \left(  \frac{ v}{\varrho^r} \right)
  \Rangle_{\varrho^q} 
  +  \Langle  \frac{u}{\varrho^r}, \frac{v}{\varrho^r}  \Rangle_{\varrho^{p+r}},
\end{equation}
which is the natural inner product induced by the bilinear form
$\Langle (\L +  \frac{1}{\varrho^{r}}) u,  \frac{v}{\varrho^r}\Rangle_{\varrho^{p+r}}$. We then introduce the following subspaces of
$L^2(\Omega, \varrho^{p+r})$ and $H^1(\Omega, \varrho)$:
\begin{equation*}
  \begin{aligned}
    V^0(\Omega, \varrho)&:= \left\{ \frac{u}{\varrho^r} \in L^2(\Omega, \varrho^{p+r})
      : \left\langle \frac{u}{\varrho^r}, 1 \right\rangle_{\varrho^{p+r}}
      = \left\langle u, \varrho^p \right\rangle= 0 \right\}, \\
    V^1(\Omega, \varrho)&:= \left\{u \in H^1(\Omega, \varrho) :
      \Langle u, \varrho^r \Rangle_{V} = 0 
  \right\} \subset V^0(\Omega, \varrho)\,.
\end{aligned}
\end{equation*}
We use $H^1(\Omega)$ and $V^1(\Omega)$ to denote
the standard $H^1$ space, and its subspace excluding constants,
given by $H^1(\Omega, \mbf{1}_{\Omega})$ and $V^1(\Omega, \mbf{1}_{\Omega})$. The
former coincides with the usual Sobolev spaces while the latter coincides with 
the subspace of $H^1(\Omega)$ consisting of mean zero functions.


In this work, we focus on the class of weighted Laplacian operators 
defined by equation \eqref{general-weighted-Laplacian}, 
for an appropriate density $\varrho$ and 
parameters $(p,q,r)\in \R^3 $. We generally suppress the dependence of $\mcl L$ on
$\varrho$ and the
constants $p,q,r$ for convenient notation and make the choice of these parameters
explicit in our statements. As we show next, the operator $\mcl L$
is positive semi-definite and since the first eigenpair 
$(\sigma_1,\varphi_1)=(0,\varrho^r \mbf{1}_\Omega)$ is known it is convenient 
to work orthogonal to $\varphi_1$ so as to make the operator strictly
positive; in other words, we consider the operator $\mcl L$ on the 
space $V^1(\Omega,\varrho)$.

\begin{lemma}\label{L-rho-is-self-adjoint-symmetric}
  If $\varrho$ satisfies \eqref{conditions-on-rho}, then the
  bilinear form
  \begin{equation}
\label{eq:BLF}
    \Langle \mcl{L} u, v \Rangle_{\varrho^{p-r}}=  \Langle \varrho^q \nabla \left( \frac{u}{\varrho^r} \right), \nabla \left( \frac{v}{\varrho^r} \right) \Rangle,
  \end{equation}
  is symmetric and positive definite on $V^1(\Omega, \varrho) \times V^1(\Omega, \varrho)$. In particular,
  the  operator
  \begin{equation*}
  \mcl L: V^1(\Omega, \varrho) \mapsto V^0(\Omega, \varrho),
\end{equation*}
defined in
the weak sense, is self-adjoint and strictly positive definite and
  the inverse operator
  \begin{equation*}
  \mcl L^{-1}: V^0(\Omega, \varrho) \mapsto V^0(\Omega, \varrho),
\end{equation*}
  exists and is compact. 
\end{lemma}
\begin{proof}
    The fact that $\mcl L$ is self-adjoint and strictly positive on $V^1(\Omega, \varrho)$
    can be verified directly.
   The fact that $\mcl L^{-1}$ is well-defined follows from the
  Lax-Milgram Lemma \cite[Lem.~2.32]{mclean2000strongly}. Compactness 
  follows from Proposition~\ref{V1-compact-embedding}.
\end{proof}
 Following the spectral theorem \cite[Thms.~D.6, D.7]{Evans} we then have:
\begin{proposition}\label{L-has-discrete-spectrum}
  Let $(p,q,r)\in\R^3$, and suppose $\varrho$  satisfies \eqref{conditions-on-rho}. Then $\mcl L: V^1(\Omega, \varrho) \mapsto V^0(\Omega, \varrho)$ has a
discrete spectrum with eigenvalues $0\le\sigma_2\le \sigma_3\le \dots$ 
  and eigenfunctions  $\{\varphi_j\}_{j \ge 2} \in V^1(\Omega, \varrho)$  
  that form an orthogonal basis in both
$V^1(\Omega, \varrho)$ and $V^0(\Omega, \varrho)$. 
Furthermore, we may extend $\mcl L$ to the operator 
$\mcl L: H^1(\Omega, \varrho) \mapsto L^2(\Omega, \varrho^{p-r})$
and include the eigenpair 
$(\sigma_1,\varphi_1)=(0, | \Omega |_{\varrho^{p+r}}^{1/2}\varrho^r \mbf{1}_\Omega)$.
\end{proposition}

\begin{remark}
\label{rem:sde}
Writing $u=\varrho^r u'$ and $v=\varrho^r v'$ we note that the identity
\eqref{eq:BLF} may be written as
\begin{equation*}
    \Langle \varrho^{p-q}\mcl{L}(\varrho^r u'), v' \Rangle_{\varrho^{q}}=  \Langle  \nabla u', \nabla v' \Rangle_{\varrho^q}.
  \end{equation*}
From this we see \cite{pavliotis2014stochastic} that the operator 
$$\mcl{G}:=-\varrho^{p-q} \circ \mcl{L} \circ \varrho^r$$
is the generator of the reversible diffusion process
$$dX_t=-\nabla  \Psi(X_t)dt+\sqrt{2} dB,$$
where $\Psi = -\log (\varrho^q)$, and $B$ denotes a $d$ dimensional Brownian motion. This diffusion process has invariant
measure proportional to $\exp(-\Psi)=\varrho^q$. This observation thus
establishes a connection between the operator $\mcl{L}$ and diffusion processes
which, when $q>0$, concentrate in regions where $\varrho$ is large and sampling density of the data
is high. For a  more detailed discussion on the connections between diffusion maps
and the operators weighted elliptic operators $\mcl L$, see Appendix~\ref{sec:diffmaps}.
\end{remark}

\subsection{Perturbations Of Densities}
\label{ssec:pod}

We now consider a specific setting of a  density $\vrhoo$ that is supported on a
strict subset $\Dp \subset \Z$, consisting of two disjoint sets $\Z^+$ and $\Z^-$. We
then consider a sequence of
probability densities  $\varrho_\eps$ supported on the whole set $\Z$ that
approximate $\vrhoo$. 
In the next two  subsections we outline
our assumptions regarding $\Dp$, $\vrhoo$ and $\varrho_\eps$ and introduce 
weighted Laplacian operators using these densities.

\subsubsection{Assumptions On The Clusters And  Densities}

We begin by introducing a set of assumptions on the domains $\Z, \Dp$,
the density $\vrhoo$, and the approximating sequence of densities $\varrho_\eps$.

\begin{assumption}\label{Assumptions-on-D}
The sets $\Z,\Dp=\Z^+\cup \Z^-\subset \R^d$ satisfy the following:
\begin{enumerate}[(a)]
\item $\Z$  is open, bounded and connected. 
\item $\Dp$ is a subset of $\Z$ consisting of 
two open connected subsets $\Z^+$ and $\Z^-$.
\item $ \Z^\pm$ are disjoint from one another and from 
 $\partial \Z$, the boundary of $\Z$:
$\exists l,l' > 0$ so that
$$
\dist(\Z^+, \Z^-) > l >0, \qquad \text{and} \qquad \dist(\Z^\pm, \partial{\Z}) > l' > 0.
$$
\item $\partial \Z$ and $\partial \Dp$ are at least $C^{1,1}$. 
\end{enumerate}
\end{assumption}

The assumption that $\Z^\pm$ are well separated from $\partial \Z$ in Assumption~\ref{Assumptions-on-D}(c) is not crucial but allows for more convenient presentation of our
results. 
 We think of $\Z^\pm$ as ``clusters'' in the continuum limit.

\begin{assumption}\label{Assumptions-on-rho}
  The density $\vrhoo$ satisfies the following: 
  \begin{enumerate}[(a)]

  \item ({\it Supported on clusters})  $\vrhoo = 0$ on $\Z \setminus \bar{\Dp}$.
     
    \item  ({\it Probability density function})
      $
      \int_{\Dp} \vrhoo(x) dx = 1.
      $
\item ({\it Uniformly bounded within clusters}) $\exists \varrho^\pm >0$ so that 
      $
      \varrho^- \le \vrhoo(x) \le \varrho^+$, for all $ x \in \bar{\Dp}$.      
      \item 
        ({\it Smoothness})
      $\vrhoo \in C^\infty(\bar{\Dp}).$
    \item ({\it Equal sized clusters}) Given  $p,r \ge 0$,
      the density $\vrhoo^{p+r}$ assigns equal mass to $\Z^+$ and $\Z^-$, i.e., 
        $$
        \int_{\Z^+} \vrhoo^{p+r}(x) dx = \int_{\Z^-}\vrhoo^{p+r}(x)dx \,.
        $$
  \end{enumerate}
\end{assumption}
We highlight that Assumption~\ref{Assumptions-on-rho}(b) and (e)
are not crucial to our analysis. Condition (b) is natural  when considering
limits of  graph Laplacian operators defined from data distributed
according to a measure with density $\varrho_0$, but
all of our analysis can be generalized to integrable $\varrho_0$
simply by observing that the eigenfunctions of $\mcl{L}$ are invariant
under scaling of $\varrho_0$ by a constant $\lambda$, whilst the  eigenvalues
scale by $\lambda^{q-p-r}.$
Condition (e) allows for a more convenient presentation with less 
cumbersome notation but can be removed at the price of a lengthier
exposition; see Remark~\ref{what-if-rho-didnt-assign-equal-weight} below.

Given a density $\vrhoo$  satisfying Assumption~\ref{Assumptions-on-rho},
we consider a sequence of densities $\varrho_\eps$ with full support on $\bD$ that  converge to
$\vrhoo$ as $\epsilon \to 0$ in a suitable sense. We have in mind densities $\varrho_\epsilon$ that become more and more concentrated in $\Dp$ as $\epsilon \to 0$. 
  In what follows, we define
\begin{equation}\label{D-eps-definition}
\Omega_\delta := \{ x : \dist(x, \Omega) \le \delta \},
\end{equation}
for any set $\Omega\subseteq \bD$ and denote the Minkowski (exterior) boundary
measure of $\Omega$ as
\begin{equation*}
  | \partial \Omega | := \liminf_{\delta \downarrow 0} \frac{1}{\delta} \left[ | \Omega_\delta| -
  |\Omega| \right].
\end{equation*}
It follows that when $\eps$ is sufficiently small, $\exists \theta >0$ so that 
 \begin{equation}\label{eqn:Depssize}
  |\Omega_\eps\setminus \Omega| \le \theta \eps | \partial \Omega |.
 \end{equation}

\begin{assumption}\label{Assumptions-on-rho-eps-3}
Let $0<L:= \min \dist(\Z^\pm, \partial \Z)$. Then there is 
$\eps_0 \in (0,L/4)$ and constants $K_1, K_2>0$ such that, 
for all $\epsilon \in (0, \eps_0)$, 
the densities $\varrho_\epsilon$ satisfy:
\begin{enumerate}[(a)]
\item ({\it Full support}) $ \text{supp} \varrho_\eps = \bD$.
    \item ({\it Probability density function}) $\int_{\Z} \varrho_\eps(x) dx = 1.$
    \item ({\it Approximation within clusters})$ \exists K_1>0$ so that
      $\| \varrho_\eps - \vrhoo\|_{C^\infty(\bar{\Z}')} \le K_1 \eps$ 
as $\eps\downarrow 0$.
\item ({\it Vanishing outside clusters}) $ \exists K_2>0$ so that
$\varrho_\eps (x) =  K_2 \eps$ for  $ x \in \Z \setminus \Dp_\eps.$
\item ({\it Controlled derivatives}) $\exists K_3>0$ so that 
  \begin{equation*}
    | \nabla \varrho_\epsilon (x) |\le K_3 \epsilon^{-1}, \qquad \forall x \in \Dp_\epsilon \setminus \Dp.
\end{equation*}

\end{enumerate}
\end{assumption}
Once again Assumption~\ref{Assumptions-on-rho-eps-3}(b) is not crucial to our analysis
but is needed to make sure the operator $\mcl L_\eps$ defined in \eqref{Leps-definition}
is the continuum limit of a
graph Laplacian.
As a consequence of Assumptions~\ref{Assumptions-on-rho}(c) and \ref{Assumptions-on-rho-eps-3}(c)-(e), it follows that $\varrho_\eps$ is uniformly bounded above and below inside $\Dp$:
there exist constants $\varrho_{\eps_0}^\pm >0$ so that 
      \begin{equation}\label{uniform-rho-bound}
        \varrho_{\eps_0}^- \le \varrho_\eps(x) \le \varrho_{\eps_0}^+,
        \qquad \forall x \in \bar{\Z}' \text{ and } \forall \eps\in(0,\eps_0)\,.
      \end{equation}
    Note that the upper bound holds on all of $\Z$ as well, whereas the lower bound clearly does not in view of Assumption~\ref{Assumptions-on-rho-eps-3}(d).

    \begin{remark}\label{rem:rho-eps-assumptions}
\changed{The above set of assumptions on $\varrho_\epsilon$ may seem very specific; 
however, the analysis we present is robust to changes in the exact construction
of the perturbed densities so long as the condition that $\varrho_\eps = K_2\epsilon$
  away from the clusters is satisfied. 
For example, given a density $\vrhoo$ we  can always construct a density $\varrho_\eps$ satisfying our assumptions by the procedure outlined in the following example.}
\end{remark}

\begin{example}\label{ex:conv-density-construction}
Consider the
standard mollifier
\begin{equation}
  \label{standard-mollifier}
  g(x) := \left\{
    \begin{aligned}
      &C^{-1} \exp\left( - \frac{1}{1 - |x|^2} \right) \quad &&|x| \le 1, \\
      & 0 \quad &&|x| > 1.   
    \end{aligned}
    \right., \qquad g_\epsilon(x) :=\frac{1}{\epsilon^d} g\left(\frac{x}{\eps}\right), 
  \end{equation}
  where $C = \int_{|x| \le 1} \exp\left( - \frac{1}{1 - |x|^2}\right) dx$ is a normalizing
  constant. Now, given $\epsilon >0$ and the density $\varrho_0$ (extended
  by zero to all of $\Z$)  define
  \begin{equation}\label{rho-eps-explicit-construction}
    \varrho_\epsilon(x) := \frac{1}{K_{\epsilon}} \Big(\epsilon + g_\epsilon \ast \vrhoo (x)\Big), \qquad K_\epsilon := \int_{\Z} \Big( \epsilon + g_\epsilon \ast
    \vrhoo(x) \Big) dx.
  \end{equation}
  One can directly verify that the above construction of $\varrho_\epsilon$ satisfies
  Assumption~\ref{Assumptions-on-rho-eps-3}.
\end{example}

\subsubsection{Assumptions On The Weighted Laplacian Operators}\label{sec:Lops}
With the densities $\vrhoo$ and $\varrho_\eps$ identified we then consider the  operators
$\mcl L_0$ and $\mcl L_\eps$ in the same form as  \eqref{general-weighted-Laplacian} as follows: 
\begin{equation}\label{L0-definition}
\left\{
    \begin{aligned}
      &\mcl L_0 u  := -\frac{1}{\vrhoo^p} {\rm div} \left( \vrhoo^q \nabla \left( \frac{u}{\vrhoo^r} \right) \right),  &&\text{ in } \Dp\\
      &\vrhoo^q \frac{\partial}{\partial n} \left( \frac{u}{\vrhoo^r} \right) = 0, 
      &&\text{ on } \partial \Dp.
  \end{aligned}
\right.
\end{equation}
Similarly for $\varrho_\eps$,
\begin{equation}\label{Leps-definition}
\left\{
    \begin{aligned}
      &\mcl L_\eps u  := -\frac{1}{\varrho_\eps^p} {\rm div} \left( \varrho_\eps^q
        \nabla \left( \frac{u}{\varrho_\eps^r} \right) \right),  &&\text{ in } \Z\\
      &\varrho_\eps^q \frac{\partial}{\partial n} \left( \frac{u}{\varrho_\eps^r} \right) = 0, 
      &&\text{ on } \partial \Z.
  \end{aligned}
\right.
\end{equation}
By Lemma~\ref{L-rho-is-self-adjoint-symmetric} and Proposition~\ref{L-has-discrete-spectrum}, the operators
\begin{equation*}
\mcl L_0: H^1(\Dp, \varrho_0) \mapsto
L^2(\Dp, \varrho_0^{p-r})
\qquad \text{ and } \qquad 
\mcl L_\eps: H^1(\Z, \varrho_\eps) \mapsto
L^2(\Z, \varrho_\eps^{p-r})
\end{equation*}
are self-adjoint and positive semi-definite. Furthermore,
these operators have positive, real, discrete eigenvalues after the first eigenvalue, which is zero.
For $j =1,2,3,...$
let $\sigma_{j,0}$ and  $\sigma_{j,\epsilon}$ denote the eigenvalues of $\mcl L_0$ 
and $\mcl L_\eps$ respectively (in increasing order  and accounting for repetitions)
and let $\varphi_{j,0}$ and  $\varphi_{j,\epsilon}$ denote the corresponding
eigenfunctions.  Recall that
$\varphi_{1,0}= | \Z'|_{\varrho_0^{p+r}}^{-1/2} \varrho_0^r\mbf{1}_{\Z'}$
and $\varphi_{1,\eps}=  | \Z|_{\varrho_\eps^{p+r}}^{-1/2}\varrho_\eps^r\mbf{1}_{\Z}$, both with corresponding zero
eigenvalues. Since we are interested in the eigenpairs for $j\ge 2$
it is more convenient to work orthogonal to the first eigenfunctions from now on, that is, to consider the spaces $V^1(\Z', \vrhoo)$ and $V^1(\Z, \varrho_\eps)$ respectively.
Thus, we consider the pairs $\{\sigma_{j,0}, \varphi_{j,0}\}$ and $\{\sigma_{j,\epsilon}, \varphi_{j,\epsilon}\}$ for $j\ge 2$ that
solve the  eigenvalue problems
\begin{align}
  &\Langle \vrhoo^{q} \nabla \left( \frac{\varphi_{j,0}}{\vrhoo^r} \right),
  \nabla \left( \frac{v}{\vrhoo^r}\right) \Rangle =  \sigma_{j,0} \Langle \vrhoo^{p-r} \varphi_{j,0} , v \Rangle,
  \qquad \varphi_{j,0}, v \in V^1(\Z', \vrhoo), \label{weak-EVP-L-rho} 
    \end{align}
    and
    \begin{align}
  &\Langle \varrho_\epsilon^{q} \nabla \left( \frac{\varphi_{j, \epsilon}}{\varrho_\epsilon^r} \right),
  \nabla \left( \frac{v}{\varrho_\epsilon^r}\right) \Rangle =
   \sigma_{j,\epsilon}  \Langle \varrho_\epsilon^{p-r} \varphi_{j, \epsilon} , v \Rangle, \qquad \varphi_{j,\epsilon},v \in V^1(\Z, \varrho_\eps).\label{weak-EVP-L-rho-eps}
    \end{align}
    Throughout the article we take  $\varphi_{j,0}$ and $\varphi_{j,\eps}$ to be normalized
      in $L^2(\Z', \vrhoo^{p-r})$ and $L^2(\Z, \varrho_\eps^{p-r})$ respectively.

We  collect some  definitions and notation concerning
the spectral gaps of the  operators $\mcl L_0$ and $\mcl L_\eps$
and Poincar{\'e} constants 
on certain subsets of $\Z$ and $\Z'$; these are used throughout the article.


\begin{definition}[Standard spectral gap $\Lambda_\Delta$]\label{Assumption-standard-gap}
  We say that the standard spectral gap condition holds for a subset $\Omega$ of $\Z$
  if  the Poincar{\'e} inequality is satisfied on $\Omega$ with an optimal
   constant $\Lambda_\Delta(\Omega) >0$, i.e., 

    \begin{align}\label{exterior-spectral-gap}
    \int_{\Omega} \left|\nabla u\right|^2 dx \ge \Lambda_\Delta(\Omega) \int_{\Omega}| u|^2 dx, \qquad \forall u \in V^1(\Omega).
    \end{align}
\end{definition}
We also define a certain
$\vrhoo$ weighted version of the above spectral gap definition.

\begin{definition}[$\mcl L_0$  spectral gap $\Lambda_0$]\label{Assumption-cluster-gap}
  We say that the $\mcl L_0$ spectral gap condition holds for a
  subset $\Omega$ of $\Z'$ if the following weighted Poincar{\'e} inequality
  is satisfied with  an optimal constant $\Lambda_0(\Omega) >0$  
     \begin{equation}\label{indivisibility-L-rho}
      \int_{\Omega} \varrho_0^{q}\left|  \nabla \left( \frac{u}{\varrho_0^r} \right) \right|^2 dx
      \ge \Lambda_0(\Omega)
      \int_{\Omega} \left| \frac{u}{\varrho_0^r} \right|^2 \varrho_0^{p + r} dx, \qquad \forall u \in V^1(\Omega, \varrho_0).
     \end{equation}
\end{definition}

 Observe that condition
  \eqref{indivisibility-L-rho} is equivalent to the
  assumption that the second eigenvalue of the operator $\mcl L_0$  restricted
  to the set $\Omega$  is bounded away from zero. Finally, we
  define the notion of a uniform spectral gap for   $\mcl L_\eps$.
  
  \begin{definition}[$\mcl L_\eps$ uniform spectral gap $\Lambda_{\eps}$]\label{Assumption-subset-gap}
    Given $\eps_0 >0$ 
    we say that the $\mcl L_\eps$ uniform spectral gap condition holds for
     a subset $\Omega$ of $\Z$ if 
     $\forall \eps \in (0, \eps_0)$ there exists an optimal constant
    $\Lambda_{\eps}(\Omega)>0$  so that
  \begin{equation}\label{spectral-gap-on-subsets}
    \int_{\Omega}  \left| \nabla \left( \frac{u}{\varrho_\epsilon^r} \right) \right|^2 \varrho_\epsilon^q dx
    \ge  \Lambda_{\eps}(\Omega)
    \int_{\Omega}  \left| \frac{u}{\varrho_\eps^r} \right|^2 \varrho_\epsilon^{p+r} dx, \qquad
    \forall u \in V^1(\Omega, \varrho_\eps).
  \end{equation}
\end{definition}
  
\begin{remark}\label{rmk:gaplimit}
To connect the spectral gaps of $\mcl{L}_{\eps}$ restricted to the clusters $\Z^\pm$ with the spectral gaps of the limiting operator $\mcl{L}_0$ on these clusters, 
one can make use of the knowledge that $\varrho_\eps$ converges to $\vrhoo$ on $\Z^\pm$ by Assumption~\ref{Assumptions-on-rho-eps-3}(c). More precisely, let us suppose \eqref{indivisibility-L-rho} holds. We show in Theorem~\ref{t:0} that $\sigma_{1,0}=\sigma_{2,0}=0$ and $ \sigma_{3,0}>0$. Since $\varrho_\eps(x)$ converges to $\vrhoo(x)$ pointwise for every $x\in \Dp$, this spectral gap translates to $\mcl{L}_{\eps}$ for small enough $\eps$ within the set $\Dp$, and so we can assert \eqref{spectral-gap-on-subsets} for $\Omega= \Z^\pm$. The assumption that the restriction of $\mcl L_{\eps}$ to $\Z^\pm$ has a spectral gap is related to the indivisibility parameter in the context of well-separated mixture models of \cite{NGTFHBH}.
\end{remark}

 \begin{remark}\label{rmk:constantgap}
   Note that for subsets $\Omega$ where $\varrho_\eps$ is constant, say $\varrho_\eps(x)=c_\eps$,
   condition \eqref{spectral-gap-on-subsets} reduces to a spectral gap of the standard Laplacian restricted to $\Omega$, with the constant
   $\Lambda_\Delta$ in \eqref{exterior-spectral-gap} replaced by $\Lambda_\Delta c_\eps^{p+r-q}$. This
    becomes important when investigating the behavior of
   $\mcl{L}_{\eps}$ away from the clusters $\Z^\pm$ and is precisely the reason why we obtain a condition on the sign of $q - p -r$ in our main theorems, see for example Theorem~\ref{thm:L-eps-low-eigenvals}.
  \end{remark}

\section{Spectral Analysis: Statement Of Theorems}\label{sec:sa}

In this section we describe the spectral properties 
of the operators $\mcl L_0$ and $\mcl L_\epsilon$ in relation to
certain  geometric features in the data summarized in the
densities $\varrho_0$ and $\varrho_\eps$.
We present precise statements of our key theoretical results, 
postponing the proofs to Section~\ref{sec:sap}.
We define, and then identify, gaps between 
the second and third eigenvalues of $\L_\eps$ together with
concentration properties of the second eigenfunction $\varphi_{2,\eps}$
as $\epsilon\downarrow 0$. More precisely, we show that the nature
and existence of a spectral gap is dependent upon the choice of $p,q$ 
and $r$ and, under general conditions,
concentration properties of $\varphi_{2,\eps}$ are
 directly related to concentration properties 
 of  $\varrho_\epsilon$.
 \changed{
In Subsection
\ref{ssec:psc} we consider the
perfectly clustered case pertaining the operator $\L_0$
while Subsection \ref{ssec:nsc} perturbs 
this setting and considers the nearly clustered case corresponding to the operator $\L_\eps$.
}


\subsection{Perfectly Separated Clusters}
\label{ssec:psc}

Recall the concept of perfectly separated clusters from the introduction,
the density $\varrho_0$ and the resulting operator $\mcl L_0$ defined on
$\mcl{Z}'$. 
The corresponding low-lying spectrum of $\mcl L_0$ can be characterized explicitly: 

  \begin{theorem}[Low-lying spectrum of $\mcl L_0$ and Fiedler vector]\label{t:0}
    \changed{Suppose $(p,q,r)\in\R^3$ and 
      Assumptions~\ref{Assumptions-on-D} and \ref{Assumptions-on-rho} hold.
    Then $\mcl L_0$ is positive semi-definite and self-adjoint
on the weighted Sobolev space $H^1(\Z', \varrho_0)$. Denote its eigenvalues by  $\sigma_{1,0} \le \sigma_{2,0} \le
\cdots$ with corresponding eigenfunctions $\varphi_{j,0}$, $j\ge 1$. Then it  holds that:}
  \begin{enumerate}[(i)]
\item The first eigenpair is given by
\begin{align*}
\sigma_{1,0}=0\,,\qquad \varphi_{1,0}= \frac{1}{| \Z'|_{\varrho_0^{p+r}}^{1/2}}\varrho_0^r(x)\mbf{1}_{\Z'}(x),\qquad \forall x\in\Z'\,.
\end{align*}

\item The second eigenpair is given by
\begin{align*}
  \sigma_{2,0}=0\,,\qquad
  \varphi_{2,0}= \frac{1}{| \Z'|_{\varrho_0^{p+r}}^{1/2}} \varrho_0^r(x)\left(\mbf{1}_{\Z^+}(x)-\mbf{1}_{\Z^-}(x)\right),\qquad \forall x\in\Z'\,.
\end{align*}

  \item $\mcl L_0$ has a uniform spectral gap, i.e.,  $\sigma_{3,0}>0$.

 \end{enumerate}
\end{theorem}

\changed{
Part (i,ii) of Theorem~\ref{t:0} can be verified directly by substituting
$\varphi_{1,0}$ and $\varphi_{2,0}$ into \eqref{weak-EVP-L-rho}. Then
it remains to show (iii), the lower bound on the third eigenvalue $\sigma_{3,0}$
which follows from Proposition~\ref{low-lying-spectrum-K-eps-tau}, 
stating that $\L_0$ has a spectral gap on $\Z'$ so long as its
restriction to each of the clusters $\Z^\pm$ has a spectral gap.
Since $\varrho_0$ is bounded away from zero on the clusters
this condition holds since $\Z^\pm$ are assumed to be
connected sets of positive Lebesgue measure. 
}


\subsection{Nearly Separated Clusters}
\label{ssec:nsc}

We now turn our attention to the densities $\varrho_\epsilon$ that have full support on $\bD$, but concentrate around  $\Dp$ as $\epsilon$ decreases. 
This represents the practical setting 
where we do not have perfect clusters 
$\Z^\pm$ and so the density $\vrhoo$ is perturbed. 
A central question here is whether
 the second eigenpair $\{ \sigma_{2,\eps},
\varphi_{2, \eps}\}$ of
$\mcl L_\eps$ exhibits behavior similar to the second eigenpair $\{ \sigma_{2,0}, \varphi_{2,0}\}$ of $\mcl L_0$ as $\varrho_\eps \to \vrhoo$; that is,
in the limit as we approach the ideal case of perfect clusters $\Z^\pm$.

In order to establish such a result we first need to approximate
the first three eigenvalues of $\L_\eps$:
\changed{
\begin{theorem}[Low-lying eigenvalues of $\L_\eps$]\label{thm:L-eps-low-eigenvals}
Let $(p,q,r)\in\R^3$ satisfy $p+r>0$ and $q>0$, and suppose  Assumptions~\ref{Assumptions-on-D}, \ref{Assumptions-on-rho},
  and \ref{Assumptions-on-rho-eps-3} hold and that $\Lambda_\Delta(\Z \setminus \Z'_{\eps_0})>0$
  for a sufficiently small $\eps_0 >0$.
 Then  the following holds for all $(\eps, \beta) \in (0, \eps_0) \times (0,1)$:
 \begin{enumerate}[(i)]
 \item  The first eigenpair is given by
   \begin{align*}
\sigma_{1,\eps}=0\,,\qquad 
\varphi_{1,\eps}= \frac{1}{| \Z|_{\varrho_\eps^{p+r}}^{1/2}}\varrho_\eps^r(x)\mbf{1}_{\Z}(x)\qquad \forall x\in\Z\,.
   \end{align*}

 \item The second eigenvalue $\sigma_{2,\eps}$  tends to zero as $\eps \to 0$,
   $$
0 \le \sigma_{2,\epsilon} \le \Xi_1 \epsilon^{{q-\beta}},
$$
with $\Xi_1 >0$ a uniform constant independent of $\eps$.

\item The third  eigenvalue behaves differently depending on the $(p,q,r)$ parameters:
  \begin{itemize}
   \item if    $q>p+r$,
    then  $ \exists \,\Xi_2,\Xi_3>0$
  independent of $\eps$ such that,
    $$
    \Xi_2 \eps^{2(q-p-r)}\le\sigma_{3,\epsilon} \le \Xi_3 \eps^{ q - p-r - 2\beta}\,,
    $$
    and so $\mcl L_\eps$ does not have a uniform spectral gap on $\Z$;
    
  \item if  $q = p+r$
    then 
  there exist  constants $\Xi_4,\Xi_5 >0$, independent of $\eps$, so that
    $$
    \Xi_4\le{\sigma}_{3, \epsilon} \le \Xi_5,
    $$
    and so
     $\mcl L_\eps$ has a uniform  spectral gap on $\Z$;
\item if $q<p+r$, then there exist constants $\Xi_6, \Xi_7>0$, independent of $\eps$, so that
    $$
   \Xi_6 \eps^{p+r-q}\le {\sigma}_{3, \epsilon} \le \Xi_7\,.
    $$
  \end{itemize}

 \end{enumerate}
  
\end{theorem}

Once again
part (i) can be verified directly by substituting $\varphi_{1,\eps}$ in \eqref{weak-EVP-L-rho-eps}.
Part (ii)
is a consequence of Proposition~\ref{prop:speceps} that obtains an upper bound on $\sigma_{2,\eps}$
using a perturbation argument. More precisely, 
 we first construct an explicit approximation $\varphi_{F,\eps}$ of $\varphi_{2,\eps}$
as a smoothed out version of $\varphi_{2,0}$,
normalized in $V^1(\Z, \varrho_\eps)$ and supported on a set slightly larger than $\Dp$.
We choose a parameter $\beta>0$ such that $|\nabla \varphi_{F,\eps}|$ is controlled by $\eps^{-\beta}$ at the boundary of $\Dp$. This is precisely the parameter $\beta$ appearing in
Theorem~\ref{thm:L-eps-low-eigenvals}. By construction, we then have that $\varphi_{F,\eps}$
converges to the normalization of $\varphi_{2,0}$ as $\eps\to 0$.
Using this approximate eigenfunction as well as $\varphi_{1,\eps}$ from part (i) in
the min-max principle (see Proposition~\ref{thm:max-min})
yields the desired upper bound on $\sigma_{2,\eps}$.

Part (iii) requires more elaborate
arguments as outlined in Subsection~\ref{ssec:t35}. The lower bounds on $\sigma_{3,\eps}$
follow from Proposition~\ref{spectral-gap-general} that is in turn based on a generalization of
Cheeger's inequality (see Proposition~\ref{thm:cheeger}). The upper bounds
follow from Proposition~\ref{spectral-gap-unbalanced-case} the proof of which uses
similar ideas as for the upper bound of $\sigma_{2,\eps}$, applying the min-max principle but with a
different candidate eigenfunction.}

Several interesting conclusions can be drawn from our arguments in Subsection~\ref{ssec:nscp}
aimed at proving  Theorem~\ref{thm:L-eps-low-eigenvals}. The
existence of spectral gaps for $\mcl{L}_{\eps}$ inside the clusters and away from the
clusters separately allows us to formally deduce bounds on the low-lying spectrum. Consider the set
\begin{equation*}
\Dp_\epsilon := \{ x : \dist(x, \Dp) \le \epsilon \},
\end{equation*}
and suppose that for some  fixed $\eps_0 >0$, we have $\Lambda_{\Delta}(\Z \setminus \Z'_{\eps_0}) >0$,
that is, the standard Laplacian has a spectral gap away from the clusters according to Definition~\ref{Assumption-standard-gap}.
Since $\varrho_\eps(x) = K_2 \eps$ for $x \in \Z \setminus \Z'_{\eps_0}$, 
we have for all 
 $u \bot \mbf{1}_{\Z \setminus \Z_{\eps_0}'}$ in $V^1(\Z \setminus \Z_{\eps_0}')$
$$
(K_2 \eps)^{2r- q}\int_{\Z \setminus \Z_{\eps_0}'} \left| \nabla \left( \frac{u}{\varrho_\eps^r} \right) \right|^2
\varrho_\eps^{q} dx
\ge \Lambda_{\Delta}(\Z \setminus \Z'_{\eps_0}) (K_2 \eps)^{r-p} \int_{\Z \setminus \Z_{\eps_0}'} \left|\frac{u}{\varrho_\eps^r} \right|^2 \varrho_\eps^{p+r} dx\,.
$$
This simple calculation shows that $ \Lambda_{\eps}(\Z \setminus \Z'_{\eps_0}) = \mcl O( \eps^{q - p -r})$,
and so the
existence of a uniform $\mcl L_\eps$ spectral gap away from the clusters is dependent
on the relation between $q$ and $p+r$, in fact we need  $q \le p +r$ 
to ensure $ \Lambda_{\eps}(\Z \setminus \Z'_{\eps_0}) >0$ independent of $\eps$ which is
in line with 
the conditions in Theorem~\ref{thm:L-eps-low-eigenvals}(iii).

\changed{
Combining parts (ii, iii) of Theorem~\ref{thm:L-eps-low-eigenvals} yields
the following corollary concerning the existence of uniform or ratio gaps in
the spectrum of $\L_\eps$ depending on $(p,q,r)$. This corollary is a
detailed statement of Main Result~\ref{main-result}(iii).

\begin{corollary}[Spectral ratio gap when $q \neq p +r $]\label{cor:ratiogap}

Suppose that the conditions of Theorem~\ref{thm:L-eps-low-eigenvals} are satisfied and that  
$q \neq p+r$. Then the following holds for all $(\eps, \beta) \in (0, \eps_0) \times (0,1)$:

      \begin{enumerate}[(i)]
      \item if $q > p +r$  then
 there exists  a constant $\Xi_1 >0$ 
 independent of $\eps$, so that 
    \begin{equation*}
     \frac{\sigma_{2,\eps} }{ \sigma_{3,\eps}} \le \Xi_1
        \eps^{2(p+r)-q - \beta};
      \end{equation*}

    \item if $q< p +r$ then
there exists  a constant $\Xi_2 >0$ 
independent of $\eps$,
\begin{equation*}
     \frac{\sigma_{2,\eps} }{ \sigma_{3,\eps}} \le \Xi_2
        \eps^{2q - p -r - \beta}.
      \end{equation*}
      
      \end{enumerate}
  \end{corollary}

  We note that while this corollary suggests that there may be no spectral ratio  
gap when $q > 2(p +r)$ or $2q < p +r$, our numerical experiments in 
Section~\ref{ssec:bp} (and in particular Tables~\ref{tab:balanced-r-eq-p} to \ref{tab:unbalanced-qlepr})
suggest that these bounds on the ratio gaps are not sharp
due to the fact that our lower bounds on $\sigma_{3,\epsilon}$ from
Theorem~\ref{thm:L-eps-low-eigenvals}(iii)  can be improved to 
match the upper bounds when $q \neq p +r$. We then conjecture that, when $q> p+r$, 
$    \frac{\sigma_{2,\eps} }{ \sigma_{3,\eps}} \le \Xi_1
\eps^{p + r - \beta},$  and when $q < p +r $
we have $\frac{\sigma_{2,\eps} }{ \sigma_{3,\eps}} \le \Xi_2 \eps^{q - \beta}$, as
summarized in Conjecture~\ref{conj:SG} in Subsection~\ref{sec:contribution-2}.

  Finally with the spectral gap results established we can characterize the geometry
  of the second eigenfunction $\varphi_{2,\eps}$ and show that
  as $\eps \downarrow 0$ this eigenfunction is nearly aligned with the second eigenfunction $\varphi_{2,0}$ of $\mcl L_0$ for certain choices of $(p,q,r)$.

\begin{theorem}[Geometry of the second eigenfunction $\varphi_{2,\eps}$]
  \label{thm:geometry-of-fiedler-vector} Suppose  the conditions of Theorem~\ref{thm:L-eps-low-eigenvals}
  are satisfied. 
  Then there exists $\Xi, \eps_0>0$ so that $\forall (\eps, \beta) 
  \in (0, \eps_0)\times(0, 1)$
\begin{equation*}
  \left| 1 -  \left\langle \frac{\varphi_{2,\eps}}{\varrho_\eps^r},
    \frac{\bar \varphi_{2,0}}{\varrho_\eps^r} \right\rangle^2_{\varrho^{p+r}_\eps} \right| 
\le \Xi \eps^{\min\{\frac{1}{2},\frac{p+r}{2}, -|q-(p+r)|+ \min\{q,p+r\}-\beta \}}     \,.
\end{equation*}
where $\bar\varphi_{2,0}$ denotes the normalization of $\varphi_{2,0}$ in $L^2(\Z,\varrho_\eps^{p-r})$.
\end{theorem}

We prove this theorem in Subsection~\ref{ssec:t39} by bounding
the difference between $\varphi_{2,\eps}$ and $\varphi_{F, \eps}$ in Proposition~\ref{varphi-F-is-close-to-fiedler-vector} and then the difference between $\bar{\varphi}_{2,0}$ and $\varphi_{F, \eps}$
in Proposition~\ref{varphi-F-is-close-to-varphi-2} and invoking the triangle inequality.
  Note that the  above bound blows up if $2q < p +r$ in the unbalanced case where $q < p +r$ and
  if $2(p +r) < q$ in the unbalanced case where $q > p +r$. Put simply, if the difference between
  $q$ and $p+r$ is too large then we may lose convergence of the second eigenfunctions.
  However, we also expect these conditions are not sharp since they 
rely on our lower
bounds on $\sigma_{3,\eps}$ in Theorem~\ref{thm:L-eps-low-eigenvals}(iii) that we conjectured can be
sharpened above. Theorem~\ref{thm:geometry-of-fiedler-vector} is a detailed statement of
Main Result~\ref{main-result}(ii).
}

\changed{
\begin{remark} \label{rem:add}
  Two concrete messages follow from Theorems~\ref{thm:L-eps-low-eigenvals} and
    \ref{thm:geometry-of-fiedler-vector}:
(1) Theorem \ref{thm:L-eps-low-eigenvals}(iii) tells us that particular care is needed when looking for
a spectral gap characterizing the number of  clusters if $q\neq p+r$
as the gap may only be
manifest in ratio form, not absolutely, leading to potential
overestimation of the number of clusters; 
(2) Theorem \ref{thm:geometry-of-fiedler-vector} tells us the form  and geometry of the Fiedler vector
which characterizes the two clusters, and its dependence on $\varrho_0$
and on $\eps$; whether or not the problem is balanced determines whether 
the Fielder vector is approximately piecewise constant, or whether it
exhibits smoother transitions across the data. These two observations
may be useful to practitioners when interpreting graph Laplacian based
 analysis of large data sets.
\end{remark}
}

\section{Numerical Experiments In The Continuum}\label{sec:num}
In this section we exemplify, and extend, the main theoretical results
stated in the previous section. In Subsections \ref{ssec:b}
and \ref{ssec:bp} 
we study binary clustered data. The numerical results in these
subsections highlight the effects of the parameters $(p,q,r)$
on spectral properties: Subsection \ref{ssec:b} addresses the
balanced case where $q=p+r$ and Subsection \ref{ssec:bp} the
unbalanced case where $q>p+r$. 
In Subsection \ref{ssec:3c} we also
extend the main theoretical results by considering data comprised
of three clusters and five clusters, showing that the intuition from the binary case
extends naturally to more than two clusters.

Our numerical simulations in the binary, unbalanced case
extend the main theoretical results as they demonstrate the
 spectral ratio gap of \changed{Corollary~\ref{cor:ratiogap}}, arising when
$q>p+r$ is indeed of $\mcl O(\eps^{p+r})$ and\changed{ when $q < p +r$ is of $\mcl O(\eps^{q})$
suggesting the lower bound on $\sigma_{3,\eps}$ can be sharpened.}

We proceed by outlining the setting of the numerical experiments. 
Consider the eigenvalue problem (\ref{weak-EVP-L-rho-eps}) :
  \begin{equation}
\label{weak-EVP-L-rho-eps2}
\Langle \varrho_\epsilon^{q} \nabla \left( \frac{\varphi_{j, \epsilon}}{\varrho_\epsilon^r} \right),
  \nabla \left( \frac{v}{\varrho_\epsilon^r}\right) \Rangle =
  \sigma_{j,\epsilon} \Langle \varrho_\epsilon^{p+r} \frac{\varphi_{j, \epsilon}}{\varrho_\epsilon^r},\frac{v}{\varrho_\epsilon^r}  \Rangle, \qquad \varphi_{j,\epsilon},v \in V^1(\Z, \varrho_\eps).
\end{equation}
Our numerics are all performed in dimension $d=2.$
We solve this  by the finite element method using the
FEniCS software package \cite{fenics}.
We work with the variables 
$\varphi_{j, \epsilon}/\varrho_\epsilon^r$ 
and $v/\varrho_\epsilon^r$, rather than directly with
$\varphi_{j, \epsilon}$ and $v$, and discretize these $\varrho_\epsilon^r$
scaled variables using the standard linear finite element basis functions 
in $H^1(\Z)$. We approximate $\varrho_\epsilon$ using
quadratic finite element basis functions. 
Throughout  we take
$\Z \equiv (-1,1) \times (-1,1)$.
We consider $\epsilon$ in the range  $(1/1280, 1/10)$. 
For each value of $\epsilon$, we approximate 
the eigenvalue problem \eqref{weak-EVP-L-rho-eps2} using a
mesh of $1.28 \times 10^6$ triangular elements defined 
on a uniform grid of $800 \times 800$ nodes.
This finite element discretization leads to a generalized
matrix eigenvalue problem which is solved using a Krylov-Schur eigenvalue solver
in PETSc \cite{petsc} with a tolerance of $10^{-9}$.

Throughout this section
we use densities of the form 
\begin{equation}\label{numerics-rho-def}
  \varrho_\eps(s) = C^{-1} \left( \epsilon + \sum_{i = 1}^{K} \frac{{\rm erf}
      \big( \epsilon^{-1}(\theta_i - |s-c_{i}|) \big)}{4 \pi \theta_i^2} \right)\,,
  \qquad \forall s \in \Z,
\end{equation}
where $|\cdot|$ is the two dimensional Euclidean norm,
$K$ is the number of circular clusters, $c_i$ denotes  
the $i^{th}$ cluster center,  $\theta_i$ the $i^{th}$ cluster 
radius, and $C$ is a normalizing parameter to make sure that
$\varrho_{\eps}$ is a probability distribution.
In Subsections~\ref{ssec:b} and \ref{ssec:bp} we consider two clusters
with parameters
$c_1 = (-0.5,0.0)$, $\theta_1=0.25$, $c_2 = (0.5,0.3)$, and $\theta_2=0.25$ as shown in Figure~\ref{fig:density-plots-two-clusters}(a).
In Subsection~\ref{ssec:3c} we consider three and five clusters
adding the point 
 $c_3 = (0.4,-0.5)$ with radius $\theta_3=0.15$, to make
three clusters, and then adding $c_4 = (-0.35,0.65)$
and $c_5 = (-0.6,-0.6)$ with radii $\theta_4=0.20$ and 
$\theta_5=0.15$, to generate five clusters.
 We plot the resulting densities in Figure~\ref{fig:density-plots-three-five-clusters}.


\subsection{Binary Balanced Case: $q = p + r$}
 \label{ssec:b}
In Figure~\ref{fig:balanced01}(a) we plot $\sigma_{2, \eps}$ in
the balanced case $r=p$, $q = p+r$ and $p \in [0.5,2]$. For a given value of $p$ each symbol denotes the numerical approximation to $\sigma_{2,\eps}$, and the line denotes the best fit determined via linear regression; in the
regression we only use data from $\eps \le 0.025$ as consistent asymptotic
behavior for $\eps \downarrow 0$ is observed in this regime. 
Theorem~\ref{thm:L-eps-low-eigenvals}(ii) predicts that $\sigma_{2,\eps} = \mcl O( \eps^{q - \beta})$
for arbitrarily small $\beta >0$. Then we expect to observe a slope of  approximately $2p$ for each
set of simulations. We report the numerical slopes in brackets in the legend of Figure~\ref{fig:balanced01}(a), and compare the numerical slopes to the 
analytic prediction in the first four rows 
of Table~\ref{tab:balanced-r-eq-p}.

In Figure~\ref{fig:balanced01}(b), we plot the ratio $\sigma_{2,\epsilon}/\sigma_{3,\epsilon}$
for different values of $\eps$. By Corollary~\ref{cor:ratiogap}
we expect $\sigma_{3, \eps}$ to be uniformly bounded away from zero
implying that $\sigma_{2,\eps}/\sigma_{3,\eps}  =  \mcl O(\eps^{q- \beta})$
and so the numerical slopes  in Figure~\ref{fig:balanced01}(b) should be close to $2p$ 
We compare the numerical slopes to the analytic slopes for the  spectral ratio gap in
the first four rows 
of Table~\ref{tab:balanced-r-eq-p}.

In Figure \ref{fig:balanced01}(c,d) we repeat the above study of the second and third
   eigenvalues for the balanced case $q = p +r$
   but this time we fix  $r=0.5$  and vary $p \in (0.5,2)$.
   We see similar results to Figure~\ref{fig:balanced01}(a,b) in that the numerical slopes
   are in good agreement with the predicted slopes  of $q = p +r$. We
   compare the numerical and analytic slopes for  
   this experiment in the last three rows 
of  Table~\ref{tab:balanced-r-eq-p}.

In summary we note that, in this binary balanced
setting the numerical experiments match the theory,
quantitatively. 
 The slopes are less accurate for 
higher values of $p$. We attribute this to the smaller 
values of the eigenvalues in these cases, which are
evaluated with less numerical precision.

\subsection{Binary Unbalanced Case: $q > p +r $}
\label{ssec:bp}
We now turn our attention to the spectrum of $\mcl L_\eps$ when $q > p +r$.
In Figure~\ref{fig:unbalanced01}(a, b)
we plot the
second eigenvalue $\sigma_{2,\eps}$ and the ratio
$\sigma_{2,\eps}/\sigma_{3,\eps}$
 for $p = r = 0.5$ and vary $q$ in the range $(1.5,3)$.
As before we fit a line to the computed values of the eigenvalue and the ratio
for each value of $q$ and report the numerical slope in brackets in the legend; once again we
fit the line to data points with $\eps \le 0.025$ where the 
$\eps \downarrow 0$ regime is manifest.
We observe that $\sigma_{2,\eps} = \mcl O(\eps^q)$  as in the
balanced case while the ratio  $\sigma_{2,\eps} /\sigma_{3,\eps} = \mcl O(\eps^{p+r})$
  which is better than  the predicted $\mcl O(\eps^{ 2(p +r) -q})$ rate 
  in Corollary~\ref{cor:ratiogap}. As mentioned earlier, these results suggest
that the lower bound on $\sigma_{3,\eps}$ in Theorem~\ref{thm:L-eps-low-eigenvals}(iii) can be sharpened to match the upper bound.
In Figure~\ref{fig:unbalanced01}(c, d), we consider another case with $q > p+r$ but this
time we fix $r=0.5$ vary $ p \in (0.5,2)$ and take  $q = p +1$. Once again we observe that $\sigma_{2,\epsilon} \sim \epsilon^q $, which is consistent with Theorem~\ref{thm:L-eps-low-eigenvals}(ii),
and $\sigma_{2,\epsilon}/\sigma_{3,\epsilon} \sim \epsilon^{p+r}$, which is better than the 
predicted rate in Theorem~\ref{thm:L-eps-low-eigenvals}(iii); again 
the results suggest
that the lower bound on $\sigma_{3,\eps}$  can be sharpened
to match the upper bound.
We compare the numerical slopes with the analytic upper bounds 
and with the conjectured $\mcl O(\eps^{p+r})$ rate for the  spectral ratio gap in Table~\ref{tab:unbalanced}.

In summary we note that, in this binary unbalanced
setting the numerical experiments are consistent with
the theory insight that only a spectral ratio gap will manifest between
the second and third eigenvalues. Furthermore, these experiments suggest that the lower
and upper bounds on the third eigenvalue should match,
suggesting  tighter bounds on the spectral ratio gap could be achievable \changed{forming the foundation for the first component of
 Conjecture~\ref{conj:SG} in Subsection~\ref{sec:contribution-2}.}

\changed{
\subsection{Binary Unbalanced Case: $q < p +r $}
\label{ssec:bp2}

Next we turn our attention to the spectrum of $\L_\eps$ when $q < p +r$. Figure~\ref{fig:unbalanced01-qlepr}(a,b) shows the
second eigenvalues $\sigma_{2,\eps}$ as well as the ratio $\sigma_{2,\eps}/\sigma_{3,\eps}$  for $p =r =1$ and $q \in [0.5, 1.5]$. 
Once again we fit a line to the computed values of the eigenvalues and the ratios and report the slopes within brackets in the legends.
We observe that $\sigma_{2,\eps} = \mcl O(\eps^q)$ as in the $q \ge p +r$  
cases; however we also notice that the ratio $\sigma_{2,\eps}/\sigma_{3,\eps} = \mcl O(\eps^q)$,
an observations which suggests that Corollary~\ref{cor:ratiogap}(ii) can be
improved; this in turn would be possible if we could sharpen
our lower bound on $\sigma_{3,\eps}$ in Theorem~\ref{thm:L-eps-low-eigenvals}(iii)
to match the upper bound, resulting in a uniform spectral gap.

Figure~\ref{fig:unbalanced01-qlepr}(c,d) shows further examples with $q < p +r$ this time with $r =1$ fixed
and taking $q = p \in [0.5, 2.0]$. Once again we observe that $\sigma_{2,\eps} = \mcl O(\eps^q)$ while $\sigma_{2,\eps}/\sigma_{3,\eps} = \mcl O(\eps^q)$
as well, further reaffirming our conjecture that the lower bound in Theorem~\ref{thm:L-eps-low-eigenvals}(iii)
is too pessimistic. We compare the
analytic and numerical slopes for the second eigenvalues as well as the spectral ratio in Table~\ref{tab:unbalanced-qlepr}.

\changed{
To summarize we derive two  conclusions in this unbalanced case:  First, our bounds on the second eigenvalue $\sigma_{2,\eps}$ are
sharp but our bounds on the spectral ratio $\sigma_{2,\eps}/\sigma_{3,\eps}$ are not sharp similarly to the $q > p+r $ case and due to the fact that our lower bound on $\sigma_{3,\eps}$
is too pessimistic. Second, followed by this observation we expect a uniform spectral gap to manifest
between the second and third eigenvalues in the unbalanced regime where $q < p +r$, similarly to the
balanced regime $q = p +r$. These observations further support the
first component of Conjecture\ref{conj:SG}
from Subsection~\ref{sec:contribution-2}.
}
}

\subsection{Multiple Clusters}
\label{ssec:3c}
\changed{
We now consider two densities $\varrho_\eps$ which concentrate, 
respectively, on three and five clusters for small $\eps$; the
quantitative details are given in \eqref{numerics-rho-def}
and the text following; see  Figure~\ref{fig:density-plots-three-five-clusters}.
In Figures~\ref{fig:balanced01-three-clusters} and \ref{fig:balanced-five-clusters} we
display the behavior of the $K^{th}$ eigenvalue
and the  spectral ratio gap related to it, for $K=3$ and $K=5$
respectively.  In both cases we let $q = p +r$ and
plot $\log( \sigma_{K, \eps})$ and $\log( \sigma_{K, \eps }/ \sigma_{K+1, \eps})$ against
$\log(\eps)$. The numerics are consistent with the hypothesis that 
$\sigma_{K, \eps} \sim \sigma_{K,\eps}/\sigma_{K+1,\eps} \sim \mcl O(\eps^q)$. 
  This suggests a natural extension of 
Theorem~\ref{thm:L-eps-low-eigenvals} and Corollary~\ref{cor:ratiogap} from the binary case to multiple clusters.
}

\changed{
In Figures~\ref{fig:unbalanced01-three-clusters}-~\ref{fig:unbalanced-five-clusters-qlepr}  we collect
similar results for the unbalanced regime where $q \neq p +r$. Once again we
see strong evidence that the multi-cluster setting behaves similarly to the binary case
in that $\sigma_{K, \eps} \sim \eps^q$ while $\sigma_{K,\eps}/\sigma_{K+1,\eps} \sim \eps^{p+r}$ when $q  > p +r$ and
$\sigma_{K,\eps}/\sigma_{K+1,\eps} \sim \eps^{q}$ when $q < p +r$ 
in both the three and five cluster cases.
We provide further evidence for this conjecture in Tables~\ref{tab:slopes-three-clusters}
and \ref{tab:slopes-five-clusters} where we collect numerical approximations to the
above rates for different choices of $p,q,r$ in the balanced and unbalanced regimes.
The above results lead to second component of 
Conjecture~\ref{conj:SG} appearing in 
Subsection~\ref{sec:contribution-2}.
}

\section{From Discrete To Continuum}\label{app:AA}

\changed{
In this section we present formal calculations, and numerical
experiments, demonstrating that the operators
of the form $\L$ in \eqref{general-weighted-Laplacian} arise as
the large data limit of $L_N$ as in \eqref{defLN-intro} for parameters
$(p,q,r) \in \mbb R^3$, and for a density $\varrho$ supported on $\Z$ 
according to which the
vertices $\{x_n\}_{n=1}^N$ are i.i.d. Subsection~\ref{sec:discrete-setting}
discusses the construction of the discrete operators $L_N$ and their properties including
self-adjointness and invariance of the spectrum under parameter choices.
Subsection~\ref{sec:cvEnergies} outlines a roadmap for rigorous proof of convergence of
$L_N$ to $\L$ in the framework of \cite{trillos2016variational, slepvcev2017analysis, trillos2018error}
through study of the convergence of Dirichlet energies, using
the law of large numbers and localization of the weights. These arguments reveal the relationship between the discrete and continuum eigenproblems as well as the correct scaling needed in the discrete setting for the spectra to converge, the topic of Subsection~\ref{sec:discrete-vs-continuum-eigenproblems}. 
In Subsection~\ref{sec:discrete-numerics} we present numerical experiments
demonstrating the convergence of discrete graph Laplacians to continuum limit operators
of the form \eqref{general-weighted-Laplacian}, as well as manifestations of
the theoretical results of Section~\ref{sec:sa} in the discrete $N < +\infty$ setting.

\subsection{The Discrete Operator $L_N$}\label{sec:discrete-setting}
Let $X_N \in \mbb R^{d \times N}$ denote the matrix with columns $\{x_n\}_{n=1}^N$
sampled i.i.d. from a density $\varrho$ on some domain $\Z$. 
Following \cite{stuart-zeronoiseSSL}, we define a similarity graph on $X_N$
by defining a weighted  similarity matrix
$\tilde W_N$  with entries
$$
\tilde{W}_{ij} = \left\{
  \begin{aligned}
    &\eta_\delta(|x_i - x_j|)\,, \quad & i \neq j,\\
      & 0 & i = j,
  \end{aligned}
\right.
$$ 
where $|\cdot|$ denotes the Euclidean norm, $\eta_\delta(\cdot)=\delta^{-d}\eta(\cdot/\delta)$ for a suitably chosen edge weight profile $\eta:\R_{\ge 0}\to\R_{\ge 0}$ 
that is non-increasing, continuous at zero and has bounded second moment.
Furthermore, let $\tilde D_N = \text{diag}(\tilde d_i)$ where $\tilde d_i := \sum_{j=1}^N \tilde W_{ij}$ is the degree of node $i$. Since  $\eta_\delta$
is approximately a Dirac distribution for small $\delta>0$ it follows
that $\tilde d_i$ is an empirical approximation of  $\varrho(x_i)$. Without loss of generality we assume that the resulting similarity graph has no isolated points: $\tilde d_i>0$ for all $i$.
For $q\in \R$, we  introduce
the matrix $W_N=W_N(q)$, a re-weighting of  $\tilde W_N$, with entries
$$
 W_{ij}= \frac{\tilde W_{ij}}{\tilde{d_i}^{1-q/2}\tilde{d_j}^{1-q/2}},
$$
with corresponding degree matrix  $D_N = \text{diag}(d_i)$ where $d_i := \sum_{j=1}^N  W_{ij}$.
We now define the graph Laplacian  $L_N$ as in \eqref{defLN-intro}
for $(p,q,r) \in \mbb R^3$,
\begin{equation*}
L_N := 
\begin{cases}
D_N^{\frac{1-p}{q-1}}\left(D_N- W_N\right)D_N^{-\frac{r}{q-1}},
&\text{ if } q\neq 1\,,\\
D_N-W_N, &\text{ if } q= 1.\\
\end{cases}
\end{equation*}
Let $\langle \cdot, \cdot \rangle$ denote the usual Euclidean inner product.
Given a symmetric matrix $A\in\R^{N\times N}$ and vectors $\bu,\bv\in\R^N$, 
we define
$$ \langle \bu, \bv\rangle_A := \bu^TA \bv\,.$$
The matrix $L_N$ is not self-adjoint with respect to the
Euclidean inner product for general $(p, q, r)$ but it is self-adjoint 
with respect to the following  $(p,q,r)$-weighted inner product:
$$
\langle\cdot\,,\,\cdot\rangle_{(p,q,r)}:=
\begin{cases}
\langle\cdot\,,\,\cdot\rangle_{D_N^{\frac{p-1-r}{q-1}}} &\text{ if } q\neq 1\,,\\
\langle\cdot\,,\,\cdot\rangle &\text{ if } q=1\,.
\end{cases}
$$
More precisely, in the case $q\neq 1$, writing $\bv=D_N^{-\frac{r}{q-1}}\bu$ yields
\begin{align}
    \langle \bu, L_N \bu\rangle_{(p,q,r)}
    &= \langle  D_N^{\frac{p-1}{q-1}}\bv,D_N^{\frac{1-p}{q-1}}\left(D_N-W_N\right) \bv \rangle 
    = \langle \bv,\left(D_N-W_N\right) \bv \rangle \notag\\
    &=\frac{1}{2}\sum_{i,j} W_{ij}\left|v_i-v_j\right|^2
    =\frac{1}{2}\sum_{i,j} W_{ij}\left|\frac{u_i}{d_i^{r/(q-1)}}-\frac{u_j}{d_j^{r/(q-1)}}\right|^2\,.\label{Dirichlet1}
\end{align}
If $q=1$, we have instead
\begin{align*}
    \langle \bu, L_N \bu\rangle_{(p,1,r)}
    = \langle \bu,\left(D_N-W_N\right) \bu\rangle
    =\frac{1}{2}\sum_{i,j} W_{ij}\left|u_i-u_j\right|^2\,.
\end{align*}
It immediately follows that the first eigenvalue of $L_N$ is zero with corresponding eigenvector
$\pmb \varphi_1=D_N^{r/(q-1)}\mbf{1}$ if $q\neq 1$ and $\pmb \varphi_1=\mbf{1}$ if $q=1$, where $\mbf{1}$ denotes the constant vector of ones. The symmetric expression \eqref{Dirichlet1} also shows why the graph Laplacian is a useful tool 
for spectral clustering: If the corresponding similarity graph has more than one disconnected component, then
choices of $u_i$ that take different constant multiples of $d_i^{r/(q-1)}$ 
(if $q\neq 1$; different constants if $q=1$) on each component of the graph 
set $ \langle \bu, L_N \bu\rangle_{(p,q,r)}$ to zero. As a consequence, a
simple continuity argument (highlighted in \cite{Ng01onspectral}) 
demonstrates that the eigenvectors corresponding 
to the low lying spectrum of $L_N$ contain information about the clusters 
in  $X_N$. Note also that for the more common parameter choices $(p,q,r)= (1,2,0)$, $(3/2,2,1/2)$ and $(1,1,0)$ discussed in the introduction (see Subsection~\ref{ssec:BLR}), the weighted inner product $\langle\cdot\,,\,\cdot\rangle_{(p,q,r)}$ reduces to the usual Euclidean inner product. We say $(\sigma,\bu)$ is an eigenpair of  $L_N$ for parameters $(p,q,r)$ if
\begin{equation*}
\langle  L_N \bu,\bv\rangle_{(p,q,r)}=\sigma \langle \bu, \bv\rangle_{(p,q,r)}\qquad \forall \bv \in \R^N\,,
\end{equation*}
and thanks to the assumption that $\tilde d_i>0$ for all $i$, this statement is equivalent to the matrix equality $L_N \bu = \sigma \bu$.

\begin{remark}\label{rmk:discrspec}
The spectra of two graph Laplacians with  parameters $(p_1,q_1,r_1)$ and $(p_2,q_2,r_2)$  are identical if
\begin{equation}\label{1=2}
p_1+r_1=p_2+r_2\,,\qquad q_1=q_2\,.
\end{equation}
This is true both in the discrete setting for the family $L_N$ defined in \eqref{defLN-intro}, and in the continuum limit for the family of weighted elliptic operators $\mcl L$ defined in \eqref{general-weighted-Laplacian}. Here, we focus on the discrete setting; the argument in the continuum limit is analogous. 

To see that this result holds, let $L^i_N$ denote the graph Laplacian defined by 
\eqref{defLN-intro} with parameters $(p_i,q_i,r_i)$, for $i=1,2.$ 
The second condition in \eqref{1=2} ensures that the weights 
$W_N$ and degrees $D_N$ are the same for both graph Laplacians 
and the first condition suffices to make their spectra identical.

Indeed, assume that $(\sigma,\bu)$ is an eigenpair of $L_N^1$ 
in the $(p_1,q_1,r_1)$-inner product,
$$
\langle L_N^1 \bu\,,\,\bu\rangle_{(p_1,q_1,r_1)}=\sigma\langle \bu\,,\,\bu \rangle_{(p_1,q_1,r_1)} \,.
$$
Defining
$\tilde \bu:= D_N^{\frac{1}{2}\left(\frac{p_1-1-r_1}{q_1-1}-\frac{p_2-1-r_2}{q_2-1}\right)}\bu
= D_N^{\frac{p_1-p_2}{q_1-1}}\bu\,,$
we have
\[ \langle \bu\,,\,\bu \rangle_{(p_1,q_1,r_1)} =\langle \tilde \bu\,,\,\tilde \bu \rangle_{(p_2,q_2,r_2)} \,.\]
Now writing 
$ \bv:=D_N^{-\frac{r_1}{q_1-1}} \bu\,$ and 
$\tilde \bv := D_N^{-\frac{r_2}{q_2-1}} \tilde \bu$
we realize that $\tilde \bv=\bv$ for parameter choices $(p_1,q_1,r_1)$ and $(p_2,q_2,r_2)$ satisfying \eqref{1=2}. We conclude that
\begin{align*}
\langle L_N^2 \tilde \bu\,,\,\tilde \bu \rangle_{(p_2,q_2,r_2)} 
&= \langle (D_N-W_N) \tilde \bv\,,\,\tilde \bv \rangle
    = \langle (D_N-W_N) \bv\,,\,\bv \rangle\\
    &= \langle L_N^1 \bu\,,\,\bu \rangle_{(p_1,q_1,r_1)} 
    =  \sigma \langle \bu\,,\,\bu \rangle_{(p_1,q_1,r_1)} \\
    &= \sigma  \langle \tilde \bu\,,\,\tilde \bu \rangle_{(p_2,q_2,r_2)} 
\end{align*}
and so $(\sigma,\tilde \bu)$ is an eigenpair of $L_N^2$ in the $(p_2,q_2,r_2)$-inner product.
\end{remark}

\begin{remark}
There are a number of graph-based algorithms which proceed
by making a preliminary density estimate via a preliminary weight
matrix $\tilde{W}$. In the approach described above, and when $q<2$, 
the rescaling of the weights from $\tilde{W}$ to $W$ enlarges 
affinities between points in regions of low sampling density; this 
adds robustness to graph-based algorithms, minimizing unwanted impact 
from outliers in the tails of $\varrho$. This is sometimes
also achieved through a rescaling within $\eta_{\delta}$ defining
$$
{W}_{ij} = \left\{
  \begin{aligned}
    &\eta_\delta\bigl(\tilde{d_i}^{1-q/2}\tilde{d_j}^{1-q/2}|x_i - x_j|\bigr)\,, \quad & i \neq j,\\
      & 0 & i = j.
  \end{aligned}
\right.
$$
This idea of variable bandwidth
originates in the statistical density estimation literature
\cite{loftsgaarden1965nonparametric,terrell1992variable} and 
was introduced to the machine learning community, in the context of
graph based data analysis, in \cite{zelnik2005self}. It would be of
interest to study limiting continuum operators in this context. Analysis
that is relevant to this question is undertaken in \cite{berry2016variable}
where aspects of the work of \cite{CoifmanLafon2006} are generalized 
to the variable bandwidth setting.  

\end{remark}


\subsection{Convergence of Dirichlet Energies}\label{sec:cvEnergies}

In this subsection, we describe why we expect the spectra of discrete operators $L_N$
to converge to the weighted Laplacian operator $\mL$. In simple terms, the limit rests on using
the law of large numbers to capture the large data limit $N \to \infty$, 
in tandem with localizing the weight functions $\eta_{\delta}$
by sending $\delta \to 0$ so that they behave like Dirac measures. To make
these ideas rigorous the two limits need to be carefully linked. Here,
however, we simply provide intuition about the role of the two limiting
processes, considering first large $N$ and then small $\delta$.

For a vector $\bu\in\R^N$, we define the \emph{discrete weighted Dirichlet energy} $E_{N,\delta}:\R^N\to [0,\infty)$,
$$
E_{N,\delta}(\bu)
    := \frac{N^{2r-q}}{\delta^2} \langle \bu,L_N \bu\rangle_{(p,q,r)},
$$
This  energy can be extended to functions defined
on $\Z$. To achieve this, for $u: \Z \to \R$, we write $u_i:=u(x_i)$.
Our aim is to study the limiting behavior of the functional $E_{N,\delta}$ as $N\to \infty$ and $\delta\to 0$ on a formal level. 
In the limit, we obtain 
the \emph{continuous weighted Dirichlet energy}
$E:L^2(\Z, \varrho^{p-r})\to[0,\infty]$
defined
as
\begin{equation*}
    E(u):=
    \begin{cases}
\frac{1}{2}   \langle u,\mcl L u\rangle_{\rho^{p-r}}
    &\text{ if }\, u\in H^1(\Z, \varrho) \,,\\
    \infty
    &\text{ if }\, u\in  L^2(\Z, \varrho^{p-r})\setminus  H^1(\Z, \varrho)\,,
    \end{cases}
  \end{equation*} 
Once the convergence of the Dirichlet energies has been established, generalizations of the results in \cite{calder2019improved, trillos2018error, trillos2016variational, wormell2020spectral} is possible. 

The set of feature vectors $X_N$ induces the empirical measure
$
\mu_N= \frac{1}{N}\sum_{i=1}^N\delta_{x_i}
$,
which allows to define the weighted Hilbert space $L^2(\Z,\mu_N)$ with inner product
$$
\langle u,v\rangle_{L^2(\Z,\mu_N)}=\int_\Z u(x)v(x)\, d\mu_N(x)=\frac{1}{N} \sum_{i=1}^Nu(x_i)v(x_i)\,.
$$
Since the feature vectors $x_i$ are i.i.d. according to the law $\varrho$,
we have $d\mu_N(x)\rightharpoonup\varrho(x)dx$ as $N\to\infty$. Further, we introduce the functions $\tilde{d}^{N,\delta},d^{N,\delta}: \Z \to\R$ as follows:
\begin{align*}
    &\tilde{d}^{N,\delta}(x):=\int_{\Z}\eta_\delta(|x-y|)\,d\mu_N(y)\,,\\
    &d^{N,\delta}(x):=\int_{\Z}\frac{\eta_\delta(|x-y|)}{\left(\tilde{d}^{N,\delta}(x)\right)^{1-q/2}\left(\tilde{d}^{N,\delta}(y)\right)^{1-q/2}}\,d \mu_N(y)\,.
\end{align*}
Note that
\begin{equation*}
    \tilde{d_i}=N\tilde{d}^{N,\delta}(x_i)\,,\qquad 
    d_i=N^{q-1}d^{N,\delta}(x_i)\,.
\end{equation*}
For a vector $\bu\in\R^N$, we can then rewrite the discrete weighted Dirichlet energy $E_{N,\delta}$
using \eqref{Dirichlet1} (case $q\neq 1$):
\begin{align*}
    E_{N,\delta}(\bu)
    &:= \frac{N^{2r-q}}{\delta^2} \langle u,L_N u\rangle_{(p,q,r)}
    = \frac{N^{2r-q}}{2\delta^2} \sum_{i,j} W_{ij}\left|\frac{u_i}{d_i^{r/(q-1)}}-\frac{u_j}{d_j^{r/(q-1)}}\right|^2\\
    &= \frac{N^{2r-q}}{2\delta^2} \sum_{i,j}  \left(\frac{\tilde{W}_{ij}}{\tilde{d_i}^{1-q/2}\tilde{d_j}^{1-q/2}}\right)\left|\frac{u_i}{d_i^{r/(q-1)}}-\frac{u_j}{d_j^{r/(q-1)}}\right|^2\\
    &= \frac{1}{2\delta^2N^2} \sum_{i,j}  \left(\frac{\eta_\delta(|x_i-x_j|)}{\left(\tilde{d}^{N,\delta}(x_i)\right)^{1-q/2}\left(\tilde{d}^{N,\delta}(x_j)\right)^{1-q/2}}\right)  \\
   & \qquad \qquad \times \left|\frac{u_i}{\left(d^{N,\delta}(x_i)\right)^{r/(q-1)}}-\frac{u_j}{\left(d^{N,\delta}(x_j)\right)^{r/(q-1)}}\right|^2\,.
\end{align*}
This formulation allows us to extend $E_{N,\delta}$ from vectors to functions on $\Z$. More precisely, for $u:\Z \to \R$, we have
\begin{align}\label{Dirichlet2}
    E_{N,\delta}(u)
     = \frac{1}{2\delta^2} \iint_{\Z\times\Z}& \left(\frac{\eta_\delta(|x-y|)}{\left(\tilde{d}^{N,\delta}(x)\right)^{1-q/2}\left(\tilde{d}^{N,\delta}(y)\right)^{1-q/2}}\right)\notag\\
   &\times \left|\frac{u(x)}{\left(d^{N,\delta}(x)\right)^{r/(q-1)}}-\frac{u(y)}{\left(d^{N,\delta}(y)\right)^{r/(q-1)}}\right|^2\,d\mu_N(x) d\mu_N(y)\,.
\end{align}
Now notice that, by the law of large numbers,
$$
\tilde{d}^{N,\delta}(x)\to \tilde{d}^{\delta}(x)\,,\qquad
d^{N,\delta}(x)\to d^{\delta}(x) \qquad \text{ as } N\to \infty\quad \forall x \in \Z\,,
$$
where the functions $\tilde{d}^{\delta},d^{\delta}: \Z \to\R$ are given by
\begin{align*}
    &\tilde{d}^{\delta}(x):=\int_{\Z}\eta_\delta(|x-y|)\varrho(y)\,dy\,,\qquad
    d^{\delta}(x):=\int_{\Z}\frac{\eta_\delta(|x-y|)}{\left(\tilde{d}^{\delta}(x)\right)^{1-q/2}\left(\tilde{d}^{\delta}(y)\right)^{1-q/2}}\varrho(y)\,dy\,.
\end{align*}
Define
\begin{equation}\label{kernelpara}
s_0:=\int_{\Z}\eta(|x|)\,dx\,,\qquad
s_2:=\int_{\Z}|e_1\cdot x|^2\eta(|x|)\, dx\,,
\end{equation}
with $e_1$ denoting the first unit standard normal vector in $\R^d$.
Taking $\delta \to 0$ as a second step, we obtain
$$
\tilde{d}^{\delta}(x)\to s_0 \varrho(x)\,,\qquad
{d}^{\delta}(x)\to s_0^{q-1} \varrho^{q-1}(x)\qquad \forall x\in \Z \,.
$$
Therefore, for smooth enough $u: \Z \to \R$, expression \eqref{Dirichlet2} allows us to estimate
\begin{align*}
E_{N,\delta}(u)&= \frac{1}{2\delta^2} \iint_{\Z\times\Z} \left(\frac{\eta_\delta(|x-y|)}{\left(\tilde{d}^{N,\delta}(x)\right)^{1-q/2}\left(\tilde{d}^{N,\delta}(y)\right)^{1-q/2}}\right) \\ 
  & \qquad \qquad \times \left|\frac{u(x)}{\left(d^{N,\delta}(x)\right)^{r/(q-1)}}-\frac{u(y)}{\left(d^{N,\delta}(y)\right)^{r/(q-1)}}\right|^2\,d\mu_N(x) d\mu_N(y)\\
&\stackrel{N\gg 1}{\approx}
 \frac{1}{2\delta^2} \iint_{\Z\times\Z} \left(\frac{\eta_\delta(|x-y|)}{\left(\tilde{d}^{\delta}(x)\right)^{1-q/2}\left(\tilde{d}^{\delta}(y)\right)^{1-q/2}}\right) \\
  & \qquad \qquad \times  \left|\frac{u(x)}{\left(d^{\delta}(x)\right)^{r/(q-1)}}-\frac{u(y)}{\left(d^{\delta}(y)\right)^{r/(q-1)}}\right|^2\varrho(x)\varrho(y)\,dxdy\\
&\stackrel{\delta \ll 1}{\approx}
\frac{1}{2\delta^2} \iint_{\Z\times\Z} \left(\frac{\eta_\delta(|x-y|)}{\left(\tilde{d}^{\delta}(x)\right)^{1-q/2}\left(\tilde{d}^{\delta}(y)\right)^{1-q/2}}\right) \\
  &\qquad \qquad \times  \left|\nabla\left(\frac{u(x)}{\left(d^{\delta}(x)\right)^{r/(q-1)}}\right)\cdot(x-y)\right|^2\varrho(x)\varrho(y)\,dxdy\\
&\stackrel{\delta \ll 1}{\approx}
\frac{1}{2}\frac{s_2}{s_0^{2r+2-q}} \int_{\Z} \frac{1}{\varrho(x)^{2-q}}
    \left|\nabla\left(\frac{u(x)}{\varrho(x)^r}\right)\right|^2\varrho(x)^2\,dx\\
&\,\,= \frac{1}{2}\frac{s_2}{s_0^{2r+2-q}} \int_{\Z}
    \left|\nabla\left(\frac{u(x)}{\varrho(x)^r}\right)\right|^2\varrho(x)^q\,dx
= \frac{s_2}{s_0^{2r+2-q}}E(u)\,.
\end{align*}
This is the desired result. To develop a theorem based on these 
calculations requires taking $N \to\infty$ concurrently with $\delta\to 0$,
and may be done in the framework of \cite{calder2019improved,trillos2018error,wormell2020spectral}.

\begin{remark}\label{k-nn-graph-remark}
 While the above arguments  primarily concern  proximity graphs; the
  method of proof in \cite{calder2019improved} is more general and can be applied to
  $k$-NN graphs as well. However, the resulting limiting process gives a different relationship
  between the continuum operator $\mcl L$ with a certain choice of $(p,q,r)$ and the
  correct normalization of the discrete Laplacian $L_N$. 
\end{remark}

\begin{remark}
Not all graph Laplacian normalizations lead to differential operators of the type \eqref{general-weighted-Laplacian} in the large data limit, and this is the motivation for introducing the parameters $(p,q,r)$ as graph Laplacian weightings of type ~\eqref{defLN-intro}. For example, the operator $D_N^{-s}(D_N-W_N)D_N^{-t}$ with $q=1$ does \emph{not} correspond to a continuum operator of type \eqref{general-weighted-Laplacian} in the same large data limit, for any choice of $s,t\in \R \setminus \{0\}$.
\end{remark}

\subsection{Discrete vs Continuum Eigenproblems}
\label{sec:discrete-vs-continuum-eigenproblems}

  In this subsection, we make explicit the relationship between the discrete and continuum eigenproblems
 and highlight
  the correct scaling needed in the discrete setting for the spectra to converge. Let $(\sigma, \varphi)$
  be an eigenpair of $\mL$ and take a test function $\phi \in H^1(\Z, \varrho)$. The arguments in  Subsection~\ref{sec:cvEnergies} show that for vectors $\bu, \bv \in \mbb R^N$ where $u_i = \varphi(x_i), v_i = \phi(x_i)$, we have
$$
 \frac{N^{2r-q}}{\delta^2} \langle L_N \bu,\bv\rangle_{(p,q,r)} 
\stackrel{N\gg 1, \delta \ll 1}{\approx}
\frac{s_2  }{2 s_0^{ 2r + 2 -q}} \Langle \mL\varphi,\phi\Rangle_{\varrho^{p-r}}.
$$
With a similar argument, one can identify the continuum analogue of the weighted inner product
$\langle\bu,\bv\rangle_{(p,q,r)}$ by rewriting it in terms of  $\varphi$ and $\phi$:
\begin{align*}
N^{r-p} \langle \bu,\bv\rangle_{(p,q,r)}
&= N^{r-p} \langle \bu,D^{\frac{p-r-1}{q-1}} \bv\rangle
= N^{r-p} \sum_{i=1}^N \bu_i \bv_i d_i^{\frac{p-r-1}{q-1}}\\
&=N^{r-p} \sum_{i=1}^N \varphi(x_i) \phi(x_i) N^{p-r-1} \left(d^{N,\delta}(x_i)\right)^{\frac{p-r-1}{q-1}}\\
&=\int_\Z \varphi(x) \phi(x) \left(d^{N,\delta}(x)\right)^{\frac{p-r-1}{q-1}}\,d\mu_N(x)\,.
\end{align*}
Recall from Subsection~\ref{sec:cvEnergies} that by the law of large numbers,
$
d^{N,\delta}(x)\to d^{\delta}(x)$ as $ N\to \infty,$
and taking $\delta \to 0$ as a next step, we obtain
${d}^{\delta}(x)\to s_0^{q-1} \varrho^{q-1}(x)$.
Therefore,
\begin{align*}
N^{r-p} \langle \bu,\bv\rangle_{(p,q,r)}
&\stackrel{N\gg 1}{\approx} \int_\Z \varphi(x) \phi(x) \left(d^{\delta}(x)\right)^{\frac{p-r-1}{q-1}}\varrho(x)\,dx\\
&\stackrel{\delta\ll 1}{\approx} s_0^{p-r-1} \int_\Z \varphi(x) \phi(x)\varrho(x)^{p-r}\,dx\,.
\end{align*}
In other words, for an eigenpair $(\tilde \sigma_{N,\delta},\bu)$ of the weighted graph Laplacian matrix $L_N$ solving
\begin{align}\label{epb_discrete}
\langle L_N \bu, \bv\rangle_{(p,q,r)} 
= \tilde \sigma_{N,\delta} \langle \bu,\bv\rangle_{(p,q,r)}\,,\qquad \forall \bv\in\R^N
\end{align}
we expect that 
\begin{equation*}
  \frac{2 s_0^{p + r -q + 1}}{\delta^2 N^{q- p - r} s_2} \tilde{\sigma}_{N,\delta} \to \sigma, \text{ as } N \to \infty, \delta \to 0,
\end{equation*}
where $\sigma$ is an eigenvalue of $\mL$,
\begin{equation*}
\langle\mL\varphi,\phi\rangle_{\varrho^{p-r}}=\sigma \langle\varphi,\phi\rangle_{\varrho^{p-r}}\,.
\end{equation*}
These considerations imply that the discrete eigenvalues of $L_N$ need to be
scaled appropriately in order to converge to the eigenvalues of $\mL$.

\begin{remark}
It is shown in the papers
\cite{calder2019improved, trillos2016variational, slepvcev2017analysis, trillos2018error} 
that for the parameter choices $(p,q,r)=(1,2,0)$ and $(3/2,2,1/2)$ 
and in the limit as $N \to \infty$ and $\delta:=\delta_N\to 0$ at 
an appropriate rate with $N$, the discrete operators $L_N$ converge
to  $\mcl{L}$ on $\Z.$ Those papers analyze the  convergence of
 the Dirichlet forms associated with $L_N$ (defined with respect to
real-valued functions on the vertices $X_N$) to those associated
with $\mcl{L}$ (defined with respect to real-valued functions on
$\Z$). In particular,  \cite{trillos2016variational, slepvcev2017analysis}
use $\Gamma$-convergence arguments based on the $TL^2$ topology to prove convergence.
This topology may be used to study $\Gamma-$limits of other
non-quadratic functionals defined with respect to
real-valued functions on the graph -- see \cite{stuart-zeronoiseSSL},
for example.
A similar methodology can be applied to show convergence of $L_N$ 
to  $\mcl{L}$
for any choice of parameters $(p,q,r)\in \R^3$. However, the $\Gamma$-convergence
framework does not result in rates of convergence for eigenvalues and eigenvectors of $L_N$
making it difficult to extend continuum analyses, such as our Main Result~\ref{main-result}, to practical discrete problems. In contrast,
the more recent articles \cite{calder2019improved,trillos2018error,wormell2020spectral}
take a more direct approach to proving the convergence of $L_N$ to $\mcl L$ and obtain
rates. The rigorous study of this limiting procedure for the general $(p,q,r)$ family of operators is the subject of future research.
\end{remark}

\begin{remark}
  The fact that the scaling factor in front of $\tilde \sigma_{N,\delta}$  has a dependence on $N^{p + r -q}$
    once again highlights  the special role of the balanced case $q=p+r$.
\end{remark}

}

\changed{
\subsection{Numerical Experiments In The Discrete Setting}\label{sec:discrete-numerics}
In this subsection we present a set of numerical experiments concerning the spectrum of discrete graph
Laplacian matrices $L_N$. Our goal here is twofold: 1) we support the theoretical findings in Subsection~\ref{sec:discrete-vs-continuum-eigenproblems}
by showing that as $N \to \infty$ and $\delta \to 0$, the  eigenvalues of $L_N$ 
converge to those of $\mL$ after appropriate scaling by $N, \delta$
and for different choices of $(p,q,r)$; 2) we show that the continuum spectral analysis
of Section~\ref{sec:sa} manifests for the setting of finitely many samples as well. In particular, we show that
a uniform spectral gap for $L_N$ exists when $q = p +r$ but disappears when $q > p + r$.

In what follows, we display two numerical examples: choosing $\varrho$ to be (i) a piecewise constant mixture model, and (ii) a mixture model with exponential components.}

\subsubsection{A Piecewise Constant Mixture}

\changed{
For the set-up of our numerical experiments, we choose $\Z = (0,1) \times (0,1) \subset \mbb R^2$ and define
the sequence of densities
\begin{equation}\label{piecewise-constant-mixture}
  \varrho_\eps(t) = \left\{
    \begin{aligned}
      & \epsilon, && t_1 \in (0.2, 0.8), \\
      & 2.5 - 1.5 \epsilon , && t_1 \in [0, 0.2] \cup [0.8, 1],
    \end{aligned}
  \right.
  \qquad \forall t = (t_1, t_2)^T \in \Z.
\end{equation}
Thus as $\epsilon \to 0$ the density $\varrho_\epsilon$ vanishes inside a strip in the middle of
$\Z$ while the rest of the probability mass is split equally between two rectangles to the sides
of $\Z$. Note that $\varrho_\eps$ is discontinuous by definition and so it does not satisfy
all of our assumptions from Subsection~\ref{ssec:pod}.  For fixed values of $\epsilon$
we  sample vertices 
$\{ x_i \}_{i=1}^N$ i.i.d. with respect to $\varrho_\epsilon$ and construct a weighted graph
$\tilde W$ with entries $\tilde W_{ij} = \eta_\delta(| x_i - x_j|)$ as in Section~\ref{sec:discrete-setting}.
As for the kernel $\eta_\delta$ we choose
\begin{equation}\label{eta-disc-exp}
  \eta_\delta(t) = \frac{1}{\pi \delta^2} \mbf{1}_{[0, \delta)}(t), \qquad \forall t \in [0, +\infty),
\end{equation}
for which we can easily compute the  normalizing constants defined in \eqref{kernelpara} to be  $s_0 = 1$ and $s_2 =1/4$.
We can then proceed to define the graph Laplacian matrices $L_N$ as outlined in
Subsection~\ref{sec:discrete-setting}  for different choices of $(p,q,r) \in \R^3$.
It remains to choose a relationship between $\delta, N$ to ensure convergence of the
spectrum of $L_N$ as $N \to \infty$ and $\delta \to 0$. Following \cite{calder2019improved}
we choose
\begin{equation}\label{delta-disc-exp}
  \delta = \left( \frac{\log(N)}{N} \right)^{1/3}.
\end{equation}
Although this choice is not justified theoretically at this point we find that it is sufficient
numerically to achieve convergence of the eigenvalues.

In Figure~\ref{fig:eval-conv} we plot the first four non-trivial  eigenvalues $\sigma_{N,\delta}$ of
$L_N$ as a function of $N$ for  $\epsilon = 2^{-3}$and various choices of $(p,q,r)$
in both balanced and unbalanced cases. Each reported eigenvalue was averaged over
twenty redraws of the vertices. We clearly observe that as $N\to \infty$ the
eigenvalues converge although the larger eigenvalues appear to converge more slowly.
In Figure~\ref{fig:eval-conv-rel-error} we plot the relative errors between the discrete
eigenvalues $\sigma_{N,\delta}$ and the continuum eigenvalues $\sigma$ computed using our
finite element solver from Section~\ref{sec:num} with the density $\varrho_\eps$ as in \eqref{piecewise-constant-mixture}.
We observe that in both balanced and unbalanced regimes the discrete eigenvalues converge to their continuum
counterparts although the convergence plateau's in the $q > p +r$ case at around $1e-3$ most likely due to
numerical errors. We observed that convergence improves for larger values of $\epsilon$.

For our next set of experiments we consider the behavior of the discrete eigenvalues $\sigma_{N,\delta}$
as  $\epsilon$ vanishes. We fix $N= 2^{13}$ and choose $\epsilon = 2^{-2}, \dots, 2^{-4}$. Here we redraw
the vertices five times and average the computed eigenvalues over these five trials.
Figure~\ref{fig:eval-eps-dependence} shows results that are analogous to
Figure~\ref{fig:density-plots-two-clusters}(b,c). We observe that in the balanced case
where $q = p +r $ the second eigenvalue vanishes like $\epsilon^q$ while
the larger eigenvalues remain bounded away from zero as predicted
by Theorem~\ref{thm:L-eps-low-eigenvals} and confirmed by our numerical experiments in Subsection~\ref{ssec:b}.
The case where $q > p +r$ also agrees with Theorem~\ref{thm:L-eps-low-eigenvals}
as well as our continuum numerical experiments
in Subsection~\ref{ssec:bp} and in turn with the first component
of Conjecture~\ref{conj:SG},
as we observe that the second eigenvalue vanishes like $\epsilon^q$ while the third eigenvalue
vanishes like $\epsilon^{p + r}$.
Finally, in the $q < p +r$ case we observe a similar behavior to the balanced case where a
uniform spectral gap manifests while the second eigenvalue 
appears to vanish at a rate that is slightly faster than $\epsilon^q$ which we attribute to numerical errors.
Hence, our discrete experiments are once again in line with continuum experiments
from Subsection~\ref{ssec:bp2} and further support the first component
of Conjecture~\ref{conj:SG}.
}

\subsubsection{An Exponential Mixture}

\changed{
Here we give full details of the numerical experiments presented in Example~\ref{ex:mixture-model}
in Subsection~\ref{ssec:CONT}. We use the same kernel $\eta_\delta$ and parameterization of
$\delta(N)$ as in \eqref{eta-disc-exp} and \eqref{delta-disc-exp} respectively. Similarly we choose
$\Z = (0,1) \times (0,1) \subset \mbb R^2$ but sample the vertices of the graph
from the density $\varrho_\omega$ as in \eqref{exp-mixture-model}, see Figure~\ref{fig:exp-mixture-example}(a) for a plot of $\varrho_\omega$ with $\omega=1/4$.

In Figure~\ref{fig:exp-mixture-example}(b,c,d) we fix $N = 2^{13}$ and choose $\omega = (1.9)^{-5},
\dots (1.9)^{-8}$. Each data point is obtained by averaging the first four eigenvalues of $L_N$ over
five trials where the vertices of the graph are redrawn from $\varrho_\omega$. As we
already discussed in Example~\ref{ex:mixture-model} our numerical results indicate that the
relationship between $p$, $q$ and $r$ has a major impact on the gap between the second and third eigenvalues of $L_N$.
In particular, when $q \le p + r$ a uniform gap is observed while when $q > p +r$ only a ratio gap
manifests. We also note that the rate of decay of the second and third eigenvalues as
a function of $\omega$ in Figure~\ref{fig:exp-mixture-example}(b,c,d )  is different from
the rates we obtained as a function of the perturbation parameter $\epsilon$  since $\varrho_\omega$ vanishes
exponentially fast in the middle of the domain which violates our assumption that the density satisfies $\varrho = K \eps$
away from the clusters.
Finally, in Figure~\ref{fig:eval-conv-mixture} we plot the first four non-trivial eigenvalues $\sigma_{N,\delta}$
of $L_N$ for $\epsilon = 1.9^{-6}$ and for different values of $N$. Analogously to
Figure~\ref{fig:eval-conv} our results show that the first few eigenvalues of $L_N$ converge
as $N \to \infty$ for the exponential mixture model as well.
}
\section{Spectral Analysis: Proofs}\label{sec:sap}

In this section we present proofs of the theorems 
in Section~\ref{sec:sa}.
The essential  analytical tools
in our  spectral analysis are the min-max and max-min formulas from
Appendix~\ref{app:AB}, together with a new weighted version of 
Cheeger's inequality given in Appendix~\ref{sec:proof-isoperim-ineqa}. 
We adopt the same organizational format as Section~\ref{sec:sa}. 
In Subsection~\ref{ssec:pscp} 
we discuss the perfectly clustered case, and then  consider small
perturbations of this setting, the nearly clustered case,
in  Subsection~\ref{ssec:nscp}. Theorem~\ref{thm:L-eps-low-eigenvals}
is proved in
Subsections \ref{ssec:t34}, \ref{ssec:t35} and while the proof of
Theorem~\ref{thm:geometry-of-fiedler-vector} is outlined in Subsection~\ref{ssec:t39}.

\subsection{Proof Of Theorem~\ref{t:0}}
\label{ssec:pscp}
\changed{
As detailed in the discussion following Theorem~\ref{t:0} it only remains to characterize the third eigenvalue of $\mcl L_0$.}
\begin{proposition}\label{low-lying-spectrum-K-eps-tau} Suppose Assumptions~\ref{Assumptions-on-D}
  and \ref{Assumptions-on-rho}
 are satisfied and 
  the $\mcl{L}_0$  spectral gap condition holds  on the clusters $\Z^\pm$ with
  optimal constants $\Lambda_0^\pm :=\Lambda_0(\Z^\pm)>0 $ separately. 
      Then $ \sigma_{3,0} \ge \min\{ \Lambda_0^+, \Lambda_0^-\}   >0$.
\end{proposition}

\begin{proof}
  Note that Assumption~\ref{Assumptions-on-rho}(e) ensures that
  $\varphi_{2,0}= |\Z'|_{\varrho_0^{p-r}}^{1/2}\varrho^r\left(\mbf 1_{\Z^+}    - \mbf 1_{\Z^-}\right)$
  belongs to  $V^0(\Dp, \vrhoo)$.
  Let $u \in V^1(\Z', \vrhoo)$ so that
  $ u \bot \text{span} \{  \varphi_{1,0},\varphi_{2,0}\}$
  in $ L^2(\Dp, \vrhoo^{p-r})$.
A direct calculation shows that this means the restrictions  $u|_{\Z^\pm}$ of $u$
to the clusters $\Z^\pm$
  are orthogonal (with respect to the $L^2(\Z^\pm, \vrhoo^{p-r}|_{\Z^\pm})$
  inner products) to the restrictions   $\vrhoo^r|_{\Z^\pm}$ of $\vrhoo^r$ and
  belong to $V^0(\Z^\pm,\vrhoo|_{\Z^\pm})$.
  Thus following the $\mcl L_0$ spectral gap assumption, see Definition~\ref{Assumption-cluster-gap}, $u|_{\Z^\pm}$  satisfy
  Poincar{\'e} inequalities of the form \eqref{indivisibility-L-rho} on $\Z^\pm$
  with optimal constants $\Lambda_0^\pm$. Hence
  \begin{align*}
    \int_{\Dp}   \left| \nabla \left( \frac{u}{\vrhoo^r} \right) \right|^2 \vrhoo^{q} dx
    &=
      \int_{\Z^+}   \left| \nabla \left( \frac{u}{\vrhoo^r} \right) \right|^2 \vrhoo^{q}dx
      +
      \int_{\Z^-}   \left| \nabla \left( \frac{u}{\vrhoo^r} \right) \right|^2 \vrhoo^{q} dx \\
    &
      \ge \min\{ \Lambda_0^+, \Lambda_0^-\} \left( \int_{\Z^+} \left| \frac{u}{\varrho_0^r} \right|^2
      \vrhoo^{p+r} dx
      + \int_{\Z^-} \left| \frac{u}{\varrho_0^r} \right|^2 \vrhoo^{p+r} dx \right)\\
    & = \min\{ \Lambda_0^+, \Lambda_0^-\} \int_{\Dp} \left| \frac{u}{\varrho_0^r} \right|^2
      \vrhoo^{p+r} dx.
  \end{align*}
  The result now follows from the max-min formula \eqref{max-min-principle} in Theorem~\ref{thm:max-min}.
\end{proof}

\begin{remark}\label{what-if-rho-didnt-assign-equal-weight}
  If Assumption~\ref{Assumptions-on-rho}(e) is dropped then 
the two terms in the definition of $\varphi_{2,0}$  need to be weighted 
by appropriate constants to ensure $\int_{\Z'} \varphi_{2,0}(x) \vrhoo^p(x) dx =0$
so that $\varphi_{2,0} \in V^0(\Z', \vrhoo)$.
\end{remark}

\subsection{Proof Of Theorem~\ref{thm:L-eps-low-eigenvals}}
\label{ssec:nscp}
We now turn our attention to the densities $\varrho_\epsilon$ that have full support on $\bD$, but concentrate around  $\Dp$ as $\epsilon$ decreases.
Throughout this section, we routinely assume that   Assumptions~\ref{Assumptions-on-D},
\ref{Assumptions-on-rho} and  \ref{Assumptions-on-rho-eps-3} are
satisfied by the domains $\Z, \Z'$ and densities $\vrhoo$ and $\varrho_\eps$. Throughout,
the constants $\Xi$ and $\Xi_j$ for any $j$ are arbitrary and can change from one line to the next.

We start by constructing an approximation for $\varphi_{2,\epsilon}$ (the second
eigenfunction of $\mcl L_\eps$) that is used throughout this section. Fix $\eps>0$ and define  the sets $\Z^\pm_{\epsilon_1}$ and $\Z^\pm_\eps$ as in \eqref{D-eps-definition}, 
where $\epsilon_1 = \epsilon + \epsilon^\beta$ with a parameter $0<\beta<1$. We choose $\eps$ small enough so that $\Z_{\eps_1}^+$ and $\Z_{\eps_1}^-$ are disjoint. Consider functions
$\xi_\epsilon^\pm\in C^\infty(\bD)$ 
that satisfy
\begin{equation*}
\begin{aligned}
  &\xi^\pm_{\epsilon}(x) = 1, \qquad &x \in \Z^\pm_{\epsilon}, \\
  &0<\xi_{\epsilon}^\pm(x) < 1, \quad |\nabla\xi_{\epsilon}^\pm(x)| \le \vartheta \epsilon^{-\beta}, \qquad &x \in \Z^\pm_{\epsilon_1} \setminus \Z^\pm_{\epsilon},\\
  & \xi_{\epsilon}^\pm (x) = 0, \qquad &x \in \Z \setminus \Z^\pm_{\epsilon_1},
\end{aligned}
\end{equation*}
for some constant $\vartheta>0$ independent of $\beta$.
The $\xi_\epsilon^\pm$ are  smooth extensions of the set functions $\mbf{1}_{\Z^\pm_\epsilon}$.
They can be constructed by convolution with the standard mollifier $g_\epsilon$
in the same manner in which $\varrho_\epsilon$ was
constructed in \eqref{rho-eps-explicit-construction} (also see \cite[Thm.~3.6]{mclean2000strongly}).
  Now define the functions $\chi_\eps^\pm \in C^\infty (\bD)$ by renormalizing $\xi_\eps^\pm$ in $L^2(\bD,\varrho_\eps^{p-r})$,
\begin{equation}\label{chi-definition}
\chi_\eps^+ := b_\eps^+ \xi_\eps^+, \qquad \chi_\eps^- := b_\eps^- \xi_\eps^-,
\end{equation}
where the coefficients $b^\pm_\eps \in \mbb R_+$ are chosen to satisfy
\begin{equation}\label{eig1Fbot}
  \begin{aligned}
   \int_{\Z_{\eps_1}^+} \varrho_\eps^{p+r}\chi_\eps^+\, dx &= 
   \int_{\Z_{\eps_1}^-} \varrho_\eps^{p+r}\chi_\eps^-\, dx, \\
   b_\eps^+  + b_\eps^- &= 2.
 \end{aligned}
\end{equation}
The first condition ensures that $\varrho_\eps^r\left(\chi_\eps^+-\chi_\eps^-\right) \in V^0(\Z, \varrho_\eps)$, whereas the second condition is not necessary and chosen for closure and
convenience in the calculations that follow.  For a schematic depiction of these constructions, see Figure~\ref{fig:epszones}. 

    
    


\begin{figure}[htp]
  \centering{
      \begin{overpic}[width=1 \textwidth, clip = true, trim = 12ex 3ex 12ex 3ex]{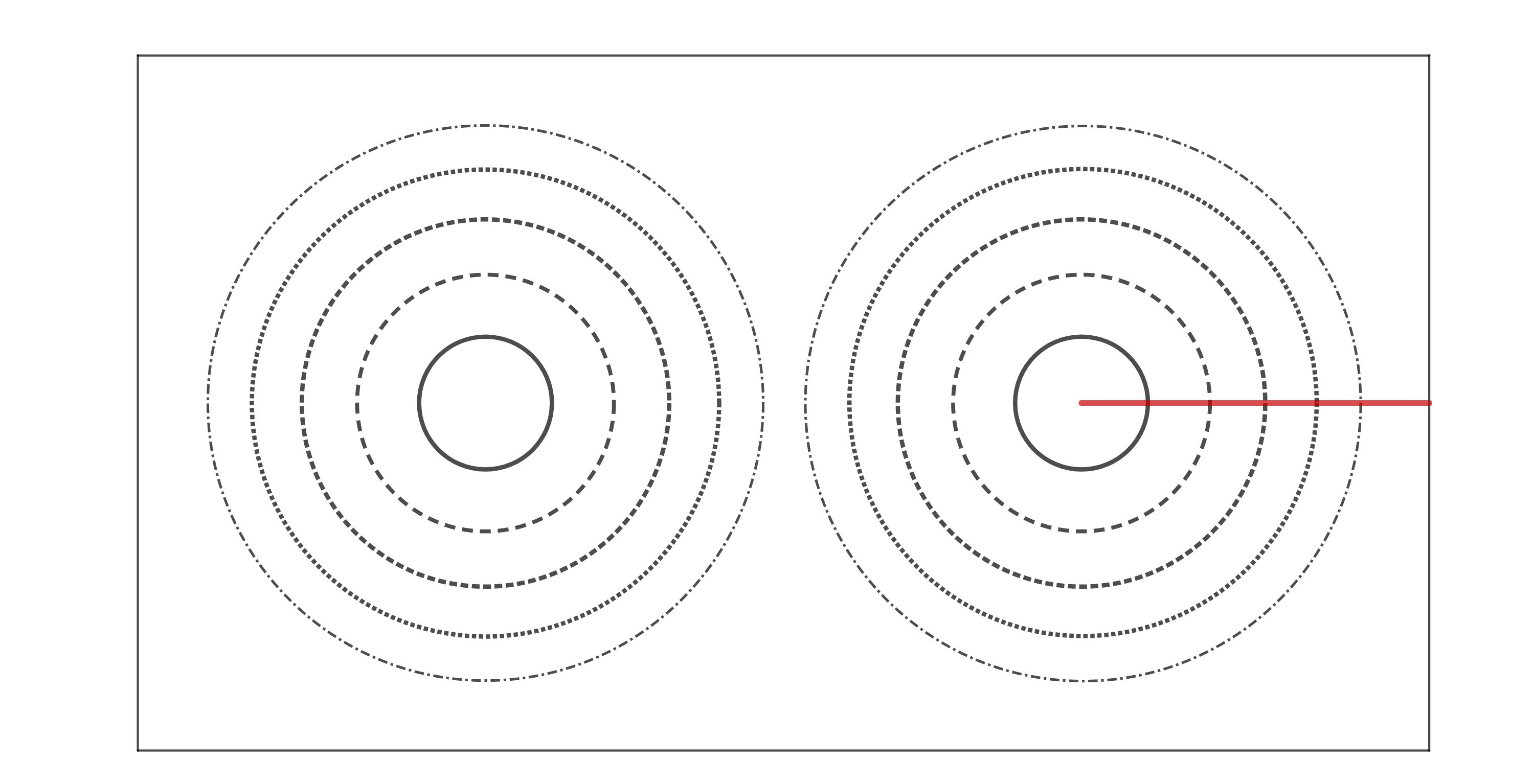}
    \put(28,25){$\Z^-$}
    \put(29,32){$\Z^-_\eps$}
    \put(30,36.5){$\Z^-_{\eps_1}$}
    \put(35.8,39){$\Z^-_{\eps_2}$}
    \put(28,45){$\Z^-_{\eps_3}$}
    
    \put(70,23.5){$\Z^+$}
    \put(65,28.5){$\Z^+_\eps$}
    \put(70,37.25){$\Z^+_{\eps_1}$}
    \put(72,41.5){$\Z^+_{\eps_2}$}
    \put(78,43.5){$\Z^+_{\eps_3}$}
    
    \put(50, 10,){$\Z$}
  \end{overpic}\\ \vspace{5ex}
  
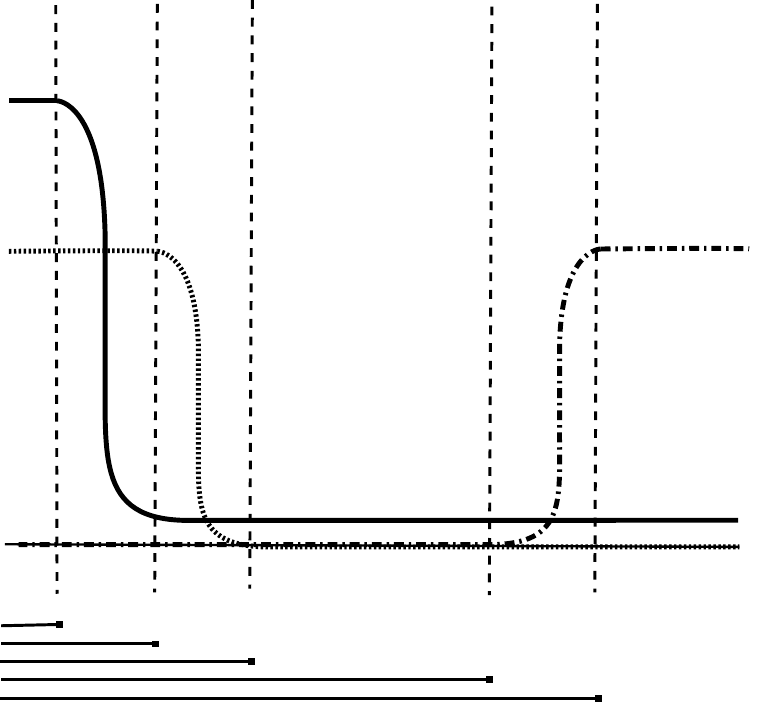
\caption{Schematic depiction of the different sets and functions used in construction of $\varphi_{F,\eps}$ from \eqref{approxFiedler} and $\tilde\varphi_{F,\eps}$ from \eqref{tilde-phi-outside-def}. Top:
  overhead schematic of the sets $\Z, \Z^\pm, \Z^\pm_\eps, \Z^\pm_{\eps_1}, \Z^\pm_{\eps_2},$ and $\Z^\pm_{\eps_3}$. Bottom:
  cross-section view of $\rho_\eps$, $\xi_\eps^+$ and $\tilde\xi_\eps$  close to the subset $\Z^+$ along the red line in the top figure.
  Here, $\eps_1:=\eps+\eps^\beta$, $\eps_2:=\eps_0+\eps^\beta$, and $\eps_3:=\eps_0+2\eps^\beta$ for 
$\eps\in (0,\eps_0)$ and $\beta\in (0,1)$. The function $\varphi_{F,\eps}$ is constructed using $\eps_1$, concentrates on the clusters, and allows to prove an upper bound on $\sigma_{2,\eps}$; the function $\tilde\varphi_{F,\eps}$ is constructed using $\eps_2$ and $\eps_3$, concentrates away from the clusters, and allows to prove an upper bound on $\sigma_{3,\eps}$.
The vertical dashed lines indicate the boundaries of the different sets as indicated below the figure.}
\label{fig:epszones}
}
\end{figure}

We define the following  ansatz as an approximation to $\varphi_{2,\eps}$ 
\begin{align}\label{approxFiedler}
\varphi_{F, \epsilon}(x) = \frac{ \varrho_\epsilon^r(x) \left[ \chi_{\epsilon}^+(x) - \chi_\epsilon^-(x) \right]}
  {\left\| \varrho_\epsilon^r(x) \left[ \chi_{\epsilon}^+(x) - \chi_\epsilon^-(x) \right]
  \right\|_{L^2(\Z, \varrho^{p-r}_\epsilon)}}.
\end{align}
Observe that $\varphi_{F,\epsilon}$ is simply a smooth approximation to the zero extension of 
$\varphi_{2,0}$ to all of $\Z$ by an element of $V^0(\Z, \varrho_\eps)$. The dependence on $\beta>0$ has been omitted in $\varphi_{F,\eps}$ for notational convenience. One should choose $\beta$ large enough in order for the set $\Z_{\eps_1}'$ to be close to $\Z_\eps'$. However, this has to be balanced with small enough $\beta$ such that the derivatives $\nabla\chi_\eps^\pm$ are allowed to be steep enough for $\varphi_{F,\eps}$ to be a good approximation of the Fiedler vector $\varphi_{2,0}$.
  The following lemma is useful throughout the rest of this section.
\begin{lemma}\label{b-min-max-are-bounded}
  Suppose that $p+ r \ge 0$ and that Assumptions~\ref{Assumptions-on-D}, \ref{Assumptions-on-rho} and
  \ref{Assumptions-on-rho-eps-3} hold and let $b_\eps^\pm$ be as in \eqref{chi-definition}.
  Suppose $\eps \in (0,\eps_0)$ for a sufficiently small $\eps_0>0$.
  Then there exists a constant $\Xi >0$, independent of $\eps$ so that
  \begin{equation*}
    |b_\eps^\pm - 1| \le \Xi \eps^{\min\{1,p+r\}}\,.
  \end{equation*}
\end{lemma}

\begin{proof}
Consider the ratio
  \begin{equation*}
    \Xi_\eps := \frac{\int_{\Z_{\eps_1}^+} \varrho_\eps^{p+r}\xi_\eps^+\, dx} 
    {\int_{\Z_{\eps_1}^-} \varrho_\eps^{p+r} \xi_\eps^-\, dx}\,.
  \end{equation*}
 Solving \eqref{eig1Fbot} for $b_\eps^\pm$ we obtain $b_\eps^+= \frac{2}{1 + \Xi_\eps}$
 and $b_\eps^- = \frac{2 \Xi_\eps}{1 + \Xi_\eps}$. Thus if we can show
 that
 \begin{equation}\label{thetaeps}
    |\Xi_\eps - 1| \le \Xi_1 \eps^{\min\{1,p+r\}},
 \end{equation}
then $|\Xi_\eps+1|=|(-2)-(\Xi_\eps-1)|\ge 2-|\Xi_\eps-1|$, and so
\begin{align*}
 |b_\eps^\pm - 1|=\frac{|\Xi_\eps-1|}{|\Xi_\eps+1|}
\le \frac{|\Xi_\eps-1|}{2-|\Xi_\eps-1|}
\le\frac{\Xi_1\eps^{\min\{1,p+r\}}}{2-\Xi_1\eps^{\min\{1,p+r\}}}\le \Xi \eps^{\min\{1,p+r\}},
\end{align*}
for some $\Xi>0$, which concludes the proof of the lemma.
  It remains to show \eqref{thetaeps}. Following Assumption~\ref{Assumptions-on-rho-eps-3}(c, d), for
  sufficiently small $\eps$, 
  \begin{align*}
    \Xi_\eps
    &\le   \frac{\int_{\Z_\eps^+} \varrho_\eps^{p+r}dx +
       K_2^{p+r} \eps^{p+r}  |\Z^+_{\eps_1} \setminus \Z^+|} 
      {\int_{\Z^-} \varrho_\eps^{p+r}  dx}
    \\
    &\le   \frac{\int_{\Z^+} (\varrho_0 + K_1 \eps)^{p+r}dx + \int_{\Z_\eps^+\setminus \Z^+} \varrho_\eps^{p+r}dx
+       K_2^{p+r} \eps^{p+r}  |\Z^+_{(\eps_0+\eps_0^\beta)} \setminus \Z^+|} 
      {\int_{\Z^-} (\varrho_0 - K_1 \eps)^{p+r}  dx}\,.
  \end{align*}
Note that $ \int_{\Z_\eps^+\setminus \Z^+} \varrho_\eps^{p+r}dx\le (\varrho_{\eps_0}^+)^{p+r}|\Z_\eps^+\setminus \Z^+|\le (\varrho_{\eps_0}^+)^{p+r}\theta \eps | \partial \Z^+|$ following the remark after \eqref{uniform-rho-bound} and using \eqref{eqn:Depssize}. For \fh{$0\le p+r\le 1$}, we use the inequality $(a+b)^{p+q}\le \left(a^{p+r}+b^{p+r}\right)$ for any $a,b\ge0$, and obtain
  \begin{align*}
    \Xi_\eps
    & \le \frac{\int_{\Z^+} \varrho_0^{p+r}dx + \Xi_2\eps^{p+r}} 
      {\int_{\Z^-} \varrho_0^{p+r}  dx - \Xi_3 \eps^{p+r}}\,.
  \end{align*}
Thanks to Assumption~\ref{Assumptions-on-rho}(e), $\int_{\Z^+} \vrhoo^{p+r}\, dx = \int_{\Z^-}\vrhoo^{p+r}\,dx$, and so Taylor expanding in $\Xi_3 \eps^{p+r}$ yields
  \begin{align*}
    \Xi_\eps
\le 1+ \left(
\frac{ \Xi_2}{\int_{\Z^-} \varrho_0^{p+r}  dx }
+ \Xi_3\right) \eps^{p+r} 
+\mcl O\left(\eps^{2(p+r)}\right)
\le \Xi_1\eps^{p+r}
  \end{align*}
since $\varrho_0$ is bounded below uniformly on $\Z^-$ by Assumption~\ref{Assumptions-on-rho}(c).

If $p+r>1$ on the other hand, we simply Taylor expand $(\varrho_0 + K_1 \eps)^{p+r}$ and $(\varrho_0 - K_1 \eps)^{p+r}$ directly, and obtain
  \begin{align*}
    \Xi_\eps
    & \le \frac{\int_{\Z^+} \varrho_0^{p+r}dx + \Xi_2\eps} 
      {\int_{\Z^-} \varrho_0^{p+r}  dx - \Xi_3 \eps}
\le 1+\Xi_1 \eps\,,
  \end{align*}
  again using the uniform upper and lower bounds for $\varrho_0$ on $\Z^\pm$.
  The lower bound on $\pm(\Xi_\eps-1)$ follows in a similar manner.
  {}
\end{proof}

\subsubsection{Proof of Theorem~\ref{thm:L-eps-low-eigenvals}(ii) (Second Eigenvalue of $\mcl L_\eps$)} 
\label{ssec:t34}

\begin{proposition}[Second eigenvalue of $\mcl L_\eps$] \label{prop:speceps}
Let $(p,q,r)\in\R^3$ satisfying $p+r>0$ and $q>0$, and suppose  Assumptions~\ref{Assumptions-on-D}, \ref{Assumptions-on-rho},
  and \ref{Assumptions-on-rho-eps-3} hold.
Then $\exists \,\eps_0 >0$ so that 
   $\forall (\eps, \beta) \in (0, \eps_0) \times (0, 1)$,
$$
0 \le \sigma_{2,\epsilon} \le \Xi \epsilon^{{q-\beta}},
$$
 where $\Xi>0$ is a uniform constant independent of $\epsilon$.
\end{proposition}

\begin{proof}
  Fix an $\eps_0>0$ and let
  $\epsilon \in (0, \epsilon_0]$.
Recall that
$\varphi_{F,\eps} \in V^0(\Z, \varrho_\eps)$ thanks to \eqref{eig1Fbot}
and is normalized with respect to the
$L^2(\Z, \varrho_\epsilon^{p-r})$ norm.
   Now consider the Rayleigh quotient 
      \begin{equation*}
     \mcl R_\eps(u):= \frac{\int_\Z  \left| \nabla \left( \frac{u}{\varrho^r_\eps} \right)\right|^2\varrho_\eps^q dx }{ \int_\Z
      \left| \frac{u}{\varrho_\eps^r} \right|^2 \varrho_\eps^{p+r} dx},
 \end{equation*}
   for   functions $u \in \text{span} \{ \varphi_{1,\eps}, \varphi_{F,\eps}\}$.
Note that $\mcl R_\eps(\varphi_{1,\eps})=0$, and so $\mcl R_\eps(u)\le \mcl R_\eps(\varphi_{F,\eps})$. Therefore, we can consider $u \in V^1(\Z, \varrho_\eps)$.
 Following the min-max principle \eqref{min-max-principle} we simply
 need to bound $\mcl R_\eps (\varphi_{F,\eps})$ to
 find an upper     bound for $\sigma_{2,\eps}$.
   Let
   \begin{equation*}
     \begin{aligned}
       \Xi_0 &= \inf_{\epsilon \in (0, \eps_0]} \| \varrho_\epsilon^{r}[\chi^+_\epsilon  - \chi^-_\epsilon] \|_{L^2(\Z, \varrho_\epsilon^{p-r})} \\
 \end{aligned}
 \end{equation*}
 and note that provided $\eps_0$ is small enough, $\Xi_0 >0$ following Lemma~\ref{b-min-max-are-bounded}, the fact that $\chi_\eps^\pm$ have disjoint supports,  $p +r \ge 0$ and  using that $\varrho_\eps$ is bounded above and  below on $\Z'$ by \eqref{uniform-rho-bound} (see also Lemma~\ref{lem:normalization-bound} in Section~\ref{ssec:t39} for a more detailed argument).
Using $0< b_\eps^\pm<2$ and Assumption~\ref{Assumptions-on-rho-eps-3}(d),
 we have
   \begin{align}\label{R-eps-bound-for-phi-F-eps}
     \mcl R_\eps(\varphi_{F,\eps})   
     &\le   \frac{4}{\Xi_0}  \int_\Z  \left| \nabla \left( \xi^+_\epsilon  - \xi^-_\epsilon  \right) \right|^2 \varrho^q_\epsilon dx\notag  \\    
     &=   \frac{4}{\Xi_0}\int_{\Z_{\eps_1}' \setminus {\Z_\epsilon}'}
     \left| \nabla \left( \xi^+_\epsilon  - \xi^-_\epsilon  \right) \right|^2\varrho^q_\epsilon  dx \\
     & \le \frac{16K_2^q \vartheta^2}{\Xi_0}
     | \Z_{\eps_1}' \setminus {\Z_\epsilon}'|\epsilon^{q - 2\beta}                                                                                     \le \Xi \eps^{q - \beta},\notag
    \end{align}
  since $ | \Z_{\eps_1}' \setminus {\Z_\epsilon}'|\le | \Z_{\eps_1}' \setminus {\Z}'|\le \theta (\eps+\eps^\beta) | \partial \Z'|\le  \Xi_1\eps^\beta$ by \eqref{eqn:Depssize} and since $\beta<1$.
 It now follows from \eqref{min-max-principle} that $\sigma_{2,\eps}\le \Xi \eps^{q -\beta}$. \\
\end{proof}

 \subsubsection{Proof of Theorem~\ref{thm:L-eps-low-eigenvals} (Third Eigenvalue of $\mcl L_\eps$)}
\label{ssec:t35}

\changed{We prove the bounds on the
  third eigenvalue of $\L_\eps$ in a series of propositions and corollaries. In particular,
  part (iii) of Theorem~\ref{thm:L-eps-low-eigenvals} follows by combining 
  Propositions~\ref{spectral-gap-general} and \ref{spectral-gap-unbalanced-case} below.}
We start with a general result that ties the existence of a $\mcl L_\eps$
spectral gap on $\Z$ to spectral gaps on subsets of $\Z$.

\begin{proposition}\label{prop:lambda23lowerbound}\label{subset-spectral-gap}
Let $(p,q,r)\in\R^3$ satisfying $p+r>0$ and $q>0$, and suppose  Assumptions~\ref{Assumptions-on-D}, \ref{Assumptions-on-rho},
  and \ref{Assumptions-on-rho-eps-3} hold.
 Let
  \begin{equation}\label{Lambda-infimum-spectral-gap}
    \Lambda(\eps):=
    \min\{  \Lambda_{\eps}(\Z_{\eps_0}^+),  \Lambda_{\eps}( \Z \setminus \Z_{\eps_0}^+)\}\ge 0\,,
  \end{equation}
  for some $\eps_0 >0$.
  Then there exist constants $s,t, \Xi_1, \Xi_2, \Xi_3 >0$ independent of $\eps$
  so that $\forall \eps \in (0, \eps_0)$,
  \begin{equation*}
    {\sigma}_{3, \epsilon} \ge \min \left\{  \frac{\Lambda(\eps)(1 -  \Xi_1 \eps^t)}
      {1 + \Xi_2 \Lambda(\eps) \eps^{s}},  \Lambda(\eps)\left(1 -  \Xi_3  \eps^{\min\{t,s\}} \right) \right\}.
\end{equation*}
\end{proposition}

\begin{proof}
  Note that it is possible that $\Lambda(\eps)=0$ if the spectral gap condition in Definition~\ref{Assumption-subset-gap} is not satisfied in
  $\Z_{\eps_0}^+$ or $\Z \setminus \Z_{\eps_0}^+$. If this happens for some $\eps\in (0,\eps_0)$, then the proposition trivially holds. Therefore, we assume from now on that $\Lambda(\eps)>0$ for all $\eps\in (0,\eps_0)$.
  
  Let $u \in V^1(\Z,\varrho_\eps)$ and $u \bot  \varphi_{F,\eps}$ with
    respect to the $\langle \cdot, \cdot \rangle_V$-inner product. Without loss of
    generality assume $\| u \|_{L^2(\Z, \varrho_\eps^{p-r})} = 1$. We will prove the
    desired lower bound for $\mcl R_\eps(u)$ and use the max-min principle (Theorem~\ref{thm:max-min}) to  infer the
   lower bound of $\sigma_{3,\eps}$.

   By definition of $ \Lambda_{\eps}$ we have 
   \begin{align*}
      \int_\Z  \left| \nabla \left( \frac{u}{\varrho_\epsilon^r} \right) \right|^2 \varrho_\epsilon^q dx
     &  =  \int_{\Z_{\eps_0}^+}  \left| \nabla \left( \frac{u}{\varrho_\epsilon^r} \right) \right|^2 \varrho_\epsilon^q dx
     +  \int_{ \Z \setminus \Z_{\eps_0}^+}  \left| \nabla \left( \frac{u}{\varrho_\epsilon^r} \right) \right|^2 \varrho_\epsilon^q dx\\
     &\quad \ge \Lambda_{\eps}(\Z_{\eps_0}^+)
       \int_{\Z_{\eps_0}^+}  \left|  \frac{u}{\varrho_\eps^r} - \bar{u}_{\Z_{\eps_0}^+} \right|^2 \varrho_\epsilon^{p+r} dx\\
      & \quad 
        + \Lambda_{\eps}( \Z \setminus \Z_{\eps_0}^+) \int_{ \Z \setminus \Z_{\eps_0}^+}
        \left|  \frac{u}{\varrho_\eps^r} - \bar{u}_{\Z \setminus \Z_{\eps_0}^+} \right|^2
        \varrho_\epsilon^{p+r}dx,
   \end{align*}
   where for subsets $\Omega \subseteq \Z$  we  used  the notation (recall \eqref{measure-of-sets}) 
   \begin{equation}
     \label{bar-u-definition}
     \bar{u}_{\Omega} := \frac{1}{|\Omega|_{\varrho_\eps^{p+r}}}\int_\Omega \left(\frac{ u}{\varrho_\eps^r}\right) \varrho_\eps^{p+r} dx.
   \end{equation}
   After expanding the squared absolute values and rearrangement we get
   \begin{align}\label{Rayleigh-bound}
     \frac{1}{\Lambda(\eps)} \int_\Z
     &\left| \nabla \left( \frac{u}{\varrho_\epsilon^r} \right) \right|^2\varrho_\epsilon^q dx
       \ge  \int_\Z  \left| \frac{u}{\varrho_\eps^r} \right|^2 \varrho_\epsilon^{p+r} dx \notag\\
     & \quad
       + \bar{u}_{\Z_{\eps_0}^+}^2 |\Z_{\eps_0}^+|_{\varrho_\epsilon^{p+r}}
       + \bar{u}_{ \Z \setminus \Z_{\eps_0}^+}^2 | \Z \setminus \Z_{\eps_0}^+|_{\varrho_\epsilon^{p+r}}\\
     & \quad
       - 2\bar{u}_{\Z_{\eps_0}^+} \int_{\Z_{\eps_0}^+} u\varrho_\eps^{p} dx
       - 2\bar{u}_{ \Z \setminus \Z_{\eps_0}^+} \int_{\Z \setminus  \Z_{\eps_0}^+} u\varrho_\eps^{p}  dx.\notag
   \end{align}
   We further discard the terms in the second line as they are positive, which leaves us with the
   lower bound:
   \begin{align*}
     \frac{1}{\Lambda(\eps)} \int_\Z \left| \nabla \left( \frac{u}{\varrho_\epsilon^r} \right) \right|^2 \varrho_\epsilon^q dx
       &\ge  \int_\Z  \left|  \frac{u}{\varrho_\eps^r} \right|^2 \varrho_\epsilon^{p+r} dx\\
      &- 2\bigg( \bar{u}_{\Z_{\eps_0}^+} \int_{\Z_{\eps_0}^+}  u\varrho_\eps^{p} dx
+ \bar{u}_{\Z \setminus \Z_{\eps_0}^+} \int_{\Z \setminus \Z_\eps^+} u\varrho_{\eps_0}^{p} dx
        \bigg).
   \end{align*}
   Using 
   H{\"o}lder's inequality and the normalization
   $\| u\|_{L^2(\Z,\varrho_\eps^{p-r})} = 1$ we obtain
   \begin{align*}
     \frac{1}{\Lambda(\eps)} \int_\Z & \left| \nabla \left( \frac{u}{\varrho_\epsilon^r} \right) \right|^2 \varrho_\epsilon^q dx  \\
&\ge 
  \left[ 1 - 2\bigg( \bar{u}_{\Z_{\eps_0}^+} |\Z_{\eps_0}^+|_{\varrho_\eps^{p+r}}^{1/2}
        + \bar{u}_{\Z \setminus \Z_{\eps_0}^+} | \Z \setminus \Z_{\eps_0}^+|_{\varrho_\eps^{p+r}}^{1/2}
\bigg) \right]
     \int_\Z \left| \frac{u}{\varrho_\eps^r}\right|^2 \varrho_\eps^{p+r}dx\\
     & =: \left[ 1 - 2(T_1 + T_2) \right] .
   \end{align*}
   It remains to 
   bound the $T_1$ and $T_2$ terms.


 Recall that $\langle u, \varrho_\eps^r \rangle_V = 0$, implying that $\int_\Z u \varrho_\eps^{p}dx =0$
 and so
      \begin{align}\label{u-orth-to-phi-1}
        \int_{\Z_{\eps_0}^+} u \varrho^p_\epsilon  dx + \int_{\Z_{\eps_0}^-} u \varrho^p_\eps dx
        = - \int_{\Z \setminus \Z_{\eps_0}'} u \varrho^p_\eps dx.
      \end{align}
On the other hand since $\langle u, \varphi_{F,\epsilon} \rangle_V =0$ as well we have that
      \begin{align*}
        0  = &b_\eps^+ \int_\Z u \xi_{\epsilon}^+ \varrho_\epsilon^p dx
            + b_\eps^+ \int_\Z \varrho_\eps^q \nabla \left( \frac{u}{\varrho_\eps^r} \right)
            \cdot \nabla \xi_\eps^+ dx \\
        & - b_\eps^- \int_\Z u \xi_{\epsilon}^- \varrho_\epsilon^p dx
        - b_\eps^- \int_\Z \varrho_\eps^q \nabla \left( \frac{u}{\varrho_\eps^r} \right)
            \cdot \nabla \xi_\eps^- dx.
      \end{align*}
Using the definition of $\xi_\eps^\pm$ we can write
      \begin{equation}\label{u-orth-to-phi-F}
        \begin{aligned}
       & \int_{\Z_{\eps_0}^+} u \varrho_\epsilon^p dx - \int_{\Z_{\eps_0}^-} u \varrho_\epsilon^p dx
        =   \int_{\Z^+_{\eps_0} \setminus \Z^+_{\eps}} u \varrho_{\eps}^p dx- \int_{\Z^-_{\eps_0} \setminus \Z^-_{\eps}} u \varrho_{\eps}^p dx \\
        &\qquad 
       -b_\eps^+ \int_{\Z^+_{\epsilon_1} \setminus \Z_\eps^+} u \varrho_\epsilon^p \xi_{\epsilon}^+ dx
          + b_\eps^-\int_{\Z^-_{\epsilon_1} \setminus \Z_\eps^-} u \varrho_\epsilon^p \xi_{\epsilon}^- dx \\
        &\qquad +  (1 - b_\eps^+) \int_{\Z_\eps^+} u \varrho_\epsilon^p dx - (1 - b_\eps^-)
        \int_{\Z_\eps^-} u \varrho_\epsilon^p dx\\
        &\qquad - b_\eps^+ \int_{\Z^+_{\eps_1} \setminus \Z^+_\eps} \varrho_\eps^q \nabla \left( \frac{u}{\varrho_\eps^r} \right)
        \cdot \nabla \xi_\eps^+ dx
         + b_\eps^- \int_{\Z^-_{\eps_1} \setminus \Z^-_\eps} \varrho_\eps^q \nabla \left( \frac{u}{\varrho_\eps^r} \right)
        \cdot \nabla \xi_\eps^- dx.
        \end{aligned}
      \end{equation}
      Furthermore,
      by
      the Cauchy-Schwartz inequality,  and
      the bound on the derivative of $\xi_\eps^\pm$  we obtain
      \begin{equation}
        \label{cauchy-schwartz-intermediate-display}
        \left|  \int_{\Z^\pm_{\eps_1} \setminus \Z^\pm_\eps} \varrho_\eps^q
          \nabla \left( \frac{u}{\varrho_\eps^r} \right)
          \cdot \nabla \xi_\eps^\pm dx \right|^2 \le \mcl R_\eps(u)
        \int_{\Z^\pm_{\eps_1} \setminus \Z_\eps^\pm} | \nabla \xi_\eps^\pm|^2 \varrho_\eps^q dx
        \le \Xi_1 \mcl R_\eps(u) \eps^{q - 2\beta} 
      \end{equation}
      for a constant $\Xi_1 >0$ independent of $\eps$. 
      Combine  \eqref{u-orth-to-phi-1}, \eqref{u-orth-to-phi-F},
      \eqref{cauchy-schwartz-intermediate-display}, and the fact that
      $0<b_\eps^\pm < 2$ to get
      \begin{equation*}
        \begin{aligned}
          2 \left| \int_{\Z_{\eps_0}^+} u \varrho_\eps^p dx \right|
          &\le  \int_{\Z\setminus \Z'_{\eps}} |u \varrho_{\eps}^p| dx 
+2  \int_{\Z_{\epsilon_1}' \setminus \Z_{\eps}'} |u \varrho_\epsilon^p| dx \\
& \qquad + \max\{|1-b_\eps^+|,|1-b_\eps^-|\}  \int_{\Z_\eps'} |u \varrho_\eps^p| dx
+ 4\sqrt{ \Xi_1} \mcl R_\eps(u)^{1/2} \eps^{\frac{q}{2}-\beta}.
      \end{aligned}
      \end{equation*}
      Multiple applications of H{\"o}lder's inequality  along with
    Lemma~\ref{b-min-max-are-bounded} then give
      \begin{align*}
        2 \left| \int_{\Z_{\eps_0}^+} u \varrho_\eps^p dx \right|
        & \le
         |\Z\setminus \Z_\eps'|_{\varrho^{p+r}_\eps}^{1/2}
		+2 \left|{\Z_{\epsilon_1}' \setminus \Z_\eps'} \right|_{\varrho_\eps^{p +r}}^{1/2} \\
        &\qquad+ \Xi\eps^{\min\{1,p+r\}} | \Z_\eps'|_{\varrho^{p+r}_\eps}^{1/2}
          + 4 \sqrt{\Xi_1} \mcl R_\eps(u)^{1/2} \eps^{\frac{q}{2}-\beta}.
       \end{align*}
       Furthermore, by Assumption~\ref{Assumptions-on-rho-eps-3}(d) and \eqref{eqn:Depssize},
       \begin{align*}
         |\Z\setminus \Z_\eps'|_{\varrho^{p+ r}_\eps} &= K_2^{p+r}\eps^{p+r}
         | \Z \setminus \Z_\eps'|
         \le \Xi_2 \eps^{p +r}\,,\\
          \left|{\Z_{\epsilon_1}' \setminus \Z_\eps'} \right|_{\varrho_\eps^{p +r}} &=
          K_2^{p+r}\eps^{p+r}
         | \Z_{\epsilon_1}' \setminus \Z_\eps'|
         \le \Xi_3 \eps^{p+r+\beta}\,.
       \end{align*}
        We can repeat the
       above calculation by replacing $\Z^+$ with $\Z^-$ and vice versa to get the bound
       \begin{equation*}
          \left|\int_{\Z_{\eps_0}^\pm} u\varrho^p_\eps dx\right| \le \Xi_4
          \eps^{\frac{1}{2}\min\left\{2, p +r\right\}}
          +  4 \sqrt{\Xi_1} \mcl R_\eps(u)^{1/2}  \eps^{\frac{q}{2}-\beta}
,     \end{equation*}
     for some constant $\Xi_4>0$.  Note that by \eqref{u-orth-to-phi-1}, we also have 
       \begin{equation*}
         \left|\int_{\Z \setminus \Z_{\eps_0}'}u\varrho_\eps^p dx \right| \le 2\Xi_4
          \eps^{\frac{1}{2}\min\left\{2, p +r\right\}}
          +  8 \sqrt{\Xi_1} \mcl R_\eps(u)^{1/2}  \eps^{\frac{q}{2}-\beta},     \end{equation*}
We conclude that
\begin{align*}
|T_1|+|T_2|
&=
\frac{|\Z_{\eps_0}^+|_{\varrho_\eps^{p+r}}^{1/2}}{|\Z_{\eps_0}^+|_{\varrho_\eps^{p+r}}} \left|\int_{\Z_{\eps_0}^+}u\varrho_\eps^p dx\right|
+
              \frac{|\Z \setminus \Z_{\eps_0}^+|_{\varrho_\eps^{p+r}}^{1/2}}{| \Z \setminus \Z_{\eps_0}^+|_{\varrho_\eps^{p+r}}} \left|\int_{\Z_{\eps_0}^-}u\varrho_\eps^p dx + \int_{\Z \setminus \Z_{\eps_0}'}u\varrho_\eps^p dx \right|  \notag\\
  & \le 
\frac{1}{|\Z_{\eps_0}^+|_{\varrho_\eps^{p+r}}^{1/2}}\left|\int_{\Z_{\eps_0}^+}u\varrho_\eps^p dx\right|
+
    \frac{1}{|\Z \setminus \Z_{\eps_0}^+|_{\varrho_\eps^{p+r}}^{1/2}} \left(
    \left|\int_{\Z_{\eps_0}^-}u\varrho_\eps^p dx \right|  +  \left|\int_{\Z \setminus \Z_{\eps_0}'}u\varrho_\eps^p dx \right| \right) 
  \notag\\
&\le\Xi_5
         \eps^{\frac{1}{2}\min\left\{2, p +r\right\}} + \Xi_6 \mcl R_{\eps}^{1/2}  \eps^{\frac{q}{2}-\beta}.
\end{align*}
Thus, we obtain
   \begin{align*}
     \mcl R_\eps(u) + \Lambda (\eps) \Xi_6 \mcl R_\eps^{1/2}(u)   \eps^{\frac{q}{2}-\beta}
       \ge \Lambda(\eps) \left[ 1 - 2\Xi_5 \eps^{\frac{1}{2}\min\left\{2, p +r\right\}}\right].
   \end{align*}
   Now if $\mcl R_\eps(u) \ge 1$ then $\mcl R_\eps(u) \ge \mcl R_\eps^{1/2}(u)$ and we have
    \begin{align*}
     \mcl R_\eps(u)  
      \ge \frac{ \Lambda(\eps) \left[ 1 - 2\Xi_5 \eps^{\frac{1}{2}\min\left\{2, p +r \right\}}\right]}
      {1 + 2 \Lambda (\eps) \Xi_6 \eps^{\frac{q}{2}-\beta}}.
    \end{align*}
    Alternatively, if $\mcl R_\eps(u) < 1$ then $\mcl R_\eps^{1/2}(u) < 1$ 
    and we instead obtain
    \begin{align*}
     \mcl R_\eps(u)  
      \ge   \Lambda(\eps) \left[ 1 - 2\Xi_8 \eps^{\frac{1}{2}\min\left\{2, p +r ,q-2\beta \right\} }\right].
    \end{align*}
    Combining these two bounds we get the desired result 
so  long as $p+r>0$, $q>0$, and $\beta$ and $\eps_0$ are small enough.
    {}
  \end{proof}

    Next, we investigate the consequences of Proposition~\ref{subset-spectral-gap} for different parameter choices $p$, $q$ and $r$.
    The main point of interest here is to analyze how the parameter $\Lambda(\eps)$ in
    \eqref{Lambda-infimum-spectral-gap} is controlled by $\eps$.
    We will show in Propositions~\ref{spectral-gap-general} that
    the choice of $q$ in relation to $p$ and $r$ plays a major role in whether  $\Lambda(\eps)$
    is uniformly bounded away from zero and hence, whether  a uniform
    spectral gap exists between $\sigma_{2,\eps}$ and $\sigma_{3,\eps}$.

    Our method of proof relies on isoperimetric-type inequalities for
    general Dirichlet forms as in  \cite[Sec.~8.5.1]{bakry2013analysis}
    viewed as a generalized form of  Cheeger's inequality. 
    For open  $\Omega \subset \Z$   define the $\varrho_\eps^q$ weighted
    Minkowski boundary measure of $\Omega$ as follows
    \begin{equation}
      \label{minkowski-outer-meas}
      |\partial \Omega|_{\varrho_\eps^q} := \liminf_{\delta \downarrow 0}
      \frac{1}{\delta} \left[ |\Omega_\delta|_{\varrho_\eps^q} - | \Omega|_{\varrho_\eps^q} \right].
    \end{equation}
    Furthermore, given $p,q,r$ we fix a subset $\Omega'\subseteq \Z$ and consider any $\Omega\subset\Omega' \subseteq \Z$.
    Define the isoperimetric function
    \begin{equation}
      \label{weighted-isoperimetric-function}
      \mcl J( \Omega, \varrho_\eps) := \frac{ | \partial \Omega |_{\varrho_\eps^q}}{\min \{ | \Omega|_{\varrho_{\eps}^{p+r}},
        |\Omega' \setminus \Omega|_{\varrho_{\eps}^{p+r}}\}}.
    \end{equation}
    The following lemma is proven in Appendix~\ref{sec:proof-isoperim-ineqa} similarly to \cite[Prop.~8.5.2]{bakry2013analysis}.
    
    \begin{lemma}\label{weighted-isoperimetric-inequality}
Let $(p,q,r)\in\R^3$, and suppose Assumptions~\ref{Assumptions-on-D}, \ref{Assumptions-on-rho} and \ref{Assumptions-on-rho-eps-3}
    hold.      
Let $\Omega' \subseteq \Z$.  Fix $\eps\in(0,\eps_0)$.
      Assume there exist  $h(\eps)>0$ so that
      \begin{equation}\label{iso-perimetric-lower-bound}
        h(\eps) \le  \inf_\Omega \mcl J( \Omega, \varrho_\eps),
      \end{equation}
      where the infimum is over open subsets $\Omega \subset \Omega' \subseteq \Z$
      such that $|\Omega |_{\varrho_\eps^{p+r}} \le \frac{1}{2} | \Omega'|_{\varrho_\eps^{p+r}}$.
      Then  $\mcl L_\eps$ has a spectral gap on $\Omega'$ according to Definition~\ref{Assumption-subset-gap} and
      \eqref{spectral-gap-on-subsets} holds with 
\as{
\begin{equation*}
\Lambda_{\eps}(\Omega') \ge \frac{ h(\eps)^2}{4}\,\left(\inf_{\Omega'} \varrho_\eps^{p+r-q}\right)\,.
\end{equation*}
}
    \end{lemma}

\begin{proposition}\label{spectral-gap-general}
  Let   $r \in \mbb R$, $q>0$, $p+r>0$,
    and suppose Assumptions~\ref{Assumptions-on-D}, \ref{Assumptions-on-rho} and \ref{Assumptions-on-rho-eps-3}
    hold.
 Then there exists $\Xi>0$ independent of $\eps\in (0,\eps_0]$ so that
\as{
   $$
   {\sigma}_{3, \epsilon} \ge \Xi\eps^{\max\{ p+r-q, 2(q - p -r)\}}.
   $$
}

\end{proposition}

\begin{proof}
  By Proposition~\ref{prop:lambda23lowerbound} we only need to find a lower bound
  on
  $\Lambda(\eps)$  which in turn
  requires us to find a lower bound on $ \Lambda_{\eps}(\Z^+_{\eps_0})$
  and $\Lambda_{\eps}(\Z \setminus \Z_{\eps_0}^+)$ separately.
We only consider $ \Lambda_{\eps}(\Z^+_{\eps_0})$ and note that the same
argument can be repeated for $\Lambda_\eps(\Z \setminus \Z^+_{\eps_0})$ possibly with different constants.

  We will find a lower bound on 
  $\inf_\Omega \mcl J( \Omega, \varrho_\eps)$ and use Lemma~\ref{weighted-isoperimetric-inequality} with $\Omega'\equiv\Z_{\eps_0}^+$
 to extend that lower bound to $\Lambda_{\eps}(\Z^+_{\eps_0})$.
For fixed $\eps$ 
let $\Omega$ be a subset of $\Z^+_{\eps_0}$ satisfying $ | \Omega |_{\varrho_\eps^{p+r}}
\le \frac{1}{2} |\Z^+_{\eps_0}|_{\varrho_\eps^{p+r}}$. First,
suppose  $ | \Omega \cap  \Z^+  |_{\varrho_\eps^{p +r}} >0$, i.e.,
part of $\Omega$ lies inside $\Z^+$. Then since 
 $\varrho_\eps$ 
 is uniformly bounded from above in $\Z^+$ and for sufficiently small $\eps_0$
 (recall \eqref{uniform-rho-bound}) we have 
\begin{align*}
  \mcl J ( \Omega, \varrho_\eps)
  &\ge \frac{ (\varrho_{\eps_0}^-)^q  | \partial \Omega \cap \Z^+| }
    {\min\{  |\Omega|_{\varrho_\eps^{p+r}}, |\Z_{\eps_0}^+ \setminus \Omega|_{\varrho_\eps^{p+r}} \} }\\
  &\ge  \frac{ (\varrho_{\eps_0}^-)^q  | \partial \Omega \cap \Z^+| }
    { (\varrho_{\eps_0}^+)^{p+r} |\Omega \cap \Z^+| +
    |\Omega \cap \left(\Z_{\eps_0}^+ \setminus \Z^+\right)|_{\varrho_\eps^{p+r}}  }\\
  & \ge \frac{ (\varrho_{\eps_0}^-)^q  | \partial \Omega \cap \Z^+| }
    { (\varrho_{\eps_0}^+)^{p+r} |\Omega \cap \Z^+|} +  \mcl O(|\Z_{\eps_0}^+ \setminus \Z^+|_{\varrho_\eps^{p+r}}) \ge \Xi_1\,,
\end{align*}
where we used Taylor expansions to write the last line.
The first ratio is uniformly bounded away from zero independent of $\eps$
by the standard isoperimetric inequality
for the set $\Omega \cap \Z^+$ while the second term is small following our
assumptions on $\varrho_\eps$. Thus, in this case $\mcl J$ is uniformly bounded
from below. 

Now consider the case where $|\Omega \cap \Z^+|_{\varrho_\eps^{p+r}} =0$, and so
$\Omega$ lies entirely in the strip $\Z^+_{\eps_0} \setminus \Z^+$ but
$|\Omega \cap \Z^+_\eps |_{\varrho_\eps^{p+r}} >0$.
Then it is possible to have  $| \partial \Omega \cap \partial \Z^+ |_{\varrho_\eps^q} >0$
or for the boundary of $\Omega$ to  touch the boundary $\partial \Z^+$ on a null set.
Then 
similar
calculations to the above yield
\begin{align*}
  \mcl J ( \Omega, \varrho_\eps)
  &= \frac{  | \partial \Omega |_{\varrho_\eps^q}}
    {\min\{  |\Omega|_{\varrho_\eps^{p+r}}, |\Z_{\eps_0}^+ \setminus \Omega|_{\varrho_\eps^{p+r}} \} }\\
  & \ge \frac{  (\varrho_{\eps_0}^-)^q | \partial \Omega \cap\Z^+ |}
    { | \Omega|_{\varrho_\eps^{p+r}} }\\
  & = \frac{  (\varrho_{\eps_0}^-)^q  | \partial \Omega \cap  \Z^+ | }
    { |\Omega \cap \Z_\eps^+|_{\varrho_\eps^{p+r}}
    + K_2^{p+r} \eps^{p+r}|\Omega \cap \left(\Z_{\eps_0}^+ \setminus \Z^+_\eps\right) | } \\
  & \ge \frac{  (\varrho_{\eps_0}^-)^q  | \partial \Omega \cap  \Z^+ | }
    { \Xi_3 \eps |\partial \Omega \cap  \bar{\Z}^+|
    + K_2^{p+r} \eps^{p+r}|\Omega \cap \left(\Z_{\eps_0}^+ \setminus \Z^+_\eps\right) | }\,,
\end{align*}
and so the lower bound on $\mcl J$ blows up as $\eps\to 0$.

Finally, we consider the case where
$|\Omega \cap \Z_\eps^+|_{\varrho_\eps^{p+r}} =  0$, and so
$\partial \Omega$ is far from $\partial \Z^+$.
Proceeding as above, we write
\begin{align*}
  \mcl J ( \Omega, \varrho_\eps)
  &\ge \frac{  | \partial \Omega |_{\varrho_\eps^q}}
    { |\Omega|_{\varrho_\eps^{p+r}} }
  =  \frac{  (K_2 \eps)^q | \partial \Omega  |}
    {  (K_2 \eps)^{p +r} |\Omega | } \\
  & \ge \Xi_4  \eps^{q - p -r},
\end{align*}
where $\Xi_4$ depends on $K_2^{q- p -r}$ and the standard isoperimetric constant.

\changed{Summarizing, if $q\le p+r$, then $\mcl J$ is bounded away from zero by a uniform constant
independent of $\epsilon$, implying that \eqref{iso-perimetric-lower-bound} holds with a uniform constant $h >0$. 
Note that $\inf_{\Z_{\eps_0}^+}\varrho_\eps^{p+r-q}=K_2^{p+r-q}\eps^{p+r-q}$ by Assumption~\ref{Assumptions-on-rho-eps-3}(d). We now investigate the different cases of $(p,q,r)$ separately:
\begin{itemize}
\item if $q=p+r$,  we obtain a uniform lower bound on $\Lambda_{\eps}(\Z_{\eps_0}^+)$ 
by Lemma~\ref{weighted-isoperimetric-inequality};
\item  if $q<p+r$ on the other hand, the lower bound on $\Lambda_{\eps}(\Z_{\eps_0}^+)$ is of order $\eps^{p+r-q}$;
\item if $q>p+r$, then we have the lower bound $ \mcl J  \ge \Xi_4  \eps^{q - p -r}$ and
Lemma~\ref{weighted-isoperimetric-inequality} implies
$ \Lambda_{\eps}(\Z^+_{\eps_0})\ge \Xi_4^2  \eps^{2(q - p -r)}/4$. 
\end{itemize}

Note that in the final bullet the factor
$\inf_{\Z_{\eps_0}^+} \varrho_\eps^{p+r-q}$ does not play a role here thanks to the uniform upper bound on $\varrho_\eps$ guaranteed in Assumption~\ref{Assumptions-on-rho-eps-3}(c).
The exact same reasoning can be applied for the set $\Omega'=\Z \setminus \Z_{\eps_0}^+$, where $\Z^-$ plays the role of $\Z^+$, and the region around $\Z^-$ where $\varrho_\eps$ is of order $\eps$ is simply extended up to the boundary of $\Z_{\eps_0}^+$. Therefore, similar bounds also hold for $\Lambda_{\eps}(\Z \setminus \Z_{\eps_0}^+)$ in each case.
By combining all the above lower bounds into one expression, Proposition~\ref{prop:lambda23lowerbound} yields the existence of a constant $\Xi>0$ so that $\sigma_{3,\eps}\ge \Xi \eps^{\max\{ p+r-q, 2(q - p -r)\}}$ as claimed.
}
{}
\end{proof}

The last proposition suggests that when \as{$q\neq p+r$} we cannot hope for a spectral gap.
Indeed, we are able to obtain a vanishing upper bound on $\sigma_{3,\eps}$ for $q>p+r$ and
quantify how fast it approaches zero in that case and ultimately obtain a
spectral ratio  gap. 

\changed{

\begin{proposition}\label{spectral-gap-unbalanced-case}
    Suppose the conditions of Proposition~\ref{spectral-gap-general}
  are satisfied.
\begin{itemize}
\item   If $q >p + r$ and $\eps_0>0$ is sufficiently small,
    then there exists a constant $\Xi_1 >0$ depending only on $\Lambda_\Delta( \Z \setminus \Z'_{\eps_0})$
    so that $\forall (\eps, \beta) \in (0, \eps_0)\times  (0, 1)$,
  \begin{equation*}
\sigma_{3,\epsilon} \le \Xi_1  \eps^{q - p-r - 2\beta}.
\end{equation*}
\item  If $q \le p + r$ and $\eps_0>0$ is sufficiently small,
    then there exists a constant $\Xi_2 >0$ depending only on $\Lambda_\Delta( \Z \setminus \Z'_{\eps_0})$
    so that $\forall \eps\in (0, \eps_0)$,
  \begin{equation*}
\sigma_{3,\epsilon} \le \Xi_2\,.
\end{equation*}
\end{itemize}
\end{proposition}
}

Note that according to Definition~\ref{Assumption-standard-gap}, $\Lambda_{\Delta}(\Z \setminus \Z'_{\eps_0})$ is the second eigenvalue of the standard Laplacian on $\Z \setminus \Z'_{\eps_0}$.

\begin{proof}
  \changed{
We apply a similar  argument to
  the proof of Proposition~\ref{prop:speceps}}  using the min-max principle. 
  Let $\tilde \varphi_2\in H^1(\Z \setminus \Z'_{\eps_0})$ denote the second eigenfunction of the standard
  Laplacian on $\Z \setminus \Z_{\epsilon_0}'$, i.e.,
  $\tilde \varphi_2 \bot \mbf{1}_{\Z \setminus \Z_{\epsilon_0}'}$ and
\begin{equation*}
  {\int_{\Z \setminus \Z_{\epsilon_0}'} |\nabla \tilde \varphi_2 |^2 dx }
  = \Lambda_\Delta(\Z \setminus \Z'_{\eps_0}) \| \tilde \varphi_2\|_{L^2(\Z \setminus \Z'_{\eps_0})}^2. 
\end{equation*}
We proceed by constructing a suitable approximation to $\tilde \varphi_2$. 
  Let 
  $\epsilon_2 := \epsilon_0 + \epsilon^\beta$
  and $\epsilon_3 := \epsilon_0 + 2\epsilon^\beta$ for $0<\beta<1$.  
  In a similar manner to \eqref{approxFiedler}, we define a function $\tilde \xi_\eps$
  (see Figure~\ref{fig:epszones})
\begin{align*}
   &\tilde \xi_\eps(x) = 1, \qquad &x \in \Z \setminus \Z_{\epsilon_3}', \\
   &0< \tilde \xi_\eps(x) < 1, \quad |\nabla\tilde \xi_\eps(x)| \le \vartheta
     \eps^{- \beta}, \qquad &x \in \Z_{\epsilon_3}'
                                                               \setminus \Z_{\epsilon_2}',
                                                              \\
  &\tilde \xi_\eps (x) = 0, \qquad &x \in \Z_{\epsilon_2}'.
\end{align*}
This allows us to define the function
\begin{equation}\label{tilde-phi-outside-def}
  \tilde \varphi_{F, \eps} := \tilde \xi_\epsilon
      \tilde \varphi_2 - \frac{\varrho_\eps^r}{|\Z \setminus \Z_{\eps_2}'|_{\varrho_\eps^{p+r}}} \int_{\Z\setminus \Z_{\eps_2}'}  \tilde \xi_\epsilon
      \tilde \varphi_2\varrho_\eps^p\, dx.
    \end{equation}
    The  shift ensures that
    $\tilde \varphi_{F, \eps} \in V^1(\Z \setminus \Z_{\eps_2}', \varrho_\eps)$.
The choice of $\eps_2$ and $\eps_3$ guarantee that the supports of
$\tilde \varphi_{F,\eps}$ and $\varphi_{F,\eps}$ are disjoint, and so they
are  orthogonal in $V^1(\Z,\varrho_\eps)$. 
Now let $u \in \text{span} \left\{ \varphi_{F,\epsilon}, \tilde \varphi_{F,\epsilon} \right\}$.
We wish to bound $\mcl{R}_\eps(u)$. A straightforward calculation shows that
since $\varphi_{F,\eps} \bot \tilde \varphi_{F, \eps}$ it suffices to
bound $\mcl R_\eps(\varphi_{F,\eps})$ and $\mcl R_\eps (\tilde \varphi_{F,\eps})$ separately.

For $\varphi_{F, \epsilon}$ 
 we showed in the proof of Proposition~\ref{prop:speceps} the existence of $\Xi_1 >0$ 
so that 
\begin{equation*}
  \mcl R_\eps (\varphi_{F,\eps}) \le \Xi_1 \eps^{ q -\beta},
\end{equation*}
for any $\beta \in (0, q)$.
To estimate $R_\eps(\tilde  \varphi_{F,\eps})$, observe that
    for $\epsilon \in (0, \epsilon_0]$ the function $\tilde \xi_\eps \tilde \varphi_2$ is in $H^1(\Z
    \setminus \Z_{\eps_2}')$. Thus, following our assumptions on $\varrho_\epsilon$ we can
    write
\begin{align}\label{R-eps-outside-upper-bound}
      \|\tilde  \varphi_{F,\eps}\|_{L^2(\Z,\varrho_\eps^{p-r})}^2 &\mcl R_\eps(\tilde  \varphi_{F,\eps})
  =\int_{\Z\setminus \Z_{\eps_2}'}  \left| \nabla \left( \frac{\tilde \varphi_{F,\eps}}{\varrho^r_{\epsilon}} \right) \right|^2 \varrho^q_\epsilon dx\notag \\
  & = K_2^{q- 2r} \eps^{q- 2r} \int_{\Z\setminus \Z_{\eps_2}'}  \left| \nabla \left( \tilde \xi_\eps \tilde \varphi_2\right) \right|^2 dx\notag\\        
  & \le2  K_2^{q- 2r}  \eps^{q- 2r}\left( \int_{\Z\setminus \Z_{\eps_2}'}  \left|\tilde \xi_\eps \nabla   \tilde \varphi_2\right|^2 dx
       +  \int_{\Z\setminus \Z_{\eps_2}'}  \left| \tilde \varphi_2\nabla \tilde \xi_\eps \right|^2 dx\right)\notag\\                       
  & \le  2 K_2^{q- 2r} \eps^{q- 2r} \left( \int_{\Z\setminus \Z_{\eps_0}'}  \left| \nabla  \tilde \varphi_2 \right|^2 dx +  \int_{\Z_{\eps_3}' \setminus \Z_{\eps_2}'}
    \left|  \tilde \varphi_2 \nabla  \tilde \xi_\eps  \right|^2 dx\right) \notag\\
  & \le  2  K_2^{q- 2r}  \eps^{q- 2r} \left( \int_{\Z\setminus \Z_{\eps_0}'}  \left| \nabla  \tilde \varphi_2 \right|^2 dx + \vartheta^2 \eps^{-2\beta}   \int_{\Z \setminus \Z_{\eps_0}'}
        \left|  \tilde \varphi_2  \right|^2 dx\right)\notag\\
  & \le 
    2 K_2^{q- 2r}  \eps^{q- 2r} \left(  \Lambda_\Delta(\Z \setminus \Z'_{\eps_0})  + \vartheta^2 \eps^{-2\beta}  \right)
    \| \tilde \varphi_2\|_{L^2(\Z \setminus \Z'_{\eps_0})}^2.
\end{align}
Next, we bound $ \|\tilde  \varphi_{F,\eps}\|_{L^2(\Z,\varrho_\eps^{p-r})}^2$ from below. We have
  \begin{align*}
     \| \tilde \xi_\eps \tilde \varphi_2 \|_{L^2(\Z, \varrho_\eps^{p-r})}^2  
     &= K_2^{p-r}\eps^{p-r} \int_{\Z\setminus \Z_{\eps_0}'} |\tilde \xi_\eps \tilde \varphi_2|^2\, dx\\
 & \ge K_2^{p-r}\eps^{p-r} \left(\int_{\Z\setminus \Z_{\eps_0}'} |\tilde \varphi_2|^2\,dx - \int_{\Z_{\eps_3}'\setminus\Z_{\eps_0}'} |\tilde \varphi_2|^2\,dx\right) \,,
  \end{align*}
and for any $k\ge 2$ by H{\"o}lder's inequality,
\begin{align*}
\int_{\Z_{\eps_3}'\setminus\Z_{\eps_0}'}  | \tilde \varphi_2|^2\, dx
\le \|  \tilde \varphi_2\|^2_{L^k(\Z\setminus \Z_{\eps_0}')} |\Z_{\eps_3}'\setminus\Z_{\eps_0}'|^{\frac{k-2}{k}}\,.
\end{align*}
  By the Sobolev embedding theorem~\cite[Thm.~4.12]{adams2003sobolev},
  $\tilde \varphi_2 \in L^k( \Z \setminus \Z'_{\eps_0})$ for $k \in[2, 2d/(d-2))$ if $d>2$ and $k\in[2,\infty)$ if $d\le 2$;
  and so using Sobolev inequalities,
 and the fact that $\| \tilde \varphi_2 \|_{L^2(\Z \setminus \Z'_{\eps_0})} \le 1$,
\begin{align*}
 \|  \tilde \varphi_2\|^2_{L^k(\Z\setminus \Z_{\eps_0}')} 
&\le \Xi_2  \|  \tilde \varphi_2\|^2_{H^1(\Z\setminus \Z_{\eps_0}')}
=  \Xi_2 \left( 1+\Lambda_\Delta( \Z\setminus \Z_{\eps_0}'))\right) \|  \tilde \varphi_2\|^2_{L^2(\Z\setminus \Z_{\eps_0}')}\\
&\le  \Xi_2 \left( 1+\Lambda_\Delta( \Z\setminus \Z_{\eps_0}'))\right)\,.
\end{align*}
Since $|\Z_{\eps_3}'\setminus\Z_{\eps_0}'|\le \Xi_3\eps^\beta |\partial \Z_{\eps_0}'|$, we can write
  \begin{align}\label{xivarphi2}
    \| \tilde \xi_\eps \tilde \varphi_2 \|_{L^2(\Z, \varrho_\eps^{p-r})}^2
    & \ge K_2^{p-r}\eps^{p-r}\left( \| \tilde \varphi_2\|_{L^2(\Z \setminus \Z'_{\eps_0})}^2
                           - \Xi_2 \left( 1+\Lambda_\Delta( \Z\setminus \Z_{\eps_0}'))\right)| \Z'_{\eps_3} \setminus \Z'_{\eps_0}|^{\frac{k-2}{k}}\right)\notag \\
    & \ge K_2^{p-r}\eps^{p-r} \left(\| \tilde \varphi_2\|_{L^2(\Z \setminus \Z'_{\eps_0})}^2
                            - \Xi_4 \eps^{\frac{\beta(k-2)}{k}}\right)\,.
  \end{align}
Furthermore,
    using Assumption~\ref{Assumptions-on-rho-eps-3}(d), the fact that
    $\tilde \varphi_2 \bot \mbf 1_{\Z \setminus \Z_{\eps_0}'}$ in $L^2(\Z \setminus \Z'_{\eps_0})$, H{\"o}lder's inequality, $\| \tilde \varphi_2 \|_{L^2(\Z \setminus \Z'_{\eps_0})} \le 1$, and the estimate \eqref{eqn:Depssize}, in that order, we can write
    \begin{align}\label{shiftbound}
      \frac{\varrho_\eps^r}{|\Z \setminus \Z_{\eps_2}'|_{\varrho_\eps^{p+r}}}
      &  \Bigg| \int_{\Z\setminus \Z_{\eps_2}'}  \tilde \xi_\epsilon
      \tilde \varphi_2\varrho_\eps^p\, dx \Bigg| \notag\\
      &= \frac{1}{|\Z \setminus \Z_{\eps_2}'|} \left| \int_{\Z\setminus \Z_{\eps_2}'}  \tilde \xi_\epsilon
        \tilde \varphi_2dx  \right|\notag \\
      & = \frac{1}{|\Z \setminus \Z_{\eps_2}'|}  \left|\int_{\Z\setminus \Z_{\eps_0}'}  
        \tilde \varphi_2dx - \int_{\Z'_{\eps_2} \setminus \Z'_{\eps_0}} \tilde \varphi_2 dx
         + \int_{\Z_{\eps_3}'\setminus \Z_{\eps_2}'}  
        (\tilde \xi_\eps - 1) \tilde \varphi_2dx \right| \notag\\
      & \le \frac{1}{|\Z \setminus \Z_{\eps_2}'|}  \left( \int_{\Z'_{\eps_2} \setminus \Z'_{\eps_0}} |\tilde \varphi_2|\, dx
         + \int_{\Z_{\eps_3}'\setminus \Z_{\eps_2}'}  
        |\tilde \xi_\eps - 1| \,|\tilde \varphi_2|\,dx \right) \notag\\
      & \le \frac{1}{|\Z \setminus \Z_{\eps_2}'|}  \int_{\Z_{\eps_3}' \setminus \Z_{\eps_0}'}  
        |\tilde \varphi_2|\,dx 
 \le \frac{| \Z_{\eps_3}' \setminus \Z_{\eps_0}' |^{1/2}}
        {| \Z \setminus \Z'_{\eps_0}|-| \Z'_{\eps_2} \setminus \Z'_{\eps_0}|} \le \Xi_5 \eps^{\beta/2}.
    \end{align}
To bound $\tilde \varphi_{F, \eps}$ on the outside set, we write explicitly
\begin{align*}
 &\|\tilde  \varphi_{F,\eps}\|_{L^2(\Z,\varrho_\eps^{p-r})}^2\\
  &\qquad= \int_\Z \left| 
    \tilde \xi_\epsilon(x) \tilde \varphi_2(x)
    - \frac{\varrho_\eps^r}{|\Z \setminus \Z_{\eps_2}'|_{\varrho_\eps^{p+r}}} \int_{\Z\setminus \Z_{\eps_2}'}  \tilde \xi_\epsilon(y)
      \tilde \varphi_2(y)\varrho_\eps^p(y)\, dy
\right|^2\varrho_\eps^{p-r}(x)\,dx\\
  &\qquad
    \ge 
    \|\tilde \xi_\eps \tilde \varphi_2\|^2_{L^2(\Z,\varrho_\eps^{p-r})} 
    -  \frac{2}{|\Z \setminus \Z_{\eps_2}'|_{\varrho_\eps^{p+r}}}
    \left|\int_{\Z\setminus \Z_{\eps_2}'}  \tilde \xi_\epsilon
    \tilde \varphi_2\varrho_\eps^p\, dy\right|^2\\
  &\qquad =
    \|\tilde \xi_\eps \tilde \varphi_2\|^2_{L^2(\Z,\varrho_\eps^{p-r})} 
    - \frac{|\Z \setminus \Z_{\eps_2}'|_{\varrho_\eps^{p+r}}^{p+r}}{\varrho_\eps^{2r}}
    \left(\frac{2 \varrho_\eps^r}{|\Z \setminus \Z_{\eps_2}'|_{\varrho_\eps^{2(p+r)}}}
    \left|\int_{\Z\setminus \Z_{\eps_2}'}  \tilde \xi_\epsilon
    \tilde \varphi_2\varrho_\eps^p\, dy  \right|^2 \right)\\
  & \qquad =
    \|\tilde \xi_\eps \tilde \varphi_2\|^2_{L^2(\Z,\varrho_\eps^{p-r})} 
    - |\Z \setminus \Z_{\eps_2}'| K_2^{p-r} \eps^{p-r}
    \left(\frac{2 \varrho_\eps^r}{|\Z \setminus \Z_{\eps_2}'|_{\varrho_\eps^{2(p+r)}}}
    \left|\int_{\Z\setminus \Z_{\eps_2}'}  \tilde \xi_\epsilon
    \tilde \varphi_2\varrho_\eps^p\, dy  \right|^2 \right)\,.
\end{align*}
Together with the bounds \eqref{xivarphi2} and \eqref{shiftbound},
we obtain for small enough $\eps_0$,
\begin{align*}
 &\|\tilde  \varphi_{F,\eps}\|_{L^2(\Z,\varrho_\eps^{p-r})}^2\\
  &\qquad\ge K_2^{p-r}\eps^{p-r} \left(\| \tilde \varphi_2\|_{L^2(\Z \setminus \Z'_{\eps_0})}^2
    - \Xi_4  \eps^{\frac{\beta(k-2)}{k}}
    - \Xi_6 \eps^{ \beta}  \right)\\
&\qquad \ge \Xi_7 \eps^{p-r}\| \tilde \varphi_2\|_{L^2(\Z \setminus \Z'_{\eps_0})}^2\,.
\end{align*}
Finally, following from \eqref{R-eps-outside-upper-bound}, we  infer the existence of  a constant
    $\Xi$, independent of $\eps \in (0, \eps_0)$, so that
    \begin{equation*}
      \mcl R_\eps( \tilde  \varphi_{F,\eps})   \le
      \Xi \eps^{q- p - r - 2\beta},
\end{equation*}
which concludes the proof.
\end{proof}


\subsection{Proof of Theorem~\ref{thm:geometry-of-fiedler-vector}
  (Geometry Of The Second Eigenfunction)}
\label{ssec:t39}
  
First, we prove a key result, that allows us to translate our bounds on
the third eigenvalue  $\sigma_{3,\eps}$ into an upper bound on the error between the
second eigenfunction $\varphi_{2,\eps}$ and the approximate Fiedler vector $\varphi_{F,\eps}$.
  
\changed{
  \begin{proposition}\label{varphi-F-is-close-to-fiedler-vector}
  Suppose there exist constants $\Xi_1, \Xi_2, \Xi_3 \ge 0$, so that
  for all $\eps \in (0, \eps_0]$,
  \begin{align*}
    \sigma_{3,\eps}  \ge \Xi_1 +\Xi_2 \eps^{q - \vartheta} + \Xi_3 \eps^{\theta-q}.
  \end{align*}
  Then
  for every 
 $0<\beta <q$,
there exists $\Xi>0$ so that
  \begin{equation*}
    \left|1 - \left\langle \frac{\varphi_{2,\eps}}{\varrho_\eps^r},
      \frac{\varphi_{F,\eps}}{\varrho_\eps^r} \right\rangle_{\varrho_\eps^{p+r}}^2\right|
    \le \frac{\Xi\eps^{ q - \beta}}{\Xi_1+\Xi_2\eps^{q-\vartheta}+ \Xi_3 \eps^{\theta-q}}\,.
  \end{equation*}
\end{proposition}
}

\begin{proof}
  Since
  $\left\langle \frac{\varphi_{2,\eps}}{\varrho_\eps^r},
    \frac{\varphi_{F,\eps}}{\varrho_\eps^r} \right\rangle_{\varrho_\eps^{p+r}}
  \equiv \langle \varphi_{2,\eps}, \varphi_{F,\eps} \rangle_{\varrho_\eps^{p-r}}$
  we will work with the $L^2(\Z, \varrho_\eps^{p-r})$ inner product for brevity.
  It follows from the spectral theorem \cite[Thm.~D.7]{Evans} that $\varphi_{j,\eps}$ form
  an orthonormal basis in $L^2(\Z, \varrho_\eps^{p-r})$. Let $\varphi_{F,\eps} = \sum_{j=1}^\infty h_j \varphi_{j,\eps}$ where $h_j = \langle \varphi_{j,\eps}, \varphi_{F,\eps} \rangle_{\varrho_\eps^{p-r}}$. Note that $h_1=0$ since $\varphi_{F,\eps} \bot \varphi_{1,\eps}$.
  It follows from the calculation in \eqref{R-eps-bound-for-phi-F-eps} that for $\beta \in (0, q)$,
  \begin{align*}
  \Xi \eps^{q-\beta} \ge \mcl{R}_{\eps}(\varphi_{F, \eps}) &= \langle \mcl L_{\varrho_\eps} \varphi_{F,\eps}, \varphi_{F,\eps} \rangle_{\varrho_\eps^{p-r}} 
     = \sigma_{2, \eps} h_2^2 + \sum_{j=3}^\infty \sigma_{j,\eps}  h_j^2 \,,
  \end{align*}
  and hence
 \begin{align*}
    \sigma_{3,\eps} \sum_{j=3}^\infty h_j^2
     \le \sum_{j=3}^\infty \sigma_{j,\eps} h_j^2
     \le \Xi \eps^{q-\beta}-\sigma_{2, \eps} h_2^2\,.
 \end{align*}
 Since $\varphi_{F,\eps}$ is normalized, it follows that $h_j^2 \le 1$ for all $j\ge 1$ and
\changed{
 \begin{align*}
     1-h_2^2=\sum_{j=3}^\infty h_j^2 \le \frac{\Xi \eps^{q-\beta}-\sigma_{2, \eps} h_2^2}{\sigma_{3,\eps}}
     \le \frac{\Xi\eps^{q-\beta}}{\Xi_1+\Xi_2\eps^{q-\vartheta}+\Xi_3\eps^{\theta-q}}.
 \end{align*}
}
  {}
\end{proof}

Now consider 
$$
\bar \varphi_{2,0}(x):= b_\eps^0 \varrho_0^r(x)\left[\mbf 1_{\Z^+}(x)       - \mbf 1_{\Z^-}(x)\right] 
\in L^2(\Z, \varrho_\eps^{p-r})\,,
$$
obtained by zero extension of $\varphi_{2,0}$ to all of $\Z$, where
\begin{equation}\label{b-eps-0-expression}
b_\eps^0:=1/\| \varrho_0^r(x)\left[\mbf 1_{\Z^+} -
  \mbf 1_{\Z^-}\right]\|_{L^2(\Z,\varrho_\eps^{p-r})}
\end{equation}
is a normalization constant. Similarly, we denote
$$
\varphi_{F,\eps}= b_\eps^F \varrho_\epsilon^r(x) \left[  \chi_{\epsilon}^+(x) - \chi_\epsilon^-(x) \right]
\in L^2(\Z, \varrho_\eps^{p-r})\,,
$$
with the normalization constant 
$$
b_\eps^F:=  1/\|\varrho_\epsilon^r \left[ \chi_{\epsilon}^+ - \chi_\epsilon^- \right]\|_{L^2(\Z,\varrho_\eps^{p-r})}>0\,.
$$
We begin by providing bounds on the normalization constants $b_\eps^0$ and $b_\eps^F$.
\begin{lemma}\label{lem:normalization-bound}
Let $(p,q,r)\in\R^3$ satisfying  $p+r>0$, and suppose Assumptions~\ref{Assumptions-on-D}, \ref{Assumptions-on-rho} and
  \ref{Assumptions-on-rho-eps-3} hold. Let $\eps_0>0$ small enough.
  Then there exist constants $\Xi_1,\Xi_2 >0$, independent of $\eps$ so that for all $\eps \in (0,\eps_0)$,
\begin{align*}
    &\left|b_\eps^0-\left(\int_{\Dp}\varrho_0^{p+r}\,dx\right)^{-1/2}\right| \le \Xi_1 \eps\,,\qquad
    \left|b_\eps^F-\left(\int_{\Dp}\varrho_0^{p+r}\,dx\right)^{-1/2}\right| \le \Xi_2 \eps^{\min\{1,p+r\}}\,.
\end{align*}
\end{lemma}
\begin{proof}
 
Using the explicit expression \eqref{b-eps-0-expression} write 
\begin{equation*}
(b_\eps^0)^{-2}
= \int_{\Dp} \varrho_0^{2r}\varrho_\eps^{p-r} dx,
\end{equation*}
It follows from Assumption~\ref{Assumptions-on-rho-eps-3}(c) that
\begin{align}\label{rho-eps-est}
\varrho_0(x)-K_1\eps\le  \varrho_\eps(x) \ \le \varrho_0(x)+K_1\eps\qquad \forall \, x\in \Dp\,.
\end{align}
Combining with Assumption~\ref{Assumptions-on-rho}(c), we can find a constant
$\Xi_3>0$ so that
\begin{align}\label{b0bound}
    \left|(b_\eps^0)^{-2}-\int_{\Dp}\varrho_0^{p+r}\,dx\right| \le \Xi_3 \eps\,.
\end{align}
 Let $b_\eps^\pm$ be as in \eqref{chi-definition}.
Using Assumption~\ref{Assumptions-on-rho-eps-3}(d), and the
definition of the $\chi_\eps^\pm$, we can write
\begin{align*}
(b_\eps^F)^{-2}
  &= \int_{\Dp_{\eps_1}} \varrho_\epsilon^{p+r}(x) \left[ (b_\eps^+)^2 \xi_{\epsilon}^+(x)
    +  (b_\eps^-)^2 \xi_\epsilon^-(x) \right]\,dx\\
   &=(b_\eps^+)^2 \int_{\Z_\eps^+ }\varrho_\epsilon^{p+r}\,dx +(b_\eps^-)^2 \int_{\Z_\eps^- } \varrho_\epsilon^{p+r}\,dx  \\
  &\quad+ K_2^{p+r}\eps^{p+r}\int_{\Dp_{\eps_1}\setminus{\Dp_\eps}} \left[ \chi_{\epsilon}^+ - \chi_\epsilon^- \right]^2\,dx \\
  & = \int_{\Dp} \varrho_\eps^{p+r} dx  \\
  & \quad + ((b_\eps^+)^2 -1) \int_{\Z^+} \varrho_\eps^{p+r} dx +
    ((b_\eps^-)^2 -1) \int_{\Z^-} \varrho_\eps^{p+r} dx \\
   &\quad +(b_\eps^+)^2 \int_{\Z_\eps^+\setminus \Z^+ }\varrho_\epsilon^{p+r}\,dx +(b_\eps^-)^2 \int_{\Z_\eps^-\setminus \Z^- } \varrho_\epsilon^{p+r}\,dx  \\
  & \quad + K_2^{p+r}\eps^{p+r}\int_{\Dp_{\eps_1}\setminus{\Dp_\eps}} \left[ \chi_{\epsilon}^+ - \chi_\epsilon^- \right]^2\,dx. 
\end{align*}
The first term is close to $\int_{\Z'} \varrho_0^{p+r} dx $ using \eqref{rho-eps-est},
whereas the terms in the second line can be controlled using Lemma~\ref{b-min-max-are-bounded} and the fact that $0<b_\eps^\pm<2$,
\begin{equation*}
  \left|(b_\eps^\pm- 1)(b_\eps^\pm+ 1)  \int_{\Z^\pm} \varrho_\eps^{p+r} dx \right|
\le   3|b_\eps^\pm - 1|\left| \int_{\Z^\pm} \varrho_\eps^{p+r} dx \right|
\le \Xi_4 \eps^{\min\{1,p+r\}} \,.
\end{equation*}
Finally, the last two lines can be estimated using  \eqref{eqn:Depssize},
\begin{align*}
    &0
\le   (b_\eps^+)^2 \int_{\Z_\eps^+\setminus \Z^+ }\varrho_\epsilon^{p+r}\,dx +(b_\eps^-)^2 \int_{\Z_\eps^-\setminus \Z^- } \varrho_\epsilon^{p+r}\,dx  \\
  & \qquad + K_2^{p+r}\eps^{p+r}\int_{\Dp_{\eps_1}\setminus{\Dp_\eps}} \left[ \chi_{\epsilon}^+ - \chi_\epsilon^- \right]^2\,dx\\
  &\quad \le 
  4|\Z_{\eps}'\setminus{\Dp}|\left(\varrho_{\eps_0}^+\right)^{p+r}
  +4 |\Z_{\eps_1}'\setminus{\Z_\eps'}|K_2^{p+r}\eps^{p+r}
 \le \Xi_5 \eps^{\min\{1,p+r+\beta\}}
\end{align*}
for some $\Xi_5>0$. Putting the above estimates together, we obtain
\begin{align}\label{bFbound}
    \left|(b_\eps^F)^{-2}-\int_{\Dp}\varrho_0^{p+r}\,dx\right| \le \Xi_6 \eps^{\min\{1,p+r\}}\,.
\end{align}
The lemma then follows from \eqref{b0bound} and \eqref{bFbound}.
\end{proof}

In order to prove Theorem~\ref{thm:geometry-of-fiedler-vector}, we aim to derive an error bound on the difference between $\bar \varphi_{2,0}$ and $\varphi_{2,\eps}$. To this end, we first estimate $ \langle \varphi_{F,\eps}, \bar \varphi_{2,0} \rangle_{\varrho_\eps^{p-r}}$ using the explicit expressions for $\varphi_{F,\eps}$ and $\bar \varphi_{2,0}$. 
\begin{proposition}\label{varphi-F-is-close-to-varphi-2}
Let $(p,q,r)\in\R^3$ satisfying  $p+r>0$, and suppose Assumptions~\ref{Assumptions-on-D}, \ref{Assumptions-on-rho} and
  \ref{Assumptions-on-rho-eps-3} hold. Let $\eps_0>0$ small enough.
  Then there exists a constant $\Xi >0$, independent of $\eps$ so that for all $\eps \in (0,\eps_0)$,
\begin{align*}
    \| \bar \varphi_{2,0}-\varphi_{F,\eps}\|^2_{L^2(\Z,\varrho_\eps^{p-r})}
    \le \Xi \eps^{\min\{1,p+r\}}\,.
\end{align*}
\end{proposition}
\begin{proof}
Since $\Z^+\cap \Z^-_{\eps_1}=\emptyset$ and $\Z^-\cap \Z^+_{\eps_1}=\emptyset$, we have
\begin{align*}
   & \langle \varrho_\epsilon^r(x) \left[  \chi_{\epsilon}^+(x) - \chi_\epsilon^-(x) \right], \varrho_0^r\left[\mbf 1_{\Z^+}(x)       - \mbf 1_{\Z^-}(x)\right] \rangle_{\varrho_\eps^{p-r}}\\
   &\quad = \int_\Z \varrho_0^r(x) \varrho_\eps^p(x) \left[ b_\eps^+ \xi_{\epsilon}^+(x)\mbf 1_{\Z^+}(x)
     - b_\eps^+ \xi_\epsilon^+(x)\mbf 1_{\Z^-}(x) \right]\,dx \\
   & \qquad - \int_\Z \varrho_0^r(x) \varrho_\eps^p(x) \left[ b_\eps^-\xi_{\epsilon}^-(x)\mbf 1_{\Z^+}(x) -
     b_\eps^- \xi_\epsilon^-(x)\mbf 1_{\Z^-}(x) \right]\,dx \\
   & \quad  =  b_\eps^+ \int_{\Z^+}  \varrho_0^r \varrho_\eps^p dx  + 
     b_\eps^- \int_{\Z^-}  \varrho_0^r \varrho_\eps^p  dx \\
   & \quad  = \int_{\Dp}\varrho_0^r \varrho_\eps^p\,dx + (b_\eps^+ - 1) \int_{\Z^+}  \varrho_0^r \varrho_\eps^p dx
     + (b_\eps^- - 1)  \int_{\Z^-}  \varrho_0^r \varrho_\eps^p dx\,.
\end{align*}
If $p\ge 0$ (and by a similar argument with the order of inequalities
  reversed if $p<0$), \eqref{rho-eps-est} implies
\begin{align*}
    \varrho_0^p(x)-\eps K_1 p\varrho_0^{p-1}(x) +O(\eps^2) \le \varrho_\eps^p(x)
    \le  \varrho_0^p(x)+\eps K_1 p\varrho_0^{p-1}(x)+O(\eps^2)\,.
\end{align*}
By Assumption~\ref{Assumptions-on-rho}(c), we conclude that there exists a constant $\Xi_1>0$ such that
\begin{align*}
    \left|\int_{\Dp}\varrho_0^r \varrho_\eps^p\,dx-\int_{\Dp}\varrho_0^{p+r}\,dx \right|\le \Xi_1 \eps\,.
\end{align*}
The above estimate together with Lemma~\ref{b-min-max-are-bounded} implies
\begin{equation*}
\left| \int_{\Dp}\varrho_0^{p+r}\,dx- \langle \varrho_\epsilon^r(x) \left[  \chi_{\epsilon}^+(x) - \chi_\epsilon^-(x) \right], \varrho_0^r\left[\mbf 1_{\Z^+}(x)       - \mbf 1_{\Z^-}(x)\right] \rangle_{\varrho_\eps^{p-r}} \right|  \le
  \Xi_2 \eps^{\min\{1, p +r\}}
\end{equation*}
for some constant $\Xi_2 >0$.
Combining this bound with Lemma~\ref{lem:normalization-bound}, and writing
\begin{align*}
\langle \varphi_{F,\eps}, \bar \varphi_{2,0} \rangle_{\varrho_\eps^{p-r}}
= b_\eps^0 b_\eps^F \langle \varrho_\epsilon^r(x) \left[  \chi_{\epsilon}^+(x) - \chi_\epsilon^-(x) \right], \varrho_0^r\left[\mbf 1_{\Z^+}(x)       - \mbf 1_{\Z^-}(x)\right] \rangle_{\varrho_\eps^{p-r}}\,,
\end{align*}
we conclude that there exists a constant $\Xi_3>0$ so that
\begin{align*}
\left|1-  \langle \varphi_{F,\eps}, \bar \varphi_{2,0} \rangle_{\varrho_\eps^{p-r}}\right|\le \Xi_3  \eps^{\min\{1,p+r\}}\,.
\end{align*}
Finally, we obtain
\begin{align*}
    \| \bar \varphi_{2,0}-\varphi_{F,\eps}\|^2_{L^2(\Z,\varrho_\eps^{p-r})}
    &=    \int_\Z\left|\bar \varphi_{2,0} -\varphi_{F,\eps}\right|^2\varrho_\eps^{p-r}\,dx\\
    &= \|\bar\varphi_{2,0}\|^2_{L^2(\Z,\varrho_\eps^{p-r})} + \|\varphi_{F,\eps}\|^2_{L^2(\Z,\varrho_\eps^{p-r})}
    -2 \langle \varphi_{F,\eps}, \bar \varphi_{2,0} \rangle_{\varrho_\eps^{p-r}}\\
    &= 2\left(1-\langle \varphi_{F,\eps}, \bar \varphi_{2,0} \rangle_{\varrho_\eps^{p-r}}\right)
    \le \Xi \eps^{\min\{1,p+r\}}\,.
\end{align*}
\end{proof}

We are now ready to provide a quantitative estimate on how close the perturbed second eigenfunction $\varphi_{2,\eps}$ is to $\bar \varphi_{2,0}$ by comparing both eigenfunctions to the approximate Fiedler vector $\varphi_{F,\eps}$.

\changed{
\begin{proof}[Proof of Theorem~\ref{thm:geometry-of-fiedler-vector}]
  We apply Proposition~\ref{varphi-F-is-close-to-fiedler-vector}
with the  eigenvalue bounds in Theorem~\ref{thm:geometry-of-fiedler-vector}(ii, iii). Depending on $(p,q,r)$, we have different lower bounds on $\sigma_{3,\eps}$. Writing the bounds from Theorem~\ref{thm:geometry-of-fiedler-vector} in the notation of Proposition~\ref{varphi-F-is-close-to-fiedler-vector}, we have
\begin{itemize}
\item If $q>p+r$, then $\Xi_1=0$, $\Xi_2>0$, $\Xi_3=0$ and $\vartheta= -q+2(p+r)$;
\item If $q=p+r$, then $\Xi_1>0$, $\Xi_2=\Xi_3=0$;
\item If $q<p+r$, then $\Xi_1=\Xi_2=0$, $\Xi_3>0$, and $\theta=p+r$.
\end{itemize} 
 We obtain that there exists a constant $\Xi_4>0$ so that
\begin{equation}\label{Fiedler-to-pert-est}
    \left|1 - \langle \varphi_{2,\eps}, \varphi_{F,\eps} \rangle_{\varrho_\eps^{p-r}}^2\right|
    \le \Xi_4 \eps^{-|q-p-r|+\min\{q,p+r\}-\beta},
  \end{equation}
for any $(p,q,r)\in\R^3$ with $q>0$ and $p+r>0$.
Combining estimate \eqref{Fiedler-to-pert-est} with Proposition~\ref{varphi-F-is-close-to-varphi-2} gives
\begin{align*}
     &\left|1 - \langle \varphi_{2,\eps},\bar \varphi_{2,0} \rangle_{\varrho_\eps^{p-r}}^2\right|\\
&\quad=  \left|1 - \left(\langle \varphi_{2,\eps}, \varphi_{F,\eps} \rangle_{\varrho_\eps^{p-r}}
    +\langle \varphi_{2,\eps},\bar \varphi_{2,0} -\varphi_{F,\eps} \rangle_{\varrho_\eps^{p-r}}\right)^2\right|\\
&\quad\le \left|1 - \langle \varphi_{2,\eps}, \varphi_{F,\eps} \rangle^2_{\varrho_\eps^{p-r}}\right|
    +\left|\langle \varphi_{2,\eps},\bar \varphi_{2,0} -\varphi_{F,\eps} \rangle_{\varrho_\eps^{p-r}}\right|\, 
    \left| \langle \varphi_{2,\eps},\bar \varphi_{2,0} +\varphi_{F,\eps} \rangle_{\varrho_\eps^{p-r}}\right|\\
&\quad\le \left|1 - \langle \varphi_{2,\eps}, \varphi_{F,\eps} \rangle^2_{\varrho_\eps^{p-r}}\right|\\
&\qquad
+ \|\varphi_{2,\eps}\|^2_{L^2(\Z,\varrho_\eps^{p-r})}
\|\bar \varphi_{2,0}-\varphi_{F,\eps}\|_{L^2(\Z,\varrho_\eps^{p-r})}
\left(\|\bar\varphi_{2,0}\|_{L^2(\Z,\varrho_\eps^{p-r})}+\|\varphi_{F,\eps}\|_{L^2(\Z,\varrho_\eps^{p-r})}\right)
  \\
&\le 
 \Xi_4\eps^{-|q-p-r|+\min\{q,p+r\}-\beta}+ \Xi_5  \eps^{\frac{1}{2} \min\{ 1,p+r\}}\\
     &\le \Xi \eps^{\min\{\frac{1}{2},\frac{p+r}{2}, q - 2(q- p -r) - \beta, q-\beta, 2q-(p+r)-\beta \}}     
\end{align*}
for some $\Xi>0$ since $\|\varphi_{2,\eps}\|_{L^2(\Z,\varrho_\eps^{p-r})}=\|\bar\varphi_{2,0}\|_{L^2(\Z,\varrho_\eps^{p-r})}=\|\varphi_{F,\eps}\|_{L^2(\Z,\varrho_\eps^{p-r})}=1$.
{}
\end{proof}
}

\section{Conclusions}
\label{sec:con}
\changed{
We have studied a three-parameter family of weighted
elliptic differential operators, motivated by spectral clustering and
semi-supervised learning problems in
the analysis of large data sets.

We analyzed the perturbative properties of the family \eqref{general-weighted-Laplacian} of 
elliptic operators $\L$, characterizing the sensitive dependence of its low-lying spectrum 
  with respect to the parameters $p, q,r$ in cases where the
  density $\varrho$ concentrates on two clusters. In particular, the theory suggests that
  there is a major change in the behavior of the spectrum of $\L$ when $q = p +r$ versus $q \neq p +r$.
  In the former regime, $\L$ has a uniform spectral gap between the third and second
  eigenvalues  indicating that two clusters are present in $\varrho$,
  while in the latter regime only a  spectral ratio gap may manifest.

In addition, we provided numerical evidence that exemplified and extended our analysis.
Most notably, our numerics show that our bounds on the second eigenvalue are sharp
and that a uniform spectral gap exists between the third and second eigenvalues of $\L$
when $q \le p +r$, whereas only a ratio spectral gap is present when $q>p+r$.
Therefore, in the $q >p +r$ and $q < p +r$ regimes, comparing with our theoretical predictions, our numerics indicate that our 
lower bounds on the third eigenvalues, and hence on 
the  spectral ratio gap, can be sharpened. 
The question of spectral gaps is of interest from a practical point of view as the low-lying
spectral properties govern many unsupervised and semi-supervised
clustering tasks.

Further, we demonstrated a rigorous connection between the geometry of the low-lying
eigenfunctions of $\L$ and the geometry of the density $\varrho$. We showed that
as $\varrho$ concentrates on two clusters, the span of the first
two eigenfunctions of $\L$ approaches certain weighted set functions on the clusters.

In fact, the family of operators $\L$ arises naturally as
  continuum limits of graph Laplacians $L_N$ of the form \eqref{defLN-intro}. We provided a roadmap for rigorous proof of convergence of $L_N$ to $\L$ as $N \to \infty$ in the framework of \cite{trillos2016variational}, but for the more general family of any parameter choices $(p,q,r)$; the full proof is the subject of future research. To support this analysis, we presented numerical evidence in the discrete graphical settings showing the
  manifestation of our continuum spectral analysis on discrete graph Laplacians that are
weighted appropriately with respect to the continuum limits, and this can be observed even in the case of more general data densities $\varrho$ than our theory provides for.

Finally, we provided numerical evidence
that extends our analysis from the binary cluster case to three or five clusters, showing
strong evidence that similar results can be proven in the setting where $\varrho$
concentrates on any number of finitely many clusters.

Our work may be of independent interest within the spectral theory of
elliptic operators.
Furthermore it will be used in our upcoming publication  \cite{HHOS2}
to build on the paper \cite{HHRS1},
which studies consistency of semi-supervised learning on
graphs, to develop a  consistency theory for semi-supervised learning in
the continuum limit.
\vspace{0.1in}

}

\noindent{\bf Acknowledgments} The authors are grateful to
Nicol{\'a}s Garc{\'i}a Trillos for helpful discussions regarding
  the results in Section~\ref{app:AA} concerning
  various graph Laplacians and their  continuum limits.
  \changed{We are also thankful to the anonymous reviewers whose comments and suggestions
    helped us improve an earlier version of this article.}
AMS is grateful to AFOSR (grant FA9550-17-1-0185)
and NSF (grant DMS 18189770) for financial support. 
FH was partially supported by Caltech's von K{\'a}rm{\'a}n postdoctoral instructorship.
BH was partially supported by an NSERC PDF fellowship. 

\bibliographystyle{abbrv}
\bibliography{references}


\appendix


\section{Diffusion maps and weighted graph Laplacians}\label{sec:diffmaps}

  We note from Remark~\ref{rem:sde} that when $p=q$ and $r=0$ the limiting
graph Laplacian $\mathcal{L}$ is the generator of a reversible
diffusion process with invariant density proportional to $\varrho^q$.
The connection between the graph Laplacian $L_N$ in \eqref{defLN-intro}
and diffusions was first
established in the celebrated paper \cite{CoifmanLafon2006}
by Coifman and Lafon, through
the diffusion maps introduced therein. In this appendix we further
elucidate these connections.

We fix a probability density $\varrho\in L^1(\Omega)$ for any set $\Omega\subset \R^d$ and introduce the following functions for $x,y\in\Omega$:
\begin{align*}
\tilde W(x,y) = \eta_\delta(|x-y|)
\end{align*}
where $\eta$ is a rotation-invariant normalized kernel, $\int_\Omega \eta_\delta(|x|)\,dx =1$, with a fixed 
scale parameter $\delta$, and with associated degree function
\begin{align*}
\tilde d(x)=\int_\Omega \tilde W(x,y)\varrho(y)\,dy\,.
\end{align*}
Note that $\tilde d(x)$ approximates $\varrho(x)$ as
  $\eta_\delta$ converges weakly to the Dirac delta distribution.
We suppress the dependence of $\tilde d$ and $\tilde W$ on $\delta$ for brevity.
Given a parameter $\alpha\in\R$, we now construct the weighted kernel
\begin{align*}
W(x,y)=\frac{\tilde W(x,y)}{\tilde d(x)^\alpha \tilde d(y)^\alpha}
\end{align*}
with associated degree function
\begin{align*}
d(x)=\int_\Omega W(x,y)\varrho(y)\,dy\,.
\end{align*}
The kernel $W$ gives rise to an integral operator $\mcl K:L^1(\Omega)\to L^1(\Omega)$,
\begin{align*}
 \mcl Kf(x)=\int_{\Omega} W(x,y)f(y)\varrho(y)\,dy\,.
\end{align*}
Then $d(x)=\mcl K\mbf{1}_\Omega(x)$. Normalizing $\mcl K$ gives a Markov operator $\mcl P:L^1(\Omega)\to L^1(\Omega)$,
\begin{align*}
\mcl Pf(x):=\frac{1}{\mcl K\mbf{1}_\Omega(x)} \mcl Kf(x)
= \int_\Omega p(x,y)f(y)\varrho(y)\,dy
\end{align*}
with anisotropic Markov transition kernel
\begin{align*}
p(x,y)=\frac{W(x,y)}{d(x)}\,.
\end{align*}
Observe that $\mcl P\mbf{1}_\Omega=\mbf{1}_\Omega$, and so $\mcl P$ leaves constants unchanged.\\
\noindent
\textbf{Discrete setting}. Given $N$ samples $x_j\sim \varrho$, we define analogously to the above the matrix $\tilde W_N$ with entries
$$
\tilde W_{ij} = \tilde W (x_i,x_j)
$$
with associated degree matrix $\tilde D_N$,
$$
\tilde D_{ij}=diag\left(\tilde d_i\right)\,,\qquad
\tilde d_i= \sum_{k=1}^N \tilde W_{ik}\,.
$$
From the above, we construct the weighted similarity matrix $W_N$ with entries
$$
W_{ij}=\frac{\tilde W_{ij}}{\tilde d_i^\alpha \tilde d_j^\alpha}
$$
with associated degree matrix $D_N$,
$$
 D_{ij}=diag\left( d_i\right)\,,\qquad
 d_i= \sum_{k=1}^N W_{ik}\,.
$$
To make the connection between this discrete setting and the continuous analogue above, we use the degree functions of Subsection~\ref{sec:cvEnergies},
\begin{align*}
\tilde d^N(x)&=\frac{1}{N}\sum_{j=1}^N\tilde W(x,x_j)\\
d^N(x)&=\frac{1}{N}\sum_{j=1}^N W(x,x_j)
=\frac{1}{N}\sum_{j=1}^N \frac{\tilde W(x,x_j)}{\left(\tilde d^N(x)\right)^\alpha \left(\tilde d^N(x_j)\right)^\alpha}\,.
\end{align*}
They correspond exactly to
$d(x)$ and $\tilde d(x)$ with $\varrho$ substituted by the empirical density $\mu_N:=\tfrac{1}{N}\sum_{i=1}^N \delta_{x_i}$. Then
$$
\tilde d_i=N\tilde d^N(x_i)\,,\qquad 
 d_i=N^{1-2\alpha} d^N(x_i)\,,\qquad 
W_{ij}=\frac{1}{N^{2\alpha}}W(x_i,x_j)\,,
$$
and so $\tilde d_i/N$ approximates $\varrho(x_i)$ as $\eta_\delta$ converges to the Dirac delta distribution for large $N$.
Finally, the operators $\mcl K$ and $\mcl P$ are approximated empirically by matrices $W_N/N$ and $P_N$, where $P_N$ has entries
\begin{align*}
P_{ij}=\frac{W(x_i,x_j)}{Nd^{N}(x_i)} = \frac{N^{2\alpha}W_{ij}}{N^{2\alpha}d_i}\,,
\end{align*}
and so
$$
P_N=D_N^{-1}W_N \,.
$$
In \cite{CoifmanLafon2006}, the graph Laplacian matrix $\bar L_N$ is defined as
\begin{align*}
\bar L_N=\frac{I_N-P_N}{\delta}= \frac{1}{\delta } D_N^{-1}\left( D_N-W_N\right)
= \frac{1}{\delta} L_N\,,
\end{align*}
where $I_N$ denotes the identity matrix, and $L_N$ is our graph Laplacian matrix as defined in \eqref{defLN-intro} with $p=q=2(1-\alpha)$ and $r=0$. Note that $\bar L_N$ is not symmetric.
\\

\noindent
\textbf{Generator of a diffusion semi-group}. Taking $\delta\to 0$, we see that 
\begin{gather*}
\tilde W(x,y)\to \delta_{x=y}\,,\\
\tilde d(x)\to \varrho(x)\,,\qquad   d(x)=\mcl K\mbf{1}_\Omega(x) \to \varrho(x)^{1-2\alpha}\,,
\end{gather*}
and so $\mcl P$ converges to the identity operator $\rm{Id}$. Defining the operator
\begin{align*}
\mcl G = \frac{{\rm Id}-\mcl P}{\delta} 
\end{align*}
analogously to the discrete setting, it was shown in \cite[Thm.~2]{CoifmanLafon2006} that 
\begin{align*}
  \lim_{\delta\to 0} \mcl G f= -\mcl L f
\end{align*}
for $f$ in any finite span of the eigenfunctions of the Laplace-Beltrami operator on a compact submanifold of $\Omega$. Here, $\mcl G$ is the
infinitesimal generator of a Markov chain, and $\mcl L$ is the weighted elliptic operator defined in \eqref{general-weighted-Laplacian} for the parameter choices $p=q=2(1-\alpha)$ and $r=0$. 
In this sense, the operator $\mcl P$ is an approximation to the semi-group
$$
e^{\delta\mcl L}= {\rm Id}+\delta \mcl L +\mcl O(\delta^2)
$$
associated with the infinitesimal generator $\mcl L$,
\begin{align*}
  -\mcl L f
  &= \frac{1}{\varrho^{2(1-\alpha)}}\nabla\cdot
    \left(\varrho^{2(1-\alpha)} \nabla f\right)\\
&= \Delta f + 2(1-\alpha)\varrho^{-1}\nabla\varrho\cdot\nabla f\\
  &= \Delta f +\nabla\log\left(\varrho^{ 2(1-\alpha)}\right)\cdot\nabla f\,.
\end{align*}
More precisely, the operator $\mcl L$ is the infinitesimal generator of the reversible diffusion process
\begin{align*}
  dX_t = -\nabla \Psi(X_t) dt + \sqrt{2}\, dB\,,
\end{align*}
where $B$ denotes a  Brownian motion in $\R^d$ with associated potential 
$$\Psi(x)=-\log\left(\varrho(x)^{ 2(1-\alpha)}\right)$$
and invariant measure proportional to $\varrho^{ 2(1-\alpha)}$
satisfying $\mcl L^* e^{-\Psi}=\mcl  L^* \rho^{2(1-\alpha)}=0$. In this sense, the discrete graph Laplacian matrix $\bar L_N$ introduced above serves as an approximation of the generator $-\mcl L$.

  In \cite{CoifmanLafon2006}, Coifman and Lafon discuss the cases (i) $\alpha =0$ ($q=2$) when the graph Laplacian has isotropic weights and $W=\tilde W$, (ii) $\alpha=1/2$ ($q=1$) when the Dirichlet energy of $\mcl L$ is linear in $\varrho$, and (iii) $\alpha=1$ ($q=0$), when $-\mcl L f = \Delta f$, and so the Markov chain corresponding to $\mcl G$ converges (as $\delta \to 0$) to the Brownian motion in $\Omega$ with reflecting boundary conditions.


There is a well-known connection between the generator of reversible
diffusion processes and Schr{\"o}dinger operators \cite{pavliotis2014stochastic}.
Following the above connections between limiting graph Laplacians
and generators of diffusion processes with invariant measures proportional
to $\varrho^{(1 - 2 \alpha)}$, we connect the  operator $\mcl L$ to certain Schr\"odinger operators as follows. Define

\begin{align*}
\mcl S u := \Delta u -u  \frac{\Delta\left(\varrho^{1-\alpha}\right)}{\varrho^{1-\alpha}} \,,
\end{align*}
then we can write for $u=f\varrho^{1-\alpha}$,
\begin{align*}
-\mcl L f
  =  \frac{\Delta\left(f\varrho^{1-\alpha}\right)}{\varrho^{1-\alpha}}
  - \frac{\Delta\left(\varrho^{1-\alpha}\right)}{\varrho^{1-\alpha}} f
= \frac{\mcl S u}{\varrho^{1-\alpha}}\,.
\end{align*}




\section{Function Spaces}\label{app:AC}
Throughout this section $\varrho$ is taken to be a smooth probability density function 
with full support on a bounded open set $\Omega \subset \mbb R^d$
with $C^1$ boundary
which is bounded from above and below by positive constants
as in \eqref{conditions-on-rho}, i.e.,
\begin{equation}\label{upper-lower-bound-assumption-on-rho}
  0<\varrho^- \le \varrho(x) \le \varrho^+ < +\infty, \qquad \forall x \in \bar{\Omega}. 
\end{equation}
Our first task is to establish the equivalence between regular $L^p(\Omega)$ spaces
and the weighted spaces $L^p(\Omega, \varrho)$. In fact, a straightforward calculation
using \eqref{upper-lower-bound-assumption-on-rho} implies the following lemma.

\begin{lemma}\label{equivalence-of-Lp-spaces}
  Let $\varrho$ be a smooth \PDF on $\Omega$ satisfying \eqref{upper-lower-bound-assumption-on-rho}
  and let $u \in L^p(\Omega)$ for $p \ge 0$. Then
  \begin{equation*}
    \varrho^- \|   u \|_{L^p(\Omega)}^p \le \| u\|_{L^p(\Omega, \varrho)}^p \le
    \varrho^+ \| u \|_{L^p(\Omega)}^p,
  \end{equation*}
  i.e.,  $L^p(\Omega) = L^p(\Omega, \varrho)$. 
\end{lemma}
Given constants $(p, q, r) \in  \R^3$ we  consider the weighted Sobolev spaces
$H^1(\Omega, \varrho)$ introduced in section~\ref{ssec:prelim}.
We now have:

\begin{lemma}\label{equivalence-of-H1-spaces}
  Let $\varrho\in C^\infty(\bar{\Omega})$ be a smooth \PDF  satisfying \eqref{upper-lower-bound-assumption-on-rho}
  and let $u \in H^1(\Omega, \varrho)$ with parameters $(p,q,r) \in \R^3$. Then
  there exist constants $C^\pm(q, \varrho^\pm) >0$
  so that
  \begin{equation*}
    C^- \left\|  \frac{u}{\varrho^r} \right\|_{H^1(\Omega)}^2 \le \| u\|_{H^1(\Omega, \varrho)}^2 \le C^+ \left\| \frac{u}{\varrho^r} \right\|_{H^1(\Omega)}^2.
  \end{equation*}
\end{lemma}
\begin{proof}
  Since $\varrho$ satisfies \eqref{upper-lower-bound-assumption-on-rho} then 
  \begin{equation*}
    (\varrho^-)^q \left| \nabla \left(  \frac{u}{\varrho^r} \right) \right|^2 dx
   \le   \int_\Omega \varrho^q \left| \nabla \left(  \frac{u}{\varrho^r} \right) \right|^2 dx  \le  (\varrho^+)^q \int_\Omega \left| \nabla \left(  \frac{u}{\varrho^r} \right) \right|^2 dx.
  \end{equation*}
  Then the desired result follows immediately by  Lemma~\ref{equivalence-of-Lp-spaces} applied to  $L^2$ norms..
{}
\end{proof}

With the equivalence between the weighted and regular $L^p$ and $H^1$ spaces
established. We can present the  following compact embedding
as a consequence of 
the 
Rellich-Kondrachov Theorem \cite[Ch.~5.7, Thm~1]{Evans}:
\begin{proposition}\label{V1-compact-embedding}
  Let $\varrho \in C^\infty(\bar{\Omega})$
  be a \PDF satisfying \eqref{upper-lower-bound-assumption-on-rho}
  and
  fix $(p, q, r) \in  \R^3$. Then $H^1(\Omega, \varrho)$
  is compactly embedded in $L^2(\Omega, \varrho^{p-r})$.
\end{proposition}

\section{Min-Max Principle}\label{app:AB}

The min-max principle \cite[Ch.~1 Sec.~6.10]{kato}
is readily applied to our specific
setting to obtain the following:

\begin{proposition}\label{thm:max-min}
Fix $(p,q,r) \in \R^3$.
For any open bounded set $\Omega\subset\mbb{R}^d$ with $\partial \Omega\in C^{1,1}$,
and for a given density $\varrho\in C^\infty(\overline{\Omega})$ satisfying Assumption~\ref{Assumptions-on-rho}, let $\sigma_1\le \sigma_2\leq...\leq \sigma_j\leq ...$ be the sequence of eigenvalues of the Neumann operator 
 $$
 \mcl{L}= -\frac{1}{\varrho^p}\nabla\cdot\left(\varrho^q\nabla\left(\frac{\cdot}{\varrho^r}\right)\right)
 $$ in $V^1(\Omega, \varrho)$, repeated in accordance with their multiplicities, and let $\left\{\varphi_{j}\right\}_{j\in\mbb{N}}$ be a corresponding Hilbertian basis of eigenvectors in $V^1(\Omega, \varrho)$; then 
\begin{equation*}
  \Langle \varrho^{q} \nabla \left( \frac{\varphi_{j}}{\varrho^r} \right),
  \nabla \left( \frac{v}{\varrho^r}\right) \Rangle =  \sigma_{j} \Langle
  \varrho^{p-r} \varphi_{j} , v \Rangle,
    \qquad \varphi_{j}, v \in V^1(\Omega, \varrho).
    \end{equation*}
Define the Rayleigh quotient of $\mcl{L}$  by
   \begin{equation*}\label{rayleigh-quotient-app}
     \mcl R(u):= \frac{\Langle \mcl{L} u,u\Rangle_{\varrho^{p-r}}}{\Langle u,u\Rangle_{\varrho^{p-r}}}
     = \frac{\int_{\Omega}  \left| \nabla \left( \frac{u}{\varrho^r} \right)\right|^2\varrho^q dx }{ \int_{\Omega}
     |u|^2 \varrho^{p-r} dx}, \qquad u \in V^1(\Omega, \varrho)\,.
 \end{equation*}
 Denote by $\mcl{S}_n$ the class of all $n$-dimensional linear subspaces in $V^1(\Omega, \varrho)$,
 and by $M^\bot$ the orthogonal subspace of $M$ in $V^1(\Omega,\varrho)$.
Then we have
\begin{align}
 \sigma_n
 &=\min_{M\in\mcl{S}_n}\,\,\max_{v\in M,v\neq 0}\,\, \mcl{R}(v)\label{min-max-principle}\\
  &=\max_{M\in\mcl{S}_{n-1}}\,\,\min_{v\in M^\bot,v\neq 0}\,\, \mcl{R}(v)\,.\label{max-min-principle}
\end{align}

\end{proposition}

\section{Weighted Cheeger's inequality}\label{sec:proof-isoperim-ineqa}
 Given positive measures $\mu$, $\nu$ on $\Omega'\subset \R^d$, define the isoperimetric function $\mcl J$ for any subset $\Omega\subset \Omega'$ by
    \begin{equation*}
      \mcl J( \Omega, \mu,\nu) := \frac{ | \partial \Omega |_{\mu}}{\min\{ | \Omega|_{\nu},
        |\Omega' \setminus \Omega|_{\nu}\}}.
    \end{equation*}
Here, we use the notation
$$
|\Omega|_\nu:=\nu(\Omega)\,,
$$
and define the $\mu$-weighted Minkowski boundary measure of $\Omega$ by
    \begin{equation*}
      |\partial \Omega|_\mu := \liminf_{\delta \downarrow 0}
      \frac{1}{\delta} \left[ |\Omega_\delta|_\mu - | \Omega|_\mu \right]\,,
    \end{equation*}
with $\Omega_\delta$ as defined in \eqref{D-eps-definition},
\begin{equation*}
\Omega_\delta := \{ x : \dist(x, \Omega) \le \delta \}\,.
\end{equation*}
We show the following weighted version of Cheeger's inequality.
\begin{proposition}[Weighted Cheeger's inequality]\label{thm:cheeger}
Let $\mu$, $\nu$ be absolutely continuous measures with respect to the Lebesgue measure with $C^\infty$ densities that are uniformly bounded above and below with positive constants on $\Omega'$. Suppose there exists a constant $h>0$ so that
      \begin{equation}\label{cheegerhyp}
        h \le  \inf_\Omega \mcl J( \Omega, \mu,\nu),
      \end{equation}
      where the infimum is over open subsets $\Omega \subset \Omega'$
      such that $|\Omega |_{\nu} \le \frac{1}{2} | \Omega'|_{\nu}$. Then the following Poincar\'e inequality holds:
\changed{
\begin{equation*}
  \left(\sup_{x\in\Omega'}\left|\frac{d\mu}{d\nu}(x)\right|\right)^{-1}\,\frac{h^2}{4} \int_{\Omega'} |f - \bar f_{\Omega'} |^2 d\nu 
  \le \int_{\Omega'} \left| \nabla f \right|^2 d\mu\,,
\end{equation*}
}
where $\bar f_{\Omega'}$ denotes the average of $f$ with respect to $\nu$,
\begin{equation*}
\bar f_{\Omega'} := \frac{\int_{\Omega'} f\,d\nu}{|\Omega'|_\nu}\,.
\end{equation*}
\end{proposition}
This is a generalization of the weighted Cheeger's inequality as here we may take different measures $\mu$ and $\nu$, whereas $\mu=\nu$ in \cite{bakry2013analysis}. The proof can readily be generalized  from \cite[Prop.~8.5.2]{bakry2013analysis} to this setting.
\begin{proof}
It follows from the  co-area formula \cite[Thm.~8.5.1]{bakry2013analysis}  that
for every Lipschitz function $f$ on $\Omega'$,
\begin{equation}
  \label{co-area-formula}
  \int_{-\infty}^\infty |\partial S(f,t) |_{\mu} dt \le \int_{\Omega'} |\nabla f | \,d\mu\,,
\end{equation}
where $S(f, t) := \{ x \in \Omega' : f(x) >t \}$ for $t \in \mbb R$. Now
let $g$ be a positive Lipschitz function on $\Omega$ such that $| S(g,  t) |_{\nu} \le \frac{1}{2} | \Omega' |_{\nu}$. Then 
by the hypothesis \eqref{cheegerhyp}  we have for $t \ge0$,
\begin{equation*}
  h \min \{ | S(g, t) |_{\nu},
        |\Omega' \setminus S(g,t)|_{\nu}\} \le |\partial S(g,t) |_{\mu},
\end{equation*}
which together with \eqref{co-area-formula} gives
\begin{equation}\label{iso-perim-proof-interim-display-1}
   h \int_0^\infty \min \{ | S(g, t) |_{\nu},
        |\Omega' \setminus S(g,t)|_{\nu}\} dt \le \int_{\Omega'} |\nabla g | \,d\mu\,.
\end{equation}
Now let $f: \Omega' \to \mbb R$ be Lipschitz and denote by $m$ a median of $f$
with respect to $\nu$, i.e., $m \in \mbb R$ such that
\begin{equation*}
  | \{x \in \Omega' : f(x) \ge m\} |_{\nu} \le \frac{1}{2} | \Omega' |_{\nu},
  \qquad \text{and} \qquad
  | \{x \in \Omega' : f(x) \le m\} |_{\nu} \le \frac{1}{2} | \Omega' |_{\nu}.
\end{equation*}
Proceeding in the same way as in proof of \cite[Prop.~8.5.2]{bakry2013analysis}
we define $F_+ = \max\{ f - m, 0\}$ and $F_- = \max\{ m - f, 0\}$ and by definition of
the median we have for $t >0$,
  \begin{equation*}
  | S(F_+^2,  t) |_{\nu} \le \frac{1}{2} | \Omega' |_{\nu},
  \qquad \text{and} \qquad
  | S(F_-^2, t) |_{\nu} \le \frac{1}{2} | \Omega' |_{\nu}.
\end{equation*}
Applying \eqref{iso-perim-proof-interim-display-1} with $g  = F_+^2$ and $g = F_-^2$
and adding the two inequalities yields
\begin{align*}
h\int_{\Omega'}|f-m|^2\,d\nu
&=h\int_{\Omega'}F_+^2\,d\nu +h\int_{\Omega'}F_-^2\,d\nu \\
&=  h \int_0^\infty | S(F_+^2, t) |_{\nu}\, dt+h \int_0^\infty | S(F_-^2, t) |_{\nu}\, dt\\
&\le  \int_{\Omega'} |\nabla (F_+^2) | \,d\mu+ \int_{\Omega'} |\nabla (F_-^2) | \,d\mu\,.
\end{align*}
By the Cauchy-Schwartz inequality,
\changed{
\begin{align*}
\int_{\Omega'} |\nabla (F_\pm^2) | \,d\mu
&=2 \int_{\Omega'} F_\pm |\nabla F_\pm| \,d\mu
\le 2\left(\int_{\Omega'} |F_\pm|^2 \,d\mu\right)^{1/2}
\left(\int_{\Omega'} |\nabla F_\pm|^2 \,d\mu\right)^{1/2}\\
&\le 2\left(\int_{\Omega'} |f-m|^2 \,d\mu\right)^{1/2}
\left(\int_{\Omega'} |\nabla F_\pm|^2 \,d\mu\right)^{1/2}\\
&\le 2\left(\sup_{x\in\Omega'}\left|\frac{d\mu}{d\nu}(x)\right|\right)^{1/2} \left(\int_{\Omega'} |f-m|^2 \,d\nu\right)^{1/2}
\left(\int_{\Omega'} |\nabla F_\pm|^2 \,d\mu\right)^{1/2}\,.
\end{align*}
The previous estimate with the fact that $F_\pm$ have disjoint support, gives
\begin{equation*}
 \left(\sup_{x\in\Omega'}\left|\frac{d\mu}{d\nu}(x)\right|\right)^{-1} \frac{h^2}{4} \int_{\Omega'} |f - m |^2 \,d\nu
  \le \int_{\Omega'} | \nabla f|^2\,d\mu\,.
\end{equation*}
}
for any median of $f$. Finally, minimizing the left-hand side over $m$ gives the desired lower bound with $m=\bar f_{\Omega'}$, which concludes the proof.
\end{proof}

\begin{proof}[Proof of Lemma~\ref{weighted-isoperimetric-inequality}]
Apply Theorem~\ref{thm:cheeger} with $d\mu(x)=\varrho_\eps^q(x)dx$ and $d\nu(x)=\varrho_\eps^{p+r}(x)dx$.
 Setting $u = f \varrho_\eps^r$ yields
\begin{equation*}
  \changed{
    \left(\sup_{x\in\Omega'} \varrho_\epsilon^{q- p -r} \right)^{-1} \frac{h^2}{4} \int_{\Omega'} \left| u -\bar{u}_{\Omega'} \varrho_\eps^r  \right|^2 \varrho_\eps^{p-r} dx
    \le \int_{\Omega'} \left| \nabla \left( \frac{u}{\varrho_\eps^r} \right) \right|^2 \varrho_\eps^q dx,
    }
\end{equation*}
which concludes the proof for Lipschitz functions $u$.
The desired result on $V^1(\Omega', \varrho_\eps)$ then follows by a density argument,
and noting that $\bar{u}=0$ in that case.
\end{proof}


\renewcommand\thefigure{4.\arabic{figure}} 
\renewcommand\thetable{4.\arabic{table}} 

 \begin{figure}[!ht]
   \centering
\subfloat[]{%
   \includegraphics[width=.46\textwidth]{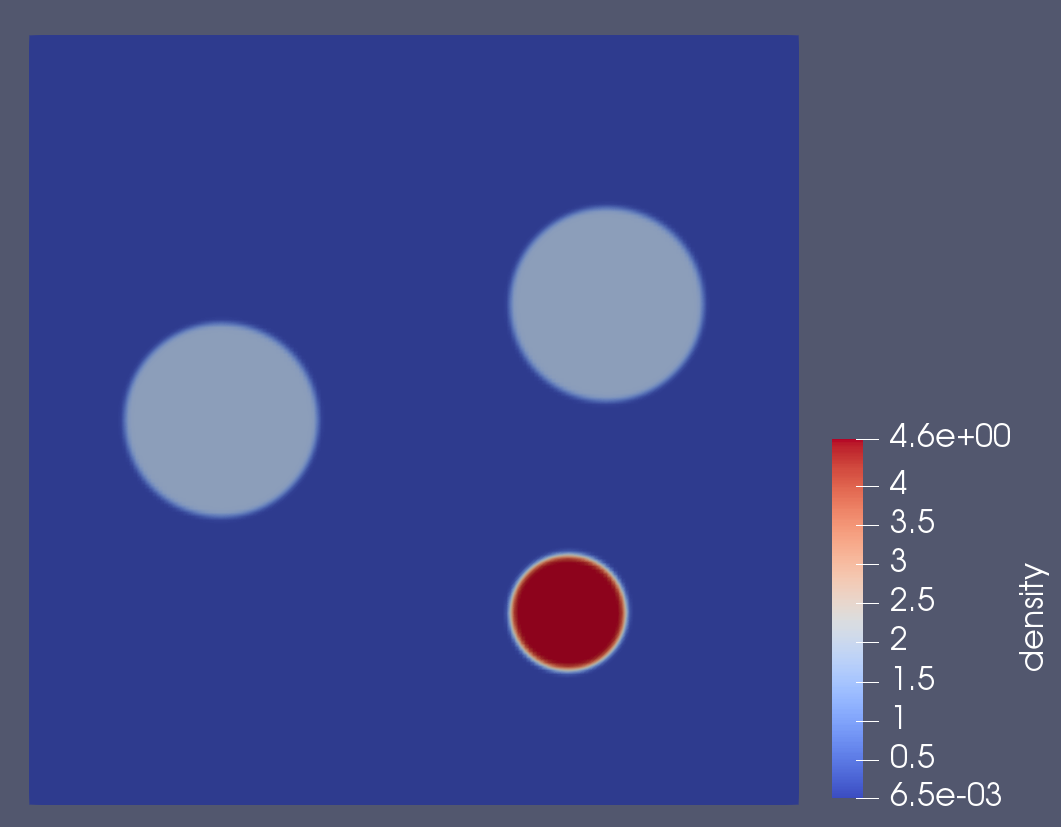}
 }
 \subfloat[]{%
   \includegraphics[width=.45\textwidth]{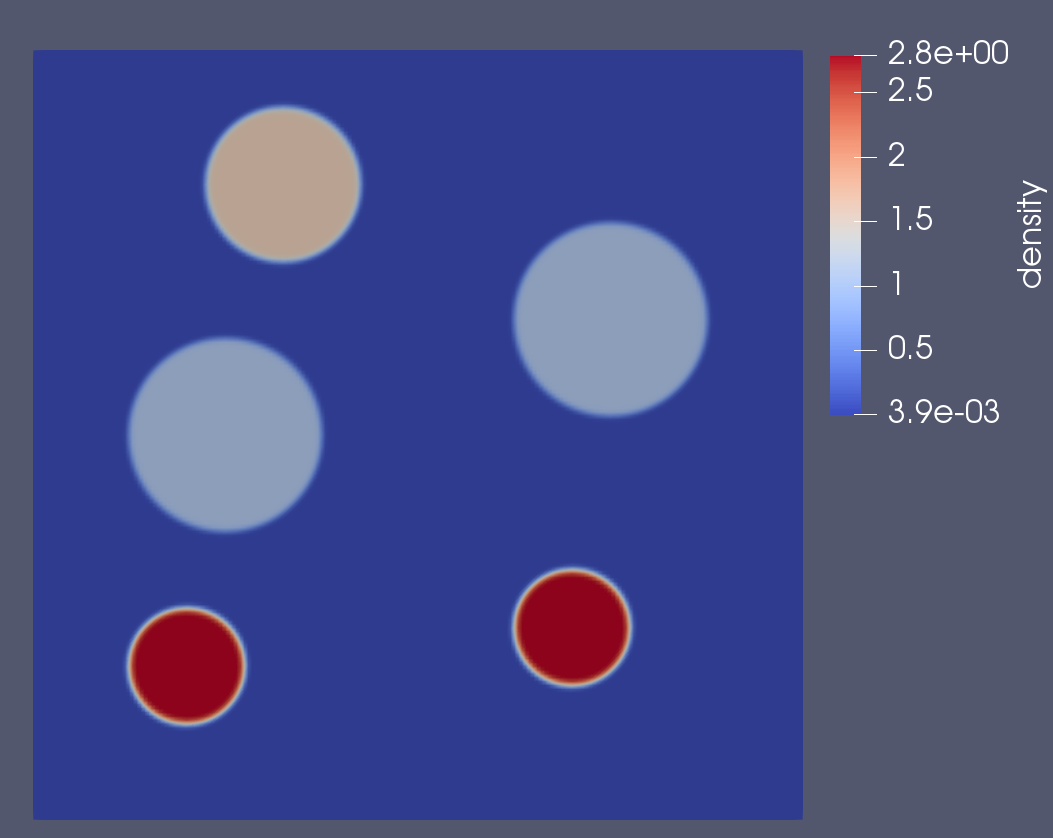}
     }
     \caption{Plot of the densities $\varrho_\eps$ of the form \eqref{numerics-rho-def} with three and five clusters for $\eps= 0.0125$.}
     \label{fig:density-plots-three-five-clusters}
 \end{figure}

   \begin{figure}[!ht]
     \subfloat[]{%
       \includegraphics[width=0.49\textwidth, clip=true, trim =1cm 6cm 2cm 5cm]{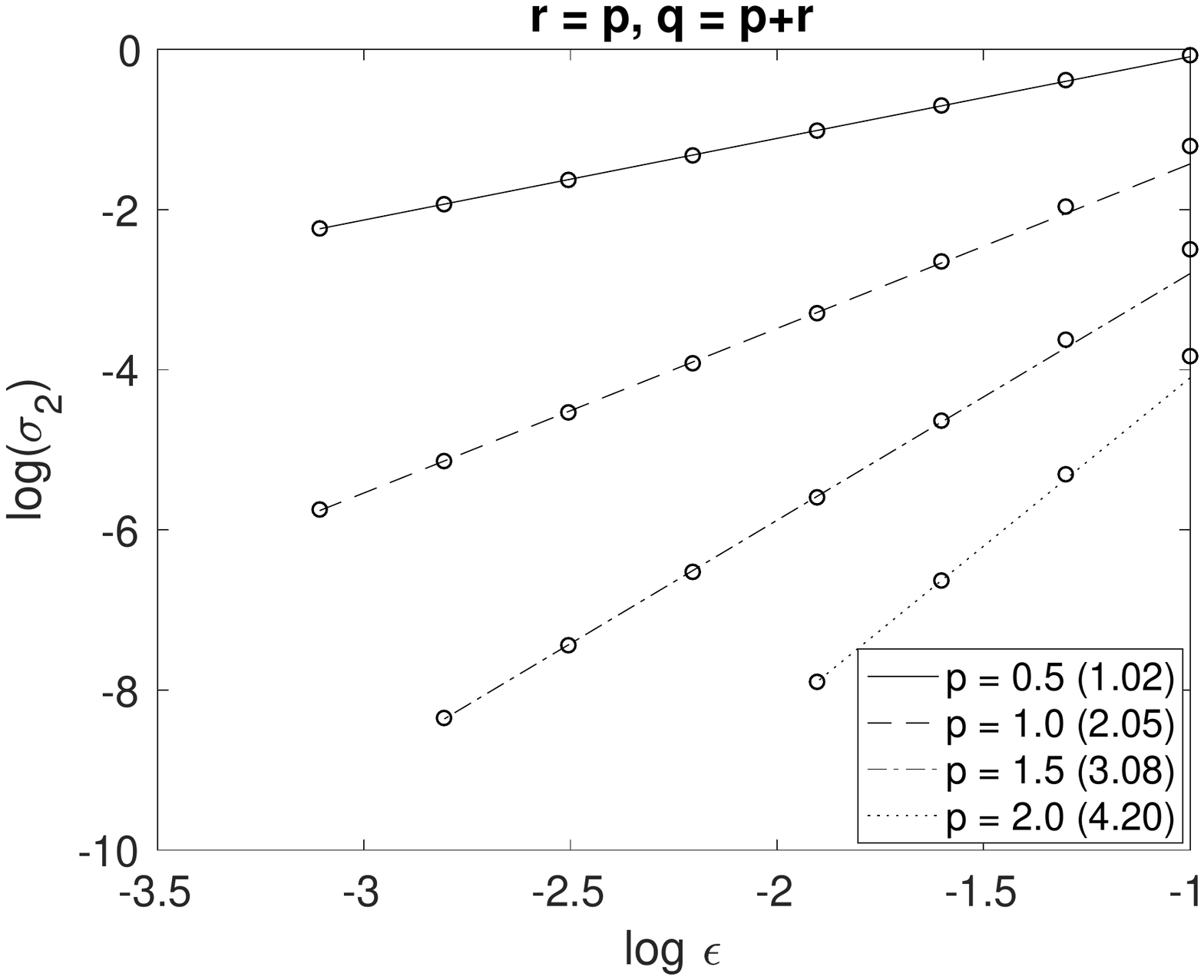}
     }
     \hfill
     \subfloat[]{%
       \includegraphics[width=0.49\textwidth, clip=true, trim =1cm 6cm 2cm 5cm]{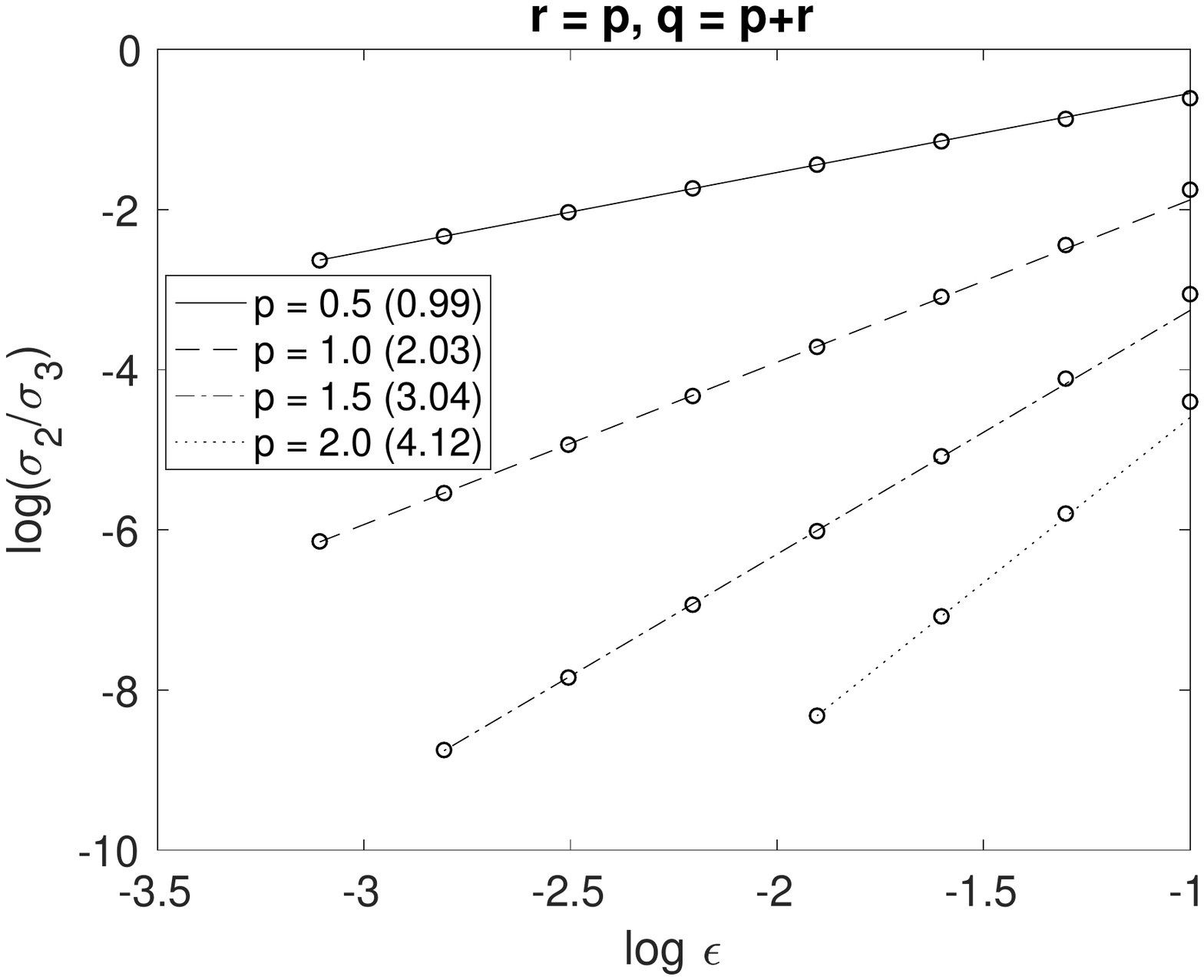}
     } \\
     \subfloat[]{%
       \includegraphics[width=0.49\textwidth, clip=true, trim =1cm 6cm 2cm 5cm]{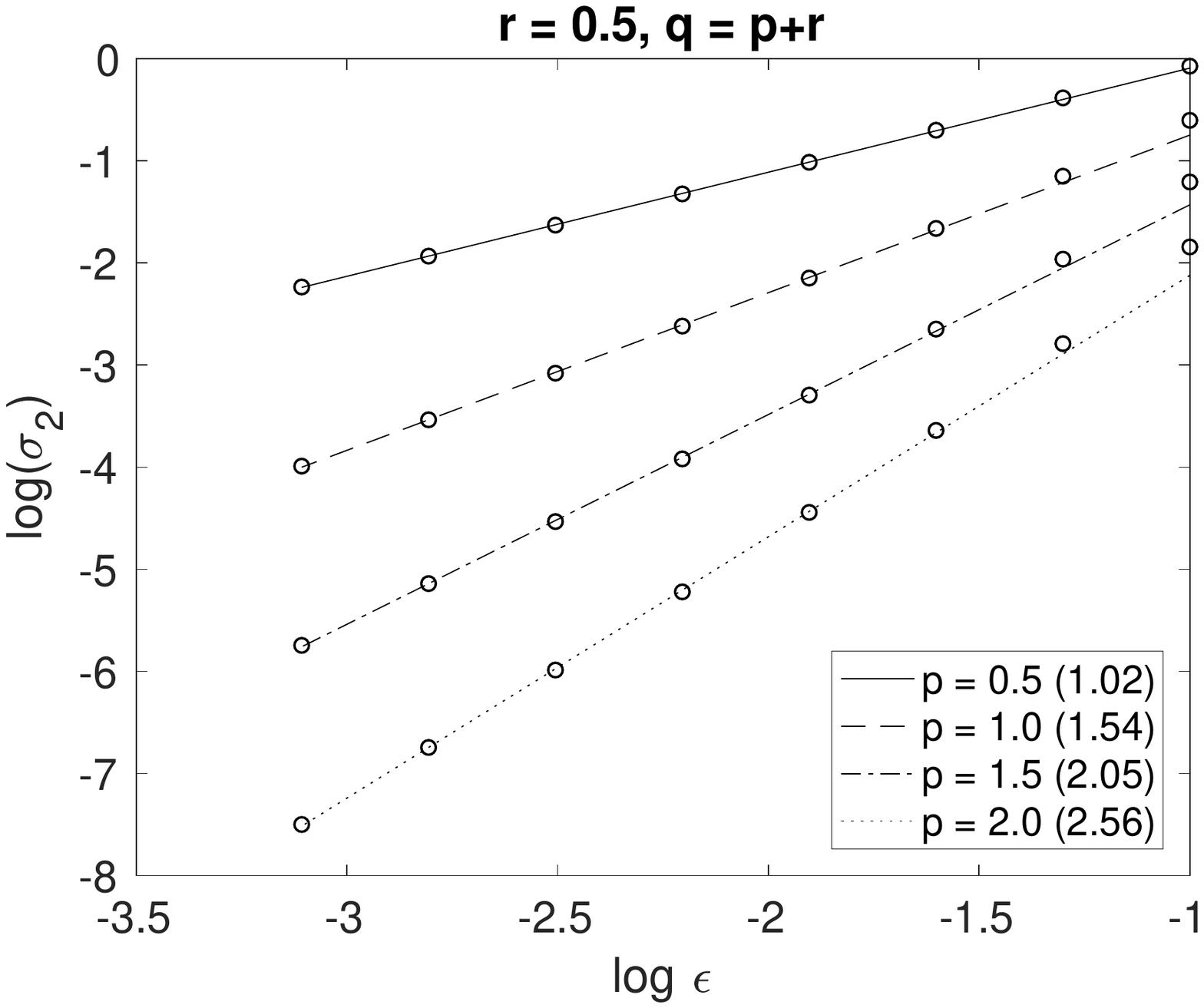}
     }
     \hfill
     \subfloat[]{%
       \includegraphics[width=0.49\textwidth, clip=true, trim =1cm 6cm 2cm 5cm]{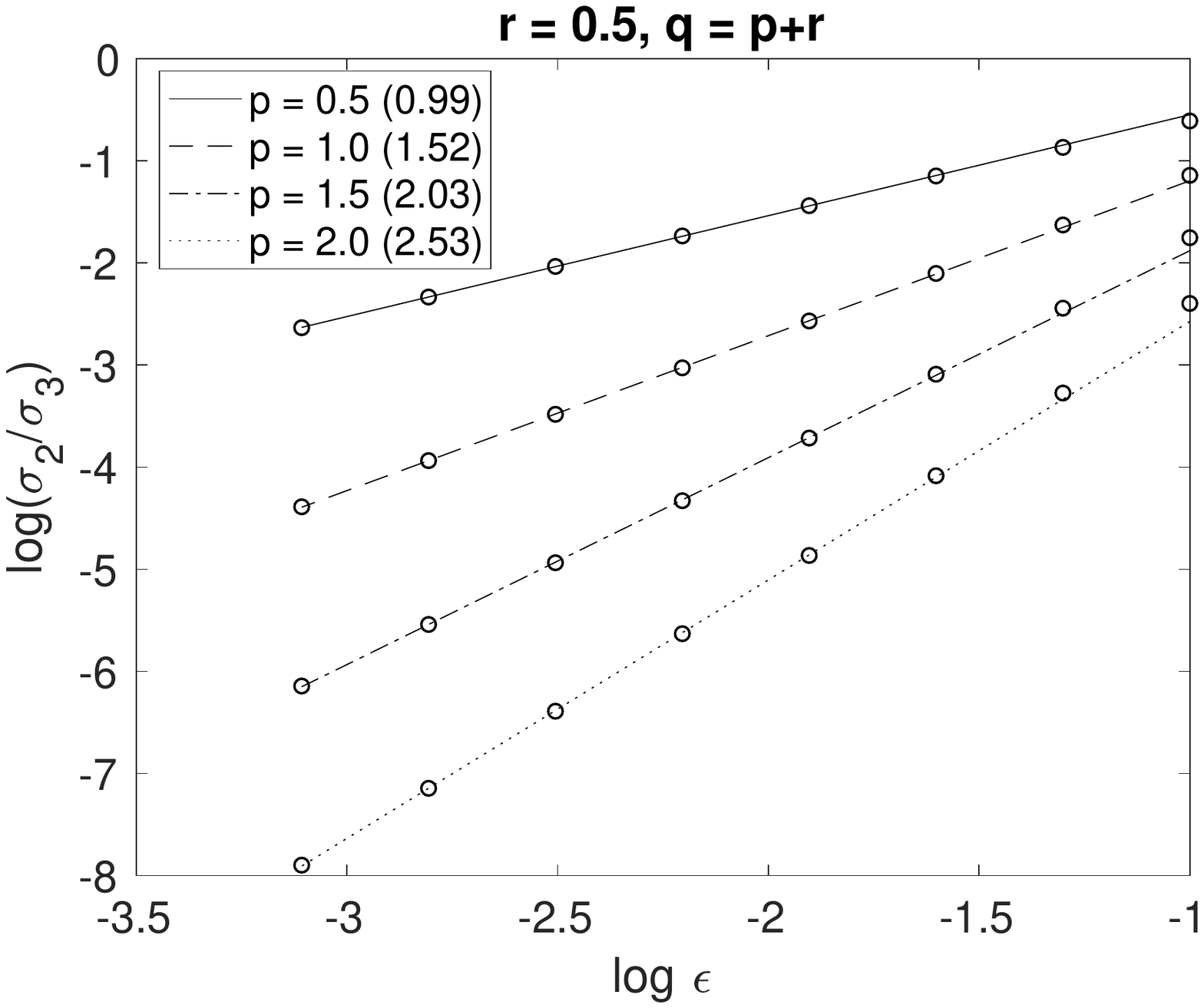}
     }
     \caption{Variation of the second and third eigenvalues of $\mcl L_\eps$ in the balanced case
       with $q = p + r$ and for various values of $p \in [0.5, 2]$. 
(a, b) consider $r=p$; (c, d) consider fixed $r=0.5$.
(a, c) show $\log(\sigma_{2,\eps})$ vs $\log(\eps)$ while (b, d) show
       $\log(\sigma_{2,\eps}/\sigma_{3,\eps})$ vs $\log(\eps)$. The values reported
       in the brackets in the legends are numerical approximations to the slope of the
       lines for different values of $p$.
}
     \label{fig:balanced01}
   \end{figure}

   \begin{table}[htp]
     \centering
     \begin{tabular}{|c |c  | c | c | c| c|}
       \hline
       & & \multicolumn{2}{c |}{$\frac{\log(\sigma_{2,\eps})}{\log{\eps}}$} &
       \multicolumn{2}{c |}{$\frac{\log(\sigma_{2,\eps}) - \log(\sigma_{3,\eps})}{\log{\eps}}$}\\
       \hline
       $p$ & $r$  & Analytic & Numerical & Analytic & Numerical \\
       \hline
       $0.5$& $0.5$ & 1.00  & 1.02   & 1.00 & 0.99  \\
       \hline
       $1.0$& $1.0$ & 2.00  & 2.05  & 2.00  & 2.03 \\
       \hline
       $1.5$& $1.5$ & 3.00  & 3.08   & 3.00 & 3.04  \\
       \hline
       $2.0$& $2.0$ & 4.00 & 4.20 &4.00 & 4.12 \\
       \hline
       $1.0$& $0.5$ & 1.50  & 1.54  & 1.50  & 1.52 \\
       \hline
       $1.5$& $0.5$ & 2.00  & 2.05   & 2.00 & 2.03  \\
       \hline
       $2.0$& $0.5$ & 2.50 & 2.56  &2.50 & 2.53 \\
       \hline
     \end{tabular}
     \caption{Comparison between  numerical approximation of
       the rate of decay of $ \log(\sigma_{2,\eps})$ and $\log(\sigma_{2,\eps}/\sigma_{3,\eps})$ as
       functions of $\log(\eps)$ and the analytic predictions in
       Theorem~\ref{thm:L-eps-low-eigenvals} and Corollary~\ref{cor:ratiogap}
        for the balanced case with $q = p +r $ and different choices of $p$ and $r$.}
      \label{tab:balanced-r-eq-p}
   \end{table}

   \begin{figure}[!ht]
     \subfloat[]{%
       \includegraphics[width=0.49\textwidth, clip=true, trim =1cm 6cm 2cm 5cm]{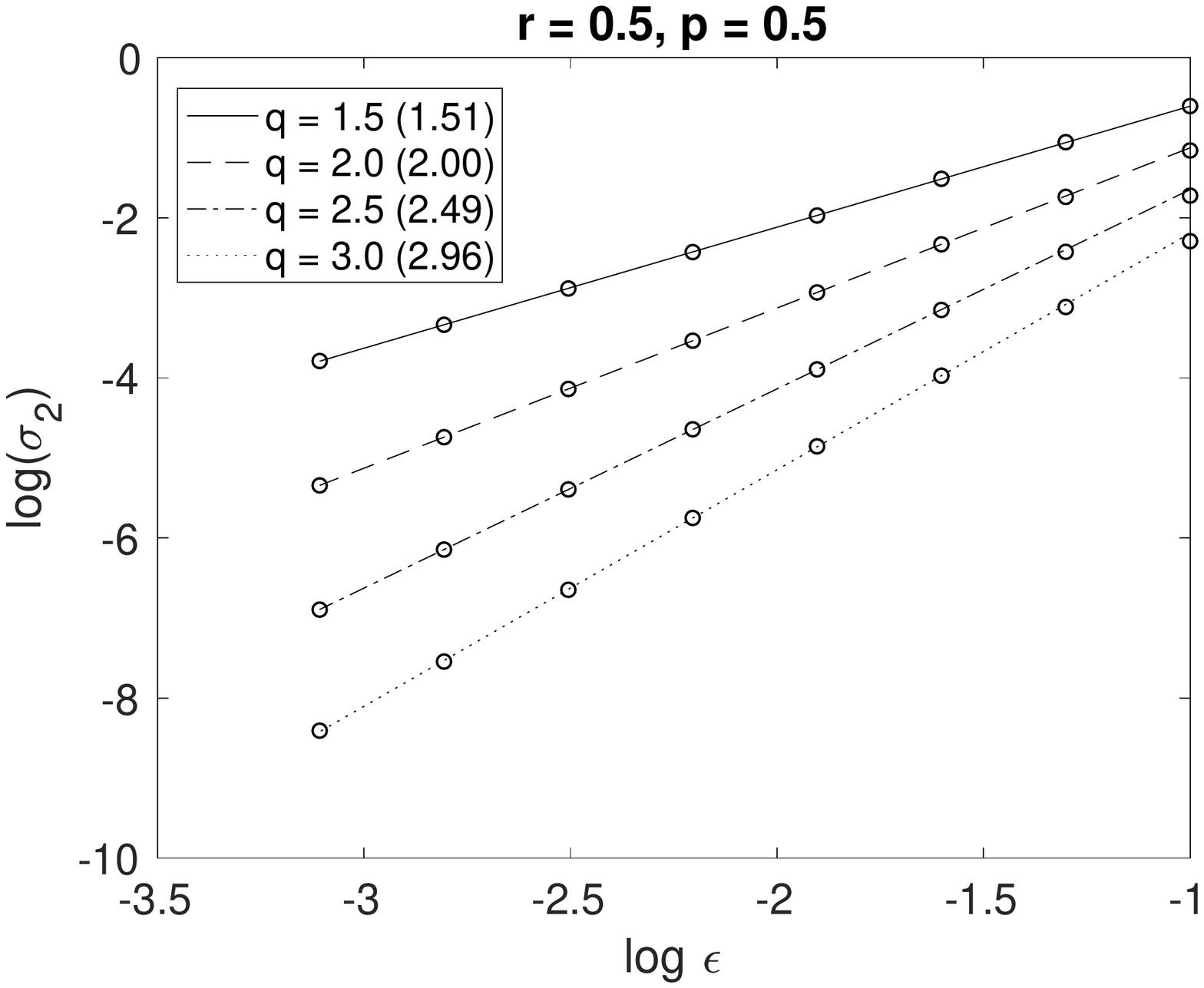}
     }
     \hfill
     \subfloat[]{%
       \includegraphics[width=0.49\textwidth, clip=true, trim =1cm 6cm 2cm 5cm]{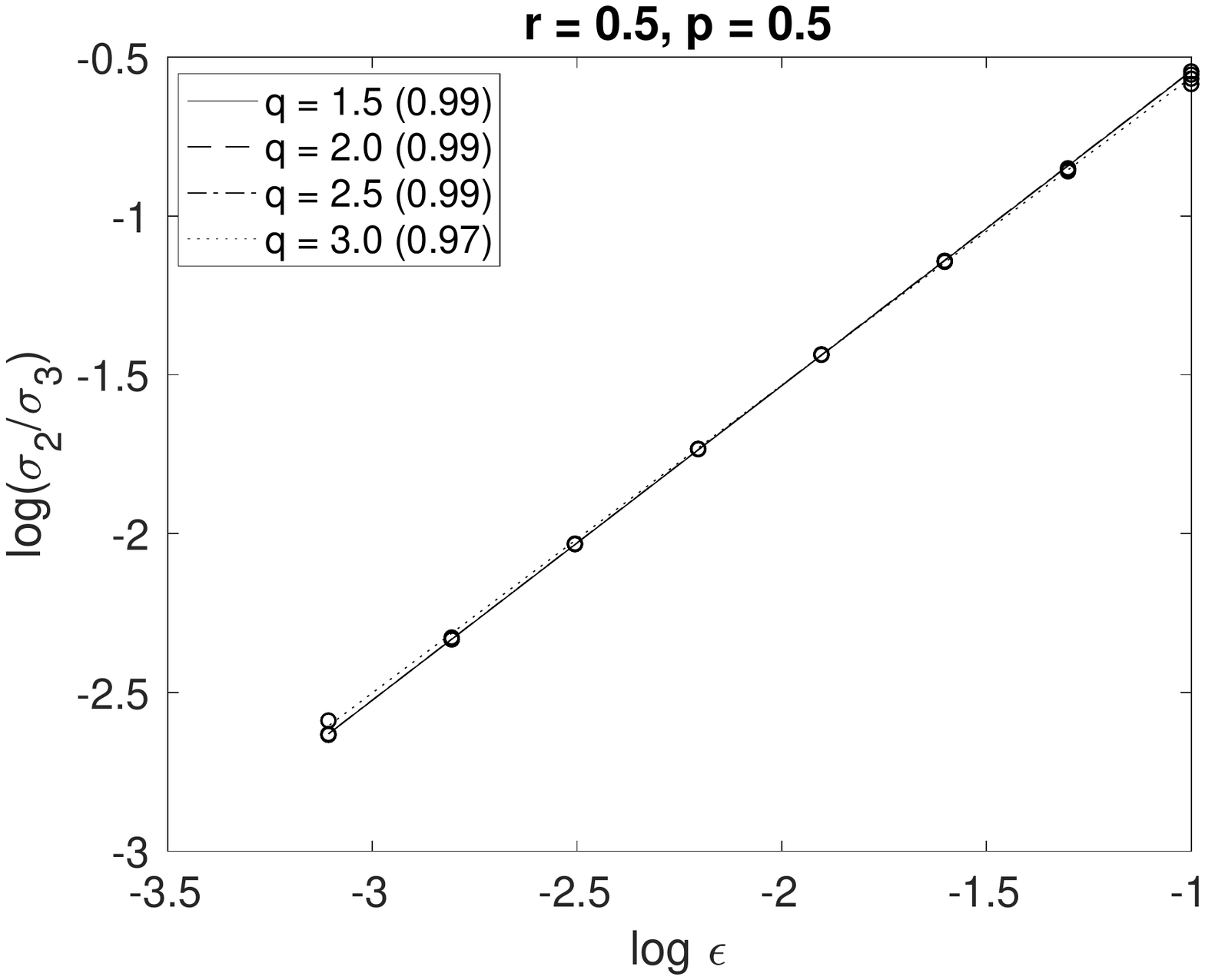}
     }\\
      \subfloat[]{%
     \includegraphics[width=0.49\textwidth, clip=true, trim =1cm 6cm 2cm 5cm]{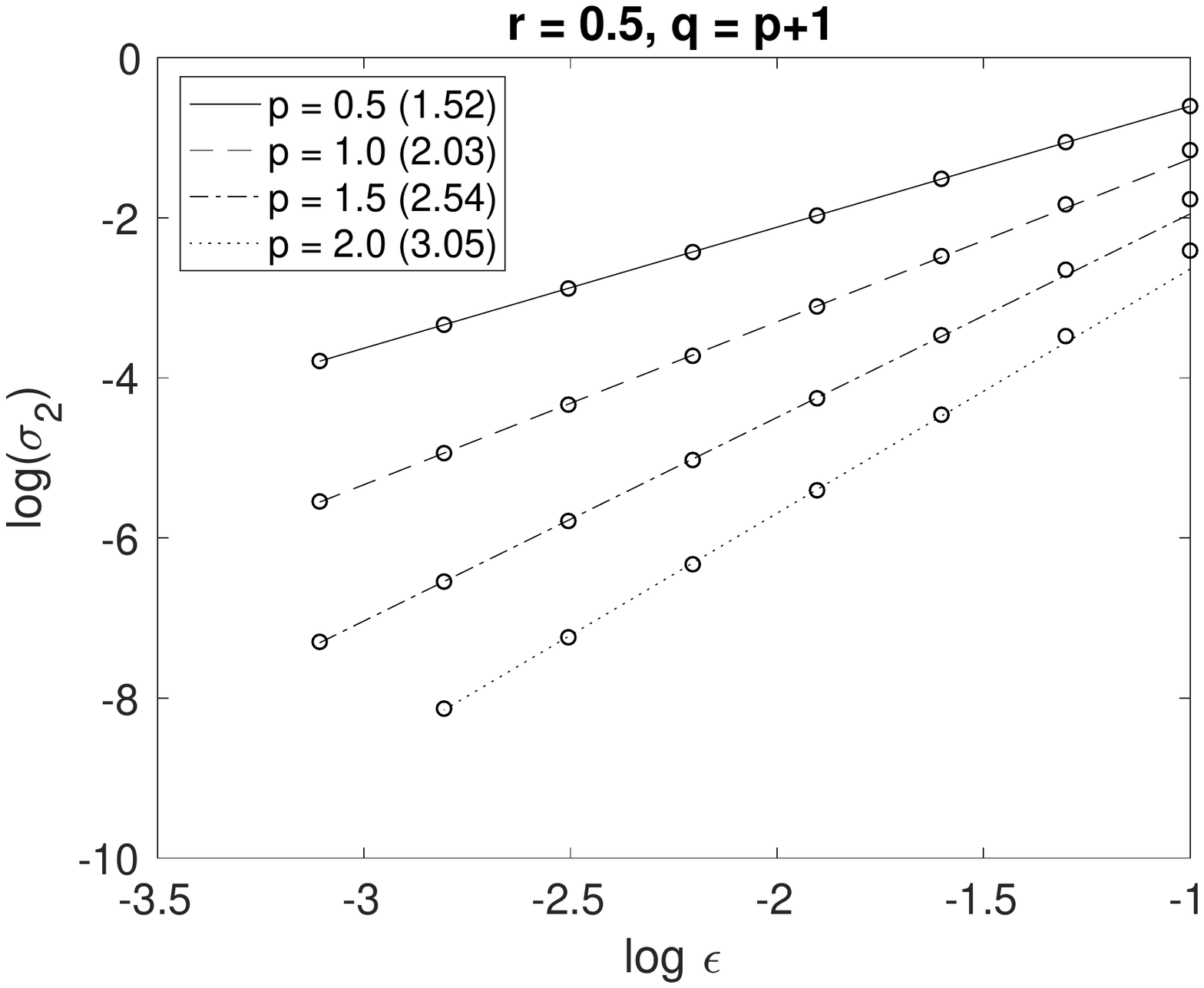}
     }
     \hfill
     \subfloat[]{%
       \includegraphics[width=0.49\textwidth, clip=true, trim =1cm 6cm 2cm 5cm]{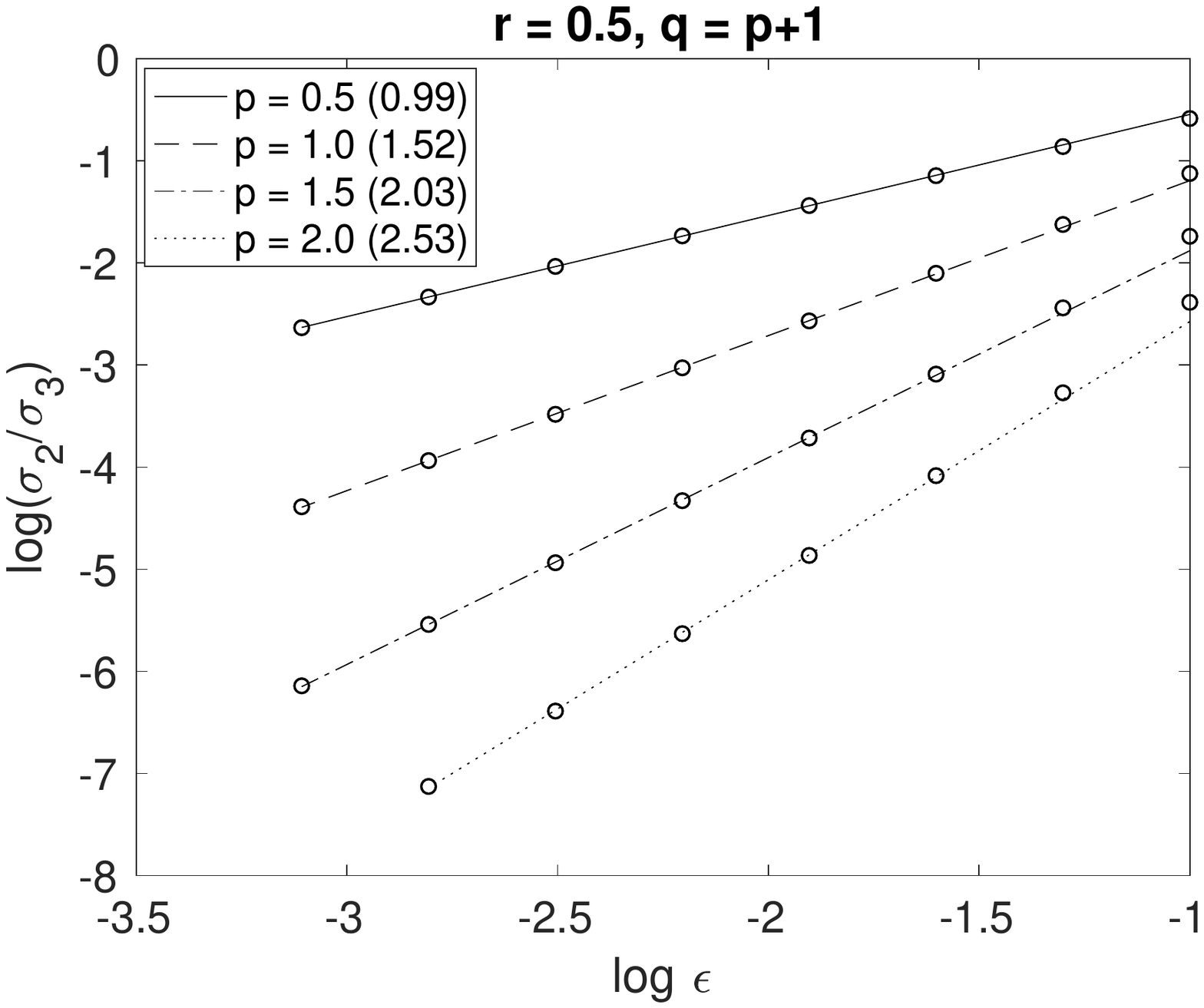}
     }
     \caption{Variation of the second and third eigenvalues of $\mcl L_\eps$ in the unbalanced case
       with $q > p + r$, and for various values of $p, q$ and $r$. In (a, b)
       we fix $p = r = 0.5$ and vary $q \in [1.5, 3]$. In (c, d) we fix $r= 0.5$,
       $q = p +1$ and vary $p \in [0.5, 2]$. 
(a, c) show $\log(\sigma_{2,\eps})$ vs $\log(\eps)$ while (b, d) show
       $\log(\sigma_{2,\eps}/\sigma_{3,\eps})$ vs $\log(\eps)$.
The values reported
       in the brackets in the legends are numerical approximations to the slope of the
       lines.}
     \label{fig:unbalanced01}
   \end{figure}


         \begin{figure}[!ht]
     \subfloat[]{%
       \includegraphics[width=0.49\textwidth, clip=true, trim =1cm 6cm 2cm 5cm]{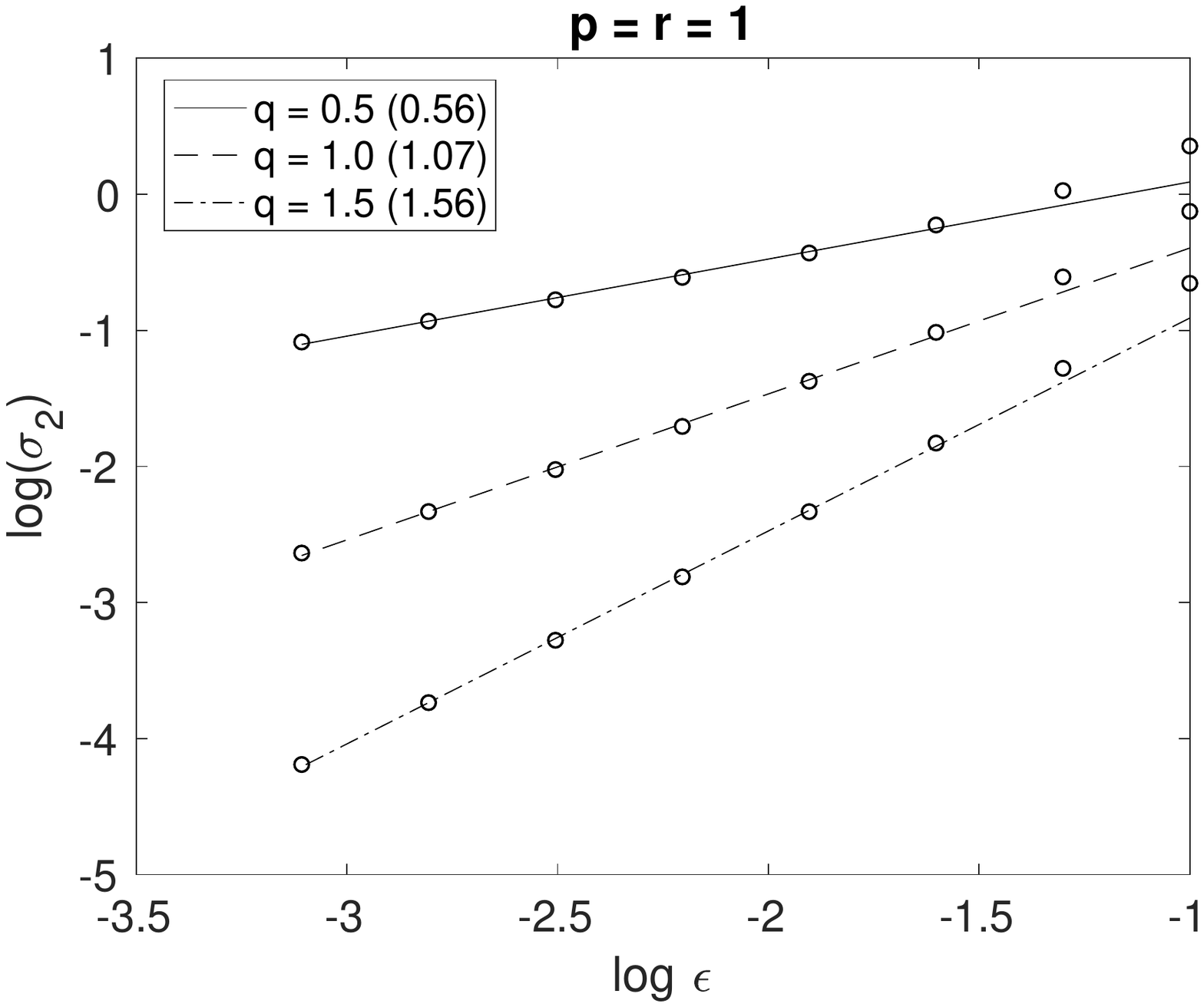}
     }
     \hfill
     \subfloat[]{%
       \includegraphics[width=0.49\textwidth, clip=true, trim =1cm 6cm 2cm 5cm]{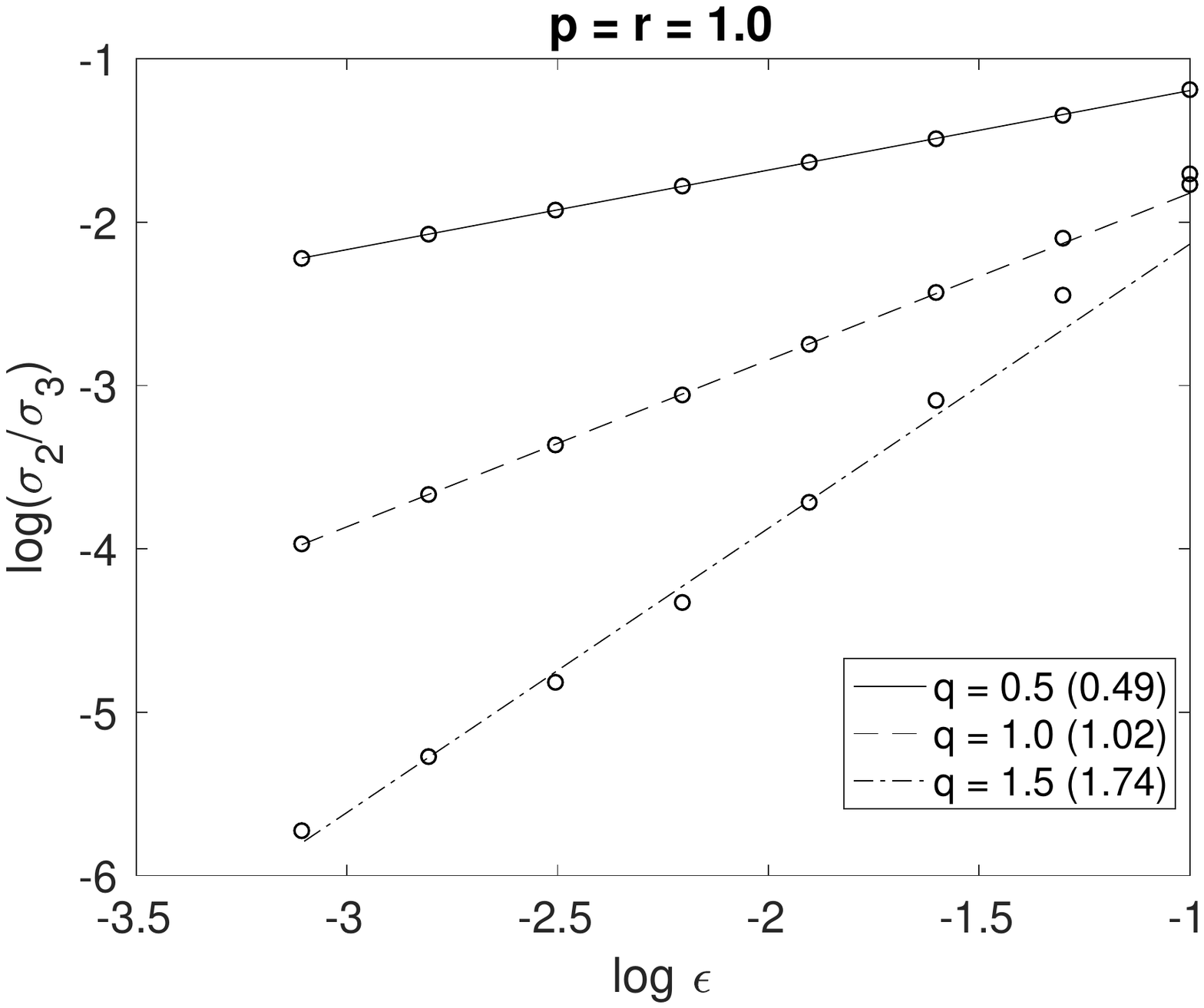}
     }\\
      \subfloat[]{%
     \includegraphics[width=0.49\textwidth, clip=true, trim =1cm 6cm 2cm 5cm]{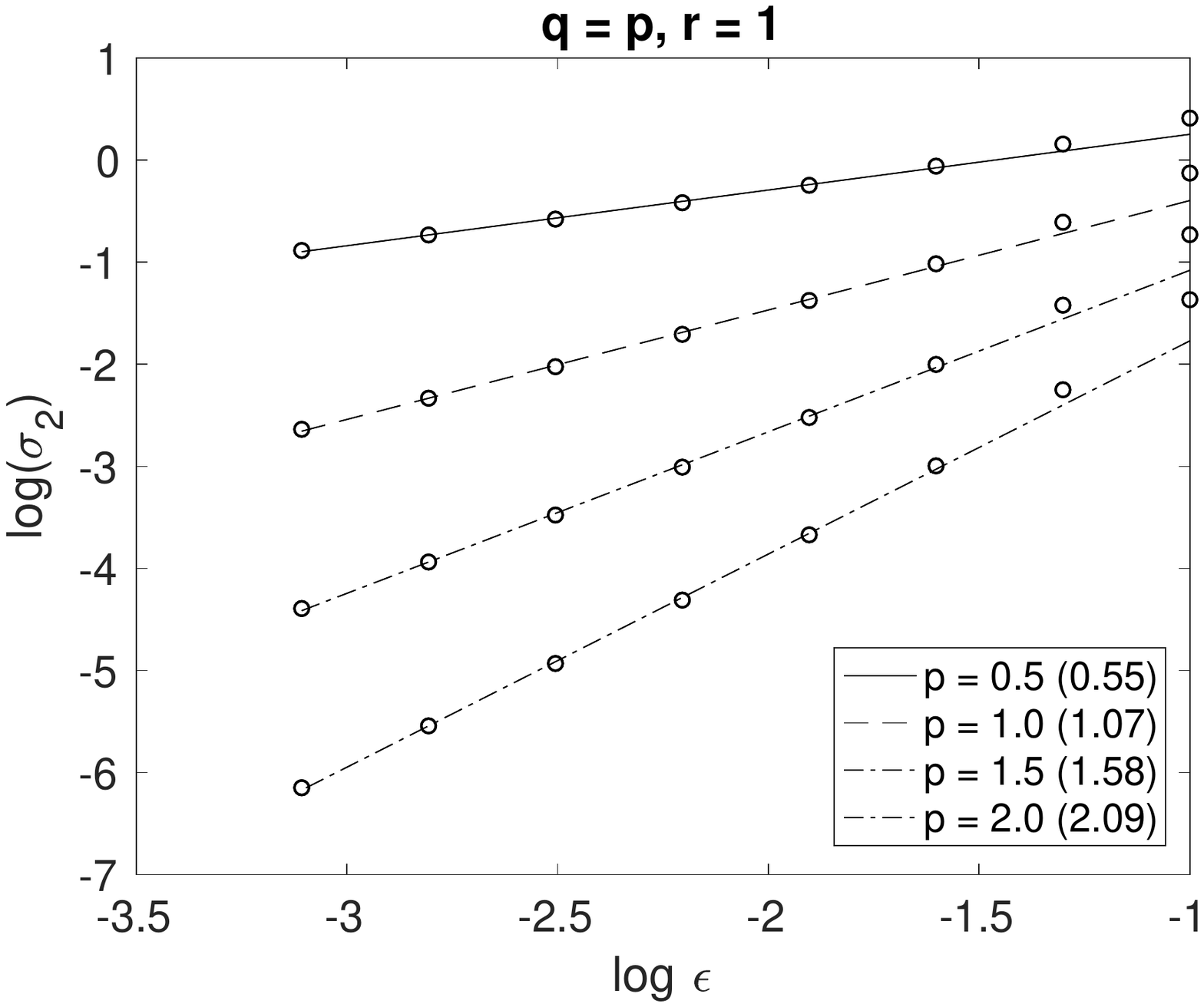}
     }
     \hfill
     \subfloat[]{%
       \includegraphics[width=0.49\textwidth, clip=true, trim =1cm 6cm 2cm 5cm]{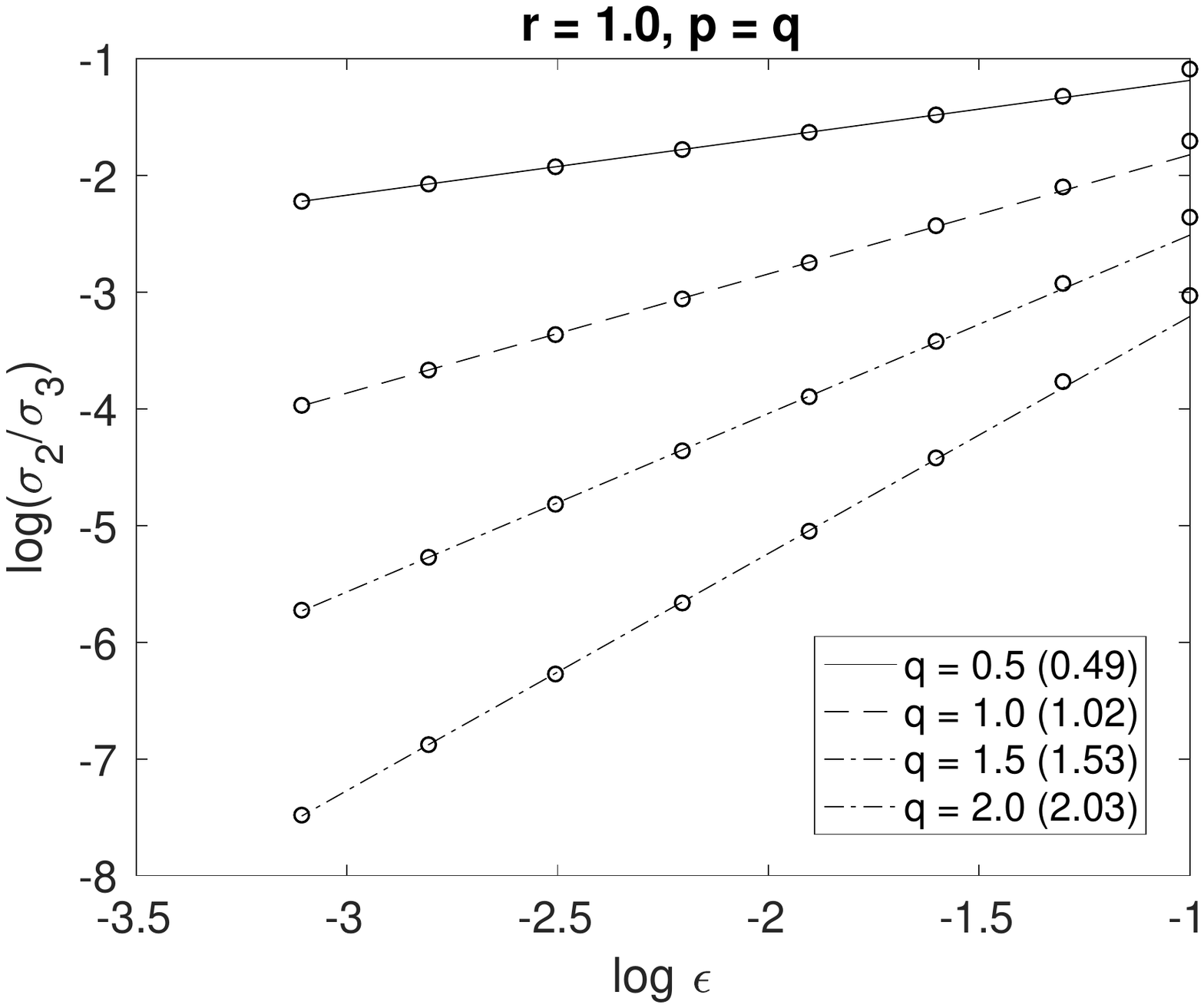}
     }
     \caption{Variation of the second and third eigenvalues of $\mcl L_\eps$ in the unbalanced case
       with $q < p + r$, and for various values of $p, q$ and $r$. In (a, b)
       we fix $p = r = 1$ and vary $q \in [0.5, 1.5]$. In (c, d) we fix $r= 1.0$,
       $q = p$ and vary $p \in [0.5, 2]$. 
(a, c) show $\log(\sigma_{2,\eps})$ vs $\log(\eps)$ while (b, d) show
       $\log(\sigma_{2,\eps}/\sigma_{3,\eps})$ vs $\log(\eps)$.
The values reported
       in the brackets in the legends are numerical approximations to the slope of the
       lines.}
     \label{fig:unbalanced01-qlepr}
   \end{figure}

   \begin{table}[htp]
     \centering
     \begin{tabular}{|c |c |c  | c | c | c| c| c|}
       \hline
       & & & \multicolumn{2}{c |}{$\frac{\log(\sigma_{2,\eps})}{\log{\eps}}$} &
       \multicolumn{3}{c |}{$\frac{\log(\sigma_{2,\eps}) - \log(\sigma_{3,\eps})}{\log{\eps}}$}\\
       \hline
       $p$ & $q$ & $r$  & Analytic & Numerical & Analytic & Numerical  & $p +r$ \\
       \hline
       $0.5$& 1.50 & $0.5$ & 1.50  & 1.51   & 0.50 & 0.99 & 1.00 \\
       \hline
       $0.5$& 2.0 & $0.5$ & 2.00  & 2.00  & 0.00  & 0.99 & 1.00\\
       \hline
       $0.5$& 2.5 & $0.5$ & 2.49  & 2.57   & - & 0.99 & 1.00 \\
       \hline
       $0.5$& 3.0 &$0.5$ & 2.96 & 3.06  & - & 0.97 & 1.00\\
       \hline
       $1.0$& 2 &$0.5$ & 2.03  & 2.11  & 1.00  & 1.52& 1.50\\
       \hline
       $1.5$& 2.5 &$0.5$ & 2.54  & 2.64   & 1.50 & 2.03 & 2.00 \\
       \hline
       $2.0$& 3.0 &$0.5$ & 3.05 & 3.20  &2.00 & 2.53 & 2.50 \\
       \hline
     \end{tabular}
     \caption{Comparison between  numerical approximation of
       the rate of decay of $ \log(\sigma_{2,\eps})$ and $\log(\sigma_{2,\eps}/\sigma_{3,\eps})$ as
       functions of $\log(\eps)$ and the analytic predictions in Theorem~\ref{thm:L-eps-low-eigenvals} and Corollary~\ref{cor:ratiogap}. The last
       column denotes the conjectured slope of $p +r$  for $\log(\sigma_{2,\eps}/\sigma_{3,\eps})$ for the unbalanced case $q > p +r$.}
      \label{tab:unbalanced}
    \end{table}

   \begin{table}[htp]
     \centering
     \begin{tabular}{|c |c |c  | c | c | c| c|}
       \hline
       & & & \multicolumn{2}{c |}{$\frac{\log(\sigma_{2,\eps})}{\log{\eps}}$} &
       \multicolumn{2}{c |}{$\frac{\log(\sigma_{2,\eps}) - \log(\sigma_{3,\eps})}{\log{\eps}}$}\\
       \hline
       $p$ & $q$ & $r$  & Analytic & Numerical & Analytic & Numerical \\
       \hline
       1& 0.5 & 1 &  0.5  & 0.56   & - & 0.49  \\
       \hline
       1& 1.0 & 1 &  1.0  & 1.07  & 0 & 1.02 \\
       \hline
       1& 1.5 & 1 &  1.5  & 1.56   & 1.0 & 1.74 \\
       \hline
       0.5& 0.5 & 1 & 0.5 & 0.5  & - & 0.49\\
       \hline
       1.5& 1.5 & 1 & 1.5  & 1.58  & 0.5  & 1.53\\
       \hline
       2.0& 2.0 & 1 & 2.0 & 2.09   & 1.0 & 2.03 \\
       \hline
     \end{tabular}
     \caption{Comparison between  numerical approximation of
       the rate of decay of $ \log(\sigma_{2,\eps})$ and $\log(\sigma_{2,\eps}/\sigma_{3,\eps})$ as
       functions of $\log(\eps)$ and the analytic predictions in Theorem~\ref{thm:L-eps-low-eigenvals} and Corollary~\ref{cor:ratiogap} for the unbalanced case $q < p +r$. Compare values in the last column with
     the prescribed values of $q$.}
      \label{tab:unbalanced-qlepr}
    \end{table}
    

  \begin{figure}[!ht]
     \subfloat[]{%
       \includegraphics[width=0.49\textwidth, clip=true, trim =1cm 6cm 2cm 5cm]{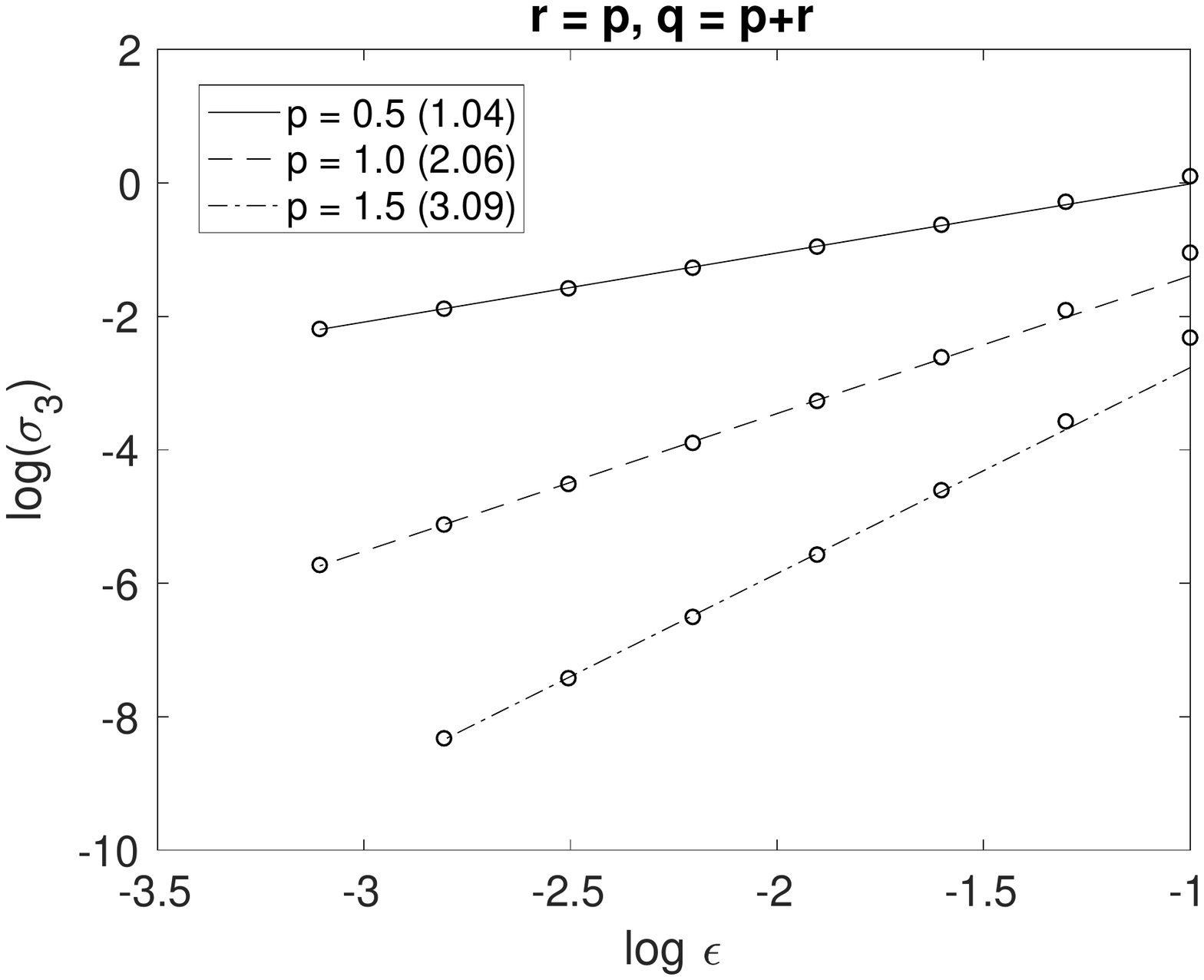}
     }
     \hfill
     \subfloat[]{%
       \includegraphics[width=0.49\textwidth, clip=true, trim =1cm 6cm 2cm 5cm]{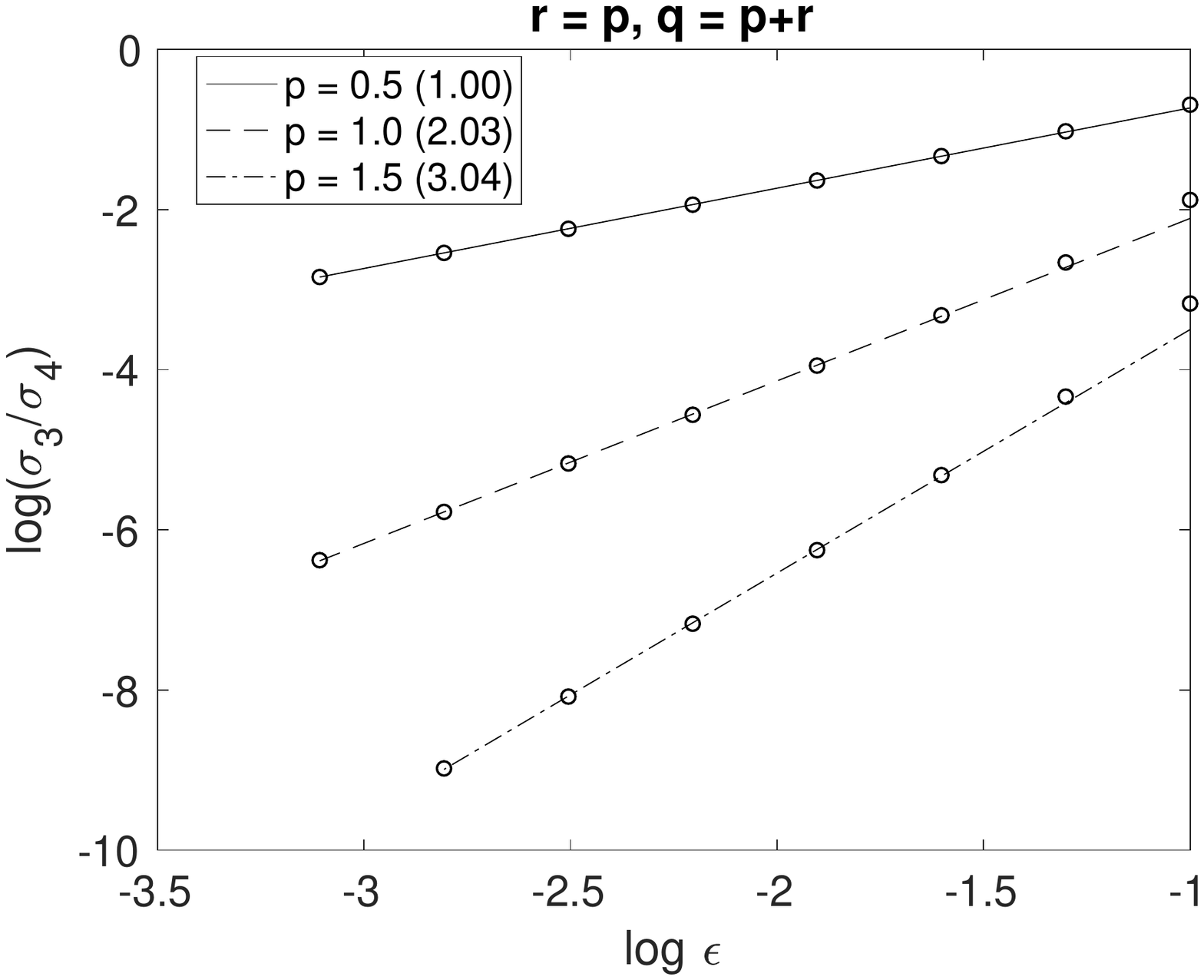}
     }
     \caption{Variation of the third and fourth eigenvalues of $\mcl L_\eps$ in the
       three cluster setting 
       with $q = p + r$, $r = p$ and for $p \in [0.5, 1.5]$. (a)
       shows $\log(\sigma_{3,\eps})$ vs $\log(\eps)$ while (b) shows
       $\log(\sigma_{3,\eps}/\sigma_{4,\eps})$ vs $\log(\eps)$. The values reported
       in the brackets in the legends are numerical approximations to the slope of the
       lines for different values of $p$. 
       }
     \label{fig:balanced01-three-clusters}
   \end{figure}

      \begin{figure}[!ht]
     \subfloat[]{%
       \includegraphics[width=0.49\textwidth, clip=true, trim =1cm 6cm 2cm 5cm]{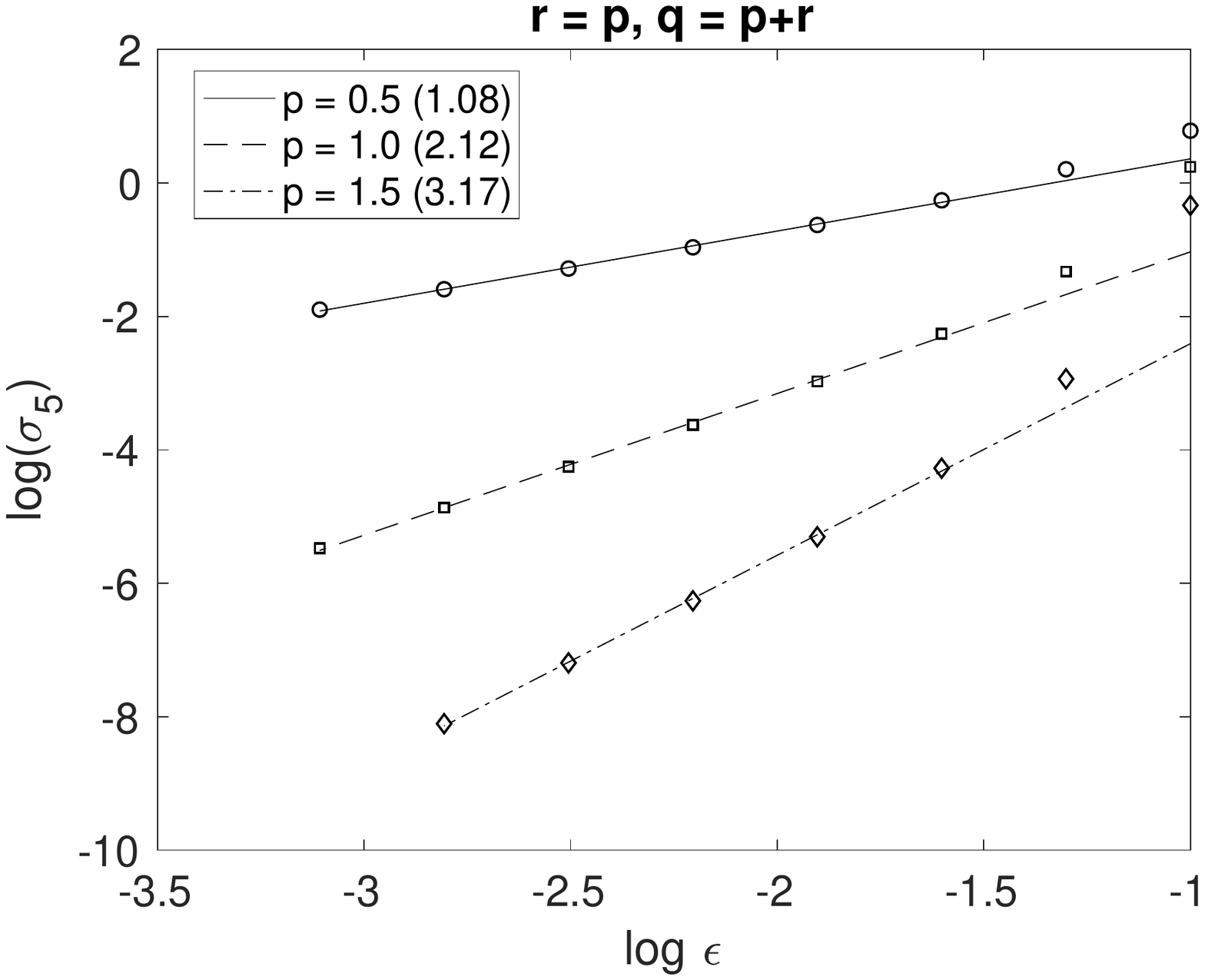}
     }
     \hfill
     \subfloat[]{%
       \includegraphics[width=0.49\textwidth, clip=true, trim =1cm 6cm 2cm 5cm]{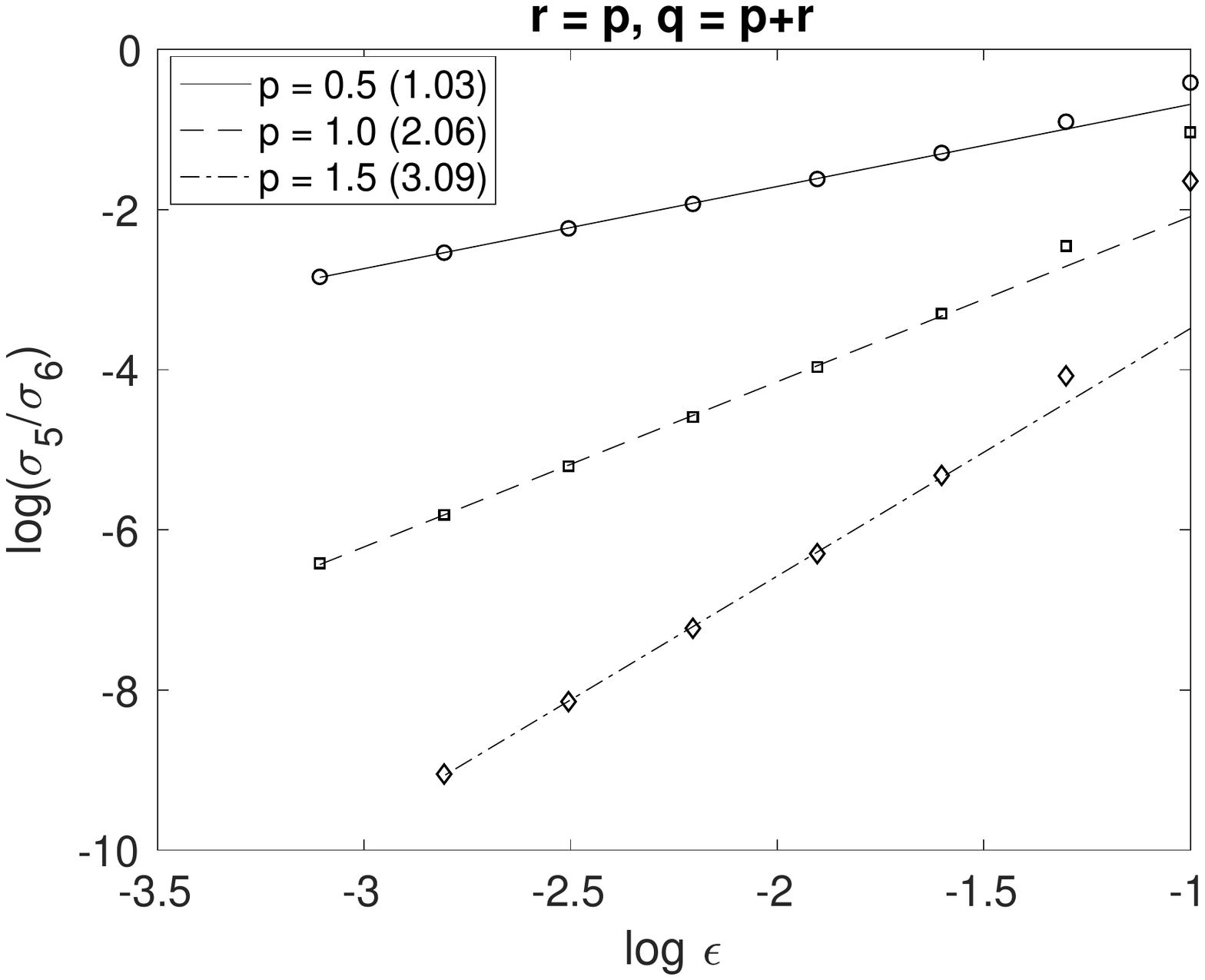}
     }
     \caption{Variation of the fifth and sixth eigenvalues of $\mcl L_\eps$ in the five cluster
       case
       with $q = p + r$, $r = p$ and for  $p \in [0.5, 1.5]$. (a)
       shows $\log(\sigma_{5,\eps})$ vs $\log(\eps)$ while (b) shows
       $\log(\sigma_{5,\eps}/\sigma_{6,\eps})$ vs $\log(\eps)$. The values reported
       in the brackets in the legends are numerical approximations to the slope of the
       lines for different values of $p$.
}
     \label{fig:balanced-five-clusters}
   \end{figure}

      \begin{figure}[!ht]
     \subfloat[]{%
       \includegraphics[width=0.49\textwidth, clip=true, trim =1cm 6cm 2cm 5cm]{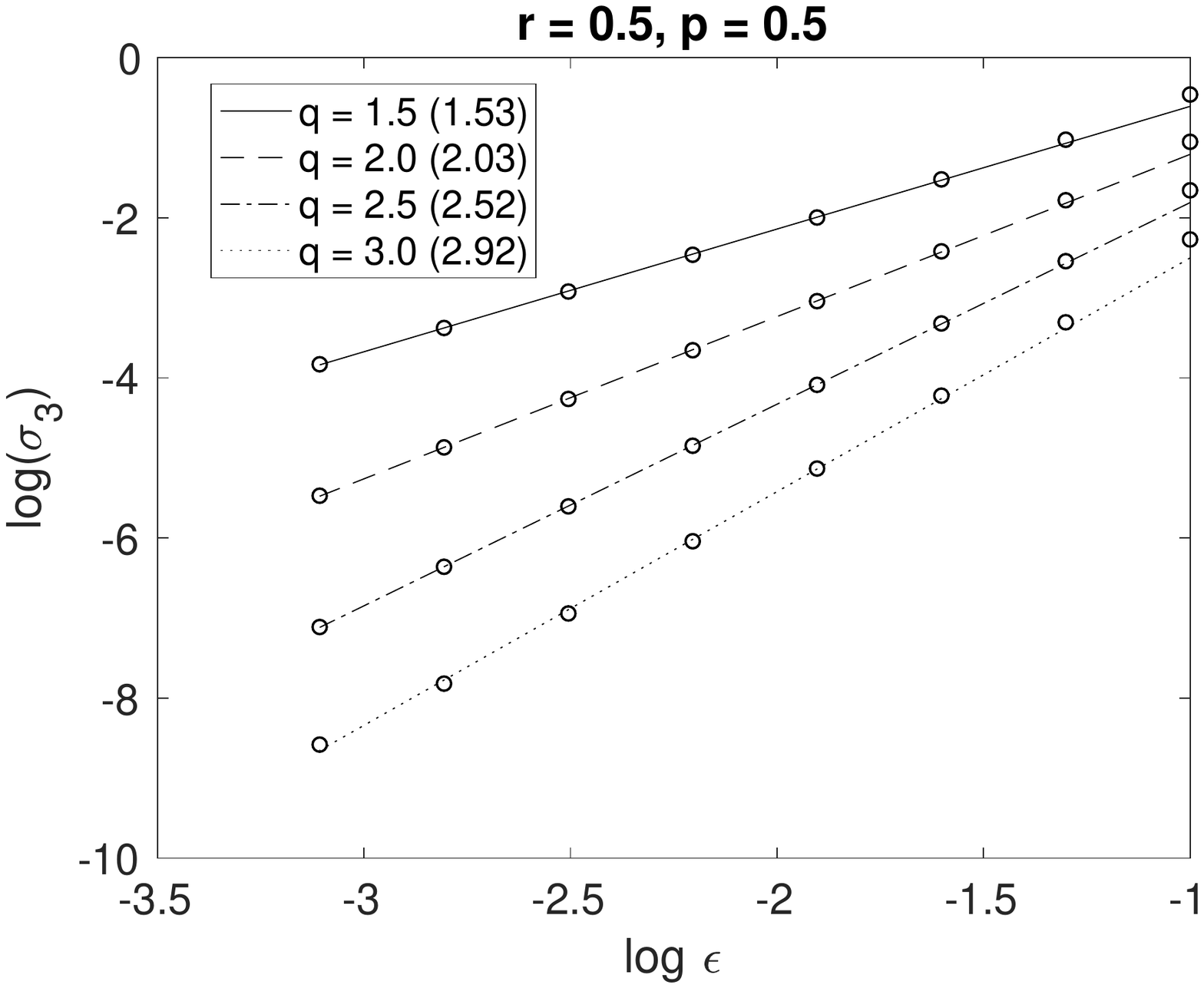}
     }
     \hfill
     \subfloat[]{%
       \includegraphics[width=0.49\textwidth, clip=true, trim =1cm 6cm 2cm 5cm]{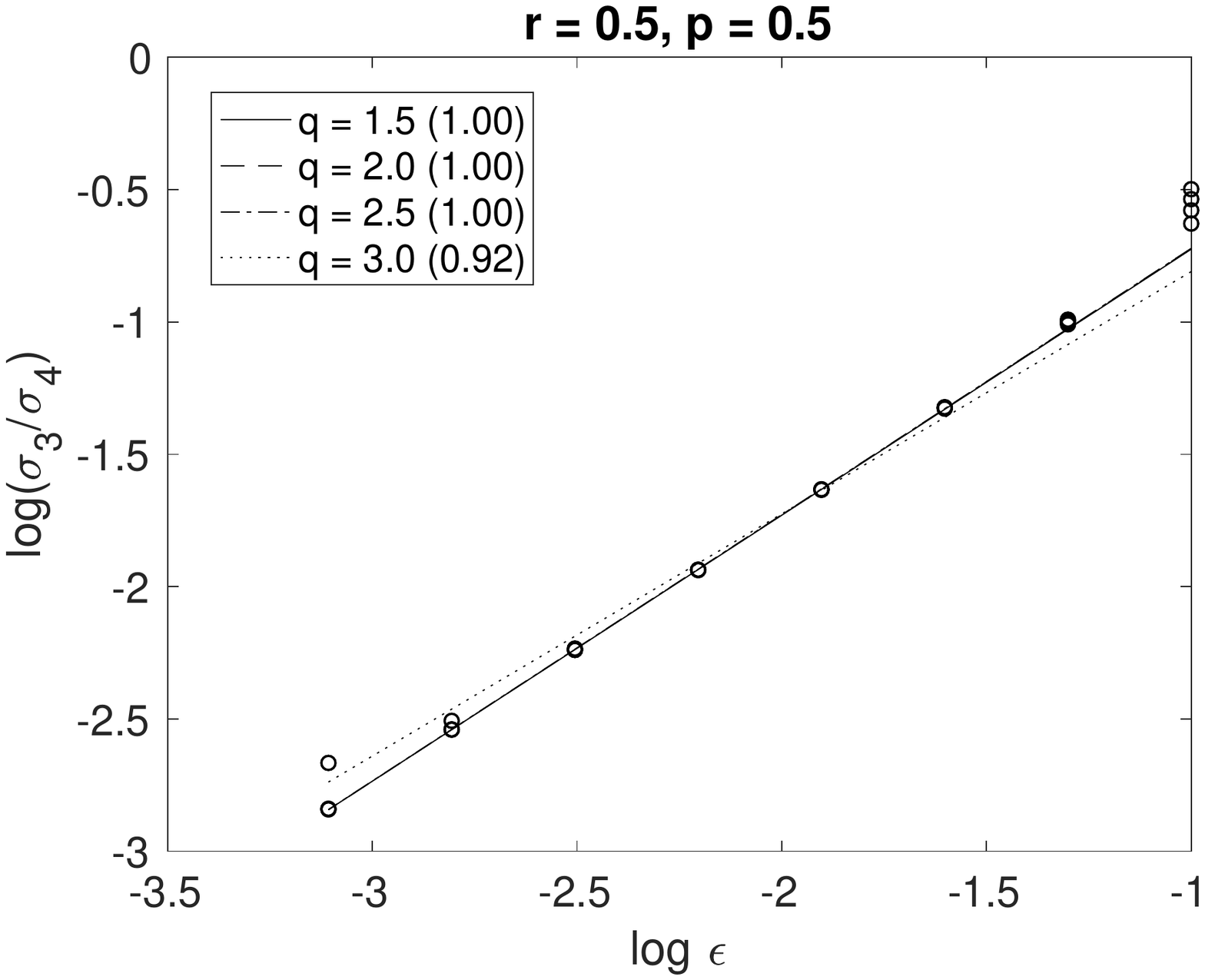}
     }
     \caption{Variation of the third and fourth eigenvalues of $\mcl L_\eps$ in the
       three cluster setting
       with $q > p + r$, $r = p = 0.5$ and for  $q \in [1.5, 3]$. (a)
       shows $\log(\sigma_{3,\eps})$ vs $\log(\eps)$ while (b) shows
       $\log(\sigma_{3,\eps}/\sigma_{4,\eps})$ vs $\log(\eps)$. The values reported
       in the brackets in the legends are numerical approximations to the slope of the
       lines for different values of $q$.
}
     \label{fig:unbalanced01-three-clusters}
   \end{figure}

      \begin{figure}[!ht]
     \subfloat[]{%
       \includegraphics[width=0.49\textwidth, clip=true, trim =1cm 6cm 2cm 5cm]{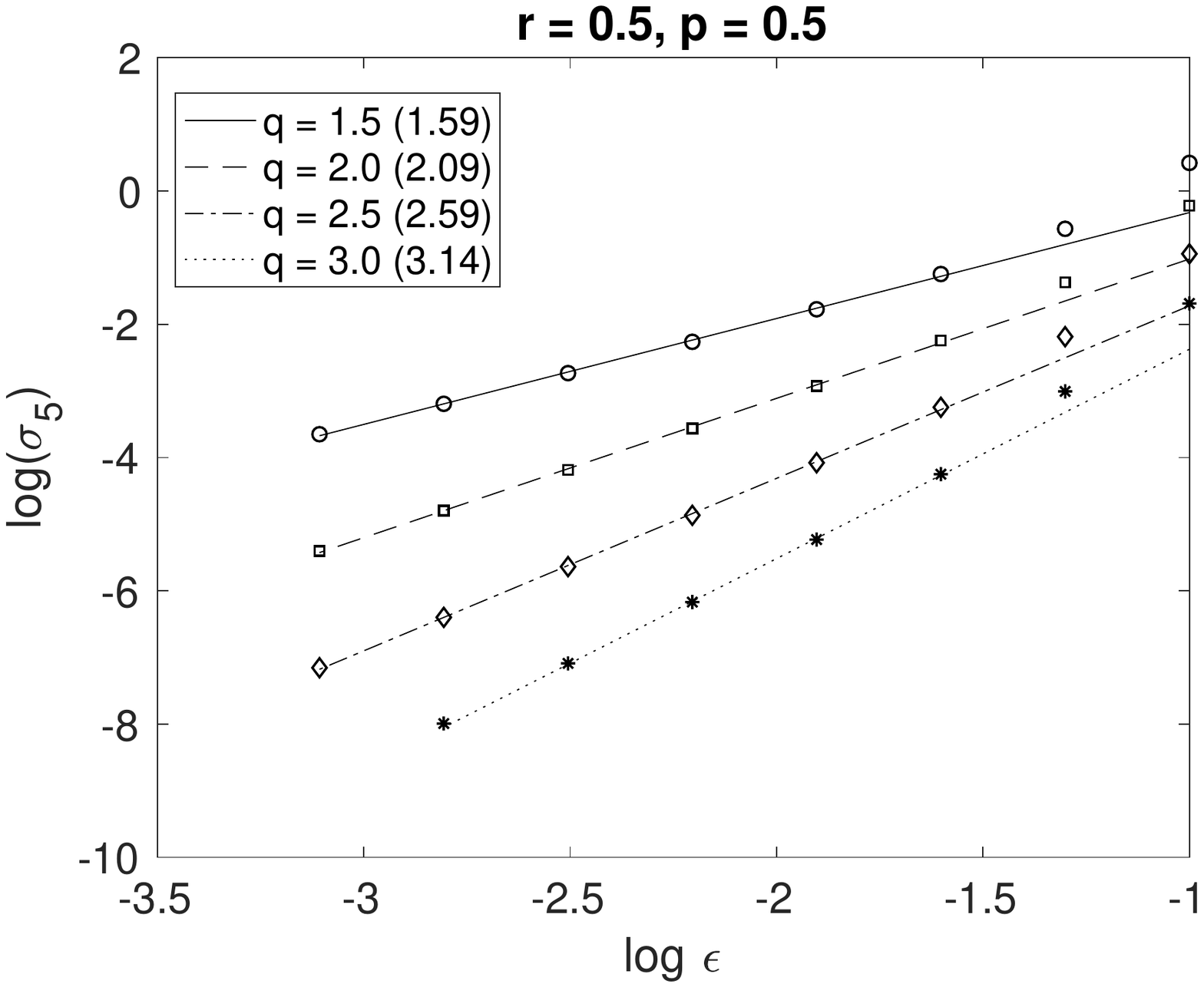}
     }
     \hfill
     \subfloat[]{%
       \includegraphics[width=0.49\textwidth, clip=true, trim =1cm 6cm 2cm 5cm]{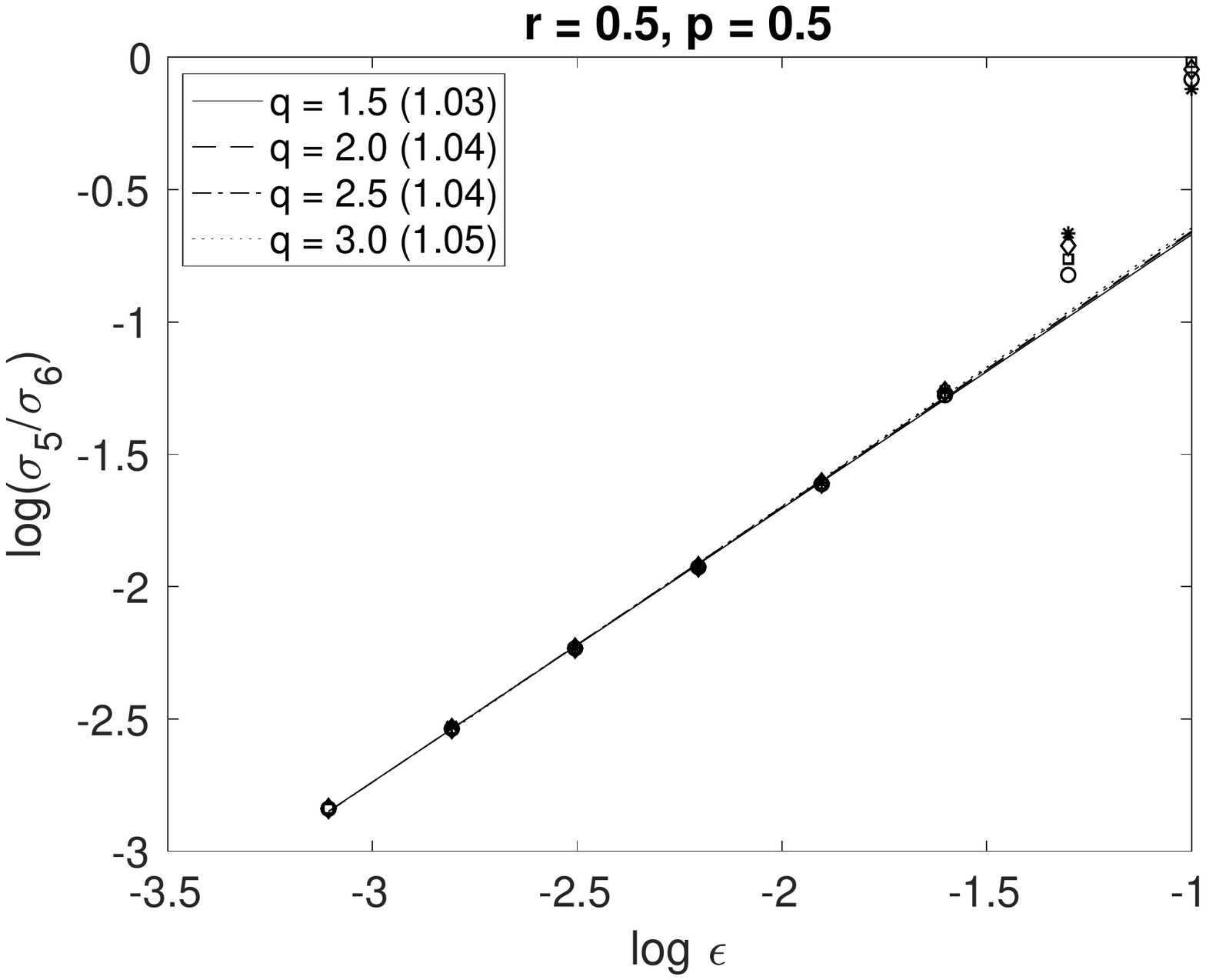}
     }
     \caption{Variation of the fifth and sixth eigenvalues of $\mcl L_\eps$ in the five cluster case
       with $q > p + r$, $r = p= 0.5$ and for  $q \in [1.5, 3]$. (a)
       shows $\log(\sigma_{5,\eps})$ vs $\log(\eps)$ while (b) shows
       $\log(\sigma_{5,\eps}/\sigma_{6,\eps})$ vs $\log(\eps)$. The values reported
       in the brackets in the legends are numerical approximations to the slope of the
       lines for different values of $q$. 
}
     \label{fig:unbalanced-five-clusters}
   \end{figure}

      \begin{figure}[!ht]
     \subfloat[]{%
       \includegraphics[width=0.49\textwidth, clip=true, trim =1cm 6cm 2cm 5cm]{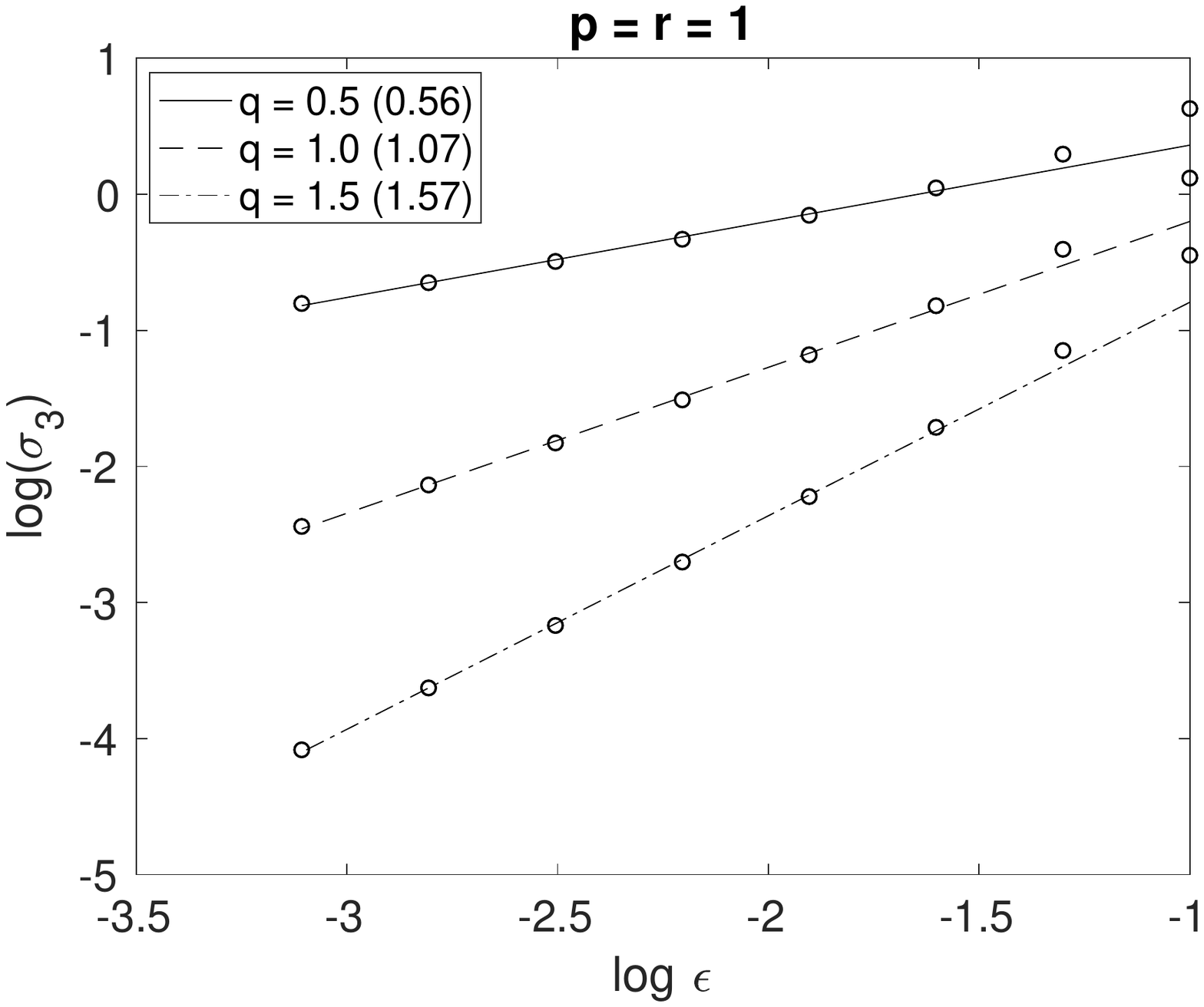}
     }
     \hfill
     \subfloat[]{%
       \includegraphics[width=0.49\textwidth, clip=true, trim =1cm 6cm 2cm 5cm]{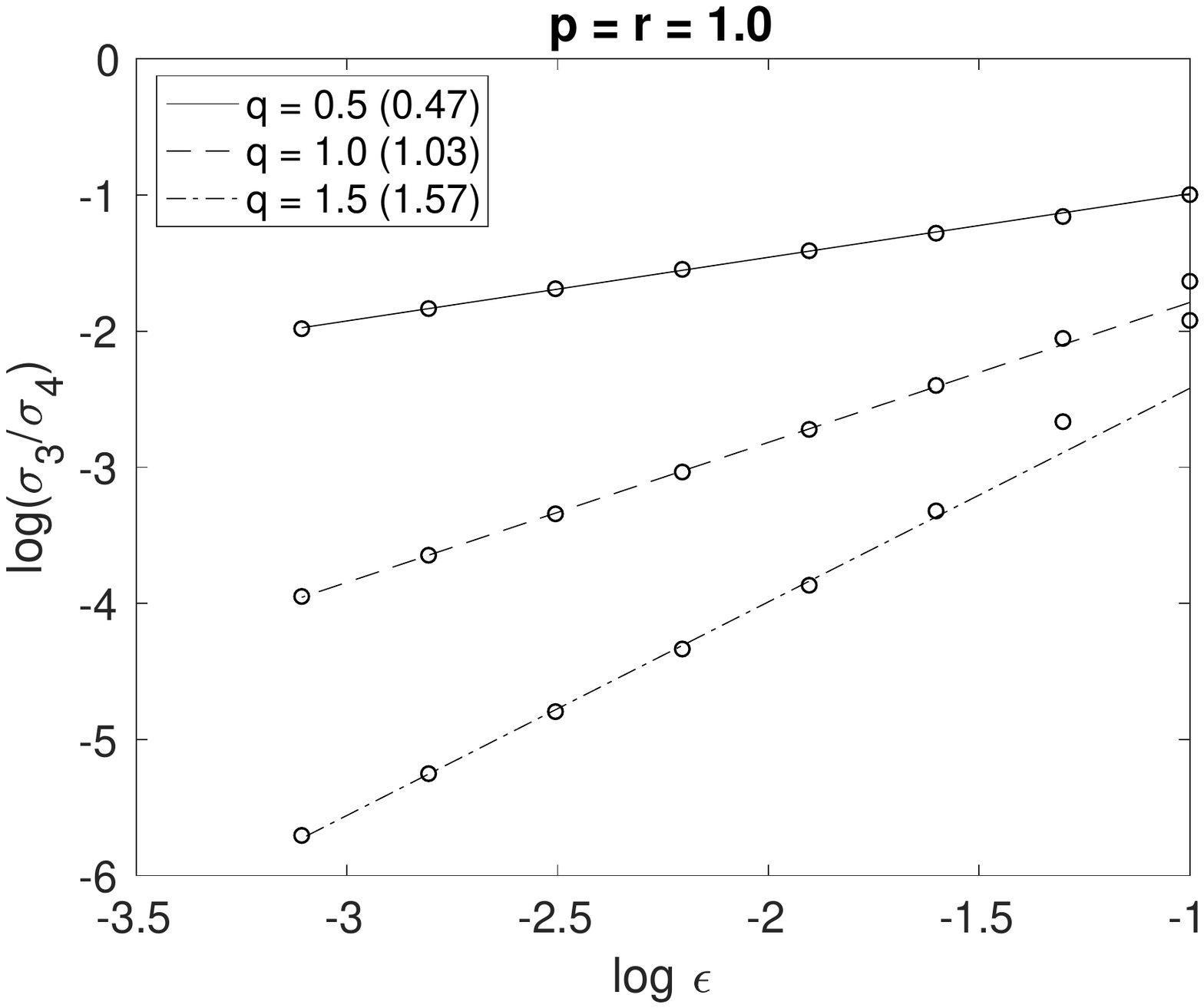}
     }
     \caption{Variation of the third and fourth eigenvalues of $\mcl L_\eps$ in the
       three cluster setting
       with $q < p + r$, $r = p = 1$ and for  $q \in [0.5, 1.5]$. (a)
       shows $\log(\sigma_{3,\eps})$ vs $\log(\eps)$ while (b) shows
       $\log(\sigma_{3,\eps}/\sigma_{4,\eps})$ vs $\log(\eps)$. The values reported
       in the brackets in the legends are numerical approximations to the slope of the
       lines for different values of $q$.
}
     \label{fig:unbalanced01-three-clusters-qlepr}
   \end{figure}

      \begin{figure}[!ht]
     \subfloat[]{%
       \includegraphics[width=0.49\textwidth, clip=true, trim =1cm 6cm 2cm 5cm]{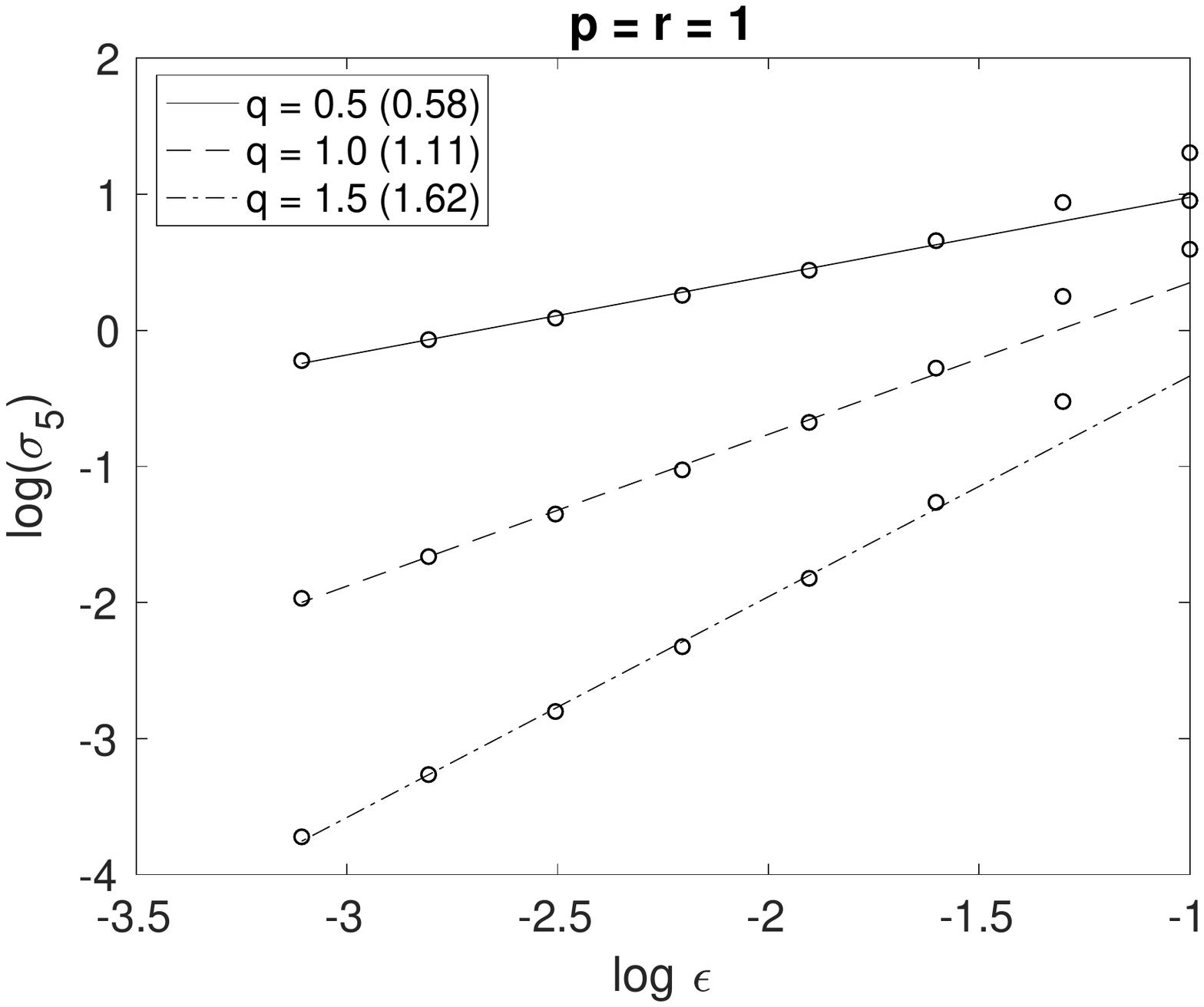}
     }
     \hfill
     \subfloat[]{%
       \includegraphics[width=0.49\textwidth, clip=true, trim =1cm 6cm 2cm 5cm]{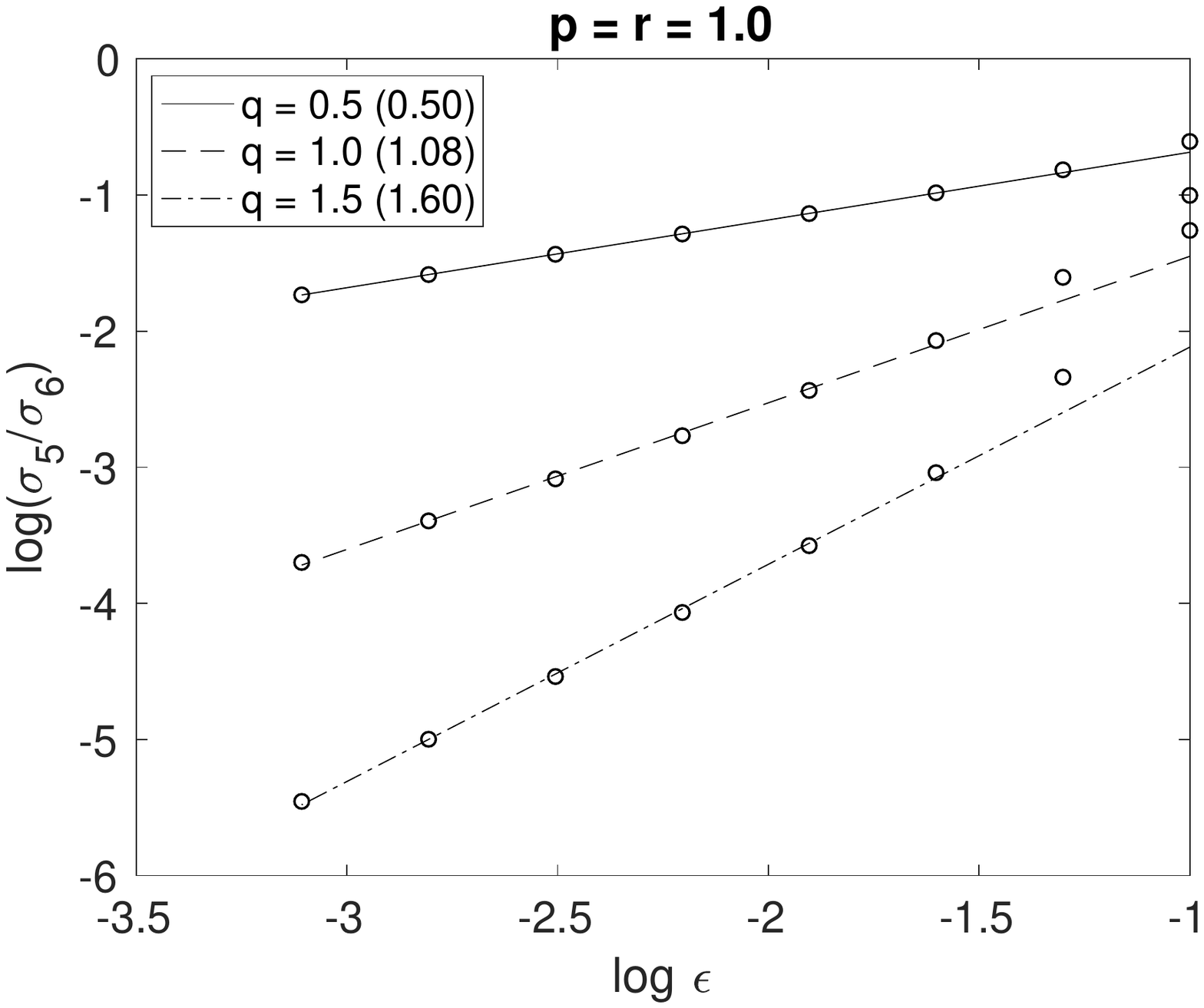}
     }
     \caption{Variation of the fifth and sixth eigenvalues of $\mcl L_\eps$ in the five cluster case
       with $q < p + r$, $r = p= 1$ and for  $q \in [0.5, 1.5]$. (a)
       shows $\log(\sigma_{5,\eps})$ vs $\log(\eps)$ while (b) shows
       $\log(\sigma_{5,\eps}/\sigma_{6,\eps})$ vs $\log(\eps)$. The values reported
       in the brackets in the legends are numerical approximations to the slope of the
       lines for different values of $q$. 
}
     \label{fig:unbalanced-five-clusters-qlepr}
   \end{figure}

       \begin{table}[htp]
     \centering
     \begin{tabular}{| c |c |c | c | c | c| c|}
       \hline
       & $p$ & $q$ & $r$ &  \multicolumn{1}{c |}{$\frac{\log(\sigma_{3,\eps})}{\log{\eps}}$} &
       \multicolumn{1}{c |}{$\frac{\log(\sigma_{4,\eps}) - \log(\sigma_{3,\eps})}{\log{\eps}}$}\\
       \hline
       \multirow{6}{*}{\rotatebox{90}{ $q= p +r$}} & $0.5$& 1.0 & 0.5    & 1.04    & 1.00 \\
       \cline{2-6}
        &$1.0$& 2.0&  $1.0$     & 2.06    & 2.03 \\
       \cline{2-6}
       &$1.5$& 3.0 & $1.5$    & 3.097   & 3.04  \\
       \cline{2-6}
       &$1.0$& 1.5 &$0.5$  & 1.55   & 1.52 \\
       \cline{2-6}
       &$1.5$& 2.0 &$0.5$    & 2.06    & 2.03\\
       \cline{2-6}
       &$2.0$& 2.5 &$0.5$    & 2.57    & 2.53  \\
       \hhline{|=|=|=|=|=|=|}
       \multirow{8}{*}{\rotatebox{90}{ $q> p +r$}}         &$0.5$& 1.5&  $0.5$     & 1.53    & 1.00 \\
       \cline{2-6}
& $0.5$& 1.0 & 0.5    & 1.04    & 1.00 \\
       \cline{2-6}
       &$0.5$& 2.0 & $0.5$    & 2.03   & 1.00  \\
       \cline{2-6}
       &$0.5$& 2.5 &$0.5$  & 2.52   & 1.00 \\
       \cline{2-6}
       &$0.5$& 3.0 &$0.5$    & 2.92    & 0.92\\
       \cline{2-6}
       &$1.0$& 2.0 &$0.5$    & 2.05    & 1.52  \\
       \cline{2-6}
       &$1.5$& 2.5 &$0.5$    & 2.55    & 2.03  \\
       \cline{2-6}
       &$2.0$& 3.0 &$0.5$    & 3.07    & 2.53  \\
              \hhline{|=|=|=|=|=|=|}
       \multirow{8}{*}{\rotatebox{90}{ $q < p +r$}}
        &1.0 & 0.5 &  1.0     & 0.56    & 0.47 \\
       \cline{2-6}
       & 1.0 & 1.0 & 1.0    & 1.07   & 1.03  \\
       \cline{2-6}
       & 1.0 & 1.5 &1.0   &  1.57   & 1.57 \\
       \cline{2-6}
       & 0.5 & 0.5 & 1.0   & 0.54    & 0.49\\
       \cline{2-6}
       & 1.5 & 1.5 & 1.0    &  1.58    & 1.54  \\
       \cline{2-6}
       & 2.0 & 2.0 &1.0     & 2.09    & 2.04  \\
       \hline
     \end{tabular}
     \caption{Numerical approximation of
       the rate of decay of $ \log(\sigma_{3,\eps})$ and $\log(\sigma_{3,\eps}/\sigma_{4,\eps})$ as
       functions of $\log(\eps)$ for different choices of  $p,q,r$ 
        in the three cluster setting.}
      \label{tab:slopes-three-clusters}
   \end{table}


          \begin{table}[htp]
     \centering
     \begin{tabular}{| c |c |c | c | c | c| c|}
       \hline
       & $p$ & $q$ & $r$ &  \multicolumn{1}{c |}{$\frac{\log(\sigma_{4,\eps})}{\log{\eps}}$} &
       \multicolumn{1}{c |}{$\frac{\log(\sigma_{5,\eps}) - \log(\sigma_{4,\eps})}{\log{\eps}}$}\\
       \hline
       \multirow{6}{*}{\rotatebox{90}{ $q= p +r$}} & $0.5$& 1.0 & 0.5    & 1.04    & 1.03 \\
       \cline{2-6}
        &$1.0$& 2.0&  $1.0$     & 2.12    & 2.06 \\
       \cline{2-6}
       &$1.5$& 3.0 & $1.5$    & 3.17   & 3.09  \\
       \cline{2-6}
       &$1.0$& 1.5 &$0.5$  & 1.61   & 1.55 \\
       \cline{2-6}
       &$1.5$& 2.0 &$0.5$    & 2.12    & 2.06\\
       \cline{2-6}
       &$2.0$& 2.5 &$0.5$    & 2.63    & 2.57  \\
       \hhline{|=|=|=|=|=|=|}
       \multirow{8}{*}{\rotatebox{90}{ $q> p +r$}} 
        &$0.5$& 1.5&  $0.5$     & 1.59    & 1.03 \\
       \cline{2-6}
       &$0.5$& 2.0 & $0.5$    & 2.09   & 1.04  \\
       \cline{2-6}
       &$0.5$& 2.5 &$0.5$  & 2.59   & 1.04 \\
       \cline{2-6}
       &$0.5$& 3.0 &$0.5$    & 3.14    & 1.05\\
       \cline{2-6}
       &$1.0$& 2.0 &$0.5$    & 2.11    & 1.56  \\
       \cline{2-6}
       &$1.5$& 2.5 &$0.5$    & 2.62   & 2.07  \\
       \cline{2-6}
       &$2.0$& 3.0 &$0.5$    & 3.16    & 2.59  \\
        \hhline{|=|=|=|=|=|=|}
       \multirow{8}{*}{\rotatebox{90}{ $q < p +r$}}
        &1.0 & 0.5 &  1.0     & 0.58    & 0.50 \\
       \cline{2-6}
       & 1.0 & 1.0 & 1.0    & 1.11   & 1.08  \\
       \cline{2-6}
       & 1.0 & 1.5 &1.0   &  1.62   & 1.60 \\
       \cline{2-6}
       & 0.5 & 0.5 & 1.0    & 0.57    & 0.52\\
       \cline{2-6}
       & 1.5 & 1.5 & 1.0    &  1.63    & 1.59  \\
       \cline{2-6}
       & 2.0 & 2.0 &1.0     & 2.14    & 2.10  \\
       \hline
     \end{tabular}
     \caption{Numerical approximation of
       the rate of decay of $ \log(\sigma_{5,\eps})$ and $\log(\sigma_{5,\eps}/\sigma_{6,\eps})$ as
       functions of $\log(\eps)$ for different choices of  $p,q,r$ 
        in the five cluster setting.}
      \label{tab:slopes-five-clusters}
   \end{table}

\begin{figure}[htp]
  \centering
  \includegraphics[width=.32 \textwidth]{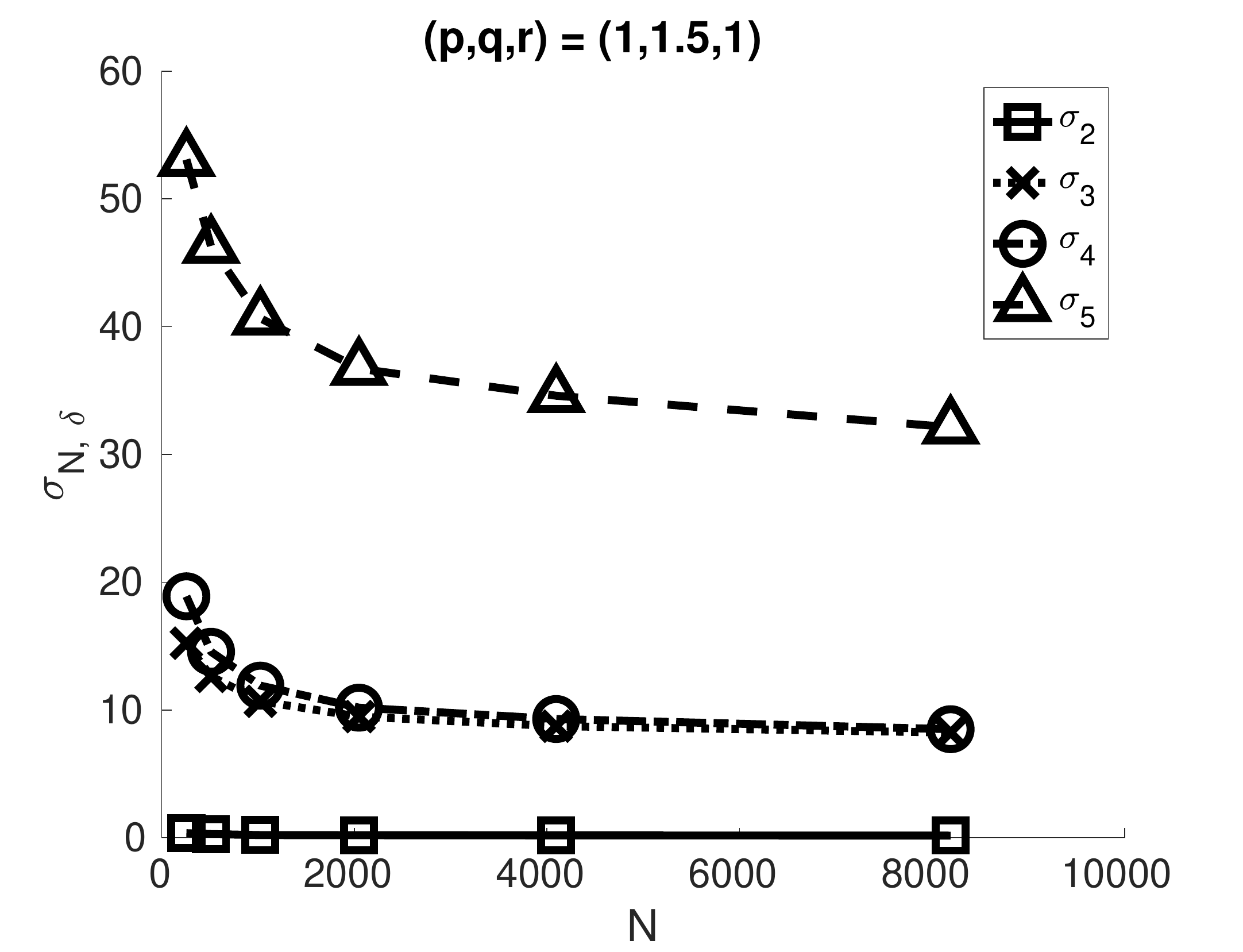}
  \includegraphics[width=.32 \textwidth]{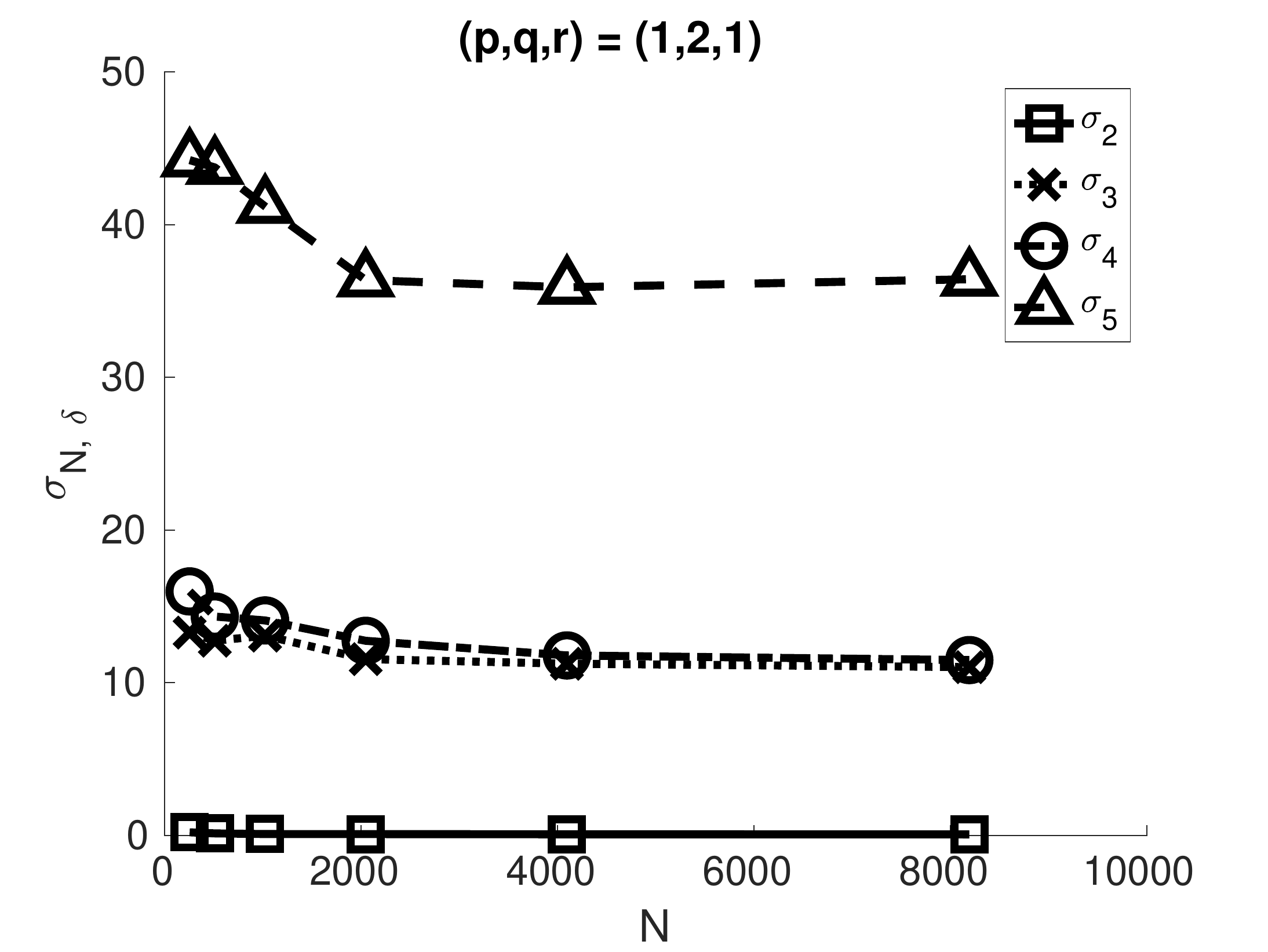}
  \includegraphics[width=.32 \textwidth]{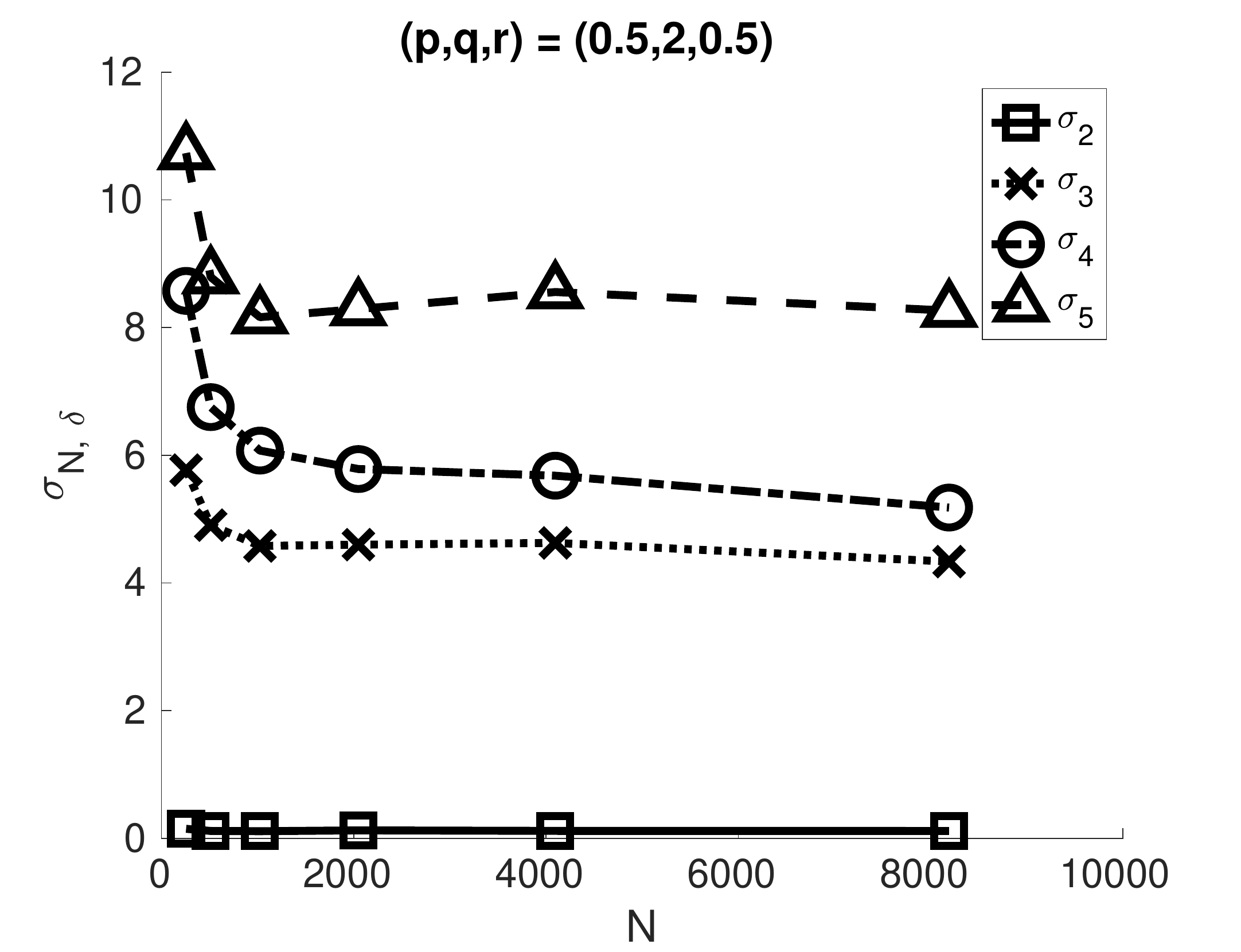}
  \caption{Convergence of the first four non-trivial  discrete eigenvalues $\sigma_{N, \delta}$
    as a function of $N$ for
    different values of $(p,q,r)$ and $\epsilon = 2^{-3}$ with vertices distributed
    according to \eqref{piecewise-constant-mixture}.}
  \label{fig:eval-conv}
\end{figure}

\begin{figure}[htp]
  \centering
  \includegraphics[width=.32 \textwidth]{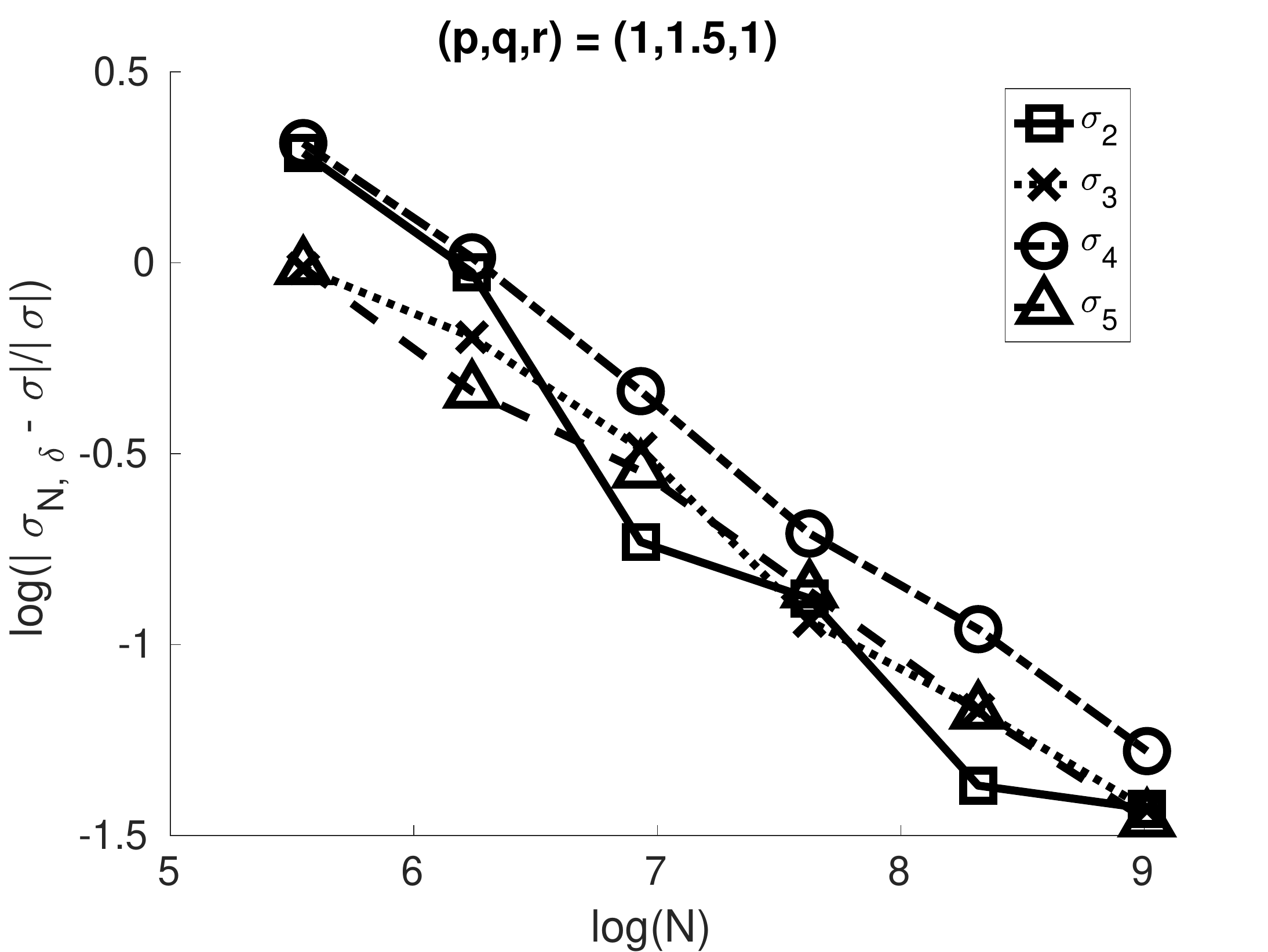}
  \includegraphics[width=.32 \textwidth]{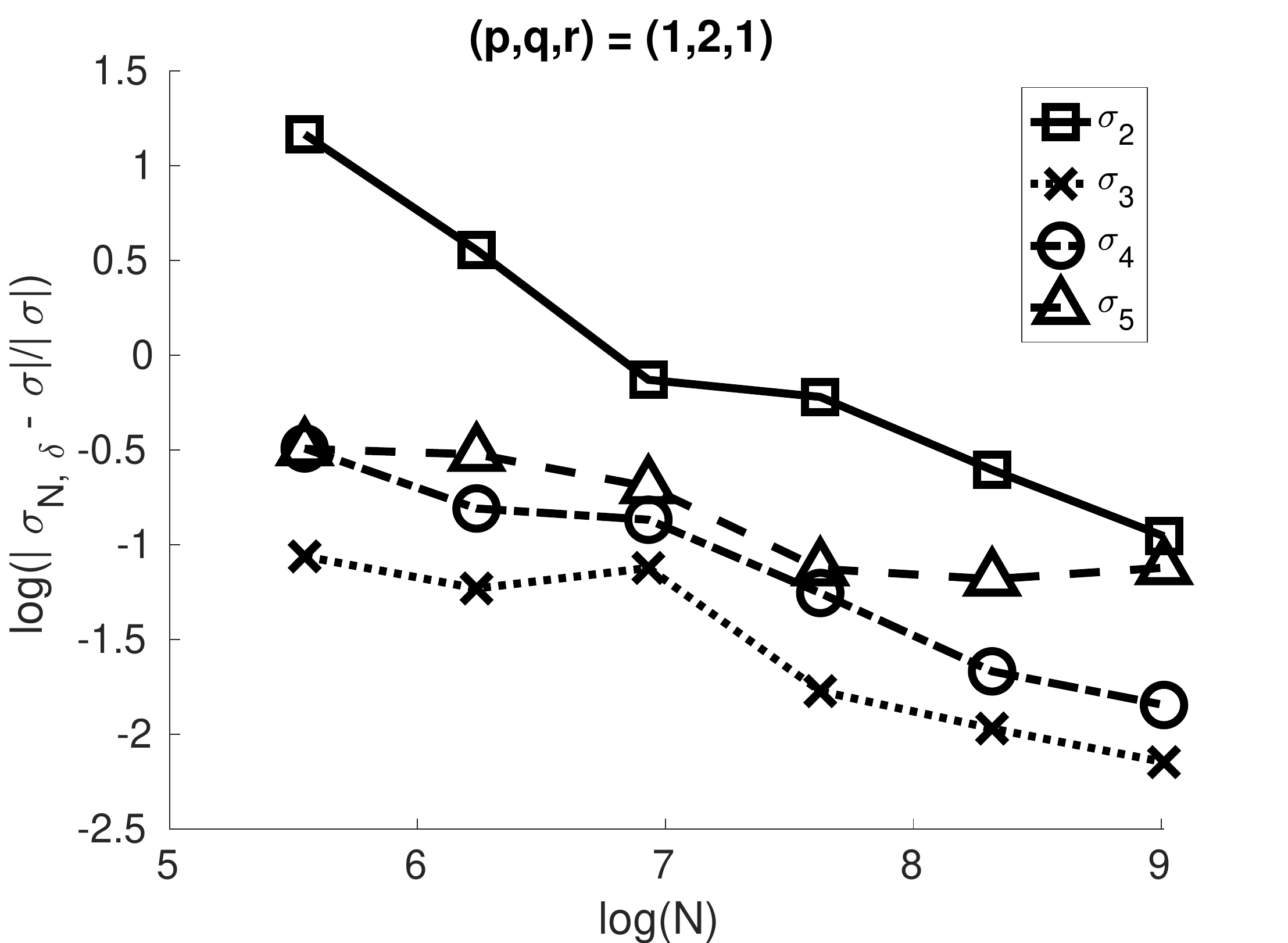}
  \includegraphics[width=.32 \textwidth]{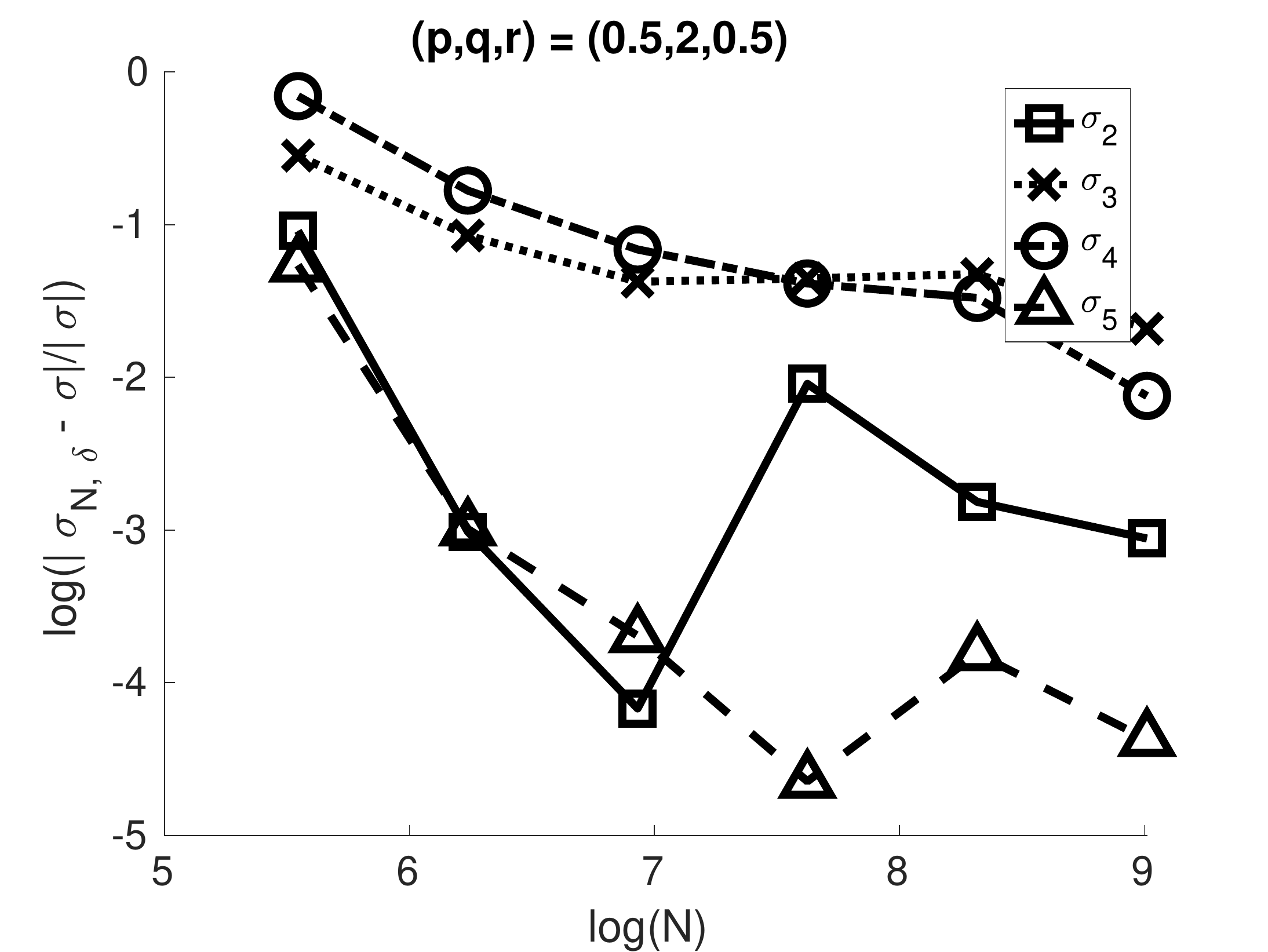}
  \caption{Relative error between the first four non-trivial discrete eigenvalues $\sigma_{N, \delta}$ and
    the continuum eigenvalues $\sigma$ 
    as a function of $N$ for
    different values of $(p,q,r)$ and $\epsilon = 2^{-3}$ with vertices distributed
    according to \eqref{piecewise-constant-mixture}.}
  \label{fig:eval-conv-rel-error}
\end{figure}

\begin{figure}[htp]
  \centering
  \includegraphics[width=.32 \textwidth]{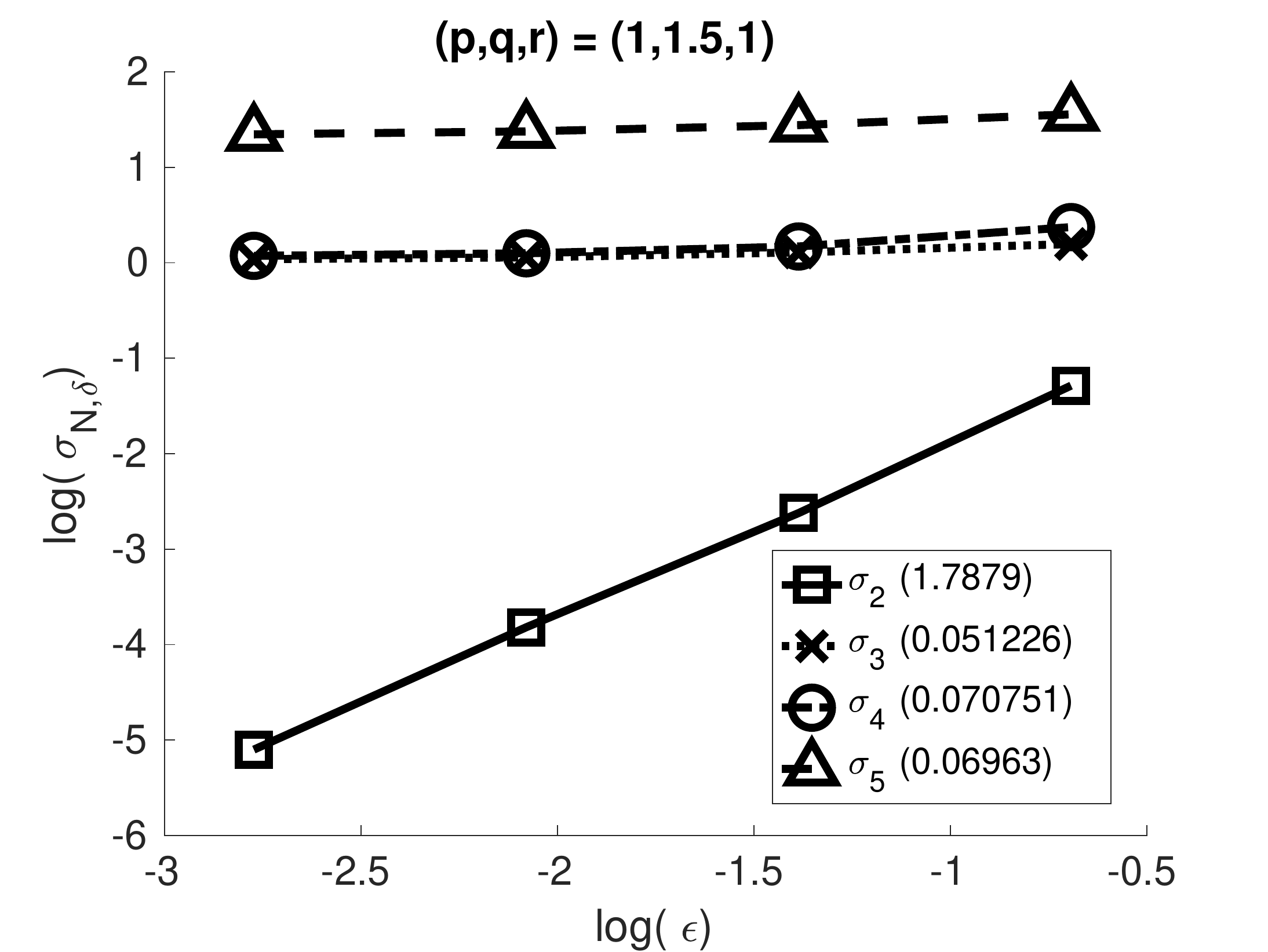}
  \includegraphics[width=.32 \textwidth]{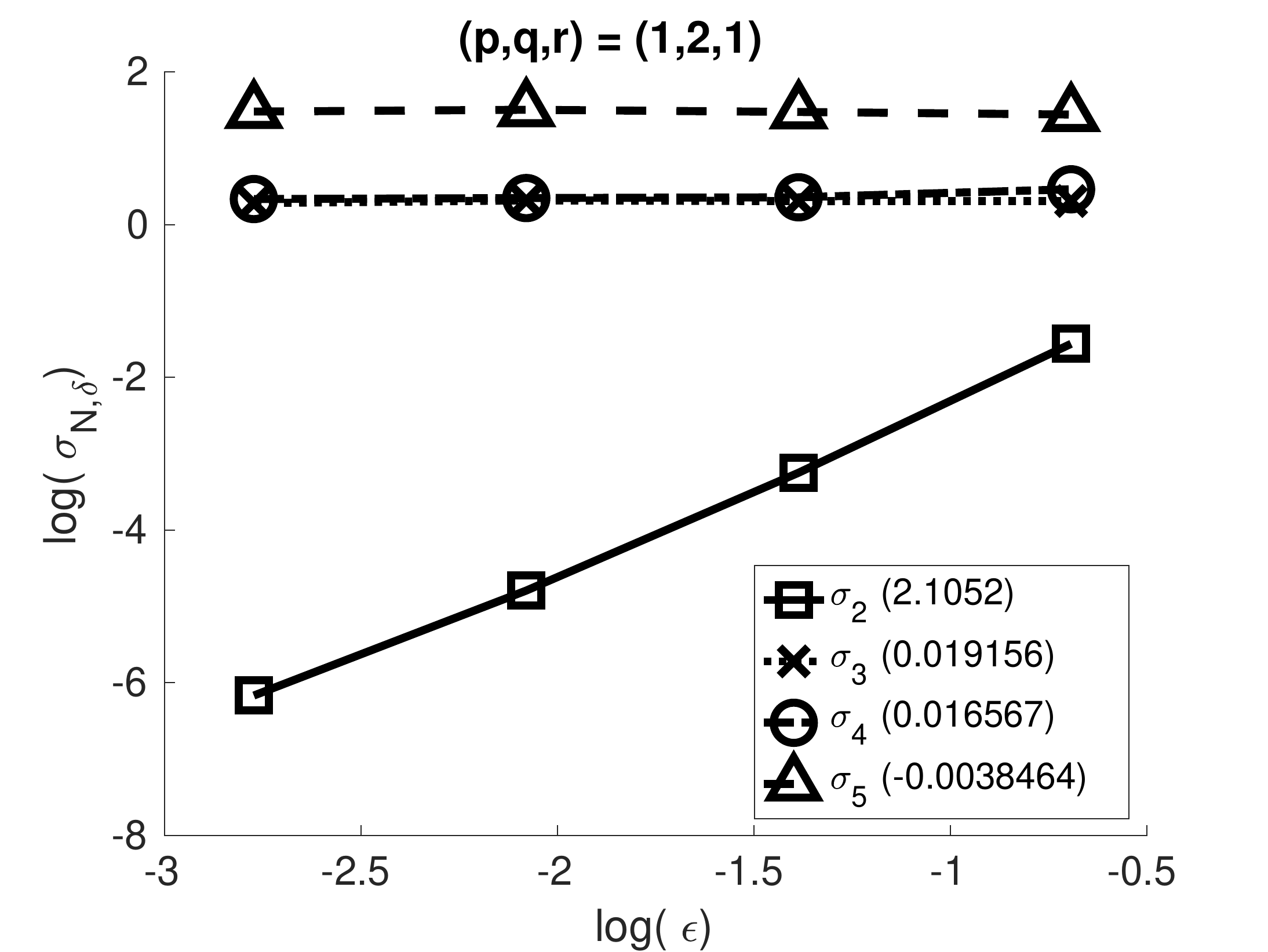}
  \includegraphics[width=.32 \textwidth]{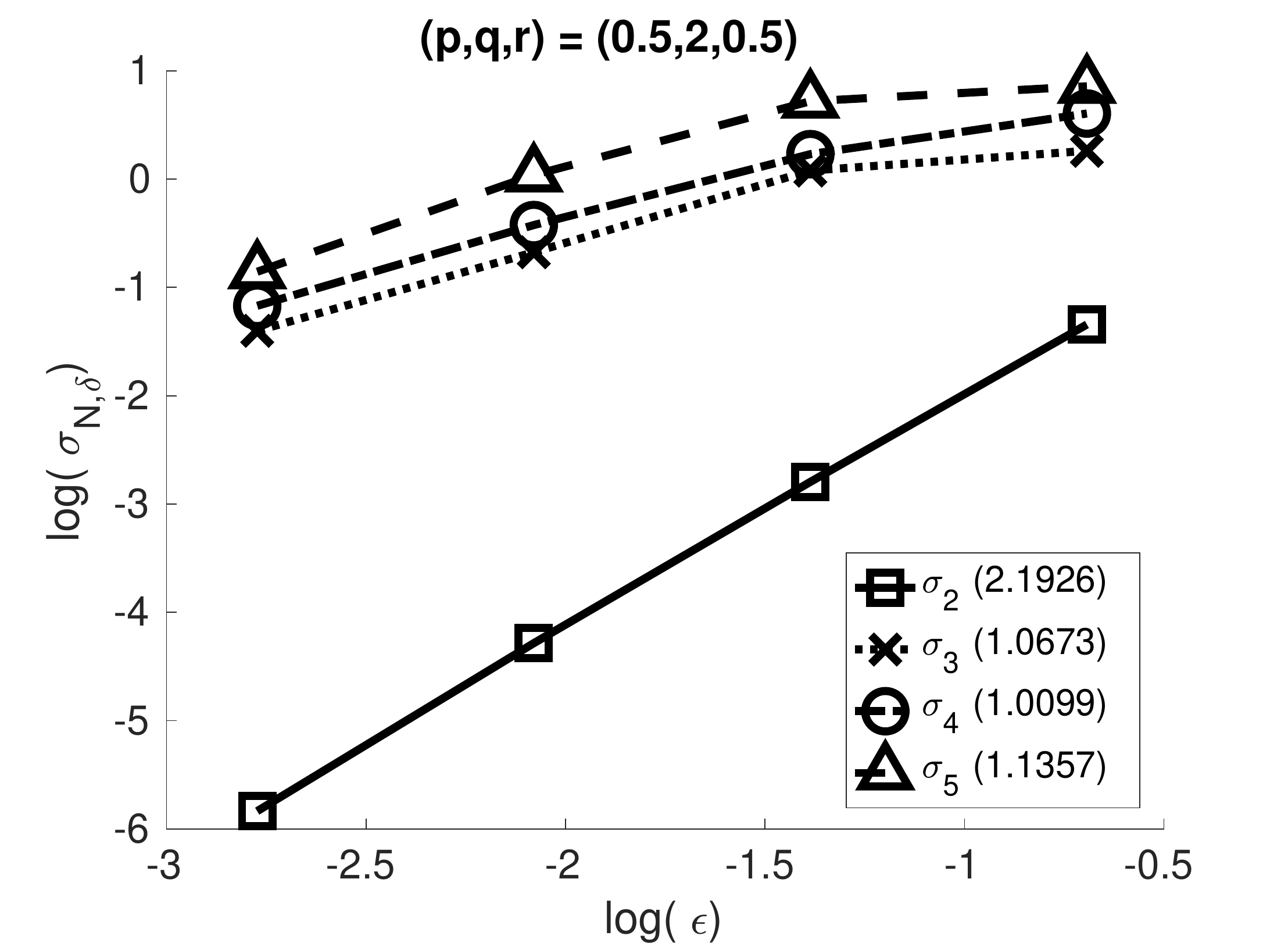}
  \caption{The dependence of the non-trivial discrete eigenvalues $\sigma_{N,\delta}$ as a function
    of $\epsilon$  for
    different values of $(p,q,r)$ and $N = 2^{13}$ with vertices
    drawn from \eqref{piecewise-constant-mixture}. The reported values within the brackets in the legend
    are the slopes of a linear fit to the last three data points indicating the rate at which
  the corresponding eigenvalues vanishes with $\epsilon$.}
  \label{fig:eval-eps-dependence}
\end{figure}

\begin{figure}[htp]
  \centering
  \includegraphics[width=.32 \textwidth]{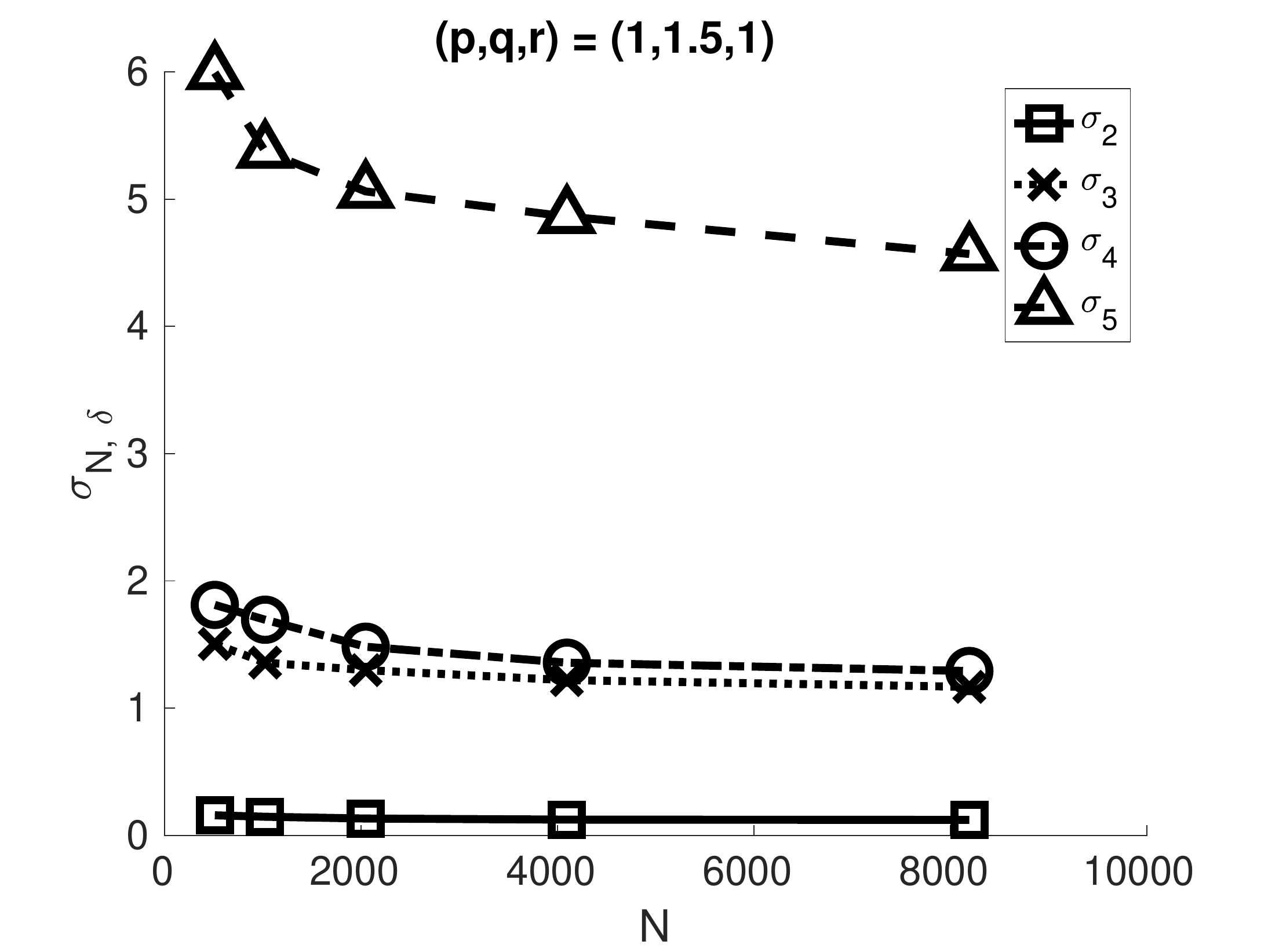}
  \includegraphics[width=.32 \textwidth]{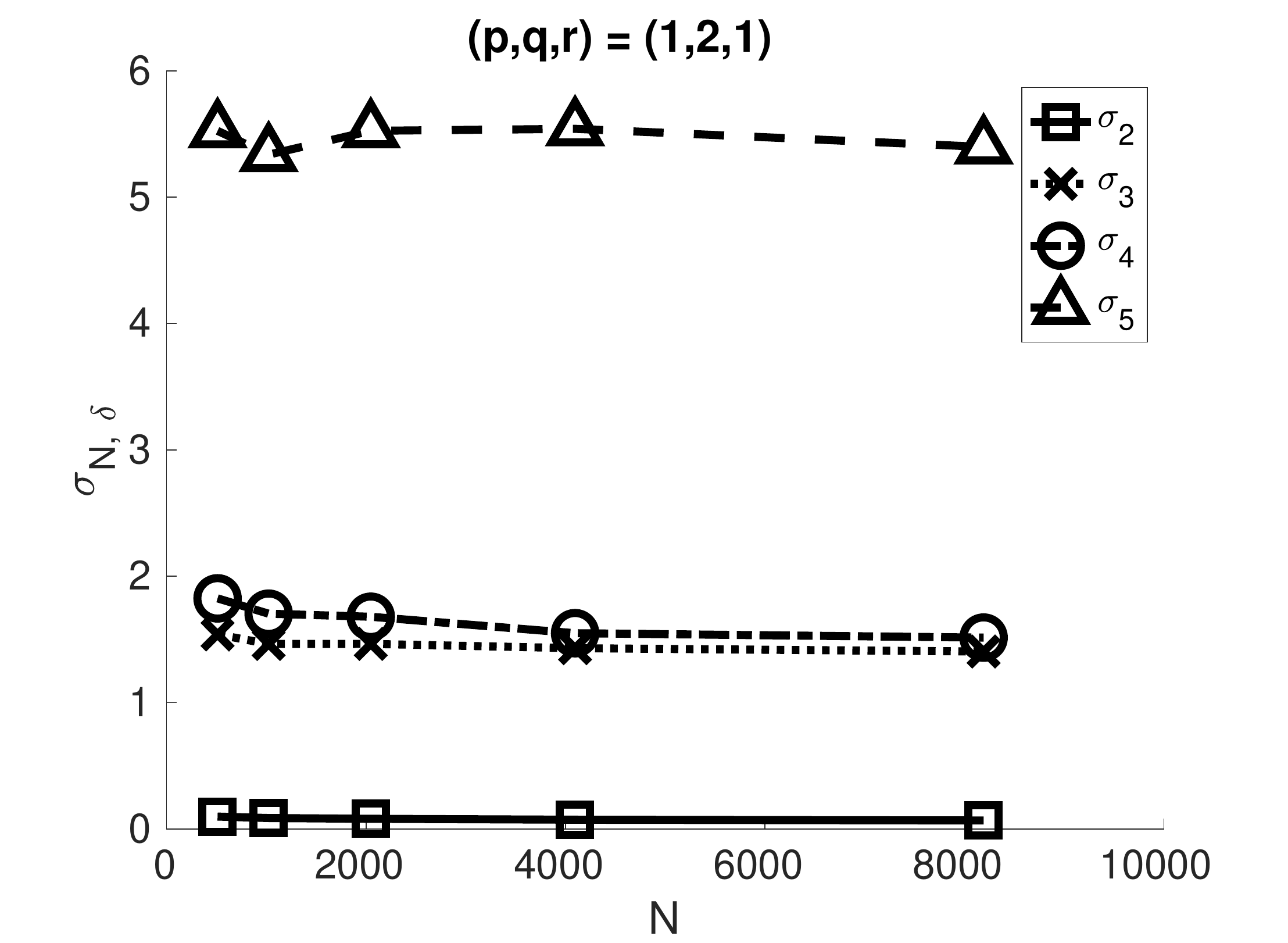}
  \includegraphics[width=.32 \textwidth]{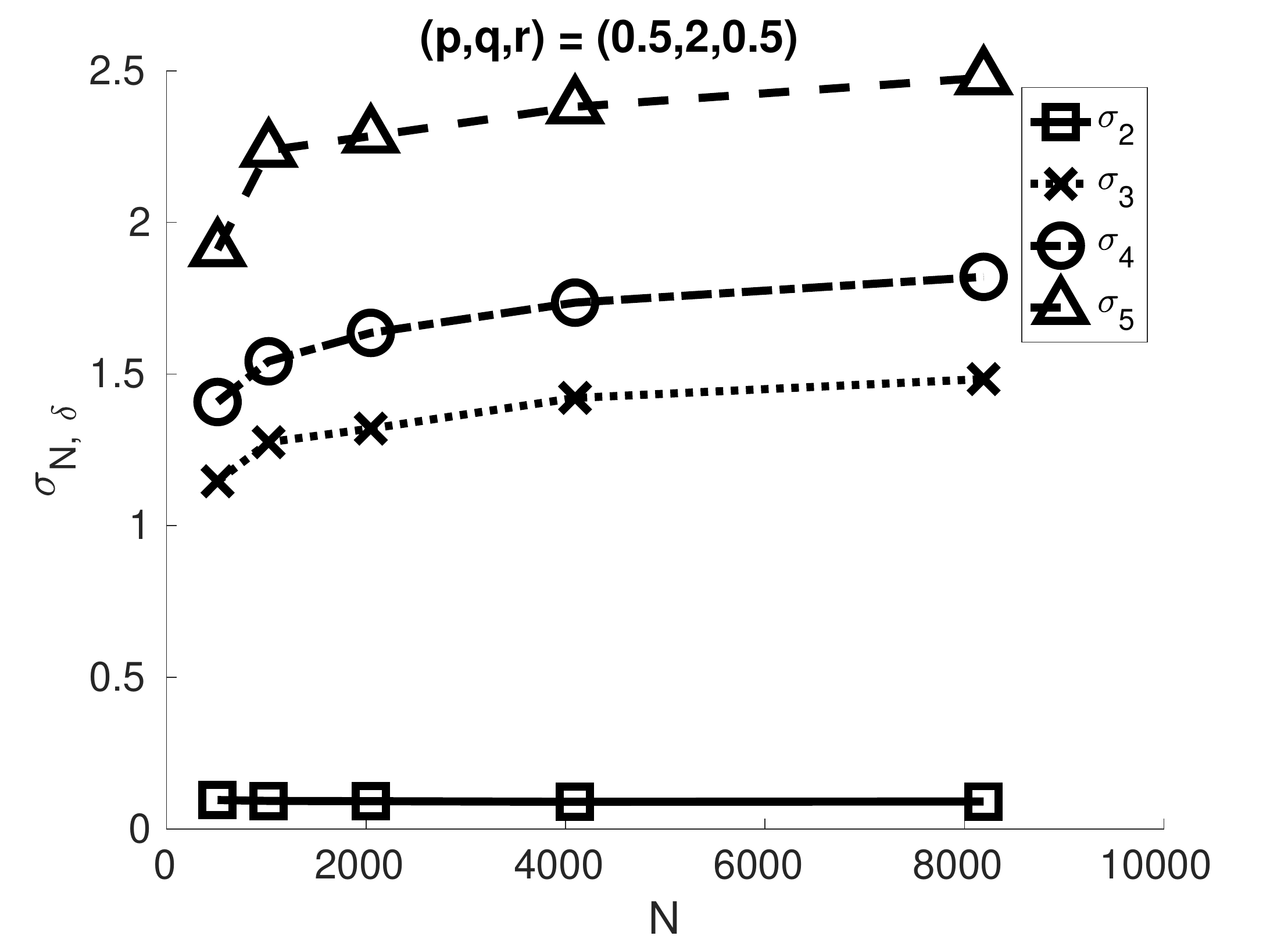}
  \caption{Convergence of the first four scaled discrete eigenvalues $\sigma_{N, \delta}$
    as a function of $N$ for
    different values of $(p,q,r)$ and $\omega = 1.9^{-3}$ with vertices distributed according
    to \eqref{exp-mixture-model}.}
  \label{fig:eval-conv-mixture}
\end{figure}

\end{document}